\documentclass{pspum-l}

\usepackage{amssymb}

\usepackage{graphicx}

\usepackage[cmtip,all]{xy}

\usepackage{amsmath,mathrsfs,amsfonts,amsthm}
\usepackage{mathtools}
\usepackage{stmaryrd}
\usepackage{hyperref}
\usepackage{epic}
\usepackage{eepic}
\usepackage[english]{babel}
\usepackage[all]{xy}
\usepackage{enumerate}
\usepackage[T1]{fontenc}
\usepackage[latin1]{inputenc}
\usepackage{imakeidx}
\makeindex[intoc, title=Notation Index]

\newenvironment{paragr}[1][]{\refstepcounter{subsubsection} \noindent \textbf{\thesubsubsection . \ #1}}{\medskip}

\newcounter{keepeqno}

\newenvironment{num}
{\setcounter{keepeqno}{\value{equation}}%
	\begin{list}{(\theequation)}{\usecounter{equation}}%
		\setcounter{equation}{\value{keepeqno}}}
	{\end{list}}

\usepackage{mathabx}
\usepackage{dsfont}

\newtheorem{theo}{Theorem}[section]
\newtheorem{lem}[theo]{Lemma}
\newtheorem{prop}[theo]{Proposition}
\newtheorem{cor}[theo]{Corollary}
\newtheorem{conj}[theo]{Conjecture}

\theoremstyle{definition}

\newtheorem{exe}[theo]{Example}

\theoremstyle{remark}
\newtheorem{rem}[theo]{Remark}

\numberwithin{equation}{section}

\newcommand{\diag}{\mathrm{diag}}
\newcommand{\Mat}{\mathrm{Mat}}
\newcommand{\ad}{\mathrm{ad}}
\newcommand{\can}{\mathrm{can}}
\newcommand{\Pet}{\mathrm{Pet}}
\newcommand{\Tam}{\mathrm{Tam}}
\newcommand{\Std}{\mathrm{Std}}
\newcommand{\Pin}{\mathrm{Pin}}
\newcommand{\cH}{\mathcal{H}}

\newcommand{\bA}{\mathbb{A}}
\newcommand{\bQ}{\mathbb{Q}}

\newcommand{\bC}{\mathbb{C}}
\newcommand{\bZ}{\mathbb{Z}}
\newcommand{\bR}{\mathbb{R}}
\newcommand{\cZ}{\mathcal{Z}}
\newcommand{\cA}{\mathcal{A}}
\newcommand{\cP}{\mathcal{P}}
\newcommand{\cC}{\mathcal{C}}

\newcommand{\GL}{\mathrm{GL}}
\newcommand{\SL}{\mathrm{SL}}
\newcommand{\Irr}{\mathrm{Irr}}
\newcommand{\St}{\mathrm{St}}
\newcommand{\Norm}{\mathrm{Norm}}
\newcommand{\Hom}{\mathrm{Hom}}
\newcommand{\Tr}{\mathrm{Trace}}

\newcommand{\Sp}{\mathrm{Sp}}

\newcommand{\End}{\mathrm{End}}
\newcommand{\Frob}{\mathrm{Frob}}
\newcommand{\Ext}{\mathrm{Ext}}

\newcommand{\Res}{\mathrm{Res}}

\newcommand{\vol}{\mathrm{vol}}

\newcommand{\Gal}{\mathrm{Gal}}

\newcommand{\temp}{\mathrm{temp}}

\newcommand{\der}{\mathrm{der}}
\newcommand{\elli}{\mathrm{ell}}

\newcommand{\cusp}{\mathrm{cusp}}
\newcommand{\scusp}{\mathrm{scusp}}
\newcommand{\Ad}{\mathrm{Ad}}
\newcommand{\Aut}{\mathrm{Aut}}
\newcommand{\Supp}{\mathrm{Supp}}
\newcommand{\Temp}{\mathrm{Temp}}
\newcommand{\Cent}{\mathrm{Cent}}
\newcommand{\Unit}{\mathrm{Unit}}
\newcommand{\Ker}{\mathrm{Ker}}

\newcommand{\Ima}{\mathrm{Im}}
\newcommand{\Lie}{\mathrm{Lie}}
\newcommand{\Id}{\mathrm{Id}}
\newcommand{\BC}{\mathrm{BC}}
\newcommand{\disc}{\mathrm{disc}}

\newcommand{\unit}{\mathrm{unit}}
\newcommand{\aff}{\mathrm{aff}}

\newcommand{\SO}{\mathrm{SO}}

\newcommand{\PGL}{\mathrm{PGL}}
\newcommand{\Specmax}{\mathrm{Specmax}}

\newcommand{\cS}{\mathcal{S}}
\newcommand{\cB}{\mathcal{B}}

\newcommand{\cF}{\mathcal{F}}

\DeclarePairedDelimiter\ceil{\lceil}{\rceil}
\DeclarePairedDelimiter\floor{\lfloor}{\rfloor}

\begin{document}
	
	\title{Introduction to the relative Langlands program}
	
	
	\author{Rapha\"el Beuzart-Plessis}
	\address{Aix Marseille Univ \\
		CNRS , I2M  \\
		Marseille, France}
	\curraddr{}
	\email{raphael.beuzart-plessis@univ-amu.fr}
	\thanks{The author was funded by the European Union ERC Consolidator Grant, RELANTRA, project number 101044930. Views and opinions expressed are however those of the authors only and do not necessarily reflect those of the European Union or the European Research Council. Neither the European Union nor the granting authority can be held responsible for them.}

	
	\date{}
	
	\begin{abstract}
	The aim of these notes is to give an overview of several aspects of what has come to be called the relative Langlands program, a theme that takes its origin in the study of automorphic periods and their relations to particular cases of Langlands functoriality and special values of (automorphic) $L$-functions. Following the work of Sakellaridis and Sakellaridis-Venkatesh, we emphasize the unifying role played by spherical varieties and harmonic analysis on them. We also review various instances of the phenomena, discovered by Jacquet, Gan-Gross-Prasad, Ichino-Ikeda and many others, that motivated these developments.
	\end{abstract}

	\maketitle

 \setcounter{tocdepth}{2}
\tableofcontents

\section{Introduction}

\begin{paragr}[Global periods.]
	Let $F$ be a global field and $G$ be a connected reductive group over $F$. Denoting by $\bA$ the adele ring of $F$, we set $[G]=G(F)\backslash G(\bA)$ for the automorphic quotient of $G$ and $\cA(G)$ (resp. $\cA_{cusp}(G)$) for the space of automorphic forms (resp. cuspidal automorphic forms) on $[G]$. Let $H$ be a closed algebraic subgroup of $G$ defined over $k$. Then, for $\varphi\in \cA_{cusp}(G)$ we call the integral (provided it converges)
	$$\displaystyle \cP_H(\varphi):=\int_{[H]} \varphi(h) dh$$
	the {\it automorphic period} of $\varphi$ along the subgroup $H$. It has long been observed that automorphic periods often bear some relations to $L$-functions and/or Langlands functoriality. Let us give three examples of this phenomenon:
	\begin{itemize}
		\item In \cite{JacLang}, as a generalization of Hecke's well-known integral representation of the standard $L$-function of modular forms, Jacquet and Langlands show that for every cuspidal automorphic representation $\pi\subset \cA_{cusp}(\GL_2)$ of $\GL_{2,F}$, the automorphic period along the subgroup
		$$\displaystyle H=\mathbb{G}_m\hookrightarrow \GL_2, \; a\mapsto \begin{pmatrix} a & \\ & 1 \end{pmatrix}$$
		essentially represents the standard $L$-function of $\pi$. More precisely, for $\varphi\in \pi$, we have an identity of the form
		$$\displaystyle \cP_H(\varphi\cdot \lvert \det\rvert^s)=\int_{[\mathbb{G}_m]} \varphi\begin{pmatrix} a & \\ & 1 \end{pmatrix} \lvert a\rvert^s da\sim L(s+1/2,\pi)$$
		where $L(s,\pi)$ is the standard $L$-function of $\pi$ and $\sim$ means that equality holds up to more elementary factors (which in particular will depend on $\varphi$).
		
		\item In \cite{Waldtoric}, Waldspurger obtains a striking extension of the previous example to an elliptic maximal torus $T=\Res_{E/F} \mathbb{G}_m/\mathbb{G}_m\hookrightarrow \PGL_{2,F}$ (where $\Res_{E/F}$ stands for Weil's restriction of scalars) associated to a quadratic extension $E/F$. Slightly more precisely, Waldspurger's result gives, for $\varphi\in \pi\subseteq \cA_{cusp}(\PGL_2)$, a relation of the form
		$$\displaystyle \left\lvert \cP_T(\varphi) \right\rvert^2\sim L(\frac{1}{2},\pi_{E})$$
		where this time $L(s,\pi_{E})$ stands for the standard $L$-function of the base-change $\pi_{E}$ of $\pi$ to $\PGL_{2,E}$. Note that this formula doesn't admit any natural deformation with a complex parameter that would yield integral representation of the base-change $L$-function.
		
		\item Finally, in their paper \cite{HLR} on Tate's conjecture for Hilbert-Blumenthal surfaces, Harder, Langlands and Rapoport \cite{HLR} make use of the following characterization of the image of functorial base-change from unitary groups in two variables to $\GL_2$ (which, in the case of Hilbert modular forms, goes back to Asai): for $E/F$ a quadratic extension and $\pi\subseteq \cA_{cusp}(\GL_{2,E})$, $\pi$ is the base-change of some cuspidal representation $\sigma\subseteq \cA_{cusp}(\mathrm{U}(2,E/F))$ on a quasi-split unitary group $\mathrm{U}(2,E/F)$ in two variables if and only if the automorphic period $\cP_{H}$ along the subgroup $H=\GL_{2,F}\subset \Res_{E/F} \GL_{2,E}$ doesn't vanish identically on $\pi$.
	\end{itemize}
\end{paragr}

\begin{paragr}[Local periods.]
	Similar phenomena occur over a local field $k$ provided we replace the period integral $\cP_H$ by the functor $\Hom_H(.,\mathbf{1}_H)$ taking a smooth representation $\pi$ of $G(k)$\footnote{All representations in these notes will have complex coefficients. Moreover, in the case where $k$ is Archimedean, ``smooth'' should be taken to mean {\it admissible smooth Fr\'echet representations of moderate growth} in the sense of Casselman-Wallach (see Section\ref{Sect distinction and Plancherel}).} to the space of $H(k)$-invariant functionals\footnote{and continuous in the Archimedean case} on it. Let $\Irr(G)$ be the set of isomorphism classes of smooth irreducible representations of $G(k)$. We usually say that $\pi\in \Irr(G)$ is {\it $H$-distinguished} if $\Hom_H(\pi,\bC)\neq 0$. In analogy with the global setting, nonzero linear forms $\ell_H\in \Hom_H(\pi,\bC)$ are also sometimes called {\it local periods}. Here are two examples of results relating local distinction to Langlands functoriality or local arithmetic invariants that are  related to the previous ones (where $\ell/k$ stands for a quadratic extension of local fields):
	
	\begin{itemize}
		\item A theorem of Tunnell and Saito (see \cite{Saito}) gives a local counterpart of Waldspurger theorem. Namely, for $T=\Res_{\ell/k}(\mathbb{G}_{m,\ell})/\mathbb{G}_{m,k}\subset \PGL_{2,k}$, an irreducible generic representation $\pi\in \Irr(\PGL_{2,k})$ is $T$-distinguished if and only if $\epsilon(\pi_{\ell})=+1$ where $\epsilon(\pi_{\ell})$ denotes the local root number of the base-change $\pi_\ell$ of $\pi$ to $\PGL_2(\ell)$.
		
		\item Similarly a local analog of the aforementioned result appearing in the Harder-Langlands-Rapoport paper is: an irreducible generic representation $\pi\in \Irr(\GL_{2,\ell})$ is $\GL_{2,k}$-distinguished if and only if it is the local base-change $\sigma_\ell$ of some irreducible representation $\sigma \in \Irr(\mathrm{U}(2,\ell/k))$ (where $\mathrm{U}(2,\ell/k)$ denotes as before a quasi-split unitary group in $2$ variables).
	\end{itemize}
\end{paragr}

\begin{paragr}
	There are actually many more examples in the literature of such relations (known or conjectural) between automorphic periods/local distinctions and Langlands functoriality and/or $L$-functions. As a far from exhaustive list, we may cite for instance the works \cite{JPSS} \cite{Jacfactorization} \cite{JacquetYe} \cite{FriedJac} \cite{JRsymplectic} \cite{OSsymplectic} \cite{BFexteriorsquare} \cite{GP} \cite{GGP} \cite{IIk} \cite{LaMao} \cite{FLO}.
	
	Broadly interpreted, the Relative Langlands Program (RLP) is an attempt to unify those seemingly disparate results by systematically relating them to a dual Galois picture similar to that of the mainstream Langlands program. It arguably appears with the foundational work of Sakellaridis Venkatesh \cite{SV}, that proposes the first unified framework to deal with these questions, but also builds on the pioneering works of many which, through a number of seminal examples, have provided insights that have not all yet been incorporated in such a general framework. However, as we will try to argue in these notes, the RLP should probably be regarded as an enhancement of the usual Langlands Program and not merely as a separate but connected subject.
	
	As a first step, it turns out to be convenient to reformulate everything in terms of the homogeneous variety $X=H\backslash G$ (as a matter of convention, we will always consider right actions of the group $G$ on set/varieties).
\end{paragr}

\begin{paragr}[Local spectrum of a variety.]
	In the relative setting, we want to replace the main players on the automorphic side of the Langlands program, that is local or automorphic (irreducible) representations, by analogous $X$-related notions that might be broadly thought as the local or automorphic spectra of $X$. Let us make this slightly more specific.
	
	Assume first that the field $k$ is local. When $X=H\backslash G$, a first attempt to define the spectrum of $X(k)$ would be as the set of irreducible representations $\pi\in \Irr(G)$ that are $H$-distinguished. This can be more directly related to the variety $X$ as follows. Let $C^\infty(X)$ be the space of smooth functions on $X(k)$. Then, by Frobenius reciprocity, and assuming that the subset $H(k)\backslash G(k)\subset X(k)$ is open\footnote{Which is automatic if $H$ is connected or $\mathrm{char}(k)=0.$}, we have
	$$\displaystyle \Hom_H(\pi,\bC)\simeq \Hom_G(\pi,C^\infty(H(k)\backslash G(k)))\subset \Hom_G(\pi,C^\infty(X)).$$
	By duality this last space is also canonically isomorphic to $\Hom_G(\cS(X),\pi^\vee)$ where $\pi^\vee$ stands for the smooth contragredient of $\pi$ and $\cS(X)$ is the space of smooth compactly supported measures on $X(k)$ (whose smooth contragredient can be identified, via the natural pairing, with $C^\infty(X)$). Note that if $X(k)$ is equipped with a $G(k)$-invariant measure (which can often be achieved, e.g. if $H$ is reductive), then multiplication by this measure identifies $\cS(X)$ $G$-equivariantly with the more classical space $C_c^\infty(X)$ of smooth and compactly supported functions. In any case, when working with the variety $X$ instead of the subgroup $H$, it turns out to be more convenient to replace $\pi$ by its contragredient and we say that $\pi\in \Irr(G)$ is {\it $X$-distinguished} if $\Hom_G(\cS(X),\pi)\neq 0$. Thus, the set of all $X$-distinguished irreducible representations is one way through which we can make sense of the ``spectrum of $X(k)$''.
	
	There is yet at least one other way to make sense of this spectrum. It comes directly from the classical theory of unitary representations and is more tied to the analytic part of the theory of automorphic forms. More precisely, let $L^2(X)$ be the Hilbert space of square-integrable {\it half-densities} on $X(k)$ (once again, when $X(k)$ carries a $G(k)$-invariant measure $dx$, it can be identified with the more classical space $L^2(X(k),dx)$). Then, $L^2(X)$ carries a natural unitary representation of $G(k)$ by right translation and we might define the {\it $L^2$-spectrum} of $X(k)$ as the set of irreducible unitary representations weakly contained in $L^2(X)$ i.e.\ the support of its Plancherel decomposition (see Section \ref{Sect distinction and Plancherel} for a quick reminder on Plancherel decomposition). As will be made more precise in Section \ref{Sect distinction and Plancherel}, this $L^2$ point of view is very important as it also naturally produces a family of functionals (so-called {\it relative characters}) which will play a prominent role in the formulation of general global conjectures.
\end{paragr}

\begin{paragr}[Theta series and automorphic spectrum.]
	Let us now return to the global setting. Thus, $F$ is a global field, $G$ is a connected reductive group defined over $F$, $H\subset G$ a closed subgroup and $X=H\backslash G$ the corresponding homogeneous variety. In order to make sense of the automorphic spectrum of $X$ we will, essentially as in the local setting, dualize and replace period integrals by the {\it theta series map}
	$$\displaystyle \Theta_X: C_c^\infty(X(\bA))\to C^\infty([G])$$
	$$\displaystyle f\mapsto \Theta_f: g\in [G]\mapsto \sum_{x\in X(F)} f(xg).$$
	As a side remark, we actually need to impose some, mild, condition on $X$ in order to ensure the convergence of theta series (e.g. already the case of the flag variety $B\subset G$ is problematic) e.g. assuming that $X$ is {\it quasi-affine} is enough. (This entails in particular that $X(F)$ is discrete in $X(\bA)$ hence the sum defining $\Theta_f$ is locally finite.)
	
	The theta series relate to $H$-automorphic periods in the following way: if $f\in C_c^\infty(X(\bA))$ is supported in the open $G(\bA)$-orbit $H(\bA)\backslash G(\bA)\subseteq X(\bA)$, then for every $\varphi\in \cA_{\cusp}(G)$ we have
	$$\displaystyle \langle \Theta_f, \varphi\rangle_{[G]}=\int_{H(\bA)\backslash G(\bA)} f(x) \cP_H(R(x)\overline{\varphi}) dx$$
	where $\langle .,.\rangle_{[G]}$ stands for the $L^2$-scalar product with respect to some invariant measure on $[G]$ and $R(x)$ is the operator of right translation by $x$. However, adopting the theta series point of view has the notable advantage of allowing to consider non-homogeneous varieties as well. More precisely, if $X$ is a $G$-variety containing an open $G$-orbit $X^{\bullet}\subset X$ and we define the $X$-theta series by summing over rational points of $X^{\bullet}$\footnote{We can also define $\Theta_f$ by summing over all the rational points of $X$ but then the pairing integral $\langle \Theta_f, \varphi\rangle_{[G]}$ might diverge and has to be interpreted in a standard regularized way.}, the previous setting now includes examples such as Tate's Zeta integral
	$$\displaystyle \chi\mapsto Z(f,\chi):=\int_{[\mathbb{G}_m]} \Theta_f(x) \chi(x)dx=\int_{\bA^\times} f(x) \chi(x)dx$$
	for $f\in \mathcal{S}(\bA)$ and where $\chi$ runs over the set of all automorphic characters $[\mathbb{G}_m]\to \bC^\times$ (or at least a proper open subcone where the integral converges absolutely). Similarly theta series allows to include Godement-Jacquet theory in the framework.
	
	Similarly to the local story, a cuspidal automorphic representation of $G(\bA)$ is said to be {\it $X$-distinguished} if for some $f\in C_c^\infty(X(\bA))$, the orthogonal projection $\Theta_{f,\pi}$ to $\pi$ is nonzero. More generally, we would like to define the {\it automorphic spectrum} of $X$ as the ``support'' of the $G(\bA)$-representation $C^\infty_X([G]):=\Theta_X(C_c^\infty(X(\bA)))$. The problem is that in general $C^\infty_X([G])$ is not a subspace of $L^2([G])$ and therefore cannot be naturally completed to a unitary representation. One way to make sense of this is to try to regularize ``in a natural way'' the Petersson inner product $\langle \Theta_{f_1},\Theta_{f_2}\rangle_{Pet}$ for $f_1,f_2\in C_c^\infty(X(\bA))$. This however is far from obvious and is one way in which Jacquet's Relative Trace Formulas can be formulated. We refer to Pierre-Henri Chaudouard's article in this proceeding for a survey on relative trace formulas.
\end{paragr}

\begin{paragr}[What these notes are about.]
	The aim of this paper is to survey various aspects of the Relative Langlands Program (RLP) mainly following the seminal work of Sakellaridis and Venkatesh \cite{SV}, where the first unified framework to deal with these questions was proposed. In particular, we will follow {\em op.\ cit.}\ by putting ourself from the beginning in the setting of {\em spherical varieties}, that is $G$-varieties $X$ admitting an open orbit for any Borel subgroup $B\subset G$. Indeed, these varieties have particularly nice finiteness properties (which have been noticed long ago by Luna, Bust, Brion, Knop among others) that make them good candidates for a development of the RLP, e.g. they tend to have certain {\em multiplicity one} properties that force the corresponding global period integrals to factorize.  As the history of the subject that mainly evolved through the study of particular examples, we have tried to keep a balance between the exposition of a general theory and special instances of the phenomena that we want to discuss. In particular, we have included a discussion of the local and global Gross-Prasas conjectures \cite{GP} and its refinement by Ichino-Ikeda \cite{IIk}.
\end{paragr}

\begin{paragr}[Outline.]
	The paper is divided in three sections. The first section is purely algebraic and contains preliminary results on the structure of spherical varieties. We recall some important invariants that can be attached to those and that will be needed in the subsequent sections. A particularly important construction is that of the {\em dual group} $\widehat{G}_X$ of a $G$-spherical variety $X$. A first version of this dual group was obtained in the work of Gaitsgory-Nadler \cite{GN} via some Tannakian formalism, but here we will construct it, following Sakellaridis-Venkatesh, in a combinatorial way by dualizing some based root datum attached to $X$. The existence of this naturally occuring root datum was observed by Brion \cite{Bri90} and later reproved by Knop \cite{KnopAB}. Moreover, as shown by Knop-Schalke \cite{Knopdualgp}, this dual group comes with a distinguished morphism $\widehat{G}_X\times \SL(2)\to \widehat{G}$ to the dual group of $G$. This morphism is at the heart of all the general conjectures stated in the later sections. The last Subsection \ref{sect boundary degenerations} concerns the construction of so-called {\em boundary degenerations}. These are auxilliary $G$-spherical varieties, that roughly play the role of parabolic subgroups in the relative setting and that are crucial ingradients in the local theory of smooth and $L^2$-asymptotics developed by Sakellaridis-Venkatesh. We introduce these boundary degenerations via the {\em affine degeneration} of $X$ (assuming that it is quasi-affine). Starting from the affine degeneration we also recall the theory of toroidal compactifications of $X$, that is due to Luna and Vust \cite{KnopLV}.
	
	The second section focus on the local aspects of the RLP. More precisely, after a general discussion of local distinction and its relation to harmonic analysis on $X$, I state the main local conjecture, due to Sakellaridis-Venkatesh, on the Plancherel decomposition of $L^2(X)$ (Subsection \ref{sect SV local conjecture}). This is followed by a general discussion of two important special properties of homogeneous varieties, temperedness and strong temperedness, giving a priori restrictions on the support of the Plancherel decomposition. After that, I review the local Gross-Prasad conjecture for orthogonal groups and revisit a conjecture of Prasad on Galois symmetric varieties (Subsections \ref{Sect GGP} and \ref{Sect Prasad conj}). These are two special instances where the local conjecture of Sakellaridis-Venkatesh can be made much more precise. Subsections \ref{Section center} to \ref{Sect Bernstein maps} are then devoted to present the local theory of asymptotics, following Sakellaridis-Venkatesh, over a local non-Archimedean field. This starts with a general discussion on the {\em center} of a spherical variety $X$ and its action on various space of functions on $X$. Indeed, as emphasized there the understanding of this action, and its finiteness properties (e.\ g.\ over the Bernstein center), is crucial for the full development of a theory of asymptotics. Our treatment of this question, leads in particular to a natural conjecture on the (non-)distinction of supercuspidal representations (Conjecture \ref{conj supercuspidal}) which to my knowledge has not been formulated before. Then Subsections \ref{Sect smooth asym} and \ref{Sect Bernstein maps} contain a review of the theory of smooth and $L^2$-asymptotics (also called Bernstein maps) respectively. In \cite{SV} these have been developed under the technical condition that $X$ is {\em wavefront}, see \S \ref{S cone Ax-} for a definition, and I explain here how this assumption can be eliminated provided a certain finiteness condition on the action of the center is satisfied (in particular this would follow from Conjecture \ref{conj supercuspidal}). Subsection \ref{Sect most cont spectrum} contains a discussion of `the most continuous part of the spectrum' and, in particular, of the part of the Plancherel decomposition that should correspond, under the local Sakellaridis-Venkatesh conjecture, to unramified parameters $W_k/I_k\to \widehat{G}_X$. The final Subsection \ref{Sect local trace formula} is devoted to a particularly powerful tool to attack the local conjecture (as well as its variants and refinements), which is the local trace formula. It is illustrated in two particular examples: the local trace formula of Arthur and a local trace formula that was develop by Waldspurger in his work on the local Gross-Prasad conjecture.
	
	The last section deals with the global aspects and aim to formulate, as precisely as possible, the general conjecture of Sakellaridis-Venkatesh on factorization of global periods. This is motivated by first reviewing the global Gross-Prasad conjecture and its refinement by Ichino-Ikeda (Subsection \ref{S global GGP}). Then, in Subsection \ref{S normalized rc}, I explain how to (conjecturally) normalize the local relative characters, which are the main ingredient entering into the local Plancherel formulas, and how in the unramified case they should relate (under the assumption that $X$ is affine) to certain ratio of local $L$-values. Once this is in place, the global conjecture can be formulated in Subsection \ref{Sect global conjecture}. 
	
	We emphasize that both the local and global conjectures stated in these notes (both being due to Sakellaridis-Venkatesh \cite{SV}) are not as precise as one would hope for, e.\ g.\ in the global conjecture this appears through the presence of certain unspecified rational constants. We encourage the reader to look at the recent work of Ben Zvi-Sakellaridis-Venkatesh \cite{BZSV} for a new, and particularly exciting, perspective on these matters.
\end{paragr}

\begin{paragr}[Acknowledgements.]
	I thank warmly the referee for going over a (far too) preliminary version of this text and for his many comments and suggestions that allowed to greatly improve its readability.
\end{paragr}

\begin{paragr}[Notation.]
	Here are some of the most common notation used in the paper (an index of notations is also included at the end):
	\begin{itemize}
		\item $k$ is a field of characteristic zero. In Section \ref{sect structure} (except for the last two Subsections \ref{S L-groups} and \ref{S Galois pairs Lgroup} discussing Galois actions on the dual groups), we assume that $k$ is algebraically closed. In Part \ref{Part local aspects}, $k$ will be a local field and we will sometimes even assume that it is non-Archimedean. Finally, in Part \ref{Part global}, $F$ will be a number field. Whenever convenient, we fix an algebraic closure $\overline{k}$ and we write $\Gamma=\Gal(\overline{k}/k)$\index{$\Gamma$} for the corresponding absolute Galois group. Unless specified otherwise, all varieties and linear algebraic groups are implicitely supposed to be defined over $k$.

		\item If $L$ is a linear algebraic group over $k$, we denote by $X^*(L)=\Hom_{k}(L,\mathbb{G}_m)$\index{$X^*(L)$} its group of algebraic characters {\em defined over $k$}. The group of algebraic characters over the algebraic closure will be denoted by $X^*(L)_{\overline{k}}=\Hom_{\overline{k}}(L_{\overline{k}},\mathbb{G}_m)$\index{$X^*(L)_{\overline{k}}$}. Similarly, for a torus $S$ defined over $k$, we denote by $X_*(S)=\Hom_k(\mathbb{G}_m,S)$\index{$X_*(S)$} and $X_*(S)_{overline{k}}=\Hom_{\overline{k}}(\mathbb{G}_m,S)$\index{$X_*(S)_{overline{k}}$} its lattices of cocharacters over $k$ and $\overline{k}$ respectively.
		
		\item Most of the time $G$ will be a connected reductive group over $k$ (or over $F$). We will always denote by $B$ a Borel subgroup (possibly only defined over $\overline{k}$) and $N$ the unipotent radical of $B$. We will also write $A$ for the {\it canonical Cartan} of $G$, that is the quotient $B/N$. When convenient, we will also identify $A$ with a maximal torus in $B$. In Subsections \ref{sect structure}, \ref{Sect smooth asym} and \ref{Sect Bernstein maps} we will assume that $G$ is split. Although the results contained in these sections should extend to nonsplit groups this would require working with appropriate notion of {\em boundary degenerations} over the field $k$.
		
		\item We set $\cA=X_*(A)\otimes \bQ$\index{$\cA$}, $\cA^*=X^*(A)_{\overline{k}}\otimes \bQ$\index{$\cA^*$}. Similar notations will be applied to other (split) tori (e.g. the canonical Cartan of a spherical variety defined in \S \ref{S canonical Cartan}).
		
		\item We write $\Delta\subset X^*(A)_{\overline{k}}$\index{$\Delta$} for the set of simple roots of $A$ (with respect to $B$), $\Delta^\vee\subset X_*(A)_{\overline{k}}$\index{$\Delta^\vee$} for the set of simple coroots and by $\alpha\mapsto \alpha^\vee$ the canonical bijection $\Delta\simeq \Delta^\vee$. We will also denote by $X^*(A)_{\overline{k}}^+\subset X^*(A)_{\overline{k}}$\index{$X^*(A)_{\overline{k}}^+$} the cone of dominant weights.
		
		\item Lie algebras will be denoted by corresponding gothic letters, such as $\mathfrak{g}$\index{$\mathfrak{g}$} for $G$. We will write $\mathfrak{g}^*$\index{$\mathfrak{g}^*$} for the dual. When working over an Archimedean field, we will denote by $\mathcal{U}(\mathfrak{g})$\index{$\mathcal{U}(\mathfrak{g})$} the enveloping algebra of the complexification of $\mathfrak{g}$.
		
		\item The neutral component of a linear algebraic group $H$ will be denoted by $H^0$\index{$H^0$}.
		
		\item For any subset $S\subset G$, we write $\Cent_G(S)$\index{$\Cent_G(S)$} and $\Norm_G(S)$\index{$\Norm_G(S)$} for the centralizer and the normalizer of $S$ in $G$ respectively.
		
		\item If a group $G$ acts on a set $X$ on the right (resp. on the left), we will denote by $R(g)$\index{$R(g)$} (resp. $L(g)$\index{$L(g)$}), $g\in G$, the operators of right translations (resp. left translations) on functions on $X$. When $G$ is a Lie group and the action $C^\infty$, similar notations will be applied to the differential of the action.
		
		\item If $G$ is a reductive group acting on a quasi-affine variety $X$, we will denote by $X\sslash G=\mathrm{Spec}\; k[X]^G$\index{$X\sslash G$} the GIT (Geometric Invariant Theoretic) quotient. This notation will also be sometimes applied to nonreductive actions but only in the following special situation: if $X=P_x\backslash P$ is a homogeneous variety under the action of a linear algebraic group $P$ and $N\subset L$ is the unipotent radical of $L$, then $X\sslash N:=L_x\backslash L$ where $L=P/N$ is the Levi quotient of $P$ and $L_x$ the image of $P_x$ in $L$.
		
		\item If $L$ is a linear algebraic group, $H\subset L$ a closed subgroup and $X$ is a $H$-variety (with $H$ acting on the right), we denote by $X\times^H L$\index{$X\times^H L$} the {\em contracted product}, that is the quotient of $X\times L$ by the free $H$-action $(x,\ell)\cdot h:=(xh,h^{-1}\ell)$.
		
		\item $X$ will most of the time be a spherical variety, usually assumed to be homogeneous quasi-affine and unimodular, see \S \ref{S def spherical var}.
	\end{itemize}
	
\end{paragr}

\section{Spherical varieties}\label{Part spherical varieties}

\subsection{Definition and examples}\label{S def spherical var}

In this subsection, we review the definition and basic properties of spherical varieties. We also discuss various examples.

\vspace{2mm}

\begin{paragr}\label{S def spherical}
A $G$-variety $X$ over $k$ is said to be {\it spherical} if it is normal and any Borel subgroup $B\subset G_{\overline{k}}$ (a priori only defined over the algebraic closure $\overline{k}$) acts with a (necessarily unique) open orbit $X_B\subset X_{\overline{k}}$ or, equivalently, if there is an open $G$-orbit in $X\times \mathcal{B}$ for the diagonal action, where $\mathcal{B}$ stands for the flag variety of $G$. In particular a spherical $G$-variety $X$ admits an open $G$-orbit $X^{\bullet}\subset X$ and in most parts of this paper we will only consider the homogeneous case $X=X^{\bullet}$\index{$X^{\bullet}$}. One reason is that most of the objects to be considered will only depend on the $G$-birational equivalence class of $X$. This is the case e.g. for all the combinatorial invariants of $X$ introduced in the next subsections (including its dual group) or of the unitary representation $L^2(X(k))$ in the local setting. One notable exception is however the discussion of the global conjecture in Section \ref{Sect global conjecture}.
	
Thus, unless explicitly specified otherwise, we will always consider a homogeneous spherical $G$-variety $X=H\backslash G$. For the sake of simplicity, we will also always assume that it satisfies the following condtion:
	\begin{itemize}
		\item $X$ is quasi-affine and unimodular (i.e.\ it admits a nonzero $G$-invariant volume form).
	\end{itemize}
	This is actually not a very restrictive hypothesis because of the following: every homogeneous $G$-variety admits an equivariant $\mathbb{G}_m$-bundle $\widetilde{X}\to X$ that is unimodular\footnote{Indeed, it suffices to take $\widetilde{X}=L^\times$ where $L\to X$ is the canonical line bundle on $X$, the one corresponding to volume forms, and $L^\times$ denotes the complement of the zero section.} and moreover $X$ is $G$-spherical if and only if $\widetilde{X}$ is $\mathbb{G}_m\times G$-spherical. Similarly, by a theorem of Borel \cite[Theorem 5.1]{BorelLAG}, every homogeneous $G$-variety admits an equivariant $\mathbb{G}_m$-cover that is quasi-affine and is spherical (resp. unimodular), for the $\mathbb{G}_m\times G$-action, if the original variety $X$ is. Thus, up to replacing $X$ by a suitable $\mathbb{G}_m^2$-cover (and enlarging $G$ to include this $\mathbb{G}_m^2$-action), we can always meet the above criterion. Actually, it suffices to take one $\mathbb{G}_m$-cover since a unimodular spherical variety is automatically quasi-affine, see Lemma \ref{lem1 appendix} below.
	
	Furthermore, our assumption that $X$ is unimodular is however mainly for convenience because, when say $k$ is local, it gives the existence of a $G(k)$-invariant measure on $X(k)$ and therefore, upon fixing such a measure, allows to identify (although not canonically) measures and half-densities with functions.
\end{paragr}

\begin{paragr}\label{S decomp algebra of regular fns}
	It turns out that spherical $G$-varieties have favorable finiteness properties that make them good candidates for the relative Langlands program. More precisely, if $X$ is a normal quasi-affine $G$-variety then $X$ is spherical if and only if the locally algebraic representation of $G$ on the ring of regular functions $k[X]$\index{$k[X]$} is multiplicity free i.e.\ we have a $G$-equivariant decomposition
	$$\displaystyle \overline{k}[X]=\bigoplus_{\lambda\in X^*(A)_{\overline{k}}^+} \overline{k}[X]_\lambda$$
	where each $\overline{k}[X]_\lambda$\index{$\overline{k}[X]_\lambda$} is either $0$ or irreducible with highest weight $\lambda$.
	
	In a similar vein, when $X$ is a spherical $G$-variety and the field $k$ is local we expect (and this is known in a lot of cases see e.g. \cite[Theorem 5.1.5]{SV}, or Corollary \ref{cor multiplicities} below, for the case where $k$ is $p$-adic and \cite{KobOsh} in the Archimedean case) that for every smooth irreducible representation $\pi$ of $G(k)$ the space $\Hom_G(\pi,C^\infty(X(k)))$ of $G(k)$-equivariant embeddings $\pi\to C^\infty(X(k))$ is finite dimensional. 
\end{paragr}


\begin{paragr}[Examples.]
	Here are some examples of homogeneous spherical varieties:
	\begin{itemize}
		\item {\it Symmetric varieties}, i.e.\ quotients of the form $X=H\backslash G$ where $H$ is of finite index in the subgroup $G^\iota$ of fixed points for some involution $\iota$ of $G$, are spherical. This is already a source of many examples such as $\mathrm{O}_n\backslash \GL_n$, $\Sp_{2n}\backslash \GL_{2n}$ or $\GL_n\times \GL_n\backslash \GL_{2n}$.
		
		\item For $H$ a connected reductive group over $k$, we can take $X=H$ equipped with the natural action of $G=H\times H$ on it by left and right translations. It is a symmetric variety, the underlying involution being given by $\iota(h_1,h_2)=(h_2,h_1)$, $(h_1,h_2)\in G$. We shall also refer to this example as the {\it group case} or to $X$ as the {\it group variety}. One rough general principle is that when specializing conjectures from the relative Langlands program to the group case we should essentially recover the original conjectures of the Langlands program.
		
		\item Another interesting class of symmetric varieties are the so-called {\it Galois symmetric varieties}. More precisely, those are associated to a connected reductive group $H$ over $k$ and a (separable) quadratic extension $\ell/k$. The group $G$ is then the Weil restriction of scalar $G=\Res_{\ell/k} H_\ell$ of the base-change of $H$ to $\ell$ and the embedding $H\subset G$ the obvious one. This corresponds to the symmetric space associated to the natural Galois involution on $G$. Furthermore, this example is a form of the group variety, since over an algebraic closure $\overline{k}$ there exists an isomorphism $X_{\overline{k}}\simeq H_{\overline{k}}$ of $G_{\overline{k}}\simeq H_{\overline{k}}\times H_{\overline{k}}$-varieties. 
		
		\item The {\it Rankin-Selberg variety} $\GL(n)^{\diag}\backslash \GL(n)\times \GL(n+1)$ and the Gan-Gross-Prasad varieties $\SO(n)^{\diag}\backslash \SO(n)\times \SO(n+1)$, $\mathrm{U}(n)^{\diag}\backslash \mathrm{U}(n)\times \mathrm{U}(n+1)$, where the superscript ``$\diag$'' indicates that we take the diagonal embedding, are examples of spherical varieties that are not symmetric.
		
		
		\item {\it Horospherical varieties} are quotients $X=S\backslash G$ where $S$ is a subgroup containing the unipotent radical $N$ of a Borel subgroup (perhaps only defined over the algebraic closure $\overline{k}$).
		
		\item More exotic examples of spherical varieties include $\mathrm{G}_2\backslash \SO(7)$ or $\mathrm{Spin}_9\backslash \mathrm{F}_4$.
		
		\item {\it Parabolic induction}: this is a general procedure to produce new spherical varieties from old ones. More precisely, let $P=LU\subset G$ be a parabolic subgroup and $X_L$ be an $L$-spherical variety then the contracted product $X=X_L\times^P G$\index{$X_L\times^P G$} (that is the quotient of $X_L\times G$ by the free $P$-action $\ell u\cdot (y,g)=(y\ell^{-1},\ell ug)$) is again spherical and is called the parabolic induction of $X_L$ from $P$ to $G$.
	\end{itemize}
\end{paragr}

\begin{paragr}[A non-example: Whittaker models.]\label{S Whittaker induction}
	Assume that $k$ is a global or local field and that $G$ is quasi-split. Let $B=TN\subset G$ be a Borel subgroup. If $k$ is global, we fix a nondegenerate character\footnote{Recall that a character, of $N(\bA_k)$ or $N(k)$, is said to be {\it nondegenerate} if it is non-trivial on each of the root subgroup $N_\alpha\subset N$ associated to simple roots $\alpha\in \Delta$.} $\psi_N: N(\bA_k)\to \bC^\times$ trivial on $N(k)$ whereas if $k$ is local, we fix a nondegenerate unitary character $\psi_N: N(k)\to \bC^\times$. In this situation, we would like to consider the ``equivariant quotient''
	$$\displaystyle X=(N,\psi_N)\backslash G$$
	as part of our general setting.\index{$(N,\psi_N)\backslash G$}  More precisely, the underlying variety is just the horopherical variety $N\backslash G$ but we should understand $X$ as a symbol whose meaning is that spaces of functions on $X(k)$ (for $k$ local) or $X(\mathbb{A}_k)$ (for $k$ global) are actually sections of the natural (complex) line bundle $\mathcal{L}_\psi$ on $(N\backslash G)(k)$ or $(N\backslash G)(\mathbb{A})$ associated to $\psi$ or, more concretely, functions $f:G(k) \mbox{ or } G(\bA_k)\to \bC$ such that $f(ug)=\psi_N(u)f(g)$ for every $u\in N(k)$ or $N(\bA_k)$. We call $X$ the {\it Whittaker variety}. In general, every statement or conjecture that we will state, only involving spaces of functions on spherical varieties, should remain true in the Whittaker case as well once appropriately reinterpreted.
	
	This example can actually be extended to the notion of {\it Whittaker induction}. More precisely, let us fix for convenience a non trivial additive character $\psi: \bA_k/k\to \bC^\times$ or $\psi: k\to \bC^\times$. Then, if $P=LU$ is a parabolic subgroup, $X_L=H_L\backslash L$ a homogeneous spherical $L$-variety and $\Psi: X_L\to \Hom_{alg}(U,\mathbb{G}_a)$ is an $L$-equivariant regular map, we shall denote by the formal symbol
	$$\displaystyle X=X_L\times^{P,\Psi} G$$
	the `variety' such that functions on $X(\bA_k)$ or $X(k)$ are functions on (the $\bA_k$- or $k$-points of) $X^M\times G$ satisfying $f(x(ul)^{-1},ulg)=\psi_{U,x}(u)f(x,g)$ for every $ul\in P(\bA_k)$ or $P(k)$. Here, for every $x\in X_L(\bA_k)$ or $X_L(k)$ we have denoted by $\psi_{U,x}$ the composition $\psi\circ \Psi(x)$. We call $X$ the {\it Whittaker induction} of $X_L$ (or more properly of the pair $(X_L,\Psi)$). Note that it can also formally be written as
	$$\displaystyle X=(H_L\ltimes U,\psi_U)\backslash G$$
	where $\psi_U=\psi_{U,1}$. We refer the reader to \cite[Section 2.6]{SV} for a more conceptual and detailed discussion of this notion as well as explanations on how to extend the definition of most of the combinatorial invariants of the next subsections to this setting.
\end{paragr}

\subsection{Structure}\label{sect structure}

The aim of this section is to review some important invariants associated to a spherical variety $X$. In particular, from the work of Brion and Knop, we can associate to $X$ a based root system $(\Phi_X,\Delta_X)$ to which we can in turn associate a dual complex group $\widehat{G}_X$ in a way that generalizes Langlands construction of the dual of a connected reductive group. A first construction of the dual group of a spherical variety was proposed is in the work of Gaitsgory and Nadler \cite{GN} through Tannakian formalism, whereas the definition given here is combinatorial in nature. There is however a slight difference between the two constructions: the dual group $\widehat{G}_X^{\mathrm{GN}}$ attached to $X$ by Gaitsgory and Nadler comes with an embedding $\widehat{G}_X^{\mathrm{GN}}\subset \widehat{G}$ in the dual group of $G$ whereas the dual group we consider here, following Sakellaridis and Venkatesh, admits an isogeny onto $\widehat{G}_X^{\mathrm{GN}}$ with finite (in general nontrivial) kernel. The drawback is that the Sakellaridis-Venkatesh dual group $\widehat{G}_X$ is not always well-defined: it will only exist under a technical condition that ``$X$ has no roots of type $N$'' (see \S \ref{S spherical roots}). However, $\widehat{G}_X$ seems to relate in better ways to local and global distinction phenomena. Moreover, the map $\widehat{G}_X\to \widehat{G}_X^{\mathrm{GN}}$ was defined in \cite{SV} under a series of assumptions on the Gaitsgory-Nadler dual group. An inconditional construction of the morphism $\iota_X: \widehat{G}_X\to \widehat{G}$ was established in general by Knop and Schalke \cite{Knopdualgp}. 
	
Everything that we will discuss in this section depends on $X$ only up to $G$-birational equivalence. Therefore, we will assume that $X$ is a spherical $G$-homogeneous variety. Moreover, except for the last two paragraphs discussing the Galois action on $\widehat{G}_X$, all the constructions happens over an algebraic closure so that \textbf{except in \S \ref{S L-groups} and \S \ref{S Galois pairs Lgroup}, we will assume that $k=\overline{k}$ is algebraically closed.}
	
For convenience, we will fix a Borel subgroup $B\subset G$\index{$B$} with unipotent radical $N$ and denote by $X_B\subset X$\index{$X_B$} the open $B$-orbit although all constructions will be independent on such choice (up to canonical isomorphism) as can readily be checked. For example the canonical Cartan of $G$, defined as the quotient $A=B/N$\index{$A$}, is independent on choices.
	
We will also denote by $W=\Norm_G(A)/A$\index{$W$} the Weyl group of $G$ and by $\Phi\supset \Phi^+\supset \Delta$ the set of all\index{$\Phi$}, positive\index{$\Phi^+$} and simple roots\index{$\Delta$} of $A$ with respect to $B$ respectively.

\vspace{2mm}

\begin{paragr}[The center.]
	The center of $X$ is its group of $G$-equivariant automorphisms
	$$\displaystyle Z(X)=\Aut_G(X).$$
	Writing $X=H\backslash G$, we have $Z(X)=\mathrm{N}_G(H)/H$\index{$Z(X)$}. It is known that $Z(X)$ is a diagonalizable group \cite[corollaire 5.2]{BrionPauer}. In particular, its neutral component $Z(X)^0$ is always a torus.
\end{paragr}

\begin{paragr}[The canonical Cartan.]\label{S canonical Cartan}
	The universal Cartan $A_X$\index{$A_X$} of $X$ is the quotient through which $A$ acts on $X_B\sslash N$. Thus, we have a canonical quotient map $A\twoheadrightarrow A_X$ and $X_B\sslash N$ is a $A_X$-torsor. The weight lattice of $X$ is simply the character group $\Lambda_X=X^*(A_X)$\index{$\Lambda_X$}. It fits in the following exact sequence:
	\begin{equation}\label{fund exact seq}
		\displaystyle 1\to k^\times \to k(X)^{(B)}\to \Lambda_X \to 1
	\end{equation}
	where $k(X)^{(B)}$\index{$k(X)^{(B)}$} stands for the group of $B$-eigenfunctions in the field of rational functions $k(X)$\index{$k(X)$}, i.e.\ functions $f\in k(X)^\times$ for which there exists an algebraic character $\chi_f\in X^*(A)=X^*(B)$ with $b\cdot f=\chi(b)f$ for every $b\in B$, and the last morphism of the short exact sequence is just $f\mapsto \chi_f$.
\end{paragr}

\begin{paragr}
	\begin{lem}\label{lem1 appendix}
		Let $X$ be a $G$-spherical variety (not necessarily quasi-affine or unimodular). Then:
		\begin{enumerate}[(i)]
			\item The $k$-algebra $k[X]$ of regular functions on $X$ is finitely generated;
			
			\item If $X=H\backslash G$ is homogeneous, spherical and unimodular. Then, $X$ is quasi-affine. 
		\end{enumerate}
	\end{lem}
	
	\begin{proof}
		We may assume that $k=\overline{k}$ is algebraically closed.
		\begin{enumerate}[(i)]
			\item Let $B\subset G$ be a Borel subgroup and $X_B\subset X$ be the open Borel orbit. By \cite[Exercise 14.2.8]{SpringerLAG}, $X_B$ is affine. Hence, since $X$ is normal, the complement $X\setminus X_B$ is the union of finitely many irreducible divisors $D_1,\ldots, D_\ell$. Let, for each $1\leq i\leq \ell$, $v_i: k(X)^{\times}\to \mathbb{Z}$ be the valuation along the divisor $D_i$. Then, the restriction of $v_i$ to $k(X)^{(B)}$ descends to a linear map $X^*(A_X)\to \mathbb{Z}$. Let $k[X]^{(B)}$ be the multiplicative monoid of nonzero $B$-eigenfunctions in $k[X]$. Then, $k[X]^{(B)}$ is the submonoid of those $f\in k(X)^{(B)}$ such that $v_i(f)\geq 0$ for every $1\leq i\leq \ell$. By \cite[Proposition 1, p.12]{Fulton}, this implies that the image of $k[X]^{(B)}\to X^*(A_X)$, $f\mapsto \chi_f$, is a finitely generated monoid. Let $f_1,\ldots,f_k\in k[X]^{(B)}$ such that $\chi_{f_1},\ldots,\chi_{f_k}$ generate this monoid. Then, by the theory of highest weights, the $G$-translates of $f_1,\ldots,f_k$ (which span a finite dimensional $k$-vector space) generate $k[X]$ as a $k$-algebra.
			
			\item Let $\varphi: X\to X^{\mathrm{aff}}:=\mathrm{Spec}(k[X])$ be the canonical morphism. By $(i)$, we need to show that $\varphi$ is an open embedding. Assume one moment that there exists a regular function $f\in k[X]$ whose zero locus $\{f=0 \}$ is set-theoretically equal to the complement $X\setminus X_B$. This would imply that the restriction of $\varphi$ to the principal open subset $D(f)\subset X^{\mathrm{aff}}$ is an isomorphism $X_B\simeq D(f)$. However, since $X$ is homogeneous it can be covered by finitely many translates of $X_B$, say $X=X_Bg_1\cup\ldots\cup X_B g_n$. But, for each $1\leq i\leq n$, the restriction of $\varphi$ to $D(f^{g_i})$ (where $f^{g_i}=$ the right $g^{-1}_i$-translate of $f$) is also an isomorphism $X_B g_i\simeq D(f^{g_i})$ which implies that $\varphi$ is indeed an open embedding. Thus, it suffices to show the existence of a regular function $f\in k[X]$ such that $\{f=0 \}=X\setminus X_B$ set theoretically. 
			
			Let $\mathfrak{g}$, $\mathfrak{h}$ and $\mathfrak{b}$ be the Lie algebras of $G$, $H$ and $B$ respectively. Without loss in generality, we may assume that the canonical basepoint $x_0=H1$ belongs to $X_B$ i.e.\ $\mathfrak{g}=\mathfrak{h}+\mathfrak{b}$. We claim that there exists a subspace $W$ of $\mathfrak{b}$ such that we have the decomposition $\Ad(b)W\oplus \mathfrak{h}\cap \mathfrak{b}=\mathfrak{b}$ for every $b\in B$. Indeed, this follows from the following general fact applied to the adjoint representation of $B$:
			\begin{num}
				\item Let $V$ be an algebraic $B$-representation and $X\subset V$ a subspace. Then, there exists a subspace $Y\subset V$ such that $V=bY\oplus X$ for every $b\in B$.
			\end{num}
			This can be proved by induction on the length $n$ of a $B$-stable filtration
			$$\displaystyle V=V_n\supset V_{n-1}\supset \ldots \supset V_0=0$$
			such that $B$ acts on every subquotient $V_i/V_{i-1}$ by a character. Indeed, when $n=0$ (i.e.\ $V=0$), the result is obvious whereas if $Y_{n-1}\subset V_{n-1}$ is a subspace such that $V_{n-1}=bY_{n-1}\oplus (X\cap V_{n-1})$ for every $b\in B$, it suffices to take for $Y$ any subspace satisfying the two following properties: $Y\cap V_{n-1}=Y_{n-1}$ and $Y/Y_{n-1}$ is a supplementary subspace to $X/(X\cap V_{n-1})$ in $V_n/V_{n-1}$.
			
			Fix a subspace $W\subset \mathfrak{b}$ as above. Then, we have (applying the condition to $b=1$) a decomposition $\mathfrak{g}=\mathfrak{h}\oplus W$ and denoting by $p_W:\mathfrak{g}\to W$ the corresponding projection, we claim that the regular function on $G$
			$$\displaystyle f: g\in G\mapsto \det(p_W\circ \Ad(g)\mid_W)$$
			descends to a regular function on $X$ satisfying the desired property. Indeed, since $X$ is unimodular, we have
			$$\displaystyle f(hg)=\det(\Ad(h)\mid_{\mathfrak{g}/\mathfrak{h}})f(g)=f(g)$$
			for every $(h,g)\in H\times G$ so that $f$ indeed descends to a function on $X=H\backslash G$. Moreover, by construction $f$ does not vanish on $B$, hence on $X_B=H\backslash HB$, whereas if $Hg\in X\setminus X_B$, we have $\mathfrak{h}+ \Ad(g)\mathfrak{b}\neq \mathfrak{g}$, a fortiori $\mathfrak{h}+ \Ad(g)W\neq \mathfrak{g}$ and therefore $f(g)=0$.
		\end{enumerate}
		
	\end{proof}
\end{paragr}

\begin{paragr}[The parabolic type]\label{S parabolic type}
	We will also denote by $P_X$\index{$P_X$} the stabilizer of the open $B$-orbit:
	$$\displaystyle P_X=\{g\in G\mid X_Bg=X_B \}.$$
	It is a parabolic subgroup of $G$ whose (conjugacy) class does not depend on $B$. As we will see, this very coarse invariant is related to a certain {\it Arthur $\SL_2$} associated to $X$ that controls in some approximative sense the size of the irreducible representations in the $X$-spectrum. If $X$ is quasi-affine, it is also the maximal parabolic subgroup containing $B$ of which $A_X$ is still a quotient \cite[Lemma 3.1]{KnopAB}.
	
	The unipotent radical $N_X$\index{$N_X$} of $P_X$ acts freely on $X_B$ and we have $X_B/N_X=X_B\sslash N$. Actually, the $N_X$-bundle $X_B\to X_B/N_X$ can be split (although non-canonically) as follows. Pick a regular function $f\in k[X]$ that vanishes, set-theoretically, on the complement $X\setminus X_B$ (this can be done thanks to our running assumption that $X$ is quasi-affine). Let $x_0\in X_B(k)$. Then, the differential $d_{x_0}f$ of $f$ at $x_0$ can be seen as an element in the dual Lie algebra $\mathfrak{g}^*$ through the natural embedding (given by differentiating the action) $T^*_{x_0} X\hookrightarrow \mathfrak{g}^*$. Then, the stabilizer $L_X$\index{$L_X$} of $d_{x_0}f$ for the coadjoint action is a Levi factor of $P_X$, the $L_X$-orbit $x_0L_X$ is isomorphic to the quotient $L_X\twoheadrightarrow A_X$ \cite[th\'eor\`eme 3.4]{BLV} thus identifying the torus $A_X$ with a subvariety of $X$ and the action map yields an isomorphism
	\begin{equation}\label{eq LST X}
		\displaystyle X_B\simeq A_X\times N_X.
	\end{equation}
	We note that the above choice of embedding $A_X\hookrightarrow X$ corresponds, in the terminology of \cite{KnopAB}, to the choice of a {\it generic flat} together with a base-point. Thus, since $X$ is homogeneous, by \cite[Corollary 7.8]{KnopAB} this embedding has a {\it closed} image. Let us also remark that the isomorphism \eqref{eq LST X} allows to define a (non-canonical) left $A_X$-action on $X_B$ commuting with the (right) $P_X$-action.
\end{paragr}

\begin{paragr}[The Weyl group]
	One crucial invariant of $X$, and probably the subtlest one, is its {\it Weyl group} $W_X$\index{$W_X$}. It is at first approximation a group of automorphisms of the canonical Cartan $A_X$ and a posteriori the Weyl group of the root system of $X$. There are multiple ways to define it but here is one that is due to Friedrich Knop \cite{KnopAB} and that is pleasantly aligned with the Hamiltonian duality picture of \cite{BZSV}. In a nutshell, Knop's approach relates $W_X$ to the cotangent bundle $T^*X$\index{$T^*X$} describing more precisely the GIT quotient $T^*X\sslash G$\footnote{Note that since $X$ is assumed quasi-affine, so is $T^*X$ and the GIT quotient $T^*X\sslash G$ is well-defined.} as $\mathfrak{a}_X^*\sslash W_X$. This identification is moreover compatible with the moment map $\Phi:T^*X\to \mathfrak{g}^*$\index{$\Phi$} (dual to the map $\mathfrak{g}\to \Gamma(X,TX)$ obtained by differentiating the action) in the sense that we have a commutative diagram \cite[Korollar 7.2]{KnopWM}:
	$$\displaystyle \xymatrix{ T^*X \ar[d] \ar[r] & T^*X\sslash G=\mathfrak{a}_X^*\sslash W_X \ar[d] \\ \mathfrak{g}^* \ar[r] & \mathfrak{g}^*\sslash G=\mathfrak{a}^*\sslash W}$$
	where $\mathfrak{a}^*$\index{$\mathfrak{a}^*$} (resp. $\mathfrak{a}_X^*$)\index{$\mathfrak{a}_X^*$} stands for the dual of the Lie algebra of $A$ (resp. of $A_X$), the identification $\mathfrak{g}^*\sslash G=\mathfrak{a}^*\sslash W$ follows from Chevalley's theorem and the right vertical arrow is induced from the natural inclusion $\mathfrak{a}_X^*\subset \mathfrak{a}^*$ (dual to the projection $A\twoheadrightarrow A_X$).
	In order to pinpoint $W_X$ more precisely (for example the previous property does not distinguish $W_X$ from its conjugates by $\Norm_W(\mathfrak{a}_X^*)$), we can pass (following Knop) to the following $W$-cover
	$$\displaystyle \widehat{T^* X}^{\bullet}:=T^*X\times_{\mathfrak{a}^*\sslash W} \mathfrak{a}^*.$$
	There is a natural map 
	$$\displaystyle X_B\times \mathfrak{a}_X^*\to \widehat{T^*X}^{\bullet}$$
	sending $(x,\chi)\in X_B\times X^*(A_X)$ to $(x,\frac{df_\chi(x)}{f_\chi(x)}, d\chi)$, where $f_\chi\in k(X)^{(B)}$ denotes any $B$-eigenfunction with eigencharacter $\chi$, and extended to all of $X_B\times \mathfrak{a}_X^*$ by linearity in the second factor. Then, this map is $P_X$-equivariant (for the trivial $P_X$-action on $\mathfrak{a}_X^*$) and there exists a non-empty open subset $\mathfrak{a}_X^{*,r}\subset \mathfrak{a}_X^*$\footnote{We can take for $\mathfrak{a}_X^{*,r}$ the open subset defined by the condition $\alpha^\vee \neq 0$ for every $\alpha\in \Phi\setminus \Phi^{L_X}$ where $\Phi^{L_X}$ stands for the root system of the Levi factor $L_X$ of $P_X$.} such that the resulting map
	$$\displaystyle X_B\times^{P_X} G\times \mathfrak{a}_X^{*,r}\to \widehat{T^*X}^{\bullet}$$
	is an open embedding with image in a particular irreducible component $\widehat{T^*X}$ of maximal dimension of $\widehat{T^*X}^{\bullet}$ whose image by the natural map $\widehat{T^*X}^{\bullet}\to \mathfrak{a}^*$ is $\mathfrak{a}_X^*$. Following Knop \cite{KnopAB}, we call $\widehat{T^*X}$ the {\it polarized cotangent bundle} of $X$. 
	
	Then, $W_X$ is the quotient of the stabilizer $W_{(X)}$\index{$W_{(X)}$} in $W$ of the irreducible component $\widehat{T^*X}$ by its pointwise stabilizer that can be shown to be equal to the Weyl group $W^{L_X}$ of the Levi quotient of $P_X$. Moreover, the identification $T^*X\sslash G=\mathfrak{a}_X^*\sslash W_X$ descends from the map $\widehat{T^*X}\to \mathfrak{a}_X^*$.
\end{paragr}

\begin{paragr}\label{embedding WX}
	This gives a description of $W_X$ as a subquotient of $W$. However, there is the following canonical way to lift it as a subgroup. Any $w\in W_X$ lifts to an unique element of $W_{(X)}$ preserving the set of positive roots $\Delta^{L_X}$ of $L_X$ (or equivalently the Borel subgroup $B\cap L_X$) and this yields a section $W_X\to W_{(X)}$ so that $W_{(X)}=W^{L_X}\rtimes W_X$.
\end{paragr}

\begin{paragr}[The cone $\mathcal{A}_X^-$.]\label{S cone Ax-}
	Let $\mathcal{V}$\index{$\mathcal{V}$} be the set of $G$-invariant valuations $v:k(X)\to \bQ\cup \{\infty \}$ that are trivial on $k^\times$. By restriction to $k(X)^{(B)}$, and the exact sequence \eqref{fund exact seq}, we have a natural map $\mathcal{V}\to \mathcal{A}_X:=X_*(A_X)\otimes_{\bZ} \bQ$ which is always injective \cite[Corollary 1.8]{KnopLV}. We will write $\mathcal{A}_X^-\subset \mathcal{A}_X$\index{$\mathcal{A}_X^-$} for the image of $\mathcal{V}$ by this map. It is known that $\mathcal{A}_X^-$ is a finitely generated cone \cite[Corollary 5.3]{KnopLV} that always contains the image of the negative Weyl chamber $\mathcal{A}^-\subset \mathcal{A}=X_*(A)\otimes \mathbb{Q}$\footnote{More precisely, $\mathcal{A}^-=\{H\in \mathcal{A}\mid \langle \alpha,H\rangle \leqslant 0 \; \forall \alpha\in \Delta \}$ where $\Delta$ stands for the set of simple roots of $A$ in $B$.} \index{$\mathcal{A}^-$} by the natural projection $\mathcal{A}\twoheadrightarrow \mathcal{A}_X$. 
	
	In the case where $X$ is quasi-affine, the valuation cone $\mathcal{A}_X^-$ admits the following alternative description. Let $X^*(A_X)^+=X^*(A_X)\cap X^*(A)^+$\index{$X^*(A_X)^+$} be the subset of dominant characters in $X^*(A_X)$. Then, the ring of regular function of $X$ decomposes according to the $G$-action as
	$$\displaystyle k[X]=\bigoplus_{\lambda\in X^*(A_X)^+} k[X]_\lambda$$
	where $k[X]_\lambda$ denotes the isotypic component for the irreducible representation with highest weight $\lambda$ (under the sphericity assumption, $k[X]_\lambda$ is either zero or irreducible itself). This decomposition is however not a grading of the ring $k[X]$ and the valuation cone is actually (negative) dual to the cone measuring the defect of this decomposition to be a grading. Namely, letting ${}^+ \mathcal{A}_X^*$\index{${}^+ \mathcal{A}_X^*$} be the cone in $\mathcal{A}_X^*=X^*(A_X)\otimes \mathbb{Q}$ generated by the set of $\nu \in X^*(A_X)$ such that
	$$\displaystyle k[X]_\lambda k[X]_\mu\cap k[X]_{\lambda+\mu-\nu}\neq 0$$
	for some $\lambda,\mu\in X^*(A_X)^+$, we have (see \cite[Lemma 5.1]{KnopLV})
	$$\displaystyle \mathcal{A}_X^-=\{ H\in \mathcal{A}_X\mid \langle \lambda,H\rangle\leq 0 \; \forall \lambda\in {}^+\mathcal{A}_X^*\}.$$
	Note that there is a natural $G$-stable (multi)filtration on the algebra $k[X]$ indexed by $X^*(A_X)$ and the partial order $\preceq$ associated to the cone ${}^+ \mathcal{A}_X^*$ (i.e.\ $\mu\preceq \lambda$ iff $\lambda-\mu\in {}^+ \mathcal{A}_X^*$) given by
	$$\displaystyle k[X]_{\preceq \lambda}=\bigoplus_{\mu\preceq \lambda} k[X]_\mu.$$
	
	Following Sakellaridis and Venkatesh \cite{SV}, we say that $X$ is {\it wavefront} if $\mathcal{A}_X^-$ is precisely the image of $\mathcal{A}^-$. This is a technical but important condition allowing to obtain a priori finiteness properties for the (points of the) spherical variety $X$ over $p$-adic fields (as well as for all its boundary degenerations see Section \ref{sect boundary degenerations}) that are crucial inputs for the development of harmonic analysis on the variety (in particular for the Plancherel decomposition of \cite[Part 3]{SV}).
\end{paragr}

\begin{paragr}[Spherical roots.]\label{S spherical roots}
	Back to the general case, it is known \cite{KnopAB} that the cone $\mathcal{A}_X^-$ is a fundamental domain for the action of $W_X$ on $\mathcal{A}_X$. This entails that $W_X$ is actually the Weyl group of a finite root system with a system of simple roots $\Delta_X\subset a_X^*$\index{$\Delta_X$} that generate the extremal rays of the dual cone
	$$\displaystyle {}^+\mathcal{A}_X^*=\{\chi\in \mathcal{A}_X^*\mid \langle \chi,\lambda\rangle \leqslant 0 \; \forall \lambda\in \mathcal{A}_X^-  \}.$$
	There are actually various ways to normalize the roots in $\Delta_X$. In the literature on spherical varieties it is customary to take the generators of the intersections of the extremal rays of ${}^+\cA_X^*$ with the weight lattice $X^*(A_X)$, although other normalizations have also been considered see e.g. \cite{KnopAutom}. However, Sakellaridis and Venkatesh have pointed out that this is not the best normalization from a representation theoretic point of view and that we should instead take $\Delta_X$ to be the set of generators of intersections of extremal rays with the root lattice $\bZ \Delta$ of $G$. See \cite[Proposition 2.2.1]{SV} for the proof that it indeed produces the simple roots of a root system with Weyl group $W_X$. In these notes, we will follow Sakellaridis-Venkatesh normalization and simply refer to elements of $\Delta_X$ as the {\it spherical roots}. We will also write $\Delta_X^\vee\subset \cA_X$ for the corresponding set of spherical coroots.
	
	We have
	\begin{equation*}
		\displaystyle Z(X)^0=\bigcap_{\alpha\in \Delta_X} \Ker(\alpha)^0.
	\end{equation*}
	Moreover, by \cite[Corollary 3.1.2]{SV}, for each spherical root $\alpha\in \Delta_X$, exactly one of the following holds:
	\begin{enumerate}[(a)]
		\item $\alpha$ is a positive root for $G$ i.e.\ $\alpha\in \Phi^+$;
		
		\item there exist two positive roots $\gamma_1,\gamma_2\in \Phi^+$ that are strongly orthogonal in the sense that $(\bQ\gamma_1+\bQ \gamma_2)\cap \Phi=\{\pm \gamma_1, \pm \gamma_2 \}$ and satisfying $\alpha=\gamma_1+\gamma_2$.
	\end{enumerate}
	Case (a) can also be divided in two subcases according to whether $\alpha\in X^*(A_X)$ or $\alpha\notin X^*(A_X)$. In the former case we say that $\alpha$ is a root of type $T$ (a typical example being the unique spherical root of $T\backslash \PGL_2$, $T\subset \PGL_2$ being a maximal torus) and in the latter case we say that $\alpha$ is a root of type $N$ (in which case $2\alpha\in X^*(A_X)$ see the discussion in \cite[pp.56-57]{SV}; a typical example is the unique spherical root of $N(T)\backslash \PGL_2$, where $N(T)$ denotes the normalizer of a maximal torus). When $\alpha$ falls in case (b) on the other hand, we shall say that $\alpha$ is a root of type $G$ (this is for example the case of all spherical roots for the group variety $X=H$). When $\alpha$ is of type $G$, it automatically belongs to the weight lattice $X^*(A_X)$ but a decomposition $\alpha=\gamma_1+\gamma_2$ as in (b) above is in general far from unique. It becomes unique however if we ask it to satisfy one of the following equivalent additional conditions (\cite[Corollary 3.1.4]{SV}, \cite[Lemma 6.1]{Knopdualgp}):
	\begin{itemize}
		\item There exist simple roots $\delta_1,\delta_2\in \Delta$ such that $\gamma_1^\vee-\gamma_2^\vee=\delta_1^\vee-\delta_2^\vee$;
		
		\item Denoting by $w_\alpha\in W_X$ the simple reflection corresponding to $\alpha$ and using the embedding $W_X\subset W$ described in \S \ref{embedding WX} to let $W_X$ act on $\mathcal{A}^*$, $\gamma_1$ and $\gamma_2$ are the only positive roots in the $-1$-eigenspace of $w_\alpha$.
	\end{itemize}
	When $\gamma_1$, $\gamma_2$ are chosen in such a way, we call them the {\it associated roots} of $\alpha$.
	
	Let $\alpha\in \Delta_X$ and $\alpha^\vee$ be the corresponding coroot for the root system of $X$. When $\alpha$ is of type $T$ or $N$ (i.e.\ $\alpha\in \Phi^+$), $\alpha^\vee$ is the image of the corresponding coroot of $G$ by the projection $\mathcal{A}\twoheadrightarrow \mathcal{A}_X$. When $\alpha$ is of type $G$ on the other hand, the projections to $\mathcal{A}_X$ of the coroots $\gamma_1^\vee$, $\gamma_2^\vee$ corresponding to the two associated roots of $\alpha$ coincide and are equal to $\alpha^\vee$.
	
	This has the following important consequence. Let $\widehat{A}$\index{$\widehat{A}$}, $\widehat{A}_X$\index{$\widehat{A}_X$} be the complex dual tori of $A$ and $A_X$ respectively. The surjection $A\to A_X$ dualizes to a morphism $\widehat{A}_X\to \widehat{A}$ and the above property implies that:
	\begin{equation}\label{eq containment coroots}
		\displaystyle \Delta_X^\vee\subset \Phi^\vee\mid_{\widehat{A}_X}.
	\end{equation}
	Since $\Phi^\vee$ is the set of roots of $\widehat{A}$ in the dual group $\widehat{G}$ the above containment will have an obvious interpretation when we define the dual group of $X$ in the next section. We emphasize that the Sakellaridis-Venkatesh normalization of spherical roots is the only one for which \eqref{eq containment coroots} holds.
\end{paragr}

\begin{paragr}[The dual group.]\label{S dual group}
	Given the existence of a root system associated to $X$ it is tempting to just mimick Langlands' definition of the dual group and to define a dual group for $X$ by just exchanging its roots and coroots. However, as discussed in the previous paragraph, the quadruple
	$$\displaystyle \mathcal{R}_X:=(X^*(A_X),\Delta_X,X_*(A_X),\Delta_X^\vee)$$\index{$\mathcal{R}_X$}
	is a based root datum only if $X$ has no spherical root of type $N$ (because otherwise $\Delta_X\not\subset X^*(A_X)$).
	
	Thus, only under the assumption that $X$ has no root of type N can we really defined a dual group. More precisely, under this assumption we let $\widehat{G}_X$\index{$\widehat{G}_X$} be a complex connected reductive group equipped with a Borel pair $(\widehat{A}_X, \widehat{B}_X)$\index{$(\widehat{A}_X, \widehat{B}_X)$} as well as an isomorphism
	$$\displaystyle \mathcal{R}_{\widehat{G}_X}\simeq \mathcal{R}_X^\vee$$
	between its based root datum $\mathcal{R}_{\widehat{G}_X}$ and the dual based root datum
	$$\displaystyle \mathcal{R}_X^\vee=(X_*(A_X),\Delta_X^\vee,X^*(A_X),\Delta_X).$$\index{$\mathcal{R}^\vee_X$}
	
	Here, we recall that a {\em Borel pair} of a complex connected reductive group $\mathcal{G}$ is a pair $(\mathcal{T},\mathcal{B})$ consisting of a Borel subgroup $\mathcal{B}\subset \mathcal{G}$ and a maximal torus $\mathcal{T}\subset \mathcal{B}$. On the other hand, a {\em pinning} of $\mathcal{G}$ is a Borel pair $(\mathcal{T},\mathcal{B})$ together with a family of basis elements $\{X_\alpha \}_{\alpha\in S}$ of the root subspaces $\Lie(\mathcal{G})_\alpha$ associated to the corresponding set of simple roots $S\subset X^*(\mathcal{T})$ of $\mathcal{T}$ in $\mathcal{B}$. The automorphisms of $\mathcal{G}$ fixing a given Borel pair $(\mathcal{T},\mathcal{B})$ (resp. fixing a  given pinning) and inducing the identity on its based root datum are the inner automorphisms associated to elements of $\mathcal{T}$ (resp. are trivial). In particular, $\widehat{G}_X$ is uniquely defined up to $\widehat{A}_X$-conjugacy.
	
	We can also associate to $X$ an algebraic morphism
	$$\displaystyle \iota^{\SL_2}_X: \SL_2(\bC)\to \widehat{G}$$\index{$\iota_X^{\SL_2}$}
	depending only on the parabolic type $P_X$ of $X$ and defined as follows. The dual group $\widehat{G}$ comes with a pinning
	$$\displaystyle \Pin_G=(\widehat{B},\widehat{A}, (X_{\alpha^\vee})_{\alpha\in \Delta}).$$\index{$\Pin_G$}
	The parabolic $P_X$ determines a standard Levi subgroup $\widehat{L}_X$ of $\widehat{G}$ equipped with its own pinning, namely $(\widehat{B},\widehat{A}, (X_{\alpha^\vee})_{\alpha\in \Delta^{L_X}})$. This, in turn, gives rise to a so-called {\it principal $\SL_2$} morphism $\SL_2(\bC)\to \widehat{L}_X$ whose differential sends $\begin{pmatrix} 0 & 1 \\ 0 & 0 \end{pmatrix}$ to the sum $\displaystyle \sum_{\alpha\in \Delta^{L_X}} X_{\alpha^\vee}$ and maps $\begin{pmatrix} 1 & 0 \\ 0 & -1 \end{pmatrix}$ into $\Lie(\widehat{A})$. Then, $\iota^{\SL_2}_X$ is the composition of this morphism with the inclusion $\widehat{L}_X\subset \widehat{G}$. 
	
	The surjection $A\to A_X$ naturally induces a morphism between dual tori $\widehat{A}_X\to \widehat{A}$. Note that this morphism has image in the center $Z(\widehat{L}_X)$ of $\widehat{L}_X$ (this is just a reflection of the fact that the quotient $A\to A_X$ extends to $L_X\to A_X$) and in particular commutes with the image of $\iota_X^{\SL_2}$. The resulting morphism $\widehat{A}_X\times \SL_2(\bC)\to \widehat{G}$ can then be further extended to $\widehat{G}_X\times \SL_2(\bC)\to \widehat{G}$ thanks to the following theorem of Knop and Schalke \cite{Knopdualgp} which was suggested by Sakellaridis and Venkatesh \cite[\S 3.2]{SV}.  

	\begin{theo}[Knop-Schalke]\label{thm KS}
		Assume that $X$ has no spherical roots of type $N$. Then, there exists an (algebraic) morphism
		$$\displaystyle \iota_X: \widehat{G}_X\times \SL_2(\bC)\to \widehat{G}$$\index{$\iota_X$}
		satisfying the following conditions:
		\begin{itemize}
			\item the restriction of $\iota_X$ to $\SL_2(\mathbb{C})$ coincides with the morphism $\iota_X^{\SL_2}$ defined above;
			
			\item the restriction of $\iota_X$ to the maximal torus $\widehat{A}_X\subset \widehat{G}_X$ is the morphism $\widehat{A}_X\to \widehat{A}$ dual to the surjection $A\to A_X$;
		\end{itemize}
		as well as the following properties on the image of root subspaces by its differential $d\iota_X$:
		\begin{itemize}
			\item for every $\alpha\in \Delta_X$ of type $T$ (i.e.\ $\alpha\in \Phi$),
			$$\displaystyle d\iota_X(\widehat{\mathfrak{g}}_{X,\alpha^\vee})= \widehat{\mathfrak{g}}_{\alpha^\vee};$$
			
			\item for every $\alpha\in \Delta_X$ of type $G$ (i.e.\ $\alpha\notin \Phi$), 
			$$\displaystyle d\iota_X(\widehat{\mathfrak{g}}_{X,\alpha^\vee})\subset \widehat{\mathfrak{g}}_{\gamma_1^\vee}\oplus \widehat{\mathfrak{g}}_{\gamma_2^\vee}$$
			where $\gamma_1,\gamma_2\in \Phi$ are the associated roots of $\alpha$ (see \S \ref{S spherical roots}).
		\end{itemize}
	\end{theo}
	
	A morphism $\iota_X$ satisfying the conditions of the theorem is called a {\it distinguished morphism}. Note that distinguished morphisms are unique up to $Z(\widehat{L}_X)$-conjugacy. Indeed, since the union of the set of spherical roots of type $T$ and of the set of associated roots to spherical roots of type $G$ form a linearly independent set of roots (of $G$), a morphism $\widehat{G}_X\to \widehat{G}$ satisfying the last three conditions above is unique up to $\widehat{A}$-conjugation and moreover the centralizer of $\iota_X^{\SL_2}$ in $\widehat{A}$ is precisely $Z(\widehat{L}_X)$.
\end{paragr}

\begin{paragr}[The $L$-group.]\label{S L-groups}
	When the spherical variety $X$ is defined over a field $k$ that is not algebraically closed, it is tempting to ask whether we can naturally define an action of the absolute Galois group $\Gamma=\Gal(\overline{k}/k)$\index{$\Gamma$} on the dual group $\widehat{G}_X$ which would lead to some sort of $L$-group ${}^L G_X=\widehat{G}_X\rtimes \Gamma$\index{${}^L G_X$} of $X$. Moreover, we would naturally require the morphism $\iota_X: \widehat{G}_X\times \SL_2(\bC)\to \widehat{G}$ to extend to a morphism ${}^L G_X\times \SL_2(\bC)\to {}^L G$ compatible with the projections to $\Gamma$.
	
	In \cite{KnopFdualgp}, motivated by some natural functorial properties of the assignment $X\to \widehat{G}_X$, Knop has proposed more stringent conditions on the images of the root subspaces $\widehat{\mathfrak{g}}_{X,\alpha^\vee}$ by $d\iota_X$ when $\alpha\in\Delta_X$ is a root of type $G$. More precisely, he has prescribed a line $L_{\alpha^\vee}\subset \widehat{\mathfrak{g}}_{\gamma_1^\vee}\oplus \widehat{\mathfrak{g}}_{\gamma_2^\vee}$, where $\gamma_1$, $\gamma_2$ denote again the associated roots of $\alpha$, and shown that there exists a distinguished morphism $\iota_X: \widehat{G}_X\times \SL_2(\bC)\to \widehat{G}$ as in Theorem \ref{thm KS} with the additional property that $d\iota_X(\widehat{g}_{X,\alpha^\vee})\subset L_{\alpha^\vee}$ for every $\alpha\in \Delta_X\setminus \Phi$. Actually, the condition that $\iota_X\mid_{\widehat{G}_X}$ commutes with $\iota_X^{\SL_2}$ already characterizes the line $L_{\alpha^\vee}$ uniquely except in one case: when $\gamma_1$, $\gamma_2$ are also simple roots for $G$. Such spherical roots are called {\it roots of type $D_2$} by Knop (because this is equivalent to the root subsystem generated by the support of $\alpha$ in $\Delta$ being of type $D_2$). Then, for $\alpha\in \Delta_X$ of type $D_2$, Knop definition is that
	$$\displaystyle L_{\alpha^\vee}=\bC(X_{\gamma_1^\vee}-X_{\gamma_2^\vee})$$
	where $X_{\gamma_1^\vee}$, $X_{\gamma_2^\vee}$ are the basis elements of $\widehat{\mathfrak{g}}_{\gamma_1^\vee}$, $\widehat{\mathfrak{g}}_{\gamma_2^\vee}$ provided by the pinning of $\widehat{G}$.
	
	Let us call a distinguished morphism $\iota_X: \widehat{G}_X\times \SL_2(\bC)\to \widehat{G}$ that satisfies Knop requirements that $d\iota_X(\widehat{\mathfrak{g}}_{X,\alpha^\vee})\subset L_{\alpha^\vee}$ for every $\alpha\in \Delta_X$ of type $D_2$, a {\it strong distinguished morphism}. Then, strong distinguished morphisms are unique up to $\widehat{A}_X$-conjugation. (In particular, the image in $\widehat{G}$ of all strong distinguished morphisms are the same.)
	
	\begin{exe}
		Consider the group case for $\SL(2)$, that is $X=\SL(2)$ as a $\SL(2)\times \SL(2)$-variety. Then, $\widehat{G}_X=\PGL_2(\bC)$, $\widehat{G}=\PGL_2(\bC)\times \PGL_2(\bC)$ with their standard pinnings. A strong distinguished morphism should send $X:= \begin{pmatrix} 0 & 1 \\ 0 & 0 \end{pmatrix}\in \widehat{\mathfrak{g}}_X$ to a multiple of
		$$\displaystyle X_{\gamma_1^\vee}-X_{\gamma_2^\vee}=(X,-X)\in \widehat{\mathfrak{g}}.$$
		In particular, the diagonal embedding $\widehat{G}_X\to \widehat{G}$ is not strongly distinguished whereas the morphism
		$$\displaystyle \widehat{G}_X\to \widehat{G},\;\;\; g\mapsto (g,\begin{pmatrix} -1 & 0 \\ 0 & 1 \end{pmatrix} g \begin{pmatrix} -1 & 0 \\ 0 & 1 \end{pmatrix})$$
		is.
	\end{exe}
	
	This now allows to define an $L$-group ${}^L G_X$ as follows. Namely, assume that the pair $(G,X)$ is defined over a field $k$ that is not algebraically closed and let $\iota_X: \widehat{G}_X\times \SL_2(\bC)\to \widehat{G}$ be a strong distinguished morphism. Because the action of $\Gamma$ on $\widehat{G}$ is pinned, it obviously preserves the image of $\iota_X$\footnote{Essentially because $\Gamma$ acts on $\cA$ and $\cA_X$ compatibly with the map $\cA\to \cA_X$ so that the subspaces $L_{\alpha^\vee}$ for $\alpha\in \Delta_X$ of type $D_2$ are permuted by $\Gamma$ hence the composition of $\iota_X$ with the action of any $\sigma\in \Gamma$ is again a strongly distinguished morphism.} and moreover the restriction of $\iota_X$ to $\widehat{A}_X\times \SL_2(\bC)$ is $\Gamma$-equivariant for the natural action of $\Gamma$ on the latter (which is trivial on the $\SL_2(\bC)$-factor). It follows that the $\Gamma$-action on $\widehat{G}$ lifts (necessarily uniquely since the kernel of $\iota_X$ is finite) to an action on $\widehat{G}_X$ such that $\iota_X$ is $\Gamma$-equivariant. The $L$-group of $X$ is then defined to be the semi-direct product ${}^L G_X=\widehat{G}_X\rtimes \Gamma$ with respect to this action. Note that $\iota_X$ extends uniquely to a morphism
	$$\displaystyle {}^L \iota_X: {}^L G_X\times \SL_2(\bC)\to {}^L G$$\index{${}^L \iota_X$}
	that is the identity on $\Gamma$.
	
	This action of $\Gamma$ on $\widehat{G}_X$ depends a priori on the choice of the strong distinguished morphism $\iota_X$ and changing $\iota_X$ conjugates the action by some element $a\in \widehat{A}_X$. Moreover, $\Gamma$ automatically preserves the Borel pair $(\widehat{A}_X,\widehat{B}_X)$ of $\widehat{G}_X$. This implies that if the $\Gamma$-action fixes any pinning of $\widehat{G}_X$ at all then up to conjugating it by some $a\in \widehat{A}_X$ it will fix the canonical pinning $\Pin_X$. However, this is not always possible i.e.\ the action of $\Gamma$ on $\widehat{G}_X$ cannot always be chosen to fix the canonical pinning.
	
	Whether this is the good definition of the $L$-group of a spherical variety remains to be seen in general but there is at least one case where this seems to fit well with known results and conjectures which is that of {\it Galois symmetric varieties}.
\end{paragr}

\begin{paragr}[Example: Galois symmetric varieties.]\label{S Galois pairs Lgroup}
	
	Let us illustrate the construction of the dual group and the morphism $\iota_X$ in the case of Galois symmetric varieties i.e.\ let us take
	$$\displaystyle X=H\backslash G,\;\; G=\Res_{\ell/k} H_\ell,$$
	where $H$ is a connected reductive group over $k$ and $\ell/k$ is a (separable) quadratic extension.
	
	As a preparation for this, let us now describe Knop's strong distinguished morphism in the group case $X=H\curvearrowleft G=H\times H$. Then, we have 
	$$\displaystyle \widehat{G}=\widehat{H}\times \widehat{H},\;\; \widehat{G}_X=\widehat{H}$$
	and the spherical root are all of type $D_2$. A particular strong distinguished morphism
	$$\displaystyle \iota_X=\iota_H: \widehat{H}\to \widehat{G}$$
	is then given by $h\mapsto (h,h^\vee)$ where
	\begin{equation}\label{Chevalley inv}
		\displaystyle \widehat{H}\to \widehat{H},\;\; h\mapsto h^\vee
	\end{equation}
	is the unique automorphism of $\widehat{H}$ sending the canonical pinning $\Pin_H=(\widehat{B}_H,\widehat{T}_H, (X_{\alpha^\vee})_{\alpha\in \Delta_H})$ of $\widehat{H}$ to its {\it opposite} 
	$$\displaystyle \Pin_H^{\mathrm{opp}}:=(\widehat{B}_H,\widehat{T}_H, (-X_{\alpha^\vee})_{\alpha\in \Delta_H})$$\index{$\Pin_H^{\mathrm{opp}}$}
	and acting by $t\mapsto w_\ell t^{-1}w_\ell$ on $\widehat{T}_H$ where $w_\ell$ stands for the longest element of the Weyl group\footnote{e.g. for $\GL_n$ with its standard pinning, $h\mapsto h^\vee$ is the automorphism $g\mapsto w{}^t g^{-1}w$ where $w=\begin{pmatrix} & & 1 \\ & \ddots & \\ 1 & & \end{pmatrix}$.}. We will call $h\mapsto h^\vee$\index{$h\mapsto h^\vee$} the {\it duality involution} of $\widehat{H}$. It is closely related to what is usually called the Chevalley involution. More precisely, the latter corresponds to the unique inner translate of $h\mapsto h^\vee$ preserving the pinning $\Pin_H$. 
	
	As the Galois action on $\widehat{H}$ preserves the pinning $\Pin_H$ (and therefore also the opposite pinning), it obviously commutes with the duality involution. This allows to extend the latter to an automorphism ${}^L H\to {}^L H$, $h\mapsto h^\vee$ that is the identity on $\Gamma$.
	
	Let now $X=\Res_{\ell/k} H_{\ell}/H$ be a Galois symmetric variety as above. Choosing an embedding $\ell\hookrightarrow \overline{k}$ yields an isomorphism $(X_{\overline{k}},G_{\overline{k}})\simeq (H_{\overline{k}},H_{\overline{k}}\times H_{\overline{k}})$ with the group variety and therefore identifications
	$$\displaystyle \widehat{G}_X=\widehat{H},\;\; \widehat{G}=\widehat{H}\times \widehat{H},\;\; \iota_X=\iota_H.$$
	However, these identifications are not Galois equivariant. Indeed, denoting by $\sigma\in \Gamma\mapsto \sigma_H\in \Aut(\widehat{H})$ the usual Galois action on $\widehat{H}$ (i.e.\ that preserving the canonical pinning $\mathcal{P}$ and giving rise to the $L$-group ${}^L H$), the Galois action on $\widehat{G}$ is given by
	\begin{equation*}
		\displaystyle \sigma_G(h_1,h_2)=\left\{ \begin{array}{ll}
			(\sigma_H(h_1),\sigma_H(h_2)) & \mbox{ if } \sigma\in \Gamma_\ell, \\
			(\sigma_H(h_2),\sigma_H(h_1)) & \mbox{ if } \sigma\in \Gamma_k\setminus \Gamma_\ell,
		\end{array}\right.\;\;\; \mbox{ for } \sigma\in \Gamma \mbox{ and } (h_1,h_2)\in \widehat{G}.
	\end{equation*}
	We readily check that this action preserves the image of $\iota_X$ and induces the following Galois action on $\widehat{G}_X$:
	\begin{equation}\label{Galois action GSspaces}
		\displaystyle \sigma_{G_X}(h)=\left\{ \begin{array}{ll}
			\sigma_H(h) & \mbox{ if } \sigma\in \Gamma_\ell, \\
			\sigma_H(h)^\vee & \mbox{ if } \sigma\in \Gamma_k\setminus \Gamma_\ell,
		\end{array}\right. \;\;\; \mbox{ for } \sigma\in \Gamma \mbox{ and } h\in \widehat{G}_X.
	\end{equation}
	We note that this action does not preserve the pinning $\Pin_H$, since the duality involution doesn't. More generally, the pinning given by a family $\{t_\alpha X_\alpha \}_{\alpha\in \Delta_H}$, where $t_\alpha\in \bC^\times$, is invariant by the Galois action $\sigma\mapsto \sigma_X$ if and only if $t_{\sigma_H(\alpha)}=t_\alpha$, for $(\sigma,\alpha)\in \Gamma_{\ell} \times \Delta_H$, and $t_{-w_\ell \sigma_H(\alpha)}=-t_\alpha$, for $(\sigma,\alpha)\in (\Gamma_k\setminus \Gamma_{\ell}) \times \Delta_H$. Thus, we can find a pinning fixed by the Galois action on $\widehat{G}_X$ unless there exists $(\sigma,\alpha)\in (\Gamma_k\setminus \Gamma_{\ell}) \times \Delta_H$ such that $w_\ell \sigma_H(\alpha)=-\alpha$. Thus, in the case where $H$ is split such a pinning exists if and only if there is no simple root $\alpha\in \Delta_H$ with $w_\ell \alpha=-\alpha$. For example, this last condition is satisfied for $H=\GL(n)$ exactly for $n$ odd.
	
	By definition, $\iota_X$ extends to a morphism
	$$\displaystyle {}^L\iota_X: {}^L G_X=\widehat{G}_X\rtimes \Gamma\to {}^L G=\widehat{G}\rtimes \Gamma.$$
	
	This morphism is related to quadratic base-change (with respect to $\ell/k$) in the following way. Let $H^{\operatorname{op}}$\index{$H^{\operatorname{op}}$} be the quasi-split outer form of $H$ over $k$ whose dual group is $\widehat{H}$ with the Galois action given by the pinned version of $\sigma\mapsto \sigma_X$ (that is the action given by the formula \eqref{Galois action GSspaces} where we replace the duality involution by the, pinning preserving, Chevalley involution). The base-change of $H^{\operatorname{op}}$ to $\ell$ is then the quasi-split inner form of $H_\ell$ and so ${}^L G={}^L (R_{\ell/k} H^{\operatorname{op}}_\ell)$. In the case where the Galois action on $\widehat{G}_X$ fixes a pinning $\Pin_X$, the natural inner class of isomorphisms $\widehat{G}_X\simeq \widehat{H}^{\operatorname{op}}$ contains a unique one sending $\Pin_X$ to $\Pin_{H^{\operatorname{op}}}$ which then extends to an isomorphism of $L$-groups ${}^L G_X\simeq {}^L H^{\operatorname{op}}$ that is the identity on $\Gamma$. Via this isomorphism, the morphism ${}^L \iota_X$ then corresponds to the usual base-change morphism, up to $\widehat{G}$-conjugacy. 
	
	On the other hand, when there is no pinning of $\widehat{G}_X$ fixed by the Galois action, ${}^L \iota_X$ is related to what is sometimes called an {\em unstable base-change} morphism as illustrated by the following example. Let $H=\GL_n$ with $n$ even and assume that $k$ is a local field. Then, the outer form $H^{\operatorname{op}}=\mathrm{U}_n$ is a quasi-split unitary group in $n$ variables. The Galois action on $\widehat{G}_X=\GL_n(\bC)$ factors through the finite Galois group $\Gal(\ell/k)$ and the nontrivial element there $c\in \Gal(\ell/k)$ acts by
	$$\displaystyle h\in \GL_n(\bC)\mapsto h^\vee=w{}^t h^{-1}w,$$
	where $w=\begin{pmatrix} & & 1 \\ & \ddots & \\ 1 & & \end{pmatrix}$. Since $n$ is even, there is a simple root $\alpha$ such that $X_\alpha^\vee=-X_\alpha$ so that this action doesn't preserve any pinning. On the other hand, the $L$-group of $\mathrm{U}_n$ is given by the $\Gal(\ell/k)$-action on $\GL_n(\bC)$ with $c$ acting by
	$$\displaystyle h\in \GL_n(\bC)\mapsto h^\vee=J{}^t h^{-1}J^{-1},$$
	where $J=\begin{pmatrix} & & 1 \\ & \ddots & \\ (-1)^{n-1} & & \end{pmatrix}$, which does preserve the standard pinning. However the Weil form of the $L$-groups $\widehat{G}_X\rtimes W_k$ and $\widehat{H}^{\operatorname{op}}\rtimes W_k$ {\em are} isomorphic. Indeed, let $\mu:\ell^\times\to \bC^\times$ be a continuous character whose restriction to $k^\times$ is the quadratic character $\eta_{\ell/k}$ associated to the extension $\ell/k$ by local class field theory. We may identify $\mu$ with a character of the Weil group $W_\ell$ of $\ell$ by the Artin reciprocity map $W_\ell^{ab}\simeq \ell^\times$. Then, an explicit isomorphism is given by the morphism
	$$\displaystyle \widehat{G}_X\rtimes W_k\to \widehat{H}^{\operatorname{op}}\rtimes W_k$$
	that is the identity on the neutral component and sends $\sigma\in W_k$ to $(J,\sigma)$ if $\sigma\in W_k\setminus W_\ell$, $(\mu(\sigma),\sigma)$ if $\sigma\in W_\ell$. However, through this isomorphism, ${}^L \iota_X$ does not correspond to the usual base-change morphism but rather to its twists by the character $\mu$ (seen as a $1$-cocyle $W_k\to Z(\widehat{G})$). The functorial lift associated to this twisted morphism is what is usually called unstable base change; it is given by the usual base-change followed by a twist by $\mu\circ \det$.
\end{paragr}

\subsection{Boundary degenerations}\label{sect boundary degenerations}

The {\it boundary degenerations} are a families of spherical varieties $(X_{\Theta})_{\Theta\subset \Delta_X}$ attached to $X$ that are important tools for the development of harmonic analysis on $X(k)$ when $k$ is a local field. In some very rough approximation they model the geometry of $X$ at infinity but they are simpler in the sense that they have more symmetries (i.e.\ their centers are bigger). Recall that $X$ is assumed quasi-affine throughout. We also assume, until the end of this section, that the group $G$ is \textbf{split}. Indeed, this allows to get rid of all subtleties related to descents over the base field $k$ (that we don't assume to be algebraically closed anymore). The theory of boundary degenerations admits an extension to nonplit $G$, see \cite{KKRGA}, but this requires introducing even more notation (starting with the introduction of a suitable set of {\em relative} spherical roots).

\vspace{2mm}

\begin{paragr}[Asymptotic cones.]\label{S Asymptotic cones}
	
	Recall the isotypic decomposition of the ring of regular functions on $X$ from \S \ref{S decomp algebra of regular fns} (which descents to $k[X]$ thanks to the assumption that $G$ is split)
	$$\displaystyle k[X]=\bigoplus_{\lambda\in X^*(A_X)^+} k[X]_\lambda.$$\index{$k[X]_\lambda$}
	As already mentioned, this does not induce a grading but rather a filtration of the algebra $k[X]$ for the order associated to the root cone ${}^+\mathcal{A}_X^*$ i.e.\ the order to which $\lambda\preceq \mu$ if and only if $\mu-\lambda\in {}^+\mathcal{A}_X^*$. The associated graded $\mathrm{gr}(k[X])$ is the ring of regular functions on some (affine) $G$-variety $X_\emptyset^{\aff}$\index{$X_\emptyset^{\aff}$} which is readily seen to be again spherical (e.g. because its affine ring is multiplicity free). Its unique open $G$-orbit $X_\emptyset\subset X_\emptyset^{\aff}$\index{$X_\emptyset$} is a horospherical variety called the {\it asymptotic cone} of $X$. 
	
	\begin{exe}\label{exe1}
		\begin{itemize}
			\item In the case of the hyperboloid
			$$\displaystyle X=SO(2)\backslash SO(2,1)=\{(x,y,z)\in \mathbb{A}^3\mid x^2+y^2-z^2=-1 \},$$
			$X_{\emptyset}^{\aff}$ can be identified with the isotropic cone
			$$\displaystyle X_{\emptyset}^{\aff}=\{ (x,y,z)\in \mathbb{A}^3\mid x^2+y^2-z^2=0 \}$$
			and the asymptotic cone $X_{\emptyset}$ is simply $X_{\emptyset}^{\aff}$ minus the origin. This example should intuitively illustrate the fact that $X_{\emptyset}$ is supposed to model $X$ close to infinity. Note that in this case $X_\emptyset$ is equipped with an extra $\mathbb{G}_m$-symmetry given by scalar multiplication.
			
			\item In the group case $X=\PGL_2$ (with $G=\PGL_2\times \PGL_2$), the asymptotic cone is $X_\emptyset= (N\times A^{diag}\times \overline{N})\backslash \PGL_2\times \PGL_2$ and can also be identified with the variety of rank one $2$ by $2$ matrices up to the scalar action of $\mu_2$.
		\end{itemize}
		
	\end{exe}
\end{paragr}

\begin{paragr}[Canonical affine degeneration and boundary degenerations.]\label{S canonical affine degeneration}
	More generally, we can let the filtration on $k[X]$ degenerate in various directions, corresponding to the faces of ${}^+\mathcal{A}^*_X$ or equivalently of its dual cone $\mathcal{A}_X^+=-\cA_X^-$ (that is, the positive Weyl chamber). Note that these faces are themselves naturally parametrized by subsets $\Theta\subset \Delta_X$ and the corresponding degeneration will yield the boundary degeneration $X_\Theta$.\index{$X_\Theta$}
	
	Formally, they can be defined via the following {\it affine degeneration} (again using the terminology introduced by Sakellaridis-Venkatesh). Consider the subalgebra of $k[A_X\times X]=k[A_X]\otimes_k k[X]$ given by the direct sum
	$$\displaystyle \bigoplus_{\lambda\in X^*(A_X)^+} t^\lambda k[X]_{\preceq \lambda}=\bigoplus_{\lambda\in X^*(A_X)^+} k[A_X]_{\succeq \lambda}\otimes k[X]_\lambda$$
	where $t^\lambda$ is a symbol for the character $\lambda$ seen as an element of $k[A_X]$, $k[X]_{\preceq \lambda}=\bigoplus_{\mu\preceq \lambda} k[X]_{\mu}$\index{$k[X]_{\preceq \lambda}$} are the pieces of the aforementioned filtration on $k[X]$ and similarly  $k[A_X]_{\succeq \lambda}=\bigoplus_{\mu\succeq \lambda} kt^\mu$\index{$k[A_X]_{\succeq \lambda}$}. The spectrum of that subalgebra is an affine variety $\mathfrak{X}^{\aff}$\index{$\mathfrak{X}^{\aff}$} equipped with an action of $A_X\times G$ (by convention, we let $G$ act on the right and $A_X$ on the left of $\mathfrak{X}^{\aff}$). Moreover, there is a natural $G$-invariant and $A_X$-equivariant morphism 
	\begin{equation}\label{eq2}
		\displaystyle \mathfrak{X}^{\aff}\to \overline{A}_{X,ad}
	\end{equation}
	where $\overline{A}_{X,ad}=Spec(k[A_X]_{\succeq 0})$\index{$\overline{A}_{X,ad}$} is the affine toric embedding of $A_{X,ad}:=A_X/Z(X)^0$\index{$A_{X,ad}$} corresponding to the cone $\mathcal{A}_X^+$. In other words, it is the unique normal equivariant embedding of $A_{X,ad}$ such that for $\lambda\in X_*(A_X)$, $x_\lambda:=\lim\limits_{t\to 0} \lambda(t)$ exists in $\overline{A}_{X,ad}$ if and only if $\lambda\in \mathcal{A}_X^+$. 
	
	For every $t\in \overline{A}_{X,ad}$, the fiber $\mathfrak{X}^{\aff}_t$ is a spherical $G$-variety and in particular contains an open $G$-orbit. Let $\mathfrak{X}\subset \mathfrak{X}^{\aff}$\index{$\mathfrak{X}$} be the union of the open $G$-orbits in the fibers of the map \eqref{eq2}. Then, $\mathfrak{X}$ is an open subset on which $A_X$ is acting freely. This can be seen as follows. Fixing a Borel $B\subset G$, the fiberwise union $\mathfrak{X}_B$\index{$\mathfrak{X}_B$} of the open $B$-orbits is the affine open subset cut out by the relations $t^\mu f_\mu\neq 0$ where $f_\mu$ runs over a set of generators of the monoid $k[X]^{(B)}$ of $B$-eigenfunctions on $X$ and $\mu\in X^*(A_X)^+$ denotes the corresponding character. As $\mathfrak{X}$ is the union of the $G$-translates of $\mathfrak{X}_B$ it is also open. Note also that $k[X_B]$ inherits a filtration $(k[X_B]_{\preceq \lambda})_{\lambda\in X^*(A_X)}$ from that on $k[X]$ by setting
	$$\displaystyle k[X_B]_{\preceq \lambda}=\bigcup_{f_\mu\in k[X]^{(B)}} \frac{k[X]_{\preceq \lambda+\mu}}{f_\mu}$$\index{$k[X_B]_{\preceq \lambda}$}
	and that $\mathfrak{X}_B\to \overline{A}_{X,ad}$ is obtained in the same way as $\mathfrak{X}^{\aff}$ by letting this filtration degenerates in a multidimensional way. However, fixing a presentation of $X_B$ as in \eqref{eq LST X} the resulting action of $A_X$ on $X_B$ gives a graduation on $k[X_B]$ that is readily seen to split the above filtration. Therefore, the affine degeneration $\mathfrak{X}_B$ splits i.e.\ we have a (non-canonical) isomorphism
	\begin{equation}\label{eq3}
		\displaystyle \mathfrak{X}_B\simeq \overline{A}_{X,ad}\times X_B\simeq \overline{A}_{X,ad}\times A_X\times N_X
	\end{equation}
	through which the map $\mathfrak{X}_B\to \overline{A}_{X,ad}$ is the second projection and the $A_X$-action is the diagonal one. In particular, this shows that $A_X$ acts freely on $\mathfrak{X}_B$ and therefore also on $\mathfrak{X}$ (since, again, the $G$-translates of $\mathfrak{X}_B$ cover $\mathfrak{X}$). It also follows from the isomorphism \eqref{eq3} that the map $\mathfrak{X}\to \overline{A}_{X,ad}$ is smooth.
	
	We denote by
	$$\displaystyle \pi:\mathfrak{X}\to \overline{A}_{X,ad}$$
	the restriction of the map \eqref{eq2} to $\mathfrak{X}$. Then, the fiber of $\pi$ at $x_0=1$ is $X$ and, setting $\mathfrak{X}^{\bullet}=\pi^{-1}(A_{X,ad})$\index{$\mathfrak{X}^{\bullet}$} (that is the open $A_X\times G$-orbit in $\mathfrak{X}$), the $A_X$-action induces an isomorphism
	\begin{equation}\label{eq6}
		\displaystyle A_X\times^{Z(X)^0} X\simeq \mathfrak{X}^{\bullet}
	\end{equation}
	compatible with the natural projections to $A_{X,ad}$.
	
	For $\lambda\in X_*(A_X)\cap \mathcal{A}_X^+$, the limit point $x_\lambda\in \overline{A}_{X,ad}$ only depends on the subset $\Theta\subset \Delta_X$ of simple roots orthogonal to $\lambda$ (that is on the face of the cone $\cA_X^+$ containing $\lambda$ in its relative interior) and can thus be denoted by $x_\Theta$\index{$x_\Theta$}. Then, for each subset $\Theta\subset \Delta_X$, we define the {\em boundary degeneration} $X_\Theta$ to be the fiber of $\pi$ above $x_\Theta$. Note that for $\Theta=\Delta_X$ we get $X_{\Delta_X}=X$. Also, $X_\Theta$ comes naturally with a (left) action of the stabilizer $A_{X,\Theta}=\bigcap_{\alpha\in \Theta} \Ker(\alpha)^0$\index{$A_{X,\Theta}$} of $x_\Theta$ that commutes with that of $G$.
	
	The $A_X$-orbits in $\overline{A}_{X,ad}$ are naturally parametrized by subsets $\Theta\subset \Delta_X$, where the orbit corresponding to $\Theta$ is that of the point $x_\Theta$ and is isomorphic to $A_X^\Theta:=A_X/A_{X,\Theta}$\index{$A_X^\Theta$}. Denote by $\mathfrak{X}_\Theta$\index{$\mathfrak{X}_\Theta$} the preimage by $\pi$ of the closure $\overline{A_X^\Theta}$. It contains a unique open $A_X\times G$-orbit $\mathfrak{X}^{\bullet}_\Theta$\index{$\mathfrak{X}^{\bullet}_\Theta$} and, generalizing the above discussion for $\Theta=\Delta_X$, the $A_X$-action induces an isomorphism
	\begin{equation}\label{eq5}
		\displaystyle A_X\times^{A_{X,\Theta}} X_\Theta\simeq \mathfrak{X}^{\bullet}_\Theta
	\end{equation}
	compatible with the projections to $A_X^\Theta$.
\end{paragr}

\begin{paragr}[Invariants of $X_\Theta$.]\label{S invts Xtheta}
	From the identification \eqref{eq3} of the open $B$-orbits in the fiber of the affine degeneration, we see that $A_{X_\Theta}=A_X$ and therefore $P_{X_\Theta}=P_X$ (since $A_{X_\Theta}$ determines $P_{X_\Theta}$, see \S \ref{S parabolic type}). Furthermore, from the description of the negative Weyl cone $\cA_{X_\Theta}^-$ via the natural $G$-stable filtration on $k[X_\Theta]$, we readily get $\Delta_{X_\Theta}=\Theta$ from which it follows that $Z(X_\Theta)^0=A_{X,\Theta}$. In particular, we have $A_{X_\Theta,ad}=A_X^\Theta$ and the previously defined family $\mathfrak{X}_\Theta\to \overline{A}_X^{\Theta}$ identifies with the canonical affine degeneration of $X_\Theta$. This, in particular, implies that boundary degenerations have the following inductive property: for every $\Omega\subset \Theta$ the corresponding boundary degeneration of $X_\Theta$ is $(X_\Theta)_\Omega=X_\Omega$. Finally, when there is no root of type $N$, the dual group $\widehat{G}_{X_\Theta}$ is the standard Levi subgroup $\widehat{L}_\Theta$\index{$\widehat{L}_\Theta$} of $\widehat{G}_X$ associated to the subset $\Theta$ and the distinguished morphism $\iota_{X_\Theta}:\widehat{G}_{X_\Theta}\times \SL_2(\mathbb{C})\to \widehat{G}$ is the restriction of $\iota_X$.
\end{paragr}

\begin{paragr}[Wonderful compactification.]\label{S wonderful compactification}
	Assume for the moment that $Z(X)^0=\{ 1\}$. Since in this case $A_{X,ad}=A_X$ we will write $\overline{A}_X$\index{$\overline{A}_X$} instead of $\overline{A}_{X,ad}$. Then, the quotient
	$$\displaystyle \overline{X}:=A_X\backslash \mathfrak{X}$$\index{$\overline{X}$}
	is a $G$-variety naturally containing $X$ that is proper by \cite[Theorem 4.2]{KnopLV}\footnote{More precisely, with the notation and terminology of {\em op.\ cit.}, the embedding $X\hookrightarrow \overline{X}$ is associated to the colored fan $(\mathcal{V},\emptyset)$.}. Associating to $\Theta$ the $G$-orbit
	$$\displaystyle Y_\Theta:=A_X\backslash \mathfrak{X}_\Theta=A_{X,\Theta}\backslash X_\Theta,$$\index{$Y_\Theta$}
	we can again parametrize $G$-orbits in $\overline{X}$ by subsets $\Theta\subset \Delta_X$. We let $Z_\Theta$\index{$Z_\Theta$} be the closure of $Y_\Theta$ in $\overline{X}$. 
	
	Note that the compactification $X\hookrightarrow \overline{X}$ is completely canonical. Since $\mathfrak{X}\to \overline{A}_X$ is smooth and surjective, it is smooth if and only if the toric variety $\overline{A}_X$ i.e.\ iff the monoid $\cA_X^+\cap X_*(A_X)$ is free. 
	
	Assume now that the $G$-variety $\overline{X}$ is smooth. Then, it is sometimes called {\it the wonderful compactification} in the literature on symmetric varieties. The structure of the boundary $\overline{X}\setminus X$ is the same as that of $\overline{A}_{X}\setminus A_X$ and is easy to describe: for every $\alpha\in \Delta_X$, $D_\alpha=Z_{\alpha^c}$\index{$D_\alpha$} (where $\alpha^c\subset \Delta_X$ denotes the complement of $\{\alpha \}$) is a smooth divisor, these divisors meets transversally and for every $\Theta\subset \Delta_X$ we have $Z_\Theta=\bigcap_{\alpha\in \Theta^c} D_\alpha$. In this case, we simply set $\infty_\Theta:=Z_\Theta$\index{$\infty_\Theta$} and refer to this (smooth) subvariety as {\it $\Theta$-infinity}.

	\begin{exe}\label{exe2}
		\begin{itemize}
			\item For $X$ the same hyperboloid as in Example \ref{exe1}, the affine degeneration is
			$$\displaystyle \mathfrak{X}^{\aff}=\{(v,t)\in \mathbb{A}^3\times \mathbb{A}^1\mid q(v)=-t^2 \}$$
			where $q$ denotes the quadratic form $q(x,y,z)=x^2+y^2-z^2$, the $A_X=\mathbb{G}_m$-action on $\mathfrak{X}^{\aff}$ is given by scalar multiplication of $v$ and $t$ and the map $\pi: \mathfrak{X}^{\aff}\to \overline{A}_X=\mathbb{A}^1$ is $(v,t)\mapsto t$. The open subset $\mathfrak{X}\subset \mathfrak{X}^{\aff}$ is the complement of $(0,0)$ and the compactification $\overline{X}=\mathfrak{X}/\mathbb{G}_m$ is the projective variety defined by the same equation $q(v)=-t^2$.
			
			\item In the group case $X=\PGL_2$, we have
			$$\displaystyle \mathfrak{X}^{\aff}=\Mat_2\sslash \mu_2$$
			where $\mu_2$ acts on the space $\Mat_2$ of $2\times 2$ matrices by scalar multiplication, $A_X=\mathbb{G}_m/\mu_2\simeq \mathbb{G}_m$ acts on $\mathfrak{X}^{\aff}$ by multiplication and $\pi: \mathfrak{X}^{\aff}\mapsto \overline{A}_X=\mathbb{A}^1$ is given by the determinant. We get $\mathfrak{X}=\Mat_2^*/\mu_2$ (where $\Mat_2^*:=\Mat_2\setminus 0$) and $\overline{X}=\mathfrak{X}/\mathbb{G}_m=\mathbb{P}(\Mat_2)$.
		\end{itemize}
		
	\end{exe}
\end{paragr}

\begin{paragr}[Normal bundles and boundary degenerations: the wonderful case.]\label{S normal bundles and wonderful compactification}
	We continue to assume that $Z(X)^0=1$ and that the canonical compactification $\overline{X}$ is smooth. Then, there is the following alternative description of the boundary degenerations.
	\begin{num}
		\item\label{Normal bundle wonderful case} For every $\Theta\subset \Delta_X$, the normal bundle $N_{\infty_\Theta}(\overline{X})$\index{$N_{\infty_\Theta}(\overline{X})$} to $\infty_\Theta$ in $\overline{X}$ contains an open $G$-orbit that can be canonically identified with $X_\Theta$.
	\end{num}
	Indeed, since $\mathfrak{X}\to \overline{A}_X$ is smooth, $N_{\infty_\Theta}(\overline{X})$ can be identified with the $A_X$-quotient (for the diagonal action) of the pullback of the normal bundle $N_{\overline{A}_X^\Theta}(\overline{A}_X)$ to $\mathfrak{X}_\Theta$. However, it follows from the theory of toric embeddings that we have a canonical identification $N_{\overline{A}_X^\Theta}(\overline{A}_X)=\overline{A}_X$\footnote{More precisely, if $\lambda_1,\ldots,\lambda_r$ are the primitive generators of the extremal rays of $\cA_{X,\Theta}^-=\cA_X^-\cap \Ker(\Theta)$ and we set $\lambda=\lambda_1+\ldots+\lambda_r$, then the isomorphism $\overline{A}_X\simeq N_{\overline{A}_X^\Theta}(\overline{A}_X)$ is given by $a\mapsto (\lim\limits_{t\to 0}\lambda(t)a, \frac{d}{dt}(\lambda(t)a)\mid_{t=0}$).} so that, in particular, this normal bundle contains an open copy of $A_X$. We can now identify $X_\Theta$ with the open subset
	$$\displaystyle A_X\backslash(A_X\times_{A_X^\Theta} \mathfrak{X}_\Theta^\bullet)\simeq A_X\backslash (A_X\times X_\Theta)$$
	of $N_{\infty_\Theta}(\overline{X})$ where the isomorphism comes from \eqref{eq5}.
\end{paragr}

\begin{paragr}[Non-wonderful case: Toroidal embeddings.]\label{S toroidal embeddings}
	When $Z(X)^0\neq 1$, the quotient $A_X\backslash \mathfrak{X}$ is still complete but is only a compactification of the variety $X_{ad}:=X/Z(X)^0$ and when it is smooth the normal bundles to the $G$-orbits only allows to recover the quotients $X_\Theta/Z(X)^0$ of the boundary degenerations by the connected center of $X$. Therefore, to have a picture similar to the wonderful case in general, we would like to replace $A_X\backslash \mathfrak{X}$ by some smooth $G$-equivariant completion of $X$. One way to do this is to pullback the canonical affine degeneration $\mathfrak{X}\to \overline{A}_{X,ad}$ along a smooth resolution of $\overline{A}_{X,ad}$ on which $A_X$ is acting faithfully. The theory of toric embeddings \cite{Fulton} allows to construct such resolutions in a very combinatorial way and lead to the so-called {\it smooth toroidal compactification} of $X$. The advantage of smooth toroidal embeddings are that they always exist and have a geometry at infinity very similar to wonderful compactifications, however there are not unique which makes certain definitions less transparently canonical.
	
	More precisely, to every {\it fan} $\mathcal{F}$\index{$\mathcal{F}$} of strongly convex finitely generated cones\footnote{Recall that a cone $C$ is {\it strongly convex} if it doesn't contain any line and that a {\it fan} $\cF$ is a family finitely generated cones, stable by taking faces, and such that for every $C,D\in \cF$, $C\cap D$ is a common face of $C$ and $D$.} in $\cA_X$, we can associate a toric embedding $A_X\hookrightarrow\overline{A}_X^{\cF}$\index{$\overline{A}_X^{\cF}$} as in \cite{Fulton}. Furthermore, if the cones in $\cF$ are all included in the positive Weyl chamber $\cA_X^+$ then the quotient map $A_X\to A_{X,ad}$ extends to a morphism $\overline{A}_X^{\cF}\to \overline{A}_{X,ad}$ and we define
	$$\displaystyle \mathfrak{X}^{\cF}:=\overline{A}_X^{\cF}\times_{\overline{A}_{X,ad}} \mathfrak{X},\;\; \overline{X}^{\cF}:=A_X\backslash \mathfrak{X}^{\cF}$$\index{$\mathfrak{X}^{\cF}$}\index{$\overline{X}^{\cF}$}
	where the quotient is taken with respect to the {\it diagonal} action of $A_X$ on $\mathfrak{X}^{\cF}$. (It is again a free action.) Note that $\mathfrak{X}^{\cF}$ contains as an open subset
	\begin{equation*}
		\displaystyle A_X\times_{A_{X,ad}} \mathfrak{X}^{\bullet}\simeq A_X\times X
	\end{equation*}
	(where the isomorphism comes from \eqref{eq6})from which we get a natural $G$-equivariant open embedding $X\subset \overline{X}^{\cF}$.
	
	The morphism $\mathfrak{X}\to \overline{A}_{X,ad}$ being smooth, we see that the variety $\overline{X}^{\cF}$ is smooth if and only if the toric embedding $\overline{A}_X^{\cF}$ is so i.e.\ if and only if for every $C\in \cF$, the intersection $C\cap X_*(A_X)$ is a free monoid. Also, since the quotient $A_X\backslash \mathfrak{X}$ is complete and the two morphisms
	$$\displaystyle A_X\backslash \mathfrak{X}\leftarrow \mathfrak{X}\to \overline{A}_{X,ad}$$
	smooth, $\overline{X}^{\cF}$ is complete if and only if the morphism $\overline{A}^{\cF}_X\to \overline{A}_{X,ad}$ is proper and by the theory of toric embedding this happens if and only if the fan $\cF$ covers all of $\cA_X^+$.
	
	Embeddings of the form $X\hookrightarrow \overline{X}^{\cF}$ are named {\it toroidal} in the literature on spherical varieties (see e.g. \cite{KnopLV}) and can be alternatively characterized as the normal $G$-varieties containing $X$ as an open subset and such that none of the $B$-divisors in $X$ contains a $G$-orbit in its closure. By a {\it smooth toroidal compactification}, we mean a toroidal embedding that is both smooth and complete. By the above characterization in terms of the fan $\cF$, smooth toroidal compactifications are readily seen to exist (see \cite[Sect. 2.6]{Fulton}) but are however far from unique.
\end{paragr}

\begin{paragr}[The local structure theorem.]\label{S LST}
	The following ``local structure theorem'' of Brion-Luna-Vust \cite[th\'eor\`eme 3.5]{BLV} is a basic tool in the study of toroidal embeddings.
	
	\begin{theo}[Brion, Luna, Vust]\label{thm LST}
		Let $\overline{X}=\overline{X}^{\cF}$ be the toroidal embedding of $X$ associated to the fan $\cF$ as above. Then, the union $\overline{X}_B$\index{$\overline{X}_B$} of the open $B$-orbits in each $G$-orbit of $\overline{X}$ is an open $P_X$-stable subset on which $N_X$ acts freely. Up to choosing a base-point, the quotient $A_{\overline{X}}:=\overline{X}_B/N_X$\index{$A_{\overline{X}}$} is isomorphic to the opposite of the toric embedding $\overline{A}^{\cF}_X$ of $A_X$ i.e.\ the toric variety $\overline{A}^{\cF}_X$ with $A_X$ acting via composition with the inversion $a\mapsto a^{-1}$ or equivalently the toric embedding associated to the opposite fan $-\cF$. Moreover, the (non-canonical) identification \eqref{eq LST X} extends to a $P_X$-equivariant isomorphism
		$$\displaystyle \overline{X}_B\simeq A_{\overline{X}}\times N_X.$$
	\end{theo}
	
	Actually, from the perspective of affine degeneration this result can be explained as follows. Set
	$$\displaystyle \mathfrak{X}_B^{\cF}:=\overline{A}_X^{\cF}\times_{\overline{A}_{X,ad}} \mathfrak{X}_B.$$\index{$\mathfrak{X}_B^{\cF}$}
	Then, $\mathfrak{X}_B^{\cF}$ is the fiberwise union of open $B$-orbits for $\mathfrak{X}^{\cF}\to \overline{A}_X^{\cF}$ and we have $\overline{X}_B=A_X\backslash \mathfrak{X}_B^{\cF}$. Moreover, by \eqref{eq3} we have a (non-canonical) isomorphism
	\begin{equation}\label{eq4}
		\displaystyle \mathfrak{X}_B^{\mathcal{F}}\simeq \overline{A}^{\cF}_{X}\times A_X\times N_X
	\end{equation}
	which after taking a quotient by $A_X$ gives the isomorphism of the local structure theorem.
\end{paragr}

\begin{paragr}[Orbits and $\Theta$-infinity: toroidal case.]\label{S theta infinity}
	
	Let $\overline{X}=\overline{X}^{\cF}$ be a smooth toroidal compactification of $X$ as above. Then, the map $\mathfrak{X}^{\cF}\to \overline{A}_X^{\cF}$ induces a bijection between $G$-orbits in $\overline{X}$ and $A_X$-orbits in $\overline{A}_X^{\cF}$. By the theory of toric embeddings, the latter are naturally parametrized by the cones in $\mathcal{F}$ and, for $C\in \mathcal{F}$, the $A_X$-orbit associated to $C$ is canonically isomorphic to $A_X^C:=A_X/A_{X,C}$\index{$A_X^C$} where $A_{X,C}:=\bigcap_{\chi\in X^*(A_X)\cap C^\vee} \Ker(\chi)^0$\index{$A_{X,C}$}. The $G$-orbit $Y_C$\index{$Y_C$} and its closure $Z_C$\index{$Z_C$} corresponding to $C$ in $\overline{X}$ can then identified with
	$$\displaystyle Y_C=A_X\backslash \mathfrak{X}_C^{\bullet} \; \mbox{ and } \; Z_C:=A_X\backslash \mathfrak{X}_C$$
	where
	$$\displaystyle \mathfrak{X}^{\bullet}_C:=A_X^C\times_{\overline{A}_X^{\cF}} \mathfrak{X}^{\cF}\subset \mathfrak{X}_C:=\overline{A}_X^C\times_{\overline{A}_X^{\cF}} \mathfrak{X}^{\cF}$$\index{$\mathfrak{X}^{\bullet}_C$}\index{$\mathfrak{X}_C$}
	and $\overline{A}_X^C$\index{$\overline{A}_X^C$} stands for the closure of $A_X^C$ in $\overline{A}_X^{\cF}$.\footnote{We note a slight conflict of notation here: $\overline{A}_X^C$ is not the toric embedding associated to (the fan of faces of) $C$. We will try to always distinguish between cones and fans hoping that this will not create any confusion.}
	
	The combinatorics of the $G$-orbits in $\overline{X}$ is also the same as that of $A_X$-orbits in $\overline{A}_X^{\cF}$ that is: the orbit closure $D_L=Z_L$\index{$D_L$} corresponding to each half-line $L\in \cF$ are smooth divisors meeting transversely and for every $C\in \cF$ we have
	$$\displaystyle Z_C=\bigcap_{L\subset C} D_L$$
	where the intersection runs over extremal rays of $C$.
	
	Let $\Theta\subset \Delta_X$. We say that a cone $C\in \cF$ {\it belongs to the $\Theta$-face} if $\Theta$ is exactly the set of simple roots $\alpha\in \Delta_X$ orthogonal to $C$ (or in other words if the minimal face of $\cA_X^+$ containing $C$ is the one corresponding to $\Theta$). This is equivalent to saying that the image of $A_X^C$ by the map $\overline{A}_X^{\cF}\to \overline{A}_{X,ad}$ is $A_X^\Theta$ and therefore from \eqref{eq5} we get isomorphisms
	\begin{equation}\label{eq8}
		\displaystyle  \mathfrak{X}^{\bullet}_C \simeq A_X\times^{A_{X,C}} X_\Theta \subset \mathfrak{X}_C\simeq \overline{A}_X^C\times_{\overline{A}_X^\Theta} \mathfrak{X}_\Theta
	\end{equation}
	from which it follows that $Y_C\simeq A_{X,C}\backslash X_\Theta$.
	
	Let $\infty_\Theta^\circ$\index{$\infty_\Theta^\circ$} be the union of the $G$-orbits $Y\subset G$ whose corresponding cone belongs to the $\Theta$-face and by $\infty_\Theta$ the union of $\infty_\Omega^\circ$ over all subsets $\Omega\subset \Theta$. Note that $\infty_\Theta$ is closed but is not necessarily the closure of $\infty_\Theta^\circ$.
\end{paragr}

\begin{paragr}[Identifications of normal bundles.]\label{S ident normal bundles}
	To $\overline{X}$ we associate the smooth toroidal embeddings $\overline{X}_\Theta:=\overline{X}_{\Theta}^{\cF}$\index{$\overline{X}_\Theta$} of the boundary degenerations $X_\Theta$ corresponding to the same fan $\cF$. Since (see \S \ref{S invts Xtheta}) $\cA_X^+\subset \cA_{X_\Theta}^+$ we see that $\cF$ is indeed a fan in $\cA_{X_\Theta}^+$ so that $\overline{X}_\Theta$ is well-defined. Moreover, it can be obtained as before using the canonical affine degeneration $\mathfrak{X}_\Theta\to \overline{A}_X^\Theta$:
	\begin{equation*}
		\displaystyle \overline{X}_\Theta=A_X\backslash \mathfrak{X}_\Theta^{\cF} \;\;\mbox{ where }\;\; \mathfrak{X}_\Theta^{\cF}=\overline{A}_X^{\cF}\times_{\overline{A}_X^\Theta} \mathfrak{X}_\Theta.
	\end{equation*}
	As for $\overline{X}$, every cone $C\in \cF$ corresponds to the closure of a $G$-orbit $Z_{\Theta,C}\subset \overline{X}_\Theta$\index{$Z_{\Theta,C}$}. When $C$ belongs to the $\Omega$-face for some subset $\Omega\subset \Theta$, the isomorphism \eqref{eq8} (and its analog for $X_\Theta$) gives an identification $Z_C\simeq Z_{\Theta,C}$. In other words, we can naturally identify $\infty_\Theta$ with a subvariety of $\overline{X}_\Theta$.
	
	As a direct generalization of \eqref{Normal bundle wonderful case}, if $Y\subset \infty_\Theta^\circ$ is a $G$-orbit, the normal cone $N_{Y}(\overline{X})$ contains an open $G$-orbit which can be canonically identified with the boundary degeneration $X_\Theta$. Actually, by the same kind of argument we have the following more general statement: for every $G$-orbit closure $Z\subset \infty_\Theta$ there is a natural identification
	\begin{equation}\label{can iden normal bundles}
		\displaystyle N_Z(\overline{X})\simeq N_Z(\overline{X}_\Theta).
	\end{equation}
\end{paragr}

\begin{paragr}[Toral actions on $X_\Theta$.]\label{S toral actions}
	Let $C\in \cF$ and $D_1,\ldots, D_r$ be the $G$-divisors in $\overline{X}$ corresponding to the extremal rays $L_1,\ldots,L_r$ of $C$ so that $\bigcap_{i=1}^r D_i=Z_C$. For each $1\leq i\leq r$, $L_i$ has a unique generator $\lambda_i$ that is primitive in the lattice of cocharacters $X_*(A_X)$. Then, as $C\cap X_*(A_X)$ is freely generated as a monoid by $\lambda_1,\ldots,\lambda_r$, the tuple $(\lambda_1,\ldots,\lambda_r)$ induces an isomorphism
	\begin{equation}\label{eq7}
		\displaystyle \mathbb{G}_m^r\simeq A_{X,C}.
	\end{equation}
	Let $\Theta\subset \Delta_X$ be such that the cone $C$ belongs to the $\Theta$-face. Then, $A_{X,C}\subset A_{X,\Theta}$ and the resulting action on $X_\Theta\subset N_{Z_C}(\overline{X})$ can be described as follows: through the isomorphism \eqref{eq7}, the action of the $i$-th $\mathbb{G}_m$-factor on $N_{Z_C}(\overline{X})$ is induced by the inverse of the natural scaling action on $N_{D_i}(\overline{X})$ via the natural identification $N_{Z_C}(\overline{X})\simeq N_{Z_C}(N_{D_i}(\overline{X}))$.
\end{paragr}

\begin{paragr}[The wavefront case.]\label{S wavefront case}
	By \cite[Proposition 2.7.2]{SV}, if $X$ is wavefront then for every $\Theta\subset \Delta_X$ (including $\Theta=\Delta_X$!), there exists a parabolic subgroup $P_\Theta^-=L_\Theta U_\Theta^-$\index{$P_\Theta^-$}\index{$L_\Theta$}\index{$U_\Theta^-$} such that $X_\Theta$ can be obtained by parabolic induction from an $L_\Theta$-variety $X^L_\Theta$\index{$X^L_\Theta$} i.e.\
	\begin{equation}\label{eq boundary deg parabind}
		\displaystyle X_\Theta=X^L_\Theta\times^{P_\Theta^-} G
	\end{equation}
	with the natural morphism between (connected) centers $Z(L_\Theta)^0\to Z(X_\Theta)^0$ surjective. Actually, if $Z(G)^0\to Z(X)^0$ is already surjective, this last property characterizes wavefront varieties. (And in general it characterizes spherical varieties that are parabolically induced from wavefront.)
	
	More precisely, we can take for $P_\Theta^-$ the parabolic opposite to the standard parabolic subgroup $P_\Theta=L_\Theta U_\Theta\supset P_X$\index{$P_\Theta$}\index{$U_\Theta$} corresponding to the set of simple roots $\Supp(\Theta)\cup \Delta(X)$\index{$\Supp(\Theta)$}, where $\Supp(\Theta)\subset \Delta$ is the minimal subset of simple roots containing $\Theta$ in their span and $\Delta(X)\subset \Delta$\index{$\Delta(X)$} is the subset corresponding to the parabolic $P_X$. Moreover, by \cite[Lemma 2.8.1]{SV}, the Levi variety can be described as the quotient $X_{P_\Theta}/U_\Theta$ of the open $P_\Theta$-orbit $X_{P_\Theta}=X_B P_\Theta$\index{$X_{P_\Theta}$} by its unipotent radical.
	
	\begin{exe}\label{exe3}
		When the variety $X$ is symmetric, corresponding to an involution $\iota$, the parabolic subgroups $P_\Theta$ are exactly the (conjugacy class of) {\it $\iota$-split} parabolics\footnote{Recall that a parabolic $P\subset G$ is said to be $\iota$-split if $\iota(P)$ is opposite to $P$} and, choosing the Levi factor $L_\Theta=P_\Theta\cap \iota(P_\Theta)$, the Levi variety $X^L_\Theta$ is the symmetric variety $L_\Theta^\iota\backslash L_\Theta$.
		
		For example, in the group case, $X=H$, $G=H\times H$, $\Delta$ is the disjoint union of two copies of the set $\Delta_H$ of simple roots for $H$, and we have $\Delta_X=\{\alpha_1+(-w_\ell \alpha)_2\mid \alpha\in \Delta_H \}$, where $w_\ell$ denotes the longest element in the Weyl group of $H$ and the subscripts $1,2$ are here to distinguih the two copies of $\Delta_H$. In particular, we can identify (by the first projection), $\Delta_X$ with $\Delta_H$ and the boundary degeneration associated to $\Theta\subset \Delta_H$ is
		$$\displaystyle X_\Theta=M_\Theta\times^{P_\Theta^-} G=(N^-_\Theta\times M_\Theta^{diag}\times N_\Theta)\backslash H\times H$$
		where $Q_\Theta=M_\Theta N_\Theta$ is (the class of) the parabolic subgroup of $H$ associated to $\Theta$, $Q^-_\Theta=M_\Theta N^-_\Theta$ the opposite parabolic and $P_\Theta^-=Q_\Theta^-\times Q_\Theta$.
	\end{exe}
\end{paragr}

\section{Local aspects}\label{Part local aspects}

\begin{paragr}[General notation.]\label{S local aspects notations}
	In this chapter, $k$ is a local field of characteristic zero and $X=H\backslash G$ is a homogeneous quasi-affine $G$-spherical variety that is unimodular (see \S \ref{S def spherical}). In particular, $X(k)$ admits $G(k)$-invariant measures and we will fix one for convenience, so that we can identify smooth measures/half-densities with functions all along.
	
	We will denote by $\mathcal{L}_k$\index{$\mathcal{L}_k$} the {\it Langlands group} of $k$ that is
	$$\displaystyle \mathcal{L}_k=\left\{\begin{array}{ll}
		W_k & \mbox{ if } k \mbox{ is Archimedean,} \\
		W_k \times \SL_2(\bC) & \mbox{ if } k \mbox{ is non-Archimedean,}
	\end{array} \right.$$
	where $W_k$\index{$W_k$} stands for the {\it Weil group} of $k$. In the non-Archimedean case, we will also denote by $I_k\subset W_k$\index{$I_k$} the inertia subgroup.
\end{paragr}

\begin{paragr}[$L$-parameters.]\label{S Lparameters}
	At some points, it will be more convient to work with the Weil form of $L$-groups. Thus, provided $X$ has no root of type $N$, we redefine $L$-groups by
	$$\displaystyle {}^L G=\widehat{G}\rtimes W_k,\;\;\; {}^L G_X=\widehat{G}_X\rtimes W_k$$
	where in the latter case the Weil action comes from the Galois action on $\widehat{G}_X$ discussed in \S \ref{S L-groups}.
	
	If $\mathcal{G}$ is an extension of $W_k$ by a connected complex reductive group $\mathcal{G}^0$ (typically, ${}^L G$ or ${}^L G_X$ as above), by a {\it $L$-parameter} in $\mathcal{G}$ we mean as usual a $\mathcal{G}^0$-conjugacy class of continuous semisimple morphisms $\mathcal{L}_k\to \mathcal{G}$ that are algebraic on $\SL_2(\bC)$ (when $k$ is non-Archimedean) and commute with the two natural projections $\mathcal{L}_k\to W_k$, $\mathcal{G}\to W_k$.
	
	We will denote by $\Phi(G)$\index{$\Phi(G)$} and $\Phi(X)$\index{$\Phi(X)$} the set of $L$-parameters in ${}^L G$ and ${}^L G_X$ up to $\widehat{G}$- and $\widehat{G}_X$-conjugacy respectively. The subsets of {\it tempered $L$-parameters}, that is parameters $\phi$ such that the projection of $\phi(W_k)$ to $\widehat{G}$ or $\widehat{G}_X$ is bounded, will be denoted by $\Phi_{\temp}(G)$\index{$\Phi_{\temp}(G)$} and $\Phi_{\temp}(X)$\index{$\Phi_{\temp}(X)$} respectively.
\end{paragr}

\begin{paragr}[Representations.]\label{S representations}
	All the representations $(\pi,V_\pi)$ of the group $G(k)$ to be considered in this section have complex coefficients and, unless explicitly said, are {\it smooth}. When $k$ is non-Archimedean this means, as usual, that every vector $v\in V_\pi$ has an open stabilizer. When $k$ is Archimedean on the other hand, we will use the term `smooth' as a shorthand for {\it admissible smooth Fr\'echet and of moderate growth} in the sense of Casselman and Wallach. By the main result of \cite{CassExt} \cite[Chap. 11]{WRRG2}, the category of such topological representations is equivalent to that of Harish-Chandra modules but with the advantage that it consists of genuine representations of the group $G(k)$ (and not just of its Lie algebra). Furthermore, it turns out that for the questions we will discuss the Fr\'echet topology on $V_\pi$ will really matter. In particular, when the field $k$ is Archimedean all intertwining operators as well as linear/bilinear/sesquilinear forms on representations will implicitely be assumed to be continuous.
	
	We also briefly need to consider unitary representations in Section \ref{Sect distinction and Plancherel}. When $(\Pi,\cH)$ is a unitary representation of $G(k)$ (with $\cH$ a Hilbert space), the subspace of smooth vectors (or $C^\infty$ vectors in the Archimedean case) $\cH^\infty$\index{$\cH^\infty$} is always a smooth representation in the above sense and will be denoted by $\Pi^\infty$\index{$\Pi^\infty$}. More generally, we will often confuse a representation with the space on which it is realized.
	
	We denote by $\Irr(G)$\index{$\Irr(G)$} the set of isomorphism classes of smooth irreducible representations of $G(k)$\footnote{Of course, irreducible is to be taken in the topological sense (i.e.\ no closed nontrivial invariant subspace) in the Archimedean case and the admissibility condition is superfluous in the non-Archimedean case.}. Every $\pi\in \Irr(G)$ has a smooth contragredient $\pi^\vee\in \Irr(G)$\index{$\pi^\vee$}. In the non-Archimedean case, it is just the space of smooth linear functionals with the obvious $G(k)$-action whereas in the Archimedean case this can either be described using Casselman-Wallach equivalence with Harish-Chandra modules or directly as the space of linear forms on $V_\pi$ that are continuous with respect to {\it every} $G(k)$-continuous norm\footnote{Here we say that a continuous norm $q$ on $V_\pi$ is {\it $G(k)$-continuous} is the action map $G(k)\times V_\pi\to V_\pi$ is continuous when we equip $V_\pi$ with the topology induced by $q$.}.

	We let $\Temp(G)\subset \Irr(G)$\index{$\Temp(G)$} the subset of {\it tempered} representations. There are actually various possible characterizations of the tempered representations: either by a growth condition on matrix coefficients (see \eqref{ineq MC tempered} below) or using the support of the Plancherel measure (see \S \ref{S remark Plancherel}). 
	
	We will also denote by $\Pi_2(G)\subset \Temp(G)$\index{$\Pi_2(G)$} the subset of square-integrable representations (aka discrete series) and by $\Pi_{2,\operatorname{ess}}(G)\subset \Irr(G)$\index{$\Pi_{2,\operatorname{ess}}(G)$} the subset of essentially square-integrable representations. We remind the reader that an irreducible representation $\pi\in \Irr(G)$ is said to be square-integrable if it has a unitary central character and the square of the norm of any of its (smooth) matrix coefficients is integrable on $G(k)/Z(G)(k)$. An irreducible representation $\pi\in \Irr(G)$ is essentially square-integrable if it admits an unramified twist that is square-integrable.
	
	We write $\Unit(G)$\index{$\Unit(G)$} for the {\it unitary dual} of $G(k)$ i.e.\ the set of isomorphism classes of irreducible unitary representations of $G(k)$. Recall that the assignment $\Pi\mapsto \Pi^\infty$ induces a natural embedding $\Unit(G)\hookrightarrow \Irr(G)$ whose image consists of {\it unitarizable} smooth irreducible representations. This image contains in particular $\Temp(G)$ and we will sometimes implicitely identify $\Unit(G)$ with the corresponding subset of $\Irr(G)$.
\end{paragr}

\begin{paragr}[The local Langlands correspondence.]\label{S LLC}
	We will need at time to assume a suitable form of the local Langlands correspondence (LLC) for $G$ (and some closely related groups such as Levi subgroups or (pure) inner forms). For a modern treatment of what this correspondence should look like in general we refer to Kaletha and Ta\"ibi's articles in these proceedings. Let us nevertheless recall some of the simplest features of this correspondence, partly to set up notation. The basic requirement is that there should be a partition of the set of irreducible representations into finite sets $\Pi^G(\phi)$\index{$\Pi^G(\phi)$} called {\it $L$-packets}, that is:
	$$\displaystyle \Irr(G)=\bigsqcup_{\phi\in \Phi(G)} \Pi^G(\phi).$$
	Moreover, this correspondence should preserve temperedness i.e.\ it should restrict to a partition $\Temp(G)=\bigsqcup_{\phi\in \Phi_{\temp}(G)} \Pi^G(\phi)$ of the tempered dual in terms of tempered $L$-parameters.
	
	Most statement in the LLC become more natural if we take into account {\it pure inner forms}. Pure inner forms of $G$ are in natural bijection with the isomorphism classes of $G$-torsors over $k$, itself classified by the Galois cohomology set
	$$\displaystyle H^1(k,G):=H^1(\Gamma,G(\overline{k})).$$\index{$H^1(k,G)$}
	More precisely, for $\alpha\in H^1(k,G)$ and denoting by $T_\alpha$\index{$T_\alpha$} the corresponding $G$-torsor, where we let $G$ act on the left of $T_\alpha$, the corresponding pure inner form is $G_\alpha:=\Aut_G(T_\alpha)$\index{$G_\alpha$} of $G$-automorphisms of $T_\alpha$. This is a group that is well-defined up to inner automorphisms. Moreover, it has the same $L$-group as $G$ so that, by the LLC, every $L$-parameter $\phi\in \Phi(G)$ should give rise to an $L$-packet $\Pi^{G_\alpha}(\phi)\subset \Irr(G_\alpha)$ for $G_\alpha$. Then, for $\phi\in \Phi(G)$, we denote by
	$$\displaystyle \Pi(\phi)=\bigsqcup_{\alpha\in H^1(k,G)} \Pi^{G_\alpha}(\phi)$$\index{$\Pi(\phi)$}
	the corresponding $L$-packet extended over all pure inner forms.
	
	One advantage of considering pure inner forms is to obtain a nice parametrization of elements in an $L$-packet. Namely, assuming that $G$ is quasi-split and fixing a{\em Whittaker datum} $\mathfrak{w}=(N,\psi_N)$\index{$\mathfrak{w}$} for $G$ (that is a pair consisting of the unipotent radical $N$ of a Borel subgroup and a generic character $\psi_N: N(k)\to \bC^\times$), there should exist for every $\phi\in \Phi(G)$ a canonical bijection
	\begin{equation}\label{param Lpackets}
		\displaystyle \Irr(S_\phi)\simeq \Pi(\phi),
	\end{equation}\index{$\Irr(S_\phi)$}
	$$\displaystyle \rho\mapsto \pi_{\mathfrak{w}}(\phi,\rho),$$\index{$\pi_{\mathfrak{w}}(\phi,\rho)$}
	between the extended $L$-packet of $\phi$ and the set of (isomorphism classes of) irreducible representation of the finite component group $S_\phi=\pi_0(\Cent_{\widehat{G}}(\phi))$\index{$S_\phi$} of the centralizer $\Cent_{\widehat{G}}(\phi)$ of $\phi$ in $\widehat{G}$. Moreover, this bijection should satisfy many properties (in particular related to the theory of endoscopy), but at least when $\phi$ is tempered, it is expected to send the trivial representation $\mathbf{1}\in \Irr(S_\phi)$, to the unique representation in $\Pi^G(\phi)$ admitting a Whittaker model of type $\mathfrak{w}$.
	
	For general $X$, we will also need to consider {\it Arthur packets}. Recall that these are supposed to be associated to {\it $A$-parameters} that is morphisms
	$$\displaystyle \psi: \mathcal{L}_k\times \SL_2(\bC)\to {}^LG$$
	such that the restriction of $\psi$ to $\mathcal{L}_k$ is a tempered $L$-parameter and the restriction to $\SL_2(\bC)$ is algebraic. Then, to such an $A$-parameter, Arthur conjecture that we can associate a packet $\Pi^{G_\alpha}(\psi)$\index{$\Pi^{G_\alpha}(\psi)$} of unitary representations of the various pure inner forms $G_\alpha$ of $G$. These packets are of course expected to satisfy a number of properties, the most important one being in relation with decompositions of global spaces of automorphic forms, cf. again Kaletha and Ta\"ibi's article for more details on this.
\end{paragr}

\begin{paragr}[Function spaces on $X$.]\label{S function spaces}
	For convenience, we fix a $G(k)$-invariant measure on $X(k)$. Here are the various function spaces on $X(k)$ that will play a role in this chapter:
	\begin{itemize}
		\item $L^2(X)$\index{$L^2(X)$} is the space of $L^2$ half-densities on $X(k)$. Since we have fixed an invariant measure on $X(k)$, this can be identified with the space of square-integrable functions (with respect to the given measure). It is naturally a unitary representation of $G(k)$ (for the action by right translation). We will denote the $L^2$-pairing on $L^2(X)$ by $\langle .,.\rangle_X$\index{$\langle .,.\rangle_X$}:
		$$\displaystyle \langle \varphi_1,\varphi_2\rangle_X:=\int_{X(k)} \varphi_1(x)\overline{\varphi_2(x)} dx.$$
		
		\item The {\it Schwartz space} $\cS(X)$\index{$\cS(X)$}. In the non-Archimedean case, it is just the space $C_c^\infty(X)$\index{$C_c^\infty(X)$} of compactly supported smooth (i.e.\ locally constant) complex-valued functions on $X(k)$ (or better measures but the two can be identified again thanks to the invariant measure fixed at the beginning). In the Archimedean case, it is the space of functions that are of rapid decay with all their derivatives that is: $C^\infty$ functions $\varphi:X(k)\to \bC$ such that for every polynomial differential operator $D$ on $X$ we have
		$$\displaystyle \lVert \varphi\rVert_{\infty,D}:=\sup_{x\in X(k)} \lvert (D\varphi)(x)\rvert<\infty.$$
		It comes equipped with the topology associated to the seminorms $(\lVert .\rVert_{\infty,D})_D$ and it is a Fr\'echet space. For a more conceptual definition of the Schwartz space (in the general setting of Nash manifolds) see \cite{AGSchwartz}.

		\item The space of smooth functions $C^\infty(X)$\index{$C^\infty(X)$}. When $k$ is non-Archimedean, it consists of the functions $X(k)\to \bC$ that are right invariant by some compact-open subgroup $J\subset G(k)$. In this case, and using the $L^2$-pairing $\langle .,.\rangle_X$, $C^\infty(X)$ can be identified with the smooth contragredient of the Schwartz space $\cS(X)$. In the Archimedean case, although it is not at all the smooth dual of $\cS(X)$ in any reasonable sense, we will simply take $C^\infty(X)$ to be the space of all $C^\infty$ functions $X(k)\to \bC$.

		\item The {\it Harish-Chandra-Schwartz space} $\cC(X)$\index{$\cC(X)$} is a very convient intermediate space between $\cS(X)$ and $L^2(X)$. It was defined quite generally by Bernstein \cite[\S 0.4]{BerPlanch} for homogeneous $G(k)$-spaces of {\it polynomial growth}. More precisely, let $\mathbb{B}\subset G(k)$ be a {\it ball} (that is a compact neighborhood of $1\in G(k)$) that moreover generates $G(k)$ as a group. A $G(k)$-homogeneous space $Y$ is then said to be of {\it polynomial growth} if, choosing a base point $y_0\in Y$, we can find a polynomial $P$ such that for every $n\geq 1$ the subset $y_0\mathbb{B}^n$ can be covered by at most $P(n)$ subsets of the form $y\mathbb{B}$, $y\in Y$. As can readily be checked, this definition doesn't depend on the choice of $y_0$ and $\mathbb{B}$. When $Y$ is a polynomial growth, we define a ``radial function'' $\sigma_Y: Y\to \mathbb{R}_{\geq 1}$ by
		$$\displaystyle \sigma_Y(y)=\inf\{n\geq 1\mid y\in y_0\mathbb{B}^n\}.$$
		Changing the choice of $(y_0,\mathbb{B})$ would yield a radial function $\sigma_Y'$ that is {\it coarsely equivalent} in the sense that there exists a constant $C>0$ such that
		$$\displaystyle C^{-1} \sigma_Y(y)\leq \sigma_Y'(y)\leq C\sigma_Y(y), \mbox{ for every } y\in Y.$$
		When $X$ is a homogeneous spherical $G$-variety, each of the $G(k)$-orbits in $X(k)$ are of polynomial growth (e.g. this can be seen using the weak Cartan decomposition of \S \ref{S wavefront and weak Cartan}) and we will denote by $\sigma_X$\index{$\sigma_X$} any function $X(k)\to \bR_{\geq 1}$ that is coarsely equivalent to a radial function on each of these $G(k)$-orbits in the above sense. We note that, in this case, there is yet another possible definition of $\sigma_X$ (up to coarse equivalence) which is that of a {\it log-norm} in the sense of Kottwitz (see \cite[Sect. 18]{KottClay} and \cite[\S 1.2]{BP3} for the definitions)\footnote{In the affine case, this can be obtained as follows: fixing a closed embedding $\iota: X\hookrightarrow \bA^n$ in some affine space we can take $\sigma_X(x)=\max(1,\log \; \lvert \iota(x)_i\rvert_v\mid 1\leq i\leq n)$.}. The Harish-Chandra Schwartz space can then be defined as the projective limit (aka decreasing intersection)
		\begin{equation*}
			\displaystyle \cC(X)=\varprojlim_{d} L^2(X,\sigma_X(x)^d dx)^\infty
		\end{equation*}\index{$\cC(X)$}
		of the spaces of smooth vectors in the weighted $L^2$-spaces $L^2(X,\sigma_X(x)^d dx)$. By \cite[Sect. 3.4, Key Lemma]{BerPlanch}, a form of Sobolev inequality, an equivalent description is as the space of functions $\varphi\in C^\infty(X)$ such that for every $d>0$ (and every $u\in \mathcal{U}(\mathfrak{g})$ in the Archimedean case) we have
		\begin{equation}\label{ineq def HCS space}
			\displaystyle \lvert\varphi(x)\rvert\ll_d \Xi^X(x) \sigma_X(x)^{-d}\;\;\; (\mbox{resp. } \lvert(R(u)\varphi)(x)\rvert\ll_{d,u} \Xi^X(x) \sigma_X(x)^{-d} )
		\end{equation}
		where $\Xi^X:X(k)\to \mathbb{R}_{>0}$\index{$\Xi^X$} is the normalizing factor (because we are considering functions instead of half-densities) given by
		$$\displaystyle \Xi^X(x)=\vol(x\mathbb{B})^{-1/2},\;\; x\in X(k),$$
		for some fixed {\it ball} $\mathbb{B}\subset G(k)$ as above. (Once again, another choice of ball would lead to a (coarsely) equivalent function). The property that $X$ is of polynomial growth then translates to:
		\begin{num}
			\item\label{conv polgrowth} There exists $d>0$ such that the integral $\displaystyle \int_{X(k)} \Xi^X(x)^2\sigma_X(x)^{-d}dx$ is convergent.
		\end{num}
	\end{itemize}
\end{paragr}

\begin{paragr}[Function spaces on $G$.]\label{S function spaces for G}
	We will also work with functions on $G(k)$ mainly to construct (smooth) operators on representations. As for $X(k)$, we fix a Haar measure on $G(k)$ in order to identify functions with measures, half-densities... The function spaces we will consider are the Schwartz space $\cS(G)$ and the Harish-Chandra-Schwartz space $\cC(G)$ and these can be obtained by specializing the previous constructions to the group case (that is $X=G$ with its natural $G\times G$-action). Actually, in the second definition of the Harish-Chandra Schwartz space, given by \eqref{ineq def HCS space}, we can replace the function $\Xi^G$\index{$\Xi^G$} by Harish-Chandra basic spherical function i.e.\ by the normalized spherical matrix coefficient, with respect to the choice of some maximal compact subgroup $K\subset G(k)$, of the (normalized) parabolic induction $I_{P_0}^G(1)$ of the trivial representation from a minimal parabolic subgroup $P_0$. (Indeed by \cite[lemme II.1.1]{WaldPlanch}, the two functions are equivalent up to a finite power of $\sigma_G$.) This has the technical advantage that the function $\Xi^G$ then satisfies the {\it doubling inequality} \cite[lemme II.1.3]{WaldPlanch}
	\begin{equation}\label{doublind ineq}
		\displaystyle \int_K \Xi^G(gkg')dk\leq \Xi^G(g)\Xi^G(g'),\;\; g,g'\in G(k).
	\end{equation}
	
	For every smooth representation $(\pi,V)$ of $G(k)$ and $f\in \cS(G)$ we can define a (continuous) operator $\pi(f)$\index{$\pi(f)$} on $V$ characterized by
	$$\displaystyle \ell(\pi(f)v)=\int_{G(k)} f(g) \ell(\pi(g)v)dg$$
	for every $v\in V$ and $\ell\in V^*$ (the continuous dual in the Archimedean case, note that the `moderate growth' condition on $\pi$ then implies that the integral is convergent).

	By \cite{CoHaHo} an irreducible representation $\pi\in \Irr(G)$ is tempered if and only if for every $(v,v^\vee)\in \pi\times \pi^\vee$ we have an inequality
	\begin{equation}\label{ineq MC tempered}
		\displaystyle \lvert \langle \pi(g)v,v^\vee\rangle\rvert\ll \Xi^G(g).
	\end{equation}
	The convergence \eqref{conv polgrowth} allows to extend, for $\pi\in \Temp(G)$, the map $f\mapsto \pi(f)$ to $\cC(G)$ by the requirement
	$$\displaystyle \langle \pi(f)v,v^\vee\rangle=\int_{G(k)} f(g) \langle \pi(g)v,v^\vee\rangle dg,\;\; \mbox{ for every } (v,v^\vee)\in \pi\times \pi^\vee.$$
	Here is yet another characterization of tempered representations that we shall use: $\pi\in \Irr(G)$ if and only is there exists a continuous $G$-equivariant surjection $\cC(G)\twoheadrightarrow \pi$.
\end{paragr}

\begin{paragr}[Space of $\pi$-coinvariants.]
	Let $\pi$ be a (smooth) $G(k)$-representation of finite length. The space of {\em $\pi$-coinvariants} $\cS(X)_\pi$\index{$\cS(X)_\pi$} of $\cS(X)$ is by definition the quotient of $\cS(X)$ by the intersection of the kernels of all the $G(k)$-equivariant continuous maps $\cS(X)\to \pi$ (where the continuity condition has to be dropped in the non-Archimedean case). When the space $\Hom_G(\cS(X),\pi)$ is finite dimensional, we have an isomorphism
	$$\displaystyle \cS(X)_\pi\simeq \Hom_G(\cS(X),\pi)\otimes \pi$$
	and $\cS(X)_\pi$ is the largest (separated in the Archimedean case) quotient of $\cS(X)$ that is isomorphic to a direct sum of copies of $\pi$.
	
	Similarly, we define the space of {\em $\pi$-coinvariants} $\cC(X)_\pi$\index{$\cC(X)_\pi$} of $\cC(X)$ of the Harish-Chandra Schwartz space the quotient of $\cS(X)$ by the intersection of the kernels of all the $G(k)$-equivariant continuous maps $\cC(X)\to \pi$; where in the non-Archimedean case, the continuity condition has to be understood in the following sense: for every compact-open subgroup $J\subset G(k)$, $\cC(X)^J\to \pi^J$ is continuous. (Recall that $\pi^J$ is then finite dimensional.) As above, if the space $\Hom_G(\cC(X),\pi)$ of continuous intertwining operators $\cC(X)\to \pi$ is finite dimensional, we have an isomorphism
	$$\displaystyle \cC(X)_\pi\simeq \Hom_G(\cC(X),\pi)\otimes \pi$$
	and $\cC(X)_\pi$ is the largest separated quotient of $\cC(X)$ that is isomorphic to a direct sum of copies of $\pi$.
\end{paragr}

\begin{paragr}[Outline.]\label{S outline local aspects}
	In Section \ref{sect SV local conjecture}, we will review some conjectures of Sakellaridis-Venkatesh describing the Plancherel decomposition of $L^2(X(k))$ in terms of Arthur packets. In the rest of the chapter, we gather some evidence towards these very general expectations.
	
	More specifically, in Sections \ref{Sect GGP} and \ref{Sect Prasad conj} we discuss at length two particular but seminal examples that are the Gan-Gross-Prasad conjectures and some other, more recent, conjectures of Prasad for Galois symmetric varieties. These predictions are actually more precise than Sakellaridis-Venkatesh' general conjecture in two ways: first, by providing explicit recipes for the contributions of each $L$-packet to the ($L^2$-)spectrum of $X(k)$ and secondly by saying something on the more general smooth spectrum in the cases at hand.
	
	In Sections \ref{Sect smooth asym} and \ref{Sect Bernstein maps} we review foundational structural results of Sakellaridis and Sakellaridis-Venkatesh on the relation between harmonic analysis on $X(k)$ and its boundary degenerations $X_\Theta(k)$ (see Section \ref{sect boundary degenerations} for the definition). The main new tools for doing so are the so-called {\it asymptotic maps} relating function spaces on $X(k)$ and $X_\Theta(k)$. There are actually two versions of such maps: the {\it smooth asymptotic maps} relating smooth function spaces and the {\it unitary asymptotic maps} (also called {\it Bernstein maps}) for $L^2$-spaces roughly corresponding to the difference between integrating Eisenstein series far from the unitary axis and on the unitary axis in a global setting. The theory of Bernstein maps eventually leads to a spectral decomposition of $L^2(X(k))$ that is parallel to the natural partition of Langlands parameters for $X$ into sets of discrete parameters for the Levi subgroups of $\widehat{G}_X$. We warn the reader that in Sections \ref{Sect smooth asym} and \ref{Sect Bernstein maps}, the field $k$ is supposed to be non-Archimedean and the group $G$ is assumed to be split over $k$. Since the appearance of Sakellaridis-Venkatesh work, the $L^2$-theory and the aforementioned Plancherel decomposition have been generalized to most spherical varieties over $\mathbb{R}$ cf. \cite{DKKS}. However, to date there is no known analogs of the smooth asymptotics maps in the Archimedean case and actually the most naive formulation leads to false statements (e.g. already in the group case Bernstein's second adjunction fails for real groups). This is clearly an interesting and fundamental problem to reflect on and we refer the reader to the note \cite{WangSL2R} for some clues in the case of $\SL_2(\mathbb{R})$. The split assumption for its part is imposed throughout the work of Sakellaridis-Venkatesh which stems from the fact that most of the literature on spherical varieties, including their compactification theory which is crucial for the methods used, is developed over an algebraically closed field (of characteristic zero). When the group is split most results of interest can be painlessly descended to an arbitrary field of definition but not in general. See however \cite{KKRGA} for the development of a general compactification theory for $G$-varieties along the lines of what is needed for the development of local harmonic analysis. This can probably be applied to extend the Sakellaridis-Venkatesh' theory to arbitrary groups but such extension hasn't been carried out so far. Obtaining similar results over local fields of positive characteristics on the other hand seems much more challenging as some of the foundational results in the geometry of spherical varieties just break down.
	
	In Section \ref{Sect local trace formula}, we briefly review local trace formulas that are important tools used to attack the above conjectures. These are actually far from being currently developed in the most desired generality but we will try to convey, in particular through a summary of (part of) the seminal works of Arthur \cite{ArtLTF} (in the group case) and Waldspurger \cite{WaldGP1} (in the setting of the orthogonal Gan-Gross-Prasad conjectures), how these can be effectively applied to the study of spectra of spherical varieties.
	
\end{paragr}

\subsection{Distinction and abstract Plancherel decompositions}\label{Sect distinction and Plancherel}

In this subsection, we recall the notion of distinction for irreducible representations with respect to a $G$-variety $X$ as well as its unitary variant; the Plancherel decomposition.

\vspace{2mm}

\begin{paragr}\label{S distinction}
	Let $\pi\in \Irr(G)$ be a smooth irreducible representation of $G(k)$. We say that $\pi$ is {\it $X$-distinguished} if the space of (continuous in the Archimedean case) intertwining homomorphisms $\Hom_{G}(\cS(X),\pi)$ is nonzero. Note that by duality and Frobenius reciprocity we have natural isomorphisms
	\begin{equation*}
		\displaystyle \Hom_{G}(\cS(X),\pi)\simeq \Hom_{G}(\pi^\vee, C^\infty(X))\simeq \bigoplus_{i\in I} \Hom_{H_i}(\pi^\vee,\bC)
	\end{equation*}
	where the subgroups $(H_i)_{i\in I}$ in the last direct sum are (representatives of) the stabilizers of the $G(k)$-orbits in $X(k)$ i.e.\ $X(k)=\bigsqcup_{i\in I} H_i(k)\backslash G(k)$. Thus, when $X(k)=H(k)\backslash G(k)$ is itself homogeneous, $\pi$ is $X$-distinguished if and only if $\pi^\vee$ is {\it $H$-distinguished} in the sense that it supports a nonzero $H(k)$-invariant linear form.
	
	If $k$ is Archimedean, it is a theorem of Kobayashi-Oshima \cite{KobOsh} that the multiplicity
	$$\displaystyle m_X(\pi)=\dim \Hom_{G(k)}(\cS(X(k)),\pi)$$\index{$m_X(\pi)$}
	is uniformly bounded as $\pi$ varies in $\Irr(G)$. The corresponding result for $k$ non-Archimedean case is certainly expected but not known in general. See however \S \ref{S smooth asymptotics cons} below for a discussion of partial results in that direction showing in particular that $m_X(\pi)$ is always {\it finite} under suitable assumptions on $X$ (including the case where $X$ is wavefront).
\end{paragr}

\begin{paragr}[Plancherel decompositions (after Bernstein).]\label{sect Plancherel}
	
	As explained in the introduction, the Plancherel decomposition of $L^2(X)$ is one convenient way to make sense of the spectrum of $X$. Let us explain what is exactly meant by that. More precisely, we will need to write direct integral decompositions of the form:
	\begin{equation}\label{eq AbPlanch}
		\displaystyle L^2(X)\simeq \int^{\oplus}_{\Phi} \Pi_\phi d\mu(\phi)
	\end{equation}
	where typically $\Phi$ will stand for some space of Langlands parameters and $\Pi_\phi$ will be some direct sum of unitary irreducible representations of $G(k)$ is some Arthur packet associated to $\phi$. We can actually make sense of decompositions like \eqref{eq AbPlanch} in a very broad context from the abstract theory of $C^*$-algebras \cite{DixC*}. However, in our context, it turns out that there is a better suited and more explicit description that is essentially due to Bernstein \cite{BerPlanch}. More precisely, in this perspective a direct integral decomposition as \eqref{eq AbPlanch} amounts to the following set of data:
	\begin{itemize}
		\item a measure space $(\Phi,d\mu)$;
		
		\item a family $(J^X_\phi)_{\phi\in \Phi}$\index{$J^X_\phi$} of $G(k)$-invariant semipositive\footnote{semipositive here means $J_\phi^X(f,f)\geq 0$ for every $f\in \cS(X)$.} Hermitian forms on $\mathcal{S}(X)$ which is measurable and square-integrable in the sense that for every $f\in \cS(X)$ the function $\phi\mapsto J^X_\phi(f,f)$ belongs to $L^2(\Phi,d\mu)$.

		
	\end{itemize}
	satisfying the following conditions:
	\begin{itemize}
		\item $\{ J^X_\phi\mid \phi\in \Phi\}$ forms a linearly independent family;
		
		\item for every $f_1,f_2\in \mathcal{S}(X)$ we have
		$$\displaystyle \langle f_1,f_2\rangle_{X}=\int_{\Phi} J^X_\phi(f_1,f_2) d\mu(\phi);$$
		
		\item for a.a. $\phi\in \Phi$, the completion of $\mathcal{S}(X)$ with respect to $J^X_\phi$ is isomorphic to the unitary representation $\Pi_\phi$.
	\end{itemize}
\end{paragr}

\begin{paragr}\label{S remark Plancherel}
	
	We can make the following remarks on this definition:
	\begin{itemize}
		\item We can always replace the triple
		$$\displaystyle (\Phi, d\mu, \{ J^X_\phi\}_{\phi\in \Phi})\; \mbox{ by } (\Phi, \kappa d\mu, \{ \kappa(\phi)^{-1}J^X_\phi\}_{\phi\in \Phi})$$
		for any measurable function $\kappa: \Phi\to \bR_{>0}$. In particular, what is canonical in the decomposition \eqref{eq AbPlanch} is really the {\it class} of the measure $d\mu$ or, even better, the product $J^X_\phi d\mu(\phi)$.
		
		\item Up to this ambiguity, the triple is unique if we require that $\Phi$ is a subset of the unitary dual $\Unit(G)$, with its usual Borel structure (coming from the Fell topology), and that, for every $\phi\in \Phi$, $\Pi_\phi$ is a direct sum of copies of $\phi$. In this case \eqref{eq AbPlanch} is called {\it the} Plancherel decomposition of $L^2(X)$ and the (class of the) the measure $\mu$ {\it the} Plancherel measure for $X$. In particular, its support
		$$\displaystyle \Supp(\mu)\subset \Unit(G)$$
		is the {\it $L^2$-spectrum of $X$}.
		
		\item The Hermitan forms $J^X_\phi$ are called {\it relative characters}. More generally, a relative character is a bilinear (rather than sesquilinear) form $J: \cS(X)\times \cS(X)\to \bC$ such that for some finite length (typically irreducible) representation $\pi$ it factors through a $G\times G$-equivariant projection $\cS(X)\otimes \cS(X)\to \pi\otimes \pi^\vee$ followed by the canonical pairing $\langle .,.\rangle: \pi\otimes \pi^\vee\to \bC$, in other words we have a factorization of the form
		$$\displaystyle J: \cS(X)\otimes \cS(X)\twoheadrightarrow \pi\otimes \pi^\vee \xrightarrow{\langle .,.\rangle} \bC.$$
		We then say that the relative character $J$ is {\it supported} on $\pi$.
		
		\item Let $x\in X(k)$ with stabilizer $H=G_x$. When $\pi$ is irreducible, a relative character $J_\pi$\index{$J_\pi$} supported on $\pi$ is related to $H$-invariant linear forms as follows. Let $\cS(X)_x$\index{$\cS(X)_x$} be the subspace of functions $f\in \cS(X)$ supported in the orbit $xG(k)=H(k)\backslash G(k)$. Then, there exist $H$-invariant linear forms $\ell_H: \pi\to \bC$ and $\ell_H^\vee: \pi^\vee\to \bC$, making the following diagram commute
		$$\displaystyle \xymatrix{ \mathcal{S}(G)\otimes \mathcal{S}(G) \ar[r]^{\int_{H\times H}} \ar[d]^{\star} & \mathcal{S}(X)_x\otimes \mathcal{S}(X)_x \ar[rr]^{\;\;\;\;\;\;\;\;\;\;J_\pi} & & \bC \\ \mathcal{S}(G) \ar[r]^{f\mapsto \pi(f)\;\;\;\;\;\;\;\;\;\;\;\;} & \pi^\vee\otimes \pi\simeq\End^\infty(\pi) \ar[urr]_{\ell_H^\vee\otimes \ell_H}}$$
		where the left vertical arrow is given by the convolution product $(f_1,f_2)\mapsto f_1\star f_2^\vee$, with $f_2^\vee(g):=f_2(g^{-1})$, the first top arrow is given by the twofold tensor product of the surjection $\cS(G)\to \cS(X)_x$, $f\mapsto f\mapsto \tilde{f}(g):=\int_{H(k)} f(hg)dh$, the bottom horizontal arrow maps $f$ to the operator $\pi(f)$ identified with an element of $\pi^\vee\otimes \pi$ via the natural identification of the latter with the space $\End^\infty(\pi)$ of $G\times G$-smooth endomorphisms of $\pi$. More precisely, the $H\times H$-invariant form $\ell_H^\vee\otimes \ell_H$ is the composition of the embedding $\pi^\vee\otimes \pi\hookrightarrow C^\infty(X)\times C^\infty(X)$, dual to the surjection $\mathcal{S}(X)\otimes \cS(X)\twoheadrightarrow \pi\otimes \pi^\vee$ with evaluation at $x\in X(k)$. Concretely, this says that
		\begin{equation*}
			\displaystyle J_\pi(\tilde{f}_1,\tilde{f}_2)=\sum_{v\in \pi} \ell_H(\pi(f_1)v)\ell_H^\vee(\pi^\vee(f_2)v^\vee)
		\end{equation*}
		where the sum runs over a suitable basis of $\pi$ and $v^\vee$ runs over the dual basis;
		
		\item An isomorphism like \eqref{eq AbPlanch} can also be interpreted as a decomposition into eigenfunctions. More precisely, the relative characters $J^X_\phi$ induce surjective intertwinings $\displaystyle \cS(X)\twoheadrightarrow \Pi_\phi^\infty$ to the subspace of smooth vectors in $\Pi_\phi$ which then dualize to embeddings $\Pi_\phi^\infty\to C^\infty(X)$. Denoting by
		$$\displaystyle \cS(X)\ni f\mapsto f^\phi\in C^\infty(X)$$
		the composition $\cS(X)\to \Pi_\phi^\infty\to C^\infty(X)$, the decomposition \eqref{eq AbPlanch} also reads
		\begin{equation}\label{decomp eigenfunctions}
			\displaystyle f=\int_{\Phi} f^\phi d\mu(\phi),\;\; \mbox{ for every } f\in \cS(X),
		\end{equation}
		where the integral can be interpreted in a naive sense, e.g. we have $f(x)=\int_{\Phi} f^\phi(x) d\mu(\phi)$ for every $x\in X(k)$, and the $f^\phi$\index{$f^\phi$} are {\it generalized eigen-functions} i.e.\ they generate finite length subrepresentations of $C^\infty(X)$.
		
		\item By the main result of \cite{BerPlanch}, for $\mu$ almost all $\phi$ the relative character $J_\phi^X$ extends by continuity to the Harish-Chandra Schwartz space $\cC(X)$. Equivalently, the equivariant map $\cS(X)\twoheadrightarrow \Pi_\phi$ extend to $\cC(X)\twoheadrightarrow \Pi_\phi$. We say that a smooth irreducible representation $\pi\in \Irr(G)$ is {\it $X$-tempered} if there exists a nonzero continuous intertwining $\cC(X)\to \pi$. Therefore, this is saying that the Plancherel measure for $X$ is supported on $X$-tempered representations which is an analytic condition usually stronger than $X$-distinction e.g. in the group case this immediately gives that the Plancherel measure for $G(k)$ is supported on tempered representations (in the usual sense). Yet another formulation of the same property says that the eigenfunctions in the decomposition \eqref{decomp eigenfunctions} are always {\it tempered} i.e.\ define, by integration, continuous functionals on $\cC(X)$\footnote{As for the Harish-Chandra Schwartz space, tempered functions $f$ can be characterized, perhaps more concretely, by a growth condition: namely there exists $d>0$ such that $\lvert f(x)\rvert\ll \Xi^X(x)\sigma_X(x)^d$, $x\in X(k)$ (resp. $\lvert (R(u)f)(x)\rvert\ll \Xi^X(x)\sigma_X(x)^{d}$ for all $u\in \mathcal{U}(\mathfrak{g})$ in the Archimedean case.)}

		\item In the group case (i.e.\ when $X=H\curvearrowleft G=H\times H$), $\Pi_\phi$ is a direct sum of unitary $G(k)$-representations of the form $\pi\widehat{\otimes}\pi^\ast\simeq \operatorname{HS}(\pi)$\index{$\operatorname{HS}(\pi)$}, $\pi\in \Unit(H)$, where $\pi^*$ is the unitary dual of $\pi$, $\widehat{\otimes}$ the completed Hilbert tensor product and we have made the natural identification with the space $\operatorname{HS}(\pi)$ of Hilbert-Schmidt operators on (the space of) $\pi$, and the relative characters $J_\phi$ is a linear combination (with positive coefficients) of the Hilbert-Schmidt inner products
		$$\displaystyle (f_1,f_2)\mapsto \Tr \pi(f_1)\pi(f_2)^*=\langle \pi(f_1),\pi(f_2)\rangle_{\operatorname{HS}}.$$\index{$\langle .,.\rangle_{\operatorname{HS}}$}

		\item Since in the group case we have canonical way to normalize relative characters (given by the trace/Hilbert-Schmidt inner product of the operators $\pi(f)$), we also have a canonical measure $d\mu_H$ on the unitary dual $\Unit(H)$. The support of this measure is known to be the tempered dual $\Temp(H)\subseteq \Unit(H)$. More precisely, for $f_1,f_2\in \mathcal{S}(H)$ we have
		$$\displaystyle (f_1,f_2)_{L^2}=\int_{\Temp(H)} \langle \pi(f_1),\pi(f_2)\rangle_{\operatorname{HS}} d\mu_H(\pi)$$
		where $d\mu_H$\index{$d\mu_H$} is a canonical measure on $\Temp(H)$ called {\it the Plancherel measure} of $H$ and that has been described very precisely by Harish-Chandra. Another equivalent version of the above Plancherel formula (obtained by formally plugging $f_2=\delta_1$) is the identity
		$$\displaystyle f(1)=\int_{\Temp(H)} \Tr (\pi(f)) d\mu_H(\pi)$$
		for every $f\in \mathcal{S}(H)$ and actually even for every $f\in \mathcal{C}(H)$.
	\end{itemize}
\end{paragr}

\subsection{The Sakellaridis-Venkatesh local conjecture}\label{sect SV local conjecture}

We now state the most general local conjecture, due to Sakellaridis and Venkatesh \cite{SV}, pertaining to the spectral decomposition of spherical varieties.

\vspace{2mm}

\begin{paragr}[Local conjecture: weak form.]\label{S local conjecture weak form}
	Assume that $X$ has no type $N$ roots so that by the general construction of \S \ref{S L-groups} it admits an $L$-group ${}^L G_X$. Recall that $\Phi_{\temp}(X)$ denote the set of (equivalence classes of) tempered $L$-parameters $\phi: \mathcal{L}_k\to {}^L G_X$ (see \S \ref{S Lparameters}). This set comes naturally equipped with a topology and a natural class of Radon measures as follows. Namely, for every $\phi: \mathcal{L}_k\to {}^L G_X\in \Phi_{\temp}(X)$, picking a minimal Levi subgroup ${}^L M\subset {}^L G_X$ through which $\phi$ factors, there is a natural action of the group of unramified parameters
	$$\displaystyle \mathfrak{X}({}^L M):=\Hom_{\operatorname{cont}}(\mathcal{L}_k/\mathcal{L}_k^1, Z({}^L M)_c^0),$$\index{$\mathfrak{X}({}^L M)$}\index{$\mathcal{L}_k^1$}\index{$Z({}^L M)_c^0$}
	where $\mathcal{L}_k^1$ denotes the kernel of the absolute value map $\lvert .\rvert:\mathcal{L}_k\to \bR_{>0}$ (e.\ g.\ if $k$ is non-Archimedean, $\mathcal{L}_k^1=I_k\times \SL_2(\bC)$ where $I_k\subset W_k$ is the inertia subgroup) and $Z({}^L M)_c^0$ is the maximal compact connected subgroup in the Galois fixed part $Z(\widehat{M})^\Gamma$ of the center of $\widehat{M}$, on $\phi$ by multiplication:
	$$\displaystyle \chi\in \mathfrak{X}({}^L M)\mapsto \phi\chi\in \Phi_{\temp}(X).$$
	Then, the topology and the measure-class on $\Phi_{\temp}(X)$ are uniquely characterized by requiring that for every $\phi$ the above map is atopological  quotient map in a neighborhood of $1\in \mathfrak{X}({}^L M)$ and the push-forward of any Haar measure on $\mathfrak{X}({}^L M)$ (which is either a compact real torus $(\mathbb{S}^1)^N$, in the non-Archimedean case, or a real vector space, in the Archimedean case) to a neighborhood of $\phi\in \Phi_{\temp}(X)$ belongs to the (restriction of the) given measure-class.

	\begin{conj}[Sakellaridis-Venkatesh]\label{conj1 SV}
		There exists a ``natural'' direct integral decomposition
		$$\displaystyle L^2(X)\simeq \int_{\Phi_{\temp}(X)}^\oplus \Pi_\phi^G d\mu(\phi)$$
		where
		\begin{itemize}
			\item $d\mu$ is in the natural class of measures on $\Phi_{\temp}(X)$ as described above.
			
			\item $\Pi_\phi^G$ is a multiplicity free direct sum of representations in the Arthur packet $\Pi^G(\psi)$ where $\psi$ is the $A$-parameter
			$$\displaystyle \mathcal{L}_k\times \SL_2(\bC)\xrightarrow{\phi\times Id} {}^LG_X\times \SL_2(\bC)\xrightarrow{{}^L \iota_X} {}^L G.$$
		\end{itemize}
	\end{conj}
	
	\begin{rem}
		\begin{itemize}
			\item The conjecture is not saying anything about the exact constituents of $\Pi_\phi^G$. It is not even asserting that $\Pi_\phi^G\neq 0$, for that we have to take into account pure inner forms, see below. In particular, it is only giving an upper bound on the Plancherel support of $L^2(X)$.
			
			\item The conjecture stated in \cite{SV} only concerns split groups (in which case the Galois action on $\widehat{G}_X$ is trivial) but here we have stated its obvious extension in general by using Knop's definition of the $L$-group ${}^L G_X$ (see \S \ref{S L-groups}).
		\end{itemize}
		
	\end{rem}
\end{paragr}

\begin{paragr}[Local conjecture: pure inner forms.]\label{S PIF X}
	We assume now that $G$ is quasi-split. A {\it pure inner form} of $X$ is by definition a spherical variety that can be obtained by the following procedure. Take $\alpha\in H^1(k,G)$. Then, $\alpha$ parametrizes a $G$-torsor $T_\alpha$ (with a left $G$-action) and the contracted product
	$$\displaystyle X_\alpha:= X\times^G T_\alpha$$\index{$X_\alpha$}
	is naturally a $G_\alpha$-variety (it is even homogeneous if $X$ is) where $G_\alpha=\Aut_G(T_\alpha)$ is the pure inner form of $G$ associated to $\alpha$. Also, we will say of the pure inner form $X_\alpha$ that it is {\it relevant} if $X_\alpha(k)\neq \emptyset$.
	
	The direct sum
	$$\displaystyle \bigoplus_{\alpha\in H^1(k,G)} L^2(X_\alpha)$$
	is now a unitary representation of $\prod_\alpha G_\alpha(k)$ (where only one factor acts nontrivially on each summand). Note that the sum can obviously be restricted to relevant pure inner forms.
	
	\begin{rem}
		Simple considerations in Galois cohomology show that
		$$\displaystyle \bigoplus_{\alpha\in H^1(k,G)} L^2(X_\alpha)=\bigoplus_{\beta\in H^1(k,H)} L^2(H_\beta(k)\backslash G_\beta(k)).$$
		Indeed, the pure inner form $X_\alpha$ corresponding to $\alpha\in H^1(k,G)$ is relevant if and only if $\alpha$ belongs to the image of the natural map $i:H^1(k,H)\to H^1(k,G)$ in which case the fiber $i^{-1}(\alpha)$ naturally parametrizes the $G_\alpha(k)$-orbits in $X_\alpha(k)$ i.e.\ we have a decomposition $X_\alpha(k)=\sqcup_{\beta\in i^{-1}(\alpha)} H_\beta(k)\backslash G_\alpha(k)$.
	\end{rem}
	
	\begin{conj}[Sakellaridis-Venkatesh]\label{Conj local strong SV}
		Assume that $G$ acts faithfully on $X$. Then, there should exist a natural decomposition
		$$\displaystyle \bigoplus_{\alpha\in H^1(k,G)} L^2(X_\alpha)\simeq \int_{\Phi_{\temp}(X)}^\oplus \Pi_\phi d\mu(\phi)$$
		where this time $\Pi_\phi$ is a {\bf non-zero} multiplicity free direct sum of representations in the (extended) Arthur packet $\Pi(\psi)=\bigsqcup_{\alpha} \Pi^{G_\alpha}(\psi)$ with $\psi$ defined as in Conjecture \ref{conj1 SV}.
	\end{conj}
\end{paragr}

\subsection{Tempered and strongly tempered varieties}

The aim of this subsection is to discuss two particular classes of (spherical) $G$-varieties $X$ roughly characterized by certain relations between their Plancherel measures $\mu_X$ and the Plancherel measure $\mu_G$ of the group $G$ itself. More specifically, we will first consider {\em tempered varieties}, that is those $G$-varieties whose Plancherel measure $\mu_X$ is supported on the set of tempered representations ($=$ the support of $\mu_G$). Following Cowling, Haagerup and Howe \cite{CoHaHo}, tempered varieties can be characterized by a growth condition on the matrix coefficients associated to functions $\varphi_1,\varphi_2\in C_c^\infty(X)$. Following Sakellaridis-Venkatesh, we will also discuss {\em strongly tempered varieties}, characterized in a similar way by a growth condition on matrix coefficients. The Plancherel measure of a strongly tempered variety $X$ is automatically absolutely continuous with respect to $\mu_G$ which, when the group $G$ is split and $X$ is spherical without type $N$ roots, is to be expected from the Sakellaridis-Venkatesh local conjecture since in this case they share the same $L$-group: ${}^L G_X={}^L G$ (at least under the assumption of Conjecture \ref{Conj local strong SV}; in general the map ${}^L G_X\to {}^L G$ is rather a finite isogeny). We however start this subsection by discussing weak Cartan decompositions for wavefront spherical varieties; a technical tool that allows to relate in a certain ad hoc way eigenfunctions on $X$ to irreducible smooth matrix coefficients on $G$.

\vspace{2mm}

\begin{paragr}[Weak Cartan decompositions and wavefront varieties.]\label{S wavefront and weak Cartan}
For simplicity, in this paragraph we assume that $G$ is split. We identify $A_X$ with a subvariety of $X$ as in \S \ref{S parabolic type}, by fixing a flat with a base-point. Then, for any large enough compact subset $\mathcal{K}\subset G(k)$ satisfies the {\it weak Cartan decomposition}
\begin{equation}\label{weak Cartan}
\displaystyle X(k)=A_X^-\mathcal{K}
\end{equation}
where
$$\displaystyle A_X^-:=\{a\in A_X(k)\mid \lvert \alpha(a)\rvert\geq 1,\; \forall \alpha\in \Delta_X \}$$\index{$A_X^-$}
is the negative Weyl chamber in $A_X(k)$. The validity of \eqref{weak Cartan} for $\mathcal{K}$ large enough can be readily deduced from the local structure theorem (Theorem \ref{thm LST}) applied to some toroidal compactification of $X$, see \cite[Lemma 5.3.1]{SV}.
	
	Recall that a spherical $G$-variety $X$ is said to be wavefront if the negative Weyl chamber $\cA_X^-$ of $X$ is the image of the negative Weyl chamber $\cA^-$ of $G$. If this is the case, the above weak Cartan decomposition allows to reduce many statements for $X$ to similar statements for the group e.g. it implies that arbitrary embeddings $\pi\hookrightarrow C^\infty(X)$ can be written in a simple way in terms of smooth matrix coefficients of $\pi$. This is mostly the content of the {\it wavefront lemma} of \cite[\S 5.3]{SV} that for convenience we reproduce here.
	
	\begin{lem}[Wavefront lemma \cite{SV}.]\label{wavefront lemma}
		Assume that $X=H\backslash G$ is a wavefront spherical variety and that $k$ is non-Archimedean. Let $x_0=H1\in X$ be the canonical base-point. Then, there exists a subset $G^+\subset G(k)$ satisfying the following conditions:
		\begin{enumerate}[(i)]
			\item $G(k)=H(k)G^+$.
			
			\item For every compact-open sugroup $J\subset G(k)$, we can find a compact-open subgroup $J'\subset G(k)$ such that
			$$\displaystyle x_0J'gJ\subset x_0gJ, \mbox{ for every } g\in G^+.$$
			In particular for every smooth $G(k)$-representation, every $G(k)$-equivariant map $L:\pi\to C^\infty(X)$ and every vector $v\in \pi^J$, we can find $v^\vee\in (\pi^\vee)^{J'}$ such that $L(v)(x_0g)=\langle \pi(g)v,v^\vee\rangle$ for every $g\in G^+$.
			
			\item $\Xi^X(x_0g)\sim \Xi^G(g)$ and $\sigma_X(x_0g)\sim \sigma_G(g)$ for $g\in G^+$.
		\end{enumerate}
	\end{lem}
	
	\begin{proof}
		The wavefront assumption entails that we can find a $\bQ$-linear section $\cA_X\to \cA_{L(X)}$ sending $\cA_X^-$ into $\cA_{L(X)}^-$. Up to scaling this by an integer, we find a morphism $s: A_X\to Z(L(X))^0$ such that, denoting by $p:Z(L(X))^0\to A_X$ the natural projection, $ps: A_X\to A_X$ is a finite isogeny and $s(A_X^-)\subset A^-$. We then take for $G^+$ a subset of the form $s(A_X^-)\mathcal{K}$ where $\mathcal{K}\subset G(k)$ is a compact. The first point then follows from the weak Cartan decomposition, provided $\mathcal{K}$ is chosen large enough.
		
		Let $J\subset G(k)$ be a compact-open subgroup. Since $G^+\subset A^-\mathcal{K}$, as $g$ runs over $G^+$, $gJg^{-1}$ contains a fixed compact-open subgroup $J_B$ of $B(k)$. Moreover, as $H(k)B(k)$ is open in $G(k)$, the product $H(k)J_B$ contains some small compact-open subgroup $J'$. This shows the first part of the second point. For the last part, let $L:\pi\to C^\infty(X)$ be $G(k)$-equivariant and $v\in \pi^J$. Then, for every $g\in G^+$ and $k'\in J'$ we have $L(v)(x_0g)=L(v)(x_0k'g)$. Letting $\ell=L^*(\mathrm{ev}_{x_0})$ be the pullback of the evaluation at $x_0$ by $L$ (a, not necessarily smooth, linear form on $\pi$), this then gives $L(v)(x_0g)=\langle \pi(g)v,v^\vee\rangle$ where $v^\vee:=\pi^*(e_{J'})\ell\in (\pi^\vee)^{J'}$.
		
		Let us fix two compact-open subgroups $J$, $J'$ as in (ii). Then, by definition of $\Xi^X$ and $\Xi^G$, the first estimate in the third point is equivalent to
		\begin{equation}\label{eq1 prop wavefront}
			\displaystyle \vol_X(x_0gJ)\sim \vol_G(J'gJ),\mbox{ for } g\in G^+.
		\end{equation}
		The inequality $\vol_X(x_0gJ)=\vol_X(x_0J'gJ)\ll\vol_G(J'gJ)$ is clear: writing $J'gJ$ as a finite disjoint union of $k$ cosets of the form $J'\gamma$ we have $\vol_X(x_0J'gJ)\leq  \vol_X(x_0J')k \ll \vol_G(J'gJ)$. For the reverse inequality, we remark that:
		\begin{itemize}
			\item By the construction above, up to shrinking $J'$ (which replaces the right hand side of \eqref{eq1 prop wavefront} by an equivalent function), we may assume the existence of a compact-open subgroup $J_H\subset H(k)$ we have $J'\subset J_H gJg^{-1}$ for every $g\in G^+$;
			
			\item Similarly, shrinking $J$ if necessary, we may assume that for each $k\in \mathcal{K}$ the compact open-subgroup $kJk^{-1}$ is included in $J_{P_X}^-J_{N_X}$ for some compact open subgroups $J_{P_X}^-\subset P_X^-(k)$ and $J_{N_X}\subset N_X(k)$. This then implies that the image of $gJg^{-1}$ in $(G/N_X)(k)$ is included in a compact subset independent of $g\in G^+$. Since $H\cap N_X=1$ and the $N_X$-orbit $x_0N_X\subset X$ is closed \cite[Exercise 8, p.115]{Humphreys}, the natural map $H(k)\to (G/N_X)(k)$ is a closed embedding and it follows that $gJg^{-1}\cap H(k)$ is included in a fixed compact subset of $H(k)$, say $\mathcal{K}_H$, as $g$ runs over $G^+$.
		\end{itemize}
		Then, for every $x\in J'gJ$ we have
		$$\displaystyle \int_{H(k)} \mathbf{1}_{J'gJ}(hx)dh=\vol_H(H(k)\cap J'gJx^{-1})\leq \vol_H(H(k)\cap J'gJg^{-1}J')\leq \vol_H(J_H\mathcal{K}_H J_H),$$
		where we use $\vol_H$ to indicate that the volume is taken with respect to the Haar measure on $H(k)$. It then follows that
		$$\displaystyle \vol(J'gJ)=\int_{H(k)\backslash G(k)}\int_{H(k)} \mathbf{1}_{J'gJ}(hx)dhdx\ll \int_{X(k)} \mathbf{1}_{x_0J'gJ}(x)dx=\vol_X(x_0J'gJ)=\vol_X(x_0gJ).$$ 
		
		Finally, the second equivalence of (iii) is just a consequence of $A_X$ being closed in $X$ (see \S \ref{S parabolic type}).
	\end{proof}
\end{paragr}

\begin{paragr}[Tempered varieties.]\label{S tempered varieties}
	We say that the variety $X$ is {\em tempered} if the support of the Plancherel measure for $L^2(X)$ is included in the tempered dual $\Temp(G)$. For $\varphi_1,\varphi_2\in L^2(X)$, we denote by 
	$$\displaystyle M_{\varphi_1,\varphi_2}: G(k)\to \bC,\; g\mapsto \langle R(g)\varphi_1,\varphi_2\rangle_X,$$\index{$M_{\varphi_1,\varphi_2}$}
	the corresponding matrix coefficient. A slight reformulation of a theorem of Cowling-Haagerup-Howe \cite{CoHaHo}\footnote{More precisely, the main result of {\em op.\ cit.}\ only concerns semisimple groups but the extension to general reductive groups is easy since a unitary representation of $G(k)$ is tempered if and only if its restriction to the derived subgroup $G_{\der}(k)$ is tempered.} provides the following characterization of tempered varieties.
	
	\begin{theo}[Cowling-Haagerup-Howe]\label{theo CHH}
		Let $X$ be a unimodular $G$-variety. The following are equivalent:
		\begin{enumerate}
			\item The variety $X$ is tempered;
			
			\item For every $\varphi_1,\varphi_2\in L^2(X)^\infty$, we can find constants $C=C_{\varphi_1,\varphi_2}>0$ and $d=d_{\varphi_1,\varphi_2}\geq 0$ such that we have the estimate
			\begin{equation}\label{ineq tempered variety}
				\displaystyle \lvert M_{\varphi_1,\varphi_2}(g)\rvert\leqslant C \Xi^G(g)\sigma_G(g)^d,\;\;\; g\in G(k).
			\end{equation}
			
			\item For every $\varphi_1,\varphi_2\in C_c^\infty(X)$, we can find constants $C=C_{\varphi_1,\varphi_2}>0$ and $d=d_{\varphi_1,\varphi_2}\geq 0$ such that the matrix coefficient $M_{\varphi_1,\varphi_2}$ satisfies the estimate \eqref{ineq tempered variety}.
		\end{enumerate}
		Furthermore, if these conditions are satisfied then we can always take $d=0$ in the estimates \eqref{ineq tempered variety}. Also, when $k$ is non-Archimedean, for every compact-open subgroup $J\subset G(k)$ there exists $C_J>0$ such that for $\varphi_1,\varphi_2\in L^2(X)^J$ the constant $C$ in \eqref{ineq tempered variety} may be taken to be $C=C_J\lVert \varphi_1\rVert_{L^2} \lVert \varphi_2\rVert_{L^2}$ (and $d=0$).
	\end{theo}
	
	This further implies the following characterization that turns out to be very easy to check in practice.
	
	\begin{lem}\label{lem tempered}
		Let $X$ be a unimodular homogeneous $G$-variety. Then, $X$ is tempered if and only if for every $x\in X(k)$ we can find $d>0$ such that the integral
		\begin{equation*}
			\displaystyle \int_{G_x(k)} \Xi^G(h) \sigma_G(h)^{-d}dh
		\end{equation*}
		converges.
	\end{lem}
	
	\begin{proof}
		The if part of the lemma is a consequence of \cite[Proposition 2.7.1]{BeuGal}. Conversely, assume that $X$ is tempered. By the previous theorem and \eqref{conv polgrowth}, we can find $d\geq 0$ such that for every $\varphi_1,\varphi_2\in \cS(X)$,
		\begin{equation}\label{conv int1}
			\displaystyle \int_{G(k)} \langle R(g)\varphi_1,\varphi_2\rangle_X \Xi^G(g)\sigma_G(g)^{-d}dg <\infty.
		\end{equation}
		Let $x\in X(k)$ and let $\delta_x$ be the Dirac at $x$, seen as a compactly supported distribution on $X(k)$, that is element of the continuous dual $C^\infty(X)'$ given by $\varphi\mapsto \varphi(x)$. Using the invariant measure on $X(k)$, we can identify for every $f\in C_c^\infty(G)$ the convolution $R(f)\delta_x$ with a smooth compactly supported function on $X(k)$. Applying \eqref{conv int1} to $\varphi_1=R(f_1)\delta_x$ and $\varphi_2=R(f_2)\delta_x$ with $f_1$, $f_2$ positive, we get
		\[\begin{aligned}
			\displaystyle \int_{G(k)} \langle R(g)\varphi_1,\varphi_2\rangle_X \Xi^G(g)\sigma_G(g)^{-d}dg=\int_{G_x(k)} (f_2\star \Xi^G \sigma_G^{-d}\star f_1^\vee)(h)dh<\infty.
		\end{aligned}\]
		However, for every compact subset $\mathcal{K}\subset G(k)$, we have $\Xi^G(k_1gk_2)\sigma_G(k_1gk_2)^{-d}\sim \Xi^G(g)\sigma_G(g)^{-d}$ (in the sense that $C^{-1}\Xi^G(g)\sigma_G(g)^{-d}\leq \Xi^G(k_1gk_2)\sigma_G(k_1gk_2)^{-d}\leq C \Xi^G(g)\sigma_G(g)^{-d}$ for some $C>0$) for $g\in G(k)$, $k_1,k_2\in \mathcal{K}$. This implies that $f_2\star \Xi^G \sigma_G^{-d}\star f_1^\vee\sim \Xi^G \sigma_G^{-d}$ and so the convergence of the above integral is equivalent to that in the lemma.

	\end{proof}
	
	Let $A_0$ be a maximal split torus in $G$ defined over $k$ and set
	$$\displaystyle A_0^-:=\{a\in A_0(k)\mid \lvert \alpha(a)\rvert\geq 1,\; \forall \alpha\in \Delta_0 \}$$
	where $\Delta_0$ is the set of simple roots associated to the choice of a minimal parabolic subgroup $P_0\supset A_0$ (also defined over $k$). Then, by \cite[lemme II.1.1]{WaldPlanch}, we can find $d_0>0$ such that
	$$\displaystyle \delta_{P_0}(a)^{-1/2}\ll \Xi^G(a)\ll \delta_{P_0}(a)^{-1/2}\sigma_G(a)^{d_0}$$
	for $a\in A_0^-$. 
	
	Whenever $H\subset G$ is a reductive subgroup, combining these estimates with a Cartan decomposition for $H(k)$ and the above lemma, this leads to a very workable combinatorial criterion to check whether the variety $X=H\backslash G$ is tempered or not (see e.g. \cite{BenKob} or \cite{GurOff} where this criterion is explicitly spelled out). Examples of tempered spherical varieties are Galois symmetric varieties $H\backslash \Res_{\ell/k} H_\ell$ ($[\ell:k]=2$) or the non-wavefront varieties $\GL_n\backslash \SO_{2n+1}$. An example of a spherical variety that is not tempered is the symmetric variety $\Sp_{2n}\backslash \GL_{2n}$.
	
	On the other hand, Conjecture \ref{conj1 SV} suggests the following alternative characterization of tempered spherical varieties.
	
	\begin{conj}\label{conj tempered varieties}
		Assume that $G$ is quasi-split and $X$ is a spherical $G$-variety. Then, $X$ is tempered if and only if $P_X=B$.
	\end{conj}
	
	Actually, the only if part is easy to show e.g. via the description of the most continuous part of the spectrum of $L^2(X)$ (see Section \ref{Sect most cont spectrum}). When $k=\bR$ and the subgroup $H$ is reductive, the above conjecture can actually be deduced from the, more general, main result of \cite{BenKobIII}, although the proof of {\it op.\ cit.}\ is mainly through a case-by-case analysis that seems difficult to adapt in the $p$-adic case.
	
\end{paragr}

\begin{paragr}
	\begin{prop}\label{prop tempered}
		Assume that $k$ is non-Archimedean and let $X$ be a homogeneous spherical tempered $G$-variety. Pick $x\in X(k)$ and let $H=G_x$, $\cS(X)_x$\index{$\cS(X)_x$} (resp. $\cC(X)_x$\index{$\cC(X)_x$}) denote the subspace of functions $f\in \cS(X)$ (resp. $f\in \cC(X)$) supported in the $G(k)$-orbit $xG(k)\simeq H(k)\backslash G(k)$.
		\begin{enumerate}[(i)]
			\item The morphism
			\begin{equation*}
				\displaystyle P_x: \cS(G)\to \cS(X)_x,\;\; P_xf: xg\mapsto \int_{H(k)} f(hg) dh,
			\end{equation*}
			extends (necessarily uniquely) to a continuous map $P_x: \cC(G)\to \cC(X)_x$. In particular, if $\pi\in \Irr(G)$ is $X$-tempered then $\pi\in \Temp(G)$ (i.e.\ it is tempered as a representation of the group $G$).
			
			\item If furthermore $X$ is wavefront,
			\begin{itemize}
				\item $P_x: \cC(G)\to \cC(X)_x$ is surjective, more precisely for every compact-open subgroup $J\subset G(k)$, we can find another compact open subgroup $J'$ such that $P_x(\cC(G)^{J'\times J})=\cC(X)^J$;
				
				\item $\cC(X)_\pi=\cS(X)_\pi$ for every $\pi\in \Temp(G)$. In particular, if an irreducible tempered (for the group $G$) representation $\pi$ is $X$-distinguished, it is automatically $X$-tempered.
			\end{itemize} 
		\end{enumerate}
	\end{prop}
	
	\begin{proof}
		Let $J\subset G(k)$ be a compact-open subgroup.
		\begin{enumerate}[(i)]
			\item We need only show that $P_x$ extends to a continuous map $\cC(G)^{J\times J}\to \cC(X)_x^J$. Set $\varphi_x=P_x(e_J)=\vol(xJ)^{-1} \mathbf{1}_{xJ}\in \cS(X)$ and more generally $\varphi_y=\vol(yJ)^{-1} \mathbf{1}_{yJ}$ for every $y\in X(k)$. Then, we have
			$$\displaystyle P_x(f)(y)=(R(f)\varphi_x)(y)=\int_{G(k)} f(g) \langle R(g) \varphi_x,\varphi_y\rangle_X dg$$
			for all $f\in \cS(G)^{J\times J}$ and $y\in X(k)$. Let $d>0$. Using the inequality $\sigma_X(z)\ll \sigma_X(zg)\sigma_G(g)$, we see that
			$$\displaystyle \langle R(g) \varphi_x,\varphi_y\rangle_X=\int_{X(k)} \sigma_X(z)^d \varphi_x(zg) \sigma_X(z)^{-d}\varphi_y(z)dz\ll \sigma_G(g)^d\langle R(g) \varphi_{x,d},\varphi_{y,-d}\rangle_X$$
			for $y\in X(k)$ and $g\in G(k)$, where we have set $\varphi_{x,d}=\sigma_X^d \varphi_x$, $\varphi_{y,-d}=\sigma_X^{-d} \varphi_y$. Therefore, up to shrinking $J$ such that it leaves $\sigma_X$ invariant and letting $C_J>0$ be a constant as in Theorem \ref{theo CHH}, we obtain
			\[\begin{aligned}
				\displaystyle \lvert P_x(f)(y)\rvert\ll C_J \int_{G(k)} \lvert f(g)\rvert \Xi^G(g)\sigma_G(g)^{d} dg \lVert \varphi_{x,d}\rVert_{L^2} \lVert \varphi_{y,-d}\rVert_{L^2}
			\end{aligned}\]
			for $f\in \cS(G)^{J\times J}$ and $y\in X(k)$. However, it is straightforward to see that
			$$\displaystyle \lVert \varphi_{y,-d}\rVert_{L^2}\sim \vol_X(yJ)^{-1/2}\sigma_X(y)^{-d}\sim \Xi^X(y)\sigma_X(y)^{-d}, \mbox{ for } y\in X(k),$$
			so that we get
			$$\displaystyle \lvert P_x(f)(y)\rvert\ll \int_{G(k)} \lvert f(g)\rvert \Xi^G(g)\sigma_G(g)^{d} dg\times \Xi^X(y)\sigma_X(y)^{-d}$$
			as before for $f\in \cS(G)^{J\times J}$ and $y\in X(k)$. By definition of the topology on $\cC(G)^{J\times J}$ and \eqref{conv polgrowth}, the right integral is a continuous semi-norm on the latter and therefore, by definition of $\cC(X)^J$ and its topology, the claim follows.

			We now explain how to deduce the implication $\pi$ is $X$-tempered $\Rightarrow$ $\pi$ is $G$-tempered. Indeed, if $\pi\in \Irr(G)$ is $X$-tempered we have a nonzero $G$-equivariant continuous morphism $\cC(X)\twoheadrightarrow \pi$ and even, since $\cC(X)$ splits as a direct sum of $\cC(X)_x$ for $x$ running over representatives of the $G(k)$-orbits, $\cC(X)_x\twoheadrightarrow \pi$ for some $x\in X(k)$. Precomposing by $P_x$ this yields a continuous $G$-morphism $\cC(G)\to \pi$ which is also nonzero as $P_x(\cS(G))=\cS(X)_x$ is dense in $\cC(X)_x$. It follows that $\pi$ is tempered as a $G$-representation.
			
			\item For the first part, it suffices to show that there exist a compact-open subgroup $J'\subset G(k)$ as well as a continuous section $S_x^J: \cC(X)_x^J\to \cC(G)^{J'\times J}$ of $P_x$ i.e.\ such that $P_xS_x^J(\varphi)=\varphi$ for every $\varphi\in \cC(X)_x^J$. We note that this section will {\it not} be $G$-equivariant. (However see the proof of Theorem \ref{theo nondeg canonical form SV} for the construction of an equivariant section under a stronger assumption.) For this we apply the Wavefront Lemma \ref{wavefront lemma}. Indeed, let $G^+$ be a subset of $G(k)$ and $J'$ be as in that lemma. Take for $S$ a set of representatives in $G^+$ for the cosets $H(k)\backslash G(k)/J$. Then, any $J$-invariant function on $xG(k)$ can be written uniquely as a series $\sum_{g\in S} a_s \mathbf{1}_{xgJ}$ and the map
			$$\displaystyle \sum_{g\in S} a_s \mathbf{1}_{xgJ}\mapsto \sum_{g\in S} a_s \mathbf{1}_{J'gJ}$$
			yields the desired section $S_x^J: \cC(X)_x^J\to \cC(G)^{J'\times J}$.
			
			Let $\pi\in \Temp(G)$. To show that $\cS(X)_\pi=\cC(X)_\pi$, we may restrict to one $G(k)$-orbit, i.e.\ it suffices to prove $\cS(X)_{x,\pi}=\cC(X)_{x,\pi}$. Let $T: \cS(X)_x\to \pi$ be a $G$-morphism. We need to show that it extends to $\cC(X)_x$. Of course, it suffices to do it at the level of $J$-invariants. Since $\pi$ is $G$-tempered, the composition $\cS(G)^{J'\times J}\to \pi^J$ of $T$ with $P_x: \cS(G)^{J'\times J}\to \cS(X)^J$ extends to a continuous morphism $\tilde{T}: \cC(G)^{J'\times J}\to\pi^J$. Now, note that the section $S_x^J: \cC(X)_x^J\to \cC(G)^{J'\times J}$ constructed above sends $\cS(X)^J$ into $\cS(G)^{J'\times J}$ and it entails that $\tilde{T}S_x^J(\varphi)=T(\varphi)$ for every $\varphi\in \cS(X)^J$. Therefore, the composition $\tilde{T}S_x^J: \cC(X)^J\to \pi^J$ provides the desired continuous extension and we are done.
		\end{enumerate}
	\end{proof}
\end{paragr}

\begin{paragr}[Strongly tempered varieties]\label{S strongly tempered}
	The variety $X$ is called {\it strongly tempered} if for every $\varphi_1,\varphi_2\in \mathcal{S}(X)$ the matrix coefficient
	$$\displaystyle M_{\varphi_1,\varphi_2}: g\in G(k)\mapsto \langle R(g)\varphi_1,\varphi_2\rangle_{X}$$
	belongs to the Harish-Chandra Schwartz space $\mathcal{C}(G)$ i.e.\ if for every $d>0$ we have an estimate
	$$\displaystyle \lvert M_{f_1,f_2}(g)\rvert\ll \Xi^G(g) \sigma_G(g)^{-d},\;\; g\in G(k).$$
	By an argument similar to Lemma \ref{lem tempered}, we can show that this definition is actually equivalent to the one given in \cite[\S 6.2]{SV}, namely that for every $x\in X(k)$ the integral
	$$\displaystyle \int_{H_x(k)} \Xi^G(h)dh$$
	converges. We think however that the first definition is better in that it allows to include the Whittaker case in a uniform way as shown by the following lemma.
	
	\begin{lem}
		Assume that $G$ is quasi-split, let $B=TN\subset G$ be a Borel subgroup and $\psi: N(k)\to \bC^\times$ be a nondegenerate unitary character. Then, the Whittaker variety $(N,\psi)\backslash G$ is strongly tempered in the following sense: for every $W_1,W_2\in \cS((N,\psi)\backslash G)$, the matrix coefficient
		$$\displaystyle g\in G(k)\mapsto \langle R(g)W_1,W_2\rangle_{N\backslash G}:=\int_{N(k)\backslash G(k)} W_1(x)\overline{W_2(x)} dx$$
		belongs to the Harish-Chandra Schwartz space $\mathcal{C}(G)$.
	\end{lem}
	
	\begin{proof}
		We give the proof in the non-Archimedean case. The Archimedean case can be dealt similarly by using the Dixmier-Malliavin's theorem in place of small compact-open subgroups. As in the proof of Lemma \ref{lem tempered}, we may assume that $W_i=R(f_i)\delta_1$, that is $W_i(x)=\int_{N(k)} f_i(ux)\psi(u)^{-1}du$, for $i=1,2$ and for some functions $f_1,f_2\in C_c^\infty(G)$. Then, we have
		\begin{equation*}
			\displaystyle \langle R(g)W_1,W_2\rangle_{N\backslash G}=\int_{G(k)\times N(k)} f_1(uxg)\overline{f_2(x)} \psi(u)dudx.
		\end{equation*}
		Let $J_T\subset T(k)$ ba a compact-open subgroup leaving $f_1$ and $f_2$ invariant on the left. Then, the above integral doesn't change if we replace $\psi(u)$ by $\psi(tut^{-1})$ for $t\in J_T$. Thus,
		\begin{equation}\label{eq1 coef Whitt}
			\displaystyle \langle R(g)W_1,W_2\rangle_{N\backslash G}=\int_{G(k)\times N(k)} f_1(uxg)\overline{f_2(x)} \widehat{\mathbf{1}}_{J_T}(u)dudx
		\end{equation}
		where we have set $\widehat{\mathbf{1}}_{J_T}(u):=\int_{J_T} \psi(tut^{-1})dt$. As readily follows from the theory of the Fourier transform, $\widehat{\mathbf{1}}_{J_T}$ is a function of compact support on the quotient $N(k)/N_{\der}(k)$ where $N_{\der}\subset N$ is the derived subgroup. Furthermore, for every $d\geq 0$ we have an estimate
		\begin{equation}\label{eq2 coef Whitt}
			\displaystyle \left\lvert \int_{G(k)} f_1(g_1xg_2)\overline{f_2(x)} dx\right\rvert\ll_d \Xi^G(g_1)\Xi^G(g_2)\sigma_G(g_1)^d \sigma_G(g_2)^{-d},\; \mbox{ for } g_1,g_2\in G(k).
		\end{equation}
		Indeed, since $f_1$ is compactly supported we certainly have
		$$\displaystyle \lvert f_1(g_1xg_2)\rvert\ll_d \Xi^G(g_1xg_2) \sigma_G(g_1xg_2)^{-d}\ll \Xi^G(g_1xg_2) \sigma_G(g_1)^d \sigma_G(x)^d\sigma_G(g_2)^{-d},$$
		and by \eqref{doublind ineq} we also have
		$$\displaystyle \int_{G(k)} \Xi^G(g_1xg_2) \sigma_G(x)^d\lvert f_2(x)\rvert dx\ll_d \Xi^G(g_1)\Xi^G(g_2).$$
		From \eqref{eq1 coef Whitt} and \eqref{eq2 coef Whitt}, to conclude it suffices to know the convergence of the integral
		$$\displaystyle \int_{N(k)} \Xi^G(u) \sigma_G(u)^d \widehat{\mathbf{1}}_{J_T}(u)du$$
		for every $d\geq 0$. But this follows from the convergence of $\int_{N_{\der}(k)} \Xi^G(v) \sigma_G(v)^d dv$ that was shown in \cite[Proposition 6.3.1]{SV} (see \cite[Appendix B.3]{BP3} for an argument in the Archimedean case).
		
	\end{proof}
	
	\begin{rem}
		When $k$ is non-Archimedean, there is a stronger result of Lapid and Mao \cite{LaMao} that says that for every $W_1,W_2\in \cS(N(k),\psi\backslash G(k))$ the matrix coefficient $g\in G(k)\mapsto \langle R(g)W_1,W_2\rangle_{N\backslash G}$ is actually compactly supported. It would be interesting to establish the analog result for $k$ Archimedean, that is $g\in G(k)\mapsto \langle R(g)W_1,W_2\rangle_{N\backslash G}$ belongs to the Schwartz space $\cS(G)$.
	\end{rem}
	
	When $X$ is strongly tempered, we can write an explicit Plancherel decomposition for $L^2(X)$ by applying the Plancherel formula for the group $G(k)$ to the matrix coefficient $M_{f_1,f_2}\in \cC(G)$. More precisely, this yields
	$$\displaystyle \langle f_1,f_2\rangle_{X}=M_{f_1,f_2}(1)=\int_{\Temp(G)} J_\pi^X(f_1,f_2) d\mu_G(\pi)$$
	where $d\mu_G(\pi)$ is the Plancherel measure for $G(k)$ and the relative character $J_\pi^X$ is given by
	\begin{equation}\label{eq Jxpi strongly tempered}
		\displaystyle J_\pi^X(f_1,f_2)=\Tr \pi(M_{f_1,f_2}).
	\end{equation}
	Note that this defines a relative character $J_\pi^X$ for {\it every} $\pi\in \Temp(G)$ (and not just $\mu_G$-almost every $\pi$.) Moreover, the above decomposition automatically entails that $J_\pi^X$ is positive semi-definite \cite[Proposition 6.1.1]{SV} (i.e.\ $J_\pi^X(f,f)\geq 0$) so that it is really a Plancherel formula in the sense of Section \ref{Sect distinction and Plancherel}.
	
	For $x\in X(k)$, setting $H=G_x$, the restriction of the relative character $J_\pi^X$ to $\cS(X)_x=\cS(H(k)\backslash G(k))$ is dual to the $H(k)$-invariant Hermitian form
	\begin{equation}\label{eq BHpi strongly tempered}
		\displaystyle \mathcal{B}^H_\pi: \pi\otimes \overline{\pi}\to \bC
	\end{equation}\index{$\mathcal{B}^H_\pi$}
	$$\displaystyle v_1\otimes v_2\mapsto \int_{H(k)} (\pi(h)v_1,v_2) dh.$$
	We refer to \cite[Theorem 6.2.1]{SV} for details. In particular, the form $\mathcal{B}^H_\pi$ is always semi-positive (i.e.\ $\mathcal{B}^H_\pi(v,v)\geq 0$) (cf. {\it op.\ cit.}).
\end{paragr}

\begin{paragr}\label{S Ichino-Ikeda local conjecture}
	The following was conjectured by Ichino-Ikeda in the setting of the Gan-Gross-Prasad conjectures (see below) and established by Sakellaridis-Venkatesh \cite[Theorem 6.4.1]{SV} in the above general setting. We also include here a small improvement that has applications for the development of certain local trace formulas (see Section \ref{Sect local trace formula}).

	\begin{theo}\label{theo nondeg canonical form SV}
		Assume that $k$ is non-Archimedean and let $X$ be a homogeneous spherical variety that is strongly tempered and wavefront. Then, for every $\pi\in \Temp(G)$, $J_\pi^X$ induces a nondegenerate hermitian form on $\cS(X)_\pi=\cC(X)_\pi$. Moreover, for $L$ a Levi subgroup and $\sigma$ a discrete series of $L(k)$, denoting by $\mathfrak{X}^{\unit}(L)$ the torus of unitary unramified characters of $L(k)$, the $X$-multiplicity $m_X(I_L^G(\sigma\otimes \lambda))$ is finite and independent of $\lambda\in \mathfrak{X}^{\unit}(L)$.
	\end{theo}
	
	\begin{rem}
		When the variety $X$ is wavefront, it is known that the multiplicity $m_X(\pi)$ is finite for every irreducible representation $\pi$ of $G(k)$, see \cite[Theorem 5.1.5]{SV} or \S \ref{S smooth asymptotics cons} below.
		
		Let $x\in X(k)$ and set $H=G_x$. The theorem admits the following, equivalent but closer to the original formulation by Ichino-Ikeda, dual formulation: for every $\pi\in \Temp(G)$, the form $\mathcal{B}^H_\pi$ descends to a nondegenerate pairing on the $H(k)$-coinvariant spaces $\pi_H\times \overline{\pi}_H$ and, for $\sigma$ a discrete series of a Levi $L$, the coinvariant space $I_L^G(\sigma\otimes \lambda)_H$ is of finite dimension independent of $\lambda\in \mathfrak{X}^{\unit}(L)$.
		
		In \cite[Theorem 6.4.1]{SV}, it is only proven that the form $\mathcal{B}^H_\pi$ is nonzero when $\pi_H\neq 0$ and that if $I_L^G(\sigma\otimes \lambda_0)_H\neq 0$ for some $\lambda_0\in \mathfrak{X}^{\unit}(L)$ then $I_L^G(\sigma\otimes \lambda)_H\neq 0$ for all $\lambda\in \mathfrak{X}^{\unit}(L)$. This is obviously sufficient to imply the claim when we are in a multiplicity one setting, i.e.\ when $\dim(\pi_H)\leq 1$ for every $\pi\in \Irr(G)$, such as for the Gan-Gross-Prasad conjectures (see Section \ref{Sect GGP}), but not in general. Actually, the argument in \cite[Theorem 6.4.1]{SV} already shows that the pairing $\mathcal{B}^H_\pi$ is nondegenerate for discrete series (of $G$) but for other tempered representations the proof needs to be pushed a bit further using a refinement of the Plancherel decomposition for the group $G(k)$ in the form of a complete matrix Paley-Wiener theorem for the Harish-Chandra Schwartz space $\cC(G)$ (which is due to Harish-Chandra see \cite{WaldPlanch}). Let us review briefly the Sakellaridis-Venkatesh argument in order to explain the refinement.
	\end{rem}
	
	\begin{proof}
		For notational simplicity we will write the proof in the case where $X(k)=H(k)\backslash G(k)$ (i.e.\ there is only one $G(k)$-orbit).
		
		Let $J\subset G(k)$ be a compact-open subgroup. Recall from Proposition \ref{prop tempered} that the $G$-morphism $P_H: \cS(G)\to \cS(X)$ extends to a continuous surjection $\cC(G)\to \cC(X)$. We shall denote by $P_J$ the restriction of $P_H$ to the subspace $\cS(J\backslash G)$ of left $J$-invariant functions. The morphism
		$$\displaystyle P^*_J: C^\infty(X)\to C^\infty(J\backslash G),\;\; \varphi\mapsto e_J\star \varphi,$$
		where we identified every function on $X(k)$ with a $H(k)$-invariant function on $G(k)$, is the adjoint of $P_J$ i.e.\ $\langle P_Jf,\varphi\rangle_X=\langle f,P_J^* \varphi\rangle_G$ for every $(f,\varphi)\in \cS(J\backslash G)\times C^\infty(X)$.
		
		First, we establish some basic properties of the maps $P_J$ and $P_J^*$.
		
		\begin{num}
			\item\label{eq0 theo nondeg canonical form SV} There exists $f_J\in \cC(J\backslash G/J)$ such that $P_J^* P_J(f)=f_J\star f$ for every $f\in \cC(J\backslash G)$ and moreover we have $P^*_J \cC(X)\subseteq \cC(G)$.
		\end{num}
		
		We start by remarking that $P_J^* \cS(X)\subset \cC(G)$. Indeed, this follows directly from the fact that $X$ is strongly tempered by noting that
		$$\displaystyle (P_J^*\varphi)(g)=(e_{J}\star \varphi)(g)=\langle R(g)\varphi, P_H(e_{J})\rangle_X,\;\; g\in G(k).$$
		Set $f_J=P_J^* P_J(e_J)$. Then, by equivariance, for every $f\in \cS(J\backslash G)$ we have
		$$\displaystyle P_J^* P_J(f)=P_J^* P_J(e_J)\star f=f_J\star f.$$
		Since the linear form $f\in \cC(J\backslash G)\mapsto P_J^* P_J(f)(g)$, for $g\in G(k)$, is continuous and $\cS(J\backslash G)$ is dense in $\cC(J\backslash G)$, it follows that the formula $P_J^* P_J(f)=f_J\star f$ holds true for every $f\in \cC(J\backslash G)$. 
		Finally, to show the containment $P^*_J \cC(X)\subseteq \cC(G)$ we use that $P_H: \cC(G)\to \cC(X)$ is surjective. Indeed, this implies that for every $\varphi\in \cC(X)$ we can find a compact-open subgroup $J'\subset J$ as well as $f\in \cC(J'\backslash G)$ such that $\varphi=P_{J'}f$. Hence, by the first part applied to $J'$ instead of $J$, $P_{J'}^* \varphi=P_{J'}^* P_{J'} f= f_{J'}\star f \in \cC(G)$ since the Harish-Chandra Schwartz space is stable by convolution. The result follows by noting that $P_J^* \varphi=e_J\star P_{J'}^* \varphi$.
		
		\begin{num}
			\item\label{eq1 theo nondeg canonical form SV} For every $\pi\in \Temp(G)$, $f_1,f_2\in \cC(J\backslash G)$ and $v_1,v_2\in \pi^J$, we have
			$$\displaystyle J_\pi^X(P_Jf_1,P_Jf_2)=J_\pi^G(f_1,f_J\star f_2)$$
			and
			$$\displaystyle \mathcal{B}_\pi^H(v_1,v_2)=(\pi(f_J)v_1,v_2)$$
			where $J_\pi^G$ denotes the relative character for the group (that is $J_\pi^G(f_1,f_2)=\Tr(\pi(f_1)\pi(f_2)^\ast)$).
		\end{num}
		
		The first equality is just a restatement of formula \eqref{eq Jxpi strongly tempered}. It can also be deduced directly from the chain of equalities
		\begin{equation}\label{eq2 theo nondeg canonical form SV}
			\displaystyle \langle P_J f_1,P_J f_2\rangle_X=\langle f_1,P_J^* P_J f_2\rangle_G=\langle f_1,f_J\star f_2\rangle_G\int_{\Temp(G)} J_\pi^G(f_1,f_J\star f_2) d\mu_G(\pi)
		\end{equation}
		which follows from adjointness and the Plancherel formula for the group $G$. The second equality is similarly just a restatement of the formula \eqref{eq BHpi strongly tempered} or can be deduced from the first one by dualizing.
		
		\begin{num}
			\item\label{eq1bis theo nondeg canonical form SV} For every $\pi\in \Temp(G)$, the operator $\pi(f)$ is a nonnegative Hermitian operator.
		\end{num}
		This follows directly from the second identity in \eqref{eq1 theo nondeg canonical form SV} since $\mathcal{B}_\pi^H$ is a positive semi-definite Hermitian form on $\pi$.
		
		Next, according to the Wavefront Lemma \ref{wavefront lemma}, we can find a subset $G^+\subset G(k)$, another compact-open subgroup $J'$ and a constant $C>0$ such that $G(k)=H(k)G^+$, $H(k)gJ\supset J'gJ$ and $\vol(H(k)\backslash H(k)gJ)\leq C \vol(J'gJ)$ for every $g\in G^+$.  Up to shrinking $J'$, we will also assume that (Proposition \ref{prop tempered}(ii))
		$$\displaystyle P_{J'}\cC(J'\backslash G/J)=\cC(X)^J.$$
		The crux of the proof is the following inequality which already appears in the proof of \cite[Theorem 6.4.1]{SV}: 
		\begin{equation}\label{eq3 theo nondeg canonical form SV}
			\displaystyle \langle \varphi,\varphi\rangle_X\leq C \langle P_{J'}^* \varphi,P_{J'}^* \varphi\rangle \mbox{ for every } \varphi\in \cC(X)^J.
		\end{equation}
		Indeed, without loss in generality, we may assume that $G^+$ is left $J'$- and right $J$-invariant. It then follows that
		\[\begin{aligned}
			\displaystyle \langle \varphi,\varphi\rangle_X=\int_{X(k)} \lvert \varphi(x)\rvert^2 dx & \leq \sum_{g\in J'\backslash G^+/J} \vol(H(k)gJ) \lvert \varphi(g)\rvert^2 \\
			& \leq C \int_{G^+} \lvert (e_{J'}\star \varphi)(g)\rvert^2 dg\leq C \langle P_{J'}^* \varphi,P_{J'}^* \varphi\rangle
		\end{aligned} \]
		for every $\varphi\in \cC(X)^J$.
		
		Combining \eqref{eq2 theo nondeg canonical form SV} with \eqref{eq3 theo nondeg canonical form SV} and a new application of the Plancherel formula for the group, we arrive at the inequality
		\[\begin{aligned}
			\displaystyle & \int_{\Temp(G)} J_\pi^G(f,f_{J'}\star f) d\mu_G(\pi)=\langle P_{J'}f, P_{J'}f\rangle_X \\
			& \leq C \langle P_{J'}^\vee P_{J'}f,P_{J'}^\vee P_{J'}f\rangle_G=C \int_{\Temp(G)} J_\pi^G(f_{J'}\star f,f_{J'}\star f) d\mu_G(\pi)
		\end{aligned}\]
		for every $f\in \cS(J'\backslash G/J)$. By a standard spectral separation argument (see e.g. \cite[Proposition 6.1.1]{SV}) this then yields
		\begin{equation}\label{eq4 theo nondeg canonical form SV}
			\displaystyle J_\pi^G(f,f_{J'}\star f)\leq C J_\pi^G(f_{J'}\star f,f_{J'}\star f), \; f\in \cS(J'\backslash G/J),
		\end{equation}
		for $\mu_G$-almost every $\pi\in \Temp(G)$.
		
		Let $\pi\in \Temp(G)$ such that the inequality \eqref{eq4 theo nondeg canonical form SV} holds. By definition of the relative character $J_\pi^G$, it can be rewritten as an inequality of trace:
		$$\displaystyle \Tr(\pi(f_{J'})\pi(f)\pi(f)^*)\leq C \Tr(\pi(f_{J'})^*\pi(f_{J'})\pi(f)\pi(f)^*),\; \mbox{ for } f\in \cS(J'\backslash G/J).$$
		Assume moreover that $\pi^J\neq 0$ (otherwise $\pi(f)=0$ for every $f\in \cS(J'\backslash G/J)$ and the inequality is trivial). Then, since $\pi$ is irreducible, $\pi(f)$ can be any linear map $\pi^J\to \pi^{J'}$. In particular, letting $\pi(f)$ run over all such rank one operators, we obtain
		\begin{equation}\label{eq5 theo nondeg canonical form SV}
			\displaystyle (\pi(f_{J'})v,v)\leq C (\pi(f_{J'})v,\pi(f_{J'})v), \mbox{ for every } v\in \pi^{J'},
		\end{equation}
		where $(.,.)$ is any invariant inner product on $\pi$.
		
		Let now $L\subset G$ be a Levi subgroup and $\sigma\in \Pi_2(L)$ be a discrete series. For each $\lambda\in \mathfrak{X}^{unit}(L)$ we set $\pi_\lambda=I_L^G(\sigma\otimes \lambda)$. Assume that $\pi_\lambda^J\neq 0$ (if it holds for one $\lambda$ then it automatically holds for all $\lambda$). Then, for almost all $\lambda\in \mathfrak{X}^{unit}(L)$, the representation $\pi_\lambda$ is irreducible and the restriction of the Plancherel measure $d\mu_G$ to this family is in the same class as (the pushforward of) any Haar measure on $\mathfrak{X}^{unit}(L)$. Therefore, the inequality \eqref{eq4 theo nondeg canonical form SV} is verified with $\pi=\pi_\lambda$ for almost all $\lambda$. However, all the unitary representations $\pi_\lambda$ can be realized on a common pre-Hilbert space $I_{P\cap K}^K(\sigma)$, where $P$ is a parabolic subgroup with Levi factor $L$ and $K\subset G(k)$ a suitable maximal compact subgroup satisfying the Iwasawa decomposition $G(k)=P(k)K$, and both sides of the inequality (for a fixed vector $v\in I_{P\cap K}^K(\sigma)^{J'}$) are readily seen to be $C^\infty$ functions in $\lambda$. Thus, the inequality holds for every $\pi_\lambda$ and we have
		\begin{equation}
			\displaystyle (\pi_\lambda(f_{J'})v,v)\leq C (\pi_\lambda(f_{J'})v,\pi_\lambda(f_{J'})v), \mbox{ for every } \lambda\in \mathfrak{X}^{\unit}(L), v\in \pi_\lambda^{J'}.
		\end{equation}
		Since $\pi_\lambda(f_{J'})$ is a nonnegative Hermitian operator (\eqref{eq1bis theo nondeg canonical form SV}), the above inequality just says that all its nonzero eigenvalues are bounded from below by $C^{-1}$. Because the map $\lambda\mapsto \pi_\lambda(f_{J'})\in \End(I_{P\cap K}^K(\sigma)^J)$ is continuous (even smooth), this implies:
		\begin{num}
			\item\label{eq6 theo nondeg canonical form SV} The rank of $\pi_\lambda(f_{J'})$ is independent of $\lambda\in \mathfrak{X}^{\unit}(L)$.
		\end{num}
		Note that by the relation \eqref{eq1 theo nondeg canonical form SV} between $\pi(f_{J'})$ and $\mathcal{B}_\pi^H$, this shows that the rank of the Hermitian form $\mathcal{B}_{\pi_\lambda}^H$ is independent of $\lambda\in \mathfrak{X}^{\unit}(L)$. Thus, the second part of the theorem, namely that the multiplicity $m_X(\pi_\lambda)$ is independent of $\lambda$, is now a consequence of the first part (according to which the rank of $\mathcal{B}_{\pi_\lambda}^H$is $m_X(\pi_\lambda)$).
		
		As a consequence of the above, we can now show:
		\begin{num}
			\item\label{eq7 theo nondeg canonical form SV} $P_{J'}^* \cC(X)^J$ is a closed direct summand of $\cC(J'\backslash G/J)$ as a topological $\cH(G,J)$-module. More precisely, we have the decomposition
			\begin{equation*}
				\displaystyle \cC(J'\backslash G/J)=\Ima(P_{J'}^*\mid_{\cC(X)^J})\oplus \Ker(P_{J'}\mid_{\cC(J'\backslash G/J)}).
			\end{equation*}
		\end{num}
		Note that, by adjunction, the two subspaces $\Ima(P_{J'}^*\mid_{\cC(X)^J})$ and $\Ker(P_{J'}\mid_{\cC(J'\backslash G/J)})$ are orthogonal with respect to the $L^2$-scalar product $\langle .,.\rangle_G$ and
		$$\displaystyle \Ker(P_{J'}\mid_{\cC(J'\backslash G/J)})=\Ker(P_{J'}^*P_{J'}\mid_{\cC(J'\backslash G/J)})=\{f\in \cC(J'\backslash G/J)\mid f_{J'}\star f=0\}=\Ker(L(f_{J'})\mid_{\cC(J'\backslash G/J)})$$
		where $L(f_{J'})$ stands for the operator of left convolution by $f_{J'}$. Moreover, by our choice of $J'$ (i.e.\ such that $P_{J'}\cC(J'\backslash G/J)=\cC(X)^J$) and \eqref{eq0 theo nondeg canonical form SV}, we have
		$$\displaystyle \Ima(P_{J'}^*\mid_{\cC(X)^J})=f_{J'}\star \cC(J'\backslash G/J)=\Ima(L(f_{J'})\mid_{\cC(J'\backslash G/J)}).$$
		Therefore, it suffices to show that the subspace $\Ima(L(f_{J'})\mid_{\cC(J'\backslash G/J)})$ is closed and
		$$\displaystyle \cC(J'\backslash G/J)=\Ima(L(f_{J'})\mid_{\cC(J'\backslash G/J)})\oplus \Ker(L(f_{J'})\mid_{\cC(J'\backslash G/J)}).$$
		For this, we shall use the spectral description of $\cC(G)$ given in \cite{WaldPlanch} that we now recall. More precisely, this theorem characterizes the image of the Fourier transform $f\in \cC(G)\mapsto (\pi\in \Temp(G)\mapsto \pi(f))$, which is injective, as follows. An assignment $\pi\in \Temp(G)\mapsto T_\pi\in \End(\pi)$\footnote{Note that $\End(\pi)$ depends on the isomorphism class of $\pi$ up to a {\it unique} isomorphism, an easy consequence of Schur's lemma.} is of the form $\pi\mapsto \pi(f)$ for some $f\in \cC(G)$ if and only if the following conditions are met:
		\begin{itemize}
			\item There exists a compact-open subgroup $J_0$ such that $T_\pi$ is bi-$J_0$-invariant for all $\pi$;
			
			\item For every parabolic subgroup $P=LU\subset G$ and $\sigma\in \Pi_2(L)$, identifying the spaces of all the induced representations $I_P^G(\sigma\otimes \lambda)$ for $\lambda\in \mathfrak{X}^{\unit}(L)$ as above (by choosing a suitable maximal compact subgroup), the map $\lambda\in \mathfrak{X}^{\unit}(L)\mapsto T_{I_P^G(\sigma\otimes\lambda)}\in I_P^G(\sigma\otimes\lambda)\otimes I_P^G(\sigma\otimes\lambda)^\vee$\footnote{Here, if $I_P^G(\sigma\otimes\lambda)$ is reducible, $T_{I_P^G(\sigma\otimes\lambda)}$ implicitely stands for the sum of the $T_\pi$ for its various irreducible summands $\pi$. (By unitarity, $I_P^G(\sigma\otimes\lambda)$ is automatically semisimple.)} is $C^\infty$. (Note that by the first condition it automatically lands in a finite dimensional subspace.)
		\end{itemize}
		Let us define $\cC(J'\backslash G/J)_X$ and $\cC(J'\backslash G/J)^X$ as the subspaces of those $f\in \cC(J'\backslash G/J)$ such that for every $\pi\in \Temp(G)$, $\Ima(\pi(f))\subset \Ima(\pi(f_{J'}))$ or $\Ima(\pi(f))\subset \Ima(\pi(f_{J'}))^\perp=\Ker(\pi(f_{J'}))$ respectively. Since the Fourier transform $f\mapsto(\pi\mapsto \pi(f))$ is injective and transforms the convolution product into pointwise composition we have
		$$\displaystyle \Ima(L(f_{J'})\mid_{\cC(J'\backslash G/J)})\subset \cC(J'\backslash G/J)_X \mbox{ and } \Ker(L(f_{J'})\mid_{\cC(J'\backslash G/J)})=\cC(J'\backslash G/J)^X.$$
		Moreover, for every $f\in \cC(J'\backslash G/J)$ and $\pi\in \Temp(G)$, the operator $\pi(f)$ can be uniquely decomposed as
		$$\displaystyle \pi(f)=\pi(f)_X+\pi(f)^X$$
		where $\Ima(\pi(f)_X)\subset \Ima(\pi(f_{J'}))$ and $\Ima(\pi(f)^X)\subset \Ima(\pi(f_{J'}))^\perp$.  Thanks to \eqref{eq6 theo nondeg canonical form SV}\footnote{More precisely, we are using the combination of \eqref{eq6 theo nondeg canonical form SV} and the following elementary fact: if $M$ is a $C^\infty$ manifold, $V$ a finite dimensional Hilbert space and $\lambda\in M\mapsto T_\lambda\in \End(V)$ a $C^\infty$ map with $T_\lambda$ of constant rank then the orthogonal projection $p_\lambda$ onto $\Ima(T_\lambda)$ also varies smoothly with $\lambda$.}, the assignments $\pi\in \Temp(G)\mapsto \pi(f)_X$ and $\pi\mapsto \pi(f)^X$ still satisfy the two conditions above and therefore are given as the Fourier transforms of two functions $f_X\in \cC(J'\backslash G/J)_X$ and $f^X\in \cC(J'\backslash G/J)^X$. This shows that
		$$\displaystyle \cC(J'\backslash G/J)=\cC(J'\backslash G/J)_X\oplus \cC(J'\backslash G/J)^X.$$
		Furthermore, for every $f\in \cC(J'\backslash G/J)_X$ and $\pi\in \Temp(G)$, $\pi(f)$ can be written uniquely as a product $\pi(f_{J'})T_\pi$ where $\Ima(T_\pi)\subset \Ker \pi(f_{J'})^\perp$ and, once again thanks to \eqref{eq6 theo nondeg canonical form SV}, the assignment $\pi\mapsto T_\pi$ also satisfies the two condition above and is therefore the Fourier transform of a unique function $g\in \cC(J'\backslash G/J)$. This implies $f=f_{J'}\star g$ and therefore
		$$\displaystyle \cC(J'\backslash G/J)_X=f_{J'}\star \cC(J'\backslash G/J).$$
		
		We can now finish the proof of the theorem. Namely, let $\pi\in \Temp(G)$. Without loss in generality we may assume that $\pi^J\neq 0$. It then suffices to prove that $J_\pi^X$ is nondegenerate on $\cC(X)_\pi^J$. By adjointness, the surjectivity $\cC(X)^J=P_{J'}\cC(J'\backslash G/J)$ implies that the restriction of $P^*_{J'}$ to $\cC(X)^J$ is injective and therefore, by \eqref{eq7 theo nondeg canonical form SV}, it identifies $\cC(X)^J$ with a (topological) direct summand of $\cC(J'\backslash G/J)$. This entails in particular that the morphism $P_{J'}^*$ descends to a map between coinvariants $P_{J',\pi}^\vee: \cC(X)_\pi\to \cC(G)_\pi$ that is injective on $\cC(X)_\pi^J$. By the relation \eqref{eq1 theo nondeg canonical form SV}, and since $J_\pi^G$ is nondegenerate on $\cC(G)_\pi$, this implies the desired result.
	\end{proof}

	It follows from the above theorem that when $X$ is strongly tempered and wavefront, the Plancherel decomposition reads
	$$\displaystyle L^2(X)\simeq \int_{\Temp(G)}^\oplus \pi^{\oplus m_X(\pi)} d\mu_G(\pi)$$
	where $m_X(\pi)=dim \Hom_G(\cS(X),\pi)$ is the (smooth) $X$-multiplicity of $\pi$. Therefore, in this case describing the $L^2$-spectrum of $X$ amounts to computing the multiplicities $m_X(\pi)$ for $\pi$ tempered.
\end{paragr}

\subsection{The Gan-Gross-Prasad local conjectures}\label{Sect GGP}

In this subsection, we discuss a particularly important instance of strongly tempered spherical varieties for which conjectures of Gan-Gross-Prasad (now established in all cases) give an important refinement of the local Conjecture \ref{Conj local strong SV}; describing the distinguished representations in each (extended) tempered $L$-packets.

\vspace{2mm}

\begin{paragr}[The orthogonal Gross-Prasad varieties.]
	One seminal example of a situation where we have precise description of the multiplicity $m_X(\pi)$ is provided by the {\em Gan-Gross-Prasad conjectures} \cite{GGP} that we now describe in a particular case for orthogonal groups. Let $(V,q)$ be quadratic space over $k$, $W\subset V$ be a nondegenerate hyperplane and $L=W^\perp$ be its orthogonal complement (a line). Set
	$$\displaystyle H=\SO(W)\hookrightarrow G=\SO(W)\times \SO(V),$$
	where the embedding of $H$ in $G$ is the product of the identity with the natural inclusion $\SO(W)\subset \SO(V)$, which identifies $\SO(W)$ with the subgroup of those $g\in \SO(V)$ that act trivially on $L$. Then, the original conjecture of Gross-Prasad \cite{GP} concerns the multiplicity function $\pi\in \Irr(G)\mapsto m(\pi)=m_X(\pi)$ associated to the spherical variety $X=H\backslash G$. A theorem of Aizenbud, Gourevitch, Rallis and Schiffmann \cite{AGRS} says that we always have $m(\pi)\leq 1$ so that the only remaining question is to determine when this multiplicity is nonzero.
\end{paragr}

\begin{paragr}[$L$-group.]
	The spherical variety $X=H\backslash G$ is strongly tempered, wavefront and its $L$-group is the same as that of $G$:
	$$\displaystyle {}^L G_X={}^L G={}^L \SO(W)\times_{W_k} {}^L \SO(V).$$
	Here, we recall that the $L$-group of $\SO(W)$ is given by
	$$\displaystyle {}^L \SO(W)=\left\{\begin{array}{ll}
		\Sp_{2m}(\bC)\times W_k & \mbox{ if } \dim(W)=2m+1 \mbox{ is odd,} \\
		\SO_{2m}(\bC)\rtimes W_k & \mbox{ if } \dim(W)=2m \mbox{ is even.}
	\end{array} \right.$$
	Moreover, in the case where $\dim(W)=2m$ is even, the Galois action on $\SO_{2m}(\bC)$ factors through the Galois group of a separable extension $\ell/k$ that is at most quadratic and if $\ell\neq k$ the nontrivial element of $\Gal(\ell/k)$ acts as the conjugation of an element in the nontrivial identity component $\mathrm{O}_{2m}(\bC)\setminus \SO_{2m}(\bC)$ (up to translation by the center $Z(\SO_{2m}(\bC))=\{ \pm 1\}$ there is a unique such element that fixes the pinning on the dual group $\SO_{2m}(\bC)$). In particular, there is a natural morphism ${}^L \SO(W)\to \mathrm{O}_{2m}(\bC)$ when $\dim(W)=2m$ and ${}^L \SO(W)\to \Sp_{2m}(\bC)$ when $\dim(W)=2m+1$. In other words, ${}^L \SO(W)$ admits a natural representation on a $\bC$-vector space $N$ that is symplectic when $\dim(W)$ is odd and orthogonal when $\dim(W)$ is even. Similarly, ${}^L \SO(V)$ admits a natural representation on a $\bC$-vector space $M$ that is symplectic when $\dim(V)$ is odd and orthogonal when $\dim(V)$ is even. The tensor product of $N$ and $M$ yields a representation of the $L$-group of $G$,
	$$\displaystyle r:{}^L G\to \GL(N\otimes M).$$
	Note that, since $\dim(W)=\dim(V)-1$ and $\dim(V)$ are of different parities, this representation of ${}^L G$ is always of symplectic type.
\end{paragr}

\begin{paragr}[Pure inner forms.]
	The pure inner forms (PIF) of the variety $X$ are of a similar form. More precisely, $H^1(k,\SO(V))$ (resp. $H^1(k,\SO(W))$) naturally parametrizes isomorphism classes of quadratic spaces of the same dimension and discriminant as $V$ (resp. as $W$) and if $\alpha\in H^1(k,\SO(V))$ (resp. $\beta\in H^1(k,\SO(W))$) corresponds to the quadratic space $V_\alpha$ (resp. $W_\beta$) then the corresponding pure inner form is its special orthogonal group $\SO(V_\alpha)$ (resp. $\SO(W_\beta)$). In particular, the {\it relevant} pure inner forms\footnote{i.e.\ the set of PIF with nonempty set of $k$-points.} of $X$ are the spherical varieties $X_\beta=\SO(W_\beta)\backslash \SO(V_\beta)$ for $\beta\in H^1(k,\SO(W))$, where $V_\beta=W_\beta\oplus^\perp L$, a quadratic space that is isomorphic to $V_\alpha$ for $\alpha\in H^1(k,\SO(V))$ the image of $\beta$ by the map $H^1(k,\SO(W))\to H^1(k,\SO(V))$, and we recall that $L$ denotes the orthogonal complement of $W$ in $V$.
\end{paragr}

\begin{paragr}[Local Langlands correspondence for orthogonal groups.]
	Let $\phi_V: \mathcal{L}_k\to {}^L \SO(V)$, $\phi_W:\mathcal{L}_k\to {}^L \SO(W)$ be $L$-parameters that we will consider as symplectic or orthogonal representations on $M$ and $N$ respectively. According to Vogan's version of the local Langlands correspondence, these $L$-parameters should give rise to $L$-packets
	$$\displaystyle \Pi^{\SO(V_\alpha)}(\phi_V)\subseteq \Irr(\SO(V_\alpha)),\;\; \Pi^{\SO(W_\beta)}(\phi_W)\subseteq \Irr(\SO(W_\beta)),$$
	for all $(\alpha,\beta)\in H^1(k,\SO(V))\times H^1(k,\SO(W))$ (that is for all the pure inner forms of $\SO(V)$ and $\SO(W)$) and there should exist bijections:
	\begin{equation}\label{LLC SOV}
		\displaystyle \bigsqcup_{\alpha\in H^1(k,\SO(V))} \Pi^{\SO(V_\alpha)}(\phi_V)\simeq \Irr(S_{\phi_V}),\;\; \sigma \mapsto \langle \sigma,.\rangle,
	\end{equation}
	
	\begin{equation}\label{LLC SOW}
		\displaystyle \bigsqcup_{\beta\in H^1(k,\SO(W))} \Pi^{\SO(W_\beta)}(\phi_W)\simeq \Irr(S_{\phi_W}),\;\; \sigma' \mapsto \langle \sigma',.\rangle
	\end{equation}
	where $S_{\phi_V}=\pi_0(Z_{\widehat{\SO(V)}}(\phi_V))$, $S_{\phi_W}=\pi_0(Z_{\widehat{\SO(W)}}(\phi_W))$ denote the (finite) component groups of the centralizers of $\phi_V$, $\phi_W$ in the respective dual groups and $\Irr(S_{\phi_V})$, $\Irr(S_{\phi_W})$ their sets of irreducible characters. (Actually, in the case considered here, that of special orthogonal groups, the component groups $S_{\phi_V}$, $S_{\phi_W}$ are always $2$-abelian groups and therefore $\Irr(S_{\phi_V})$, $\Irr(S_{\phi_W})$ are just their character groups). The bijections \eqref{LLC SOV} and \eqref{LLC SOW} are not entirely canonical and should depend on the choice of suitable Whittaker data, more precisely on the choices of quasi-split pure inner forms of $\SO(V)$, $\SO(W)$ as well as Whittaker data on those. Following Gan, Gross and Prasad \cite[\S 17]{GGP} these can be chosen in terms of the pair $(V,W)$ alone. Namely, we first observe that there exists a unique $\beta\in H^1(k,\SO(W))$ such that the relevant pure inner form $\SO(W_\beta)\times \SO(V_\beta)$ is quasi-split (where we recall that $V_\beta=W_\beta\oplus^\perp L$). Thus, it suffices to choose Whittaker data on the special orthogonal groups $\SO(U)$ for $U=W_\beta$ and $V_\beta$. One of these special orthogonal group is odd, hence adjoint, and therefore admits a unique Whittaker datum (up to conjugacy). On the other hand, the (conjugacy classes of) Whittaker data on the even special orthogonal group $\SO(U)$ can be naturally set in bijection with isomorphism classes of non-isotropic lines $L'\subset U$ with the property that the special orthogonal group of its orthogonal complement $(L')^\perp$ is split, see \cite[Proposition 12.1]{GGP}. Then, we choose the Whittaker datum corresponding to a non-isotropic line $L'\subset U$ of the same discriminant as the odd orthogonal space.
	
	We emphasize that the existence of the Local Langlands correspondence for special orthogonal groups, with its required compatibility with the theory of (twisted) endoscopy, was established by Arthur \cite{artbook} (in the quasi-split case) and M\oe{}glin-Renard \cite{MoeReLLC} (in general) with one caveat: in the even orthogonal case $\SO(V)$, $\dim(V)=2m$, this correspondence does not distinguish between $L$-parameters $\mathcal{L}_k\to {}^L \SO(V)$ that are conjugate under the full orthogonal group $\mathrm{O}_{2m}(\bC)$ (whereas the usual equivalence relation on parameters is by conjugation by the dual group $\widehat{\SO(V)}=\SO_{2m}(\bC)$). In particular, this means that the $L$-packets so constructed are all invariant by the outer automorphism given by conjugation by any element in $\mathrm{O}(V)\setminus \SO(V)$. An alternative construction of the local Langlands correspondence for special orthogonal groups, based on the local theta correspondence, is given in \cite{ChenZouLLC}, although this approach doesn't give the endoscopic relations.
\end{paragr}

\begin{paragr}[The Gross-Prasad character.]
	In \cite[\S 6]{GGP}, Gan, Gross and Prasad define a character
	$$\displaystyle \chi_{\phi_V,\phi_W}: S_{\phi_V}\times S_{\phi_W}\to \{\pm 1 \},$$
	by the following formulas
	\begin{equation}\label{GPchar1}
		\displaystyle \chi_{\phi_V,\phi_W}(s,1)=\epsilon(M^{-\widetilde{s}}\otimes N) \det(M^{-\widetilde{s}})(-1)^{\dim(N)/2} \det(N)(-1)^{\dim(M^{-\widetilde{s}})/2},
	\end{equation}
	\begin{equation}\label{GPchar2}
		\displaystyle \chi_{\phi_V,\phi_W}(1,s')=\epsilon(M\otimes N^{-\widetilde{s}'}) \det(M)(-1)^{\dim(N^{-\widetilde{s}'})/2} \det(N^{-\widetilde{s}'})(-1)^{\dim(M)/2},
	\end{equation}
	for $(s,s')\in S_{\phi_V}\times S_{\phi_W}$, where:
	\begin{itemize}
		\item $\tilde{s}\in Z_{\widehat{\SO(V)}}(\phi_V)$, $\widetilde{s}'\in Z_{\widehat{\SO(W)}}(\phi_W)$ are chosen lifts of $s$, $s'$ respectively;
		
		\item $M^{-\widetilde{s}}$ denotes the $-1$-eigenspace of $\widetilde{s}$ in $M$, considered as a representation of the local Langlands group $\mathcal{L}_k$ by composition with the $L$-parameter $\phi_V$, and $N^{-\widetilde{s}'}$ is defined similarly;
		
		\item $\det(M^{-\widetilde{s}})$, $\det(M)$, $\det(N^{-\widetilde{s}'})$ and $\det(N)$ are the determinant characters of the corresponding representations of $\mathcal{L}_k$, seen as characters of $k^\times$ via local class field theory; in particular in the above formulas they are evaluated at $-1\in k^\times$. Moreover, all the vector spaces $M^{-\widetilde{s}}$, $M$, $N^{-\widetilde{s}'}$ and $N$ are of even dimension so that all the exponents are integers.

		\item $\epsilon(M^{-\widetilde{s}}\otimes N)$, $\epsilon(M\otimes N^{-\widetilde{s}'})$ are the local root numbers of the corresponding representations of $\mathcal{L}_k$; since the two representations $M^{-\widetilde{s}}\otimes N$ and $M\otimes N^{-\widetilde{s}'}$ of $\mathcal{L}_k$ are symplectic these local root numbers do not depend on the auxilliary choice of an additive character of $k$.
	\end{itemize}
	
	Finally, it is shown in {\em op.\ cit.}\ that the above definition does not depend on the chosen lifts $\widetilde{s}$, $\widetilde{s}'$ and that this indeed yields a character of $S_{\phi_V}\times S_{\phi_W}$.
\end{paragr}

\begin{paragr}[Generic $L$-parameters.]
	The $L$-parameter $\phi_V$ is said to be {\em generic} if the corresponding Vogan $L$-packet
	$$\displaystyle \bigsqcup_{\alpha\in H^1(k,\SO(V))} \Pi^{\SO(V_\alpha)}(\phi_V),$$
	contains at least a representation that is generic with respect to some Whittaker datum. It was conjectured by Gross and Prasad \cite[Conjectures 2.5, 2.6]{GP}, and proved by Gan and Ichino \cite[Appendix B]{GIGGP} for all classical groups, that if an $L$-parameter $\phi_V$ is generic if and only if its associated local $L$-factor $L(s,\phi_V,\Ad)$ (where $\Ad$ denotes the adjoint representation of ${}^L \SO(V)$) does not have a pole at $s=1$. In particular, tempered $L$-parameter are always generic (their local $L$-factors always have poles of nonpositive real part). Moreover, if $\phi_V$ is generuc then, for each Whittaker datum $\varpi$ on one of the pure inner form $\SO(V_\alpha)$, its Vogan $L$-packet contains exactly one representation that is generic with respect to $\varpi$. Of course, all of this applies equally well to the $L$-parameter $\phi_W$.
\end{paragr}

\begin{paragr}[The Gross-Prasad conjecture.]
	The theorem below is the original local Gross-Prasad conjecture \cite{GP} in the non-Archimedean case which was proven by Waldspurger and M\oe{}glin-Waldspurger \cite{WalGP2}, \cite{MWGP}. For simplicity, we will ignore in its formulation the aforementioned ambiguity built in the currently available version of the local Langlands correspondence for even orthogonal groups and we refer the reader to {\em op.\ cit.}\ for precise statement taking this subtlety into account.
	
	\begin{theo}[Waldspurger, M\oe{}glin-Waldspurger (formerly Gross-Prasad conjecture)]
		Assume that the parameters $\phi_V$, $\phi_W$ are generic and that $k$ is non-Archimedean of characteristic zero. Then, there exists a unique representation
		$$\displaystyle \pi=\sigma'\boxtimes \sigma\in \bigsqcup_{\beta\in H^1(k,SO(W))} \Pi^{SO(W_\beta)}(\phi_W)\boxtimes \Pi^{SO(V_\beta)}(\phi_V)$$
		with $m(\pi)\neq 0$ (hence $m(\pi)=1$ by \cite{AGRS}). Moreover, the character $\langle \sigma,.\rangle\boxtimes \langle \sigma',.\rangle$ of the product of component groups $S_{\phi_V}\times S_{\phi_W}$ is exactly the character $\chi_{\phi_V,\phi_W}$ defgined by the formulas \eqref{GPchar1} and \eqref{GPchar2} above.
	\end{theo}
\end{paragr}

\begin{paragr}
	\begin{rem}
		\begin{itemize}
			\item Gan, Gross and Prasad \cite{GGP} have proposed a more general conjecture pertaining to certain pairs $W\subset V$ of quadratic spaces with $\dim(V)-\dim(W)$ odd. In this generality, $H$ is a so-called ``Bessel subgroup'' that is the semi-direct product of $\SO(W)$ (which is no longer a spherical subgroup of $G$ when $\dim(V)-\dim(W)>1$) by a certain unipotent subgroup which also comes equipped with a character invariant by $\SO(W)$-conjugation. The proof of Waldspurger and M\oe{}glin-Waldspurger actually also applies to this more general setting.
			
			\item Similar conjectures have also been formulated for products of unitary and symplectic/metaplectic groups by Gan-Gross-Prasad \cite{GGP}. This is now known in most cases, see \cite{BP3}, \cite{GIGGP}, \cite{AtobGGP}, \cite{HeGGP}, \cite{XueGGP}.
			
			\item More recently, Gan, Gross and Prasad \cite{GGPnontempered} have extended their conjectures to certain non-generic $L$-packets; more precisely those associated to Arthur parameters. This new prediction also features the same character of component groups as above but also introduce a new combinatorial condition on the Arthur parameters for the corresponding $L$-packets to contain a distinguished representation. This conjecture has been established for general linear groups \cite{KYChan} but remains open for all other classical groups.
		\end{itemize}
	\end{rem}
\end{paragr}

\subsection{Prasad's conjecture on Galois pairs}\label{Sect Prasad conj}

Continuing the discussion initiated in \S \ref{S Galois pairs Lgroup}, we will describe in this subsection general predictions made by D. Prasad \cite{PraGal} on distinction for Galois symmetric varieties. It seems that the conjectures in {\em op.\ cit.}\ didn't achieve yet their final form, in particular with respect to some of the finest aspects related to ramification points of a certain base-change map between spaces of local Langlands parameters, and we will thus restrict ourself to the case of discrete series (for which the conjectures are in a much more definitive form).

\vspace{2mm}

\begin{paragr}[Notation.]
Let $\ell/k$ be a separable quadratic extension and $c$ be the nontrivial element of $\Gal(\ell/k)$. We fix a quasi-split connected reductive group $H$ over $k$ and we set
$$\displaystyle X=H\backslash G,\;\;\; G=\Res_{\ell/k} H_{\ell},$$
for the corresponding Galois symmetric variety. Recall that in \S \ref{S Galois pairs Lgroup} we have described the $L$-group ${}^L G_X$ of $X$ as well as the strong distinguished morphism
$$\displaystyle {}^L \iota_X: {}^L G_X\to {}^L G.$$
In this section, we will adopt the notation introduced at this occasion and, in particular, we denote by $h\mapsto h^\vee$ the duality involution of ${}^L H$ (as defined in \S \ref{S Galois pairs Lgroup} i.e.\ the inner translate of the Chevalley involution that sends the canonical pinning to its opposite).
	
Using that the restriction of the duality involution to $Z(\widehat{H})$ is given by $z\mapsto z^{-1}$, we readily check that the norm map
$$\displaystyle N: Z(\widehat{G})=Z(\widehat{H})\times Z(\widehat{H})\to Z(\widehat{H}),\;\; (z_1,z_2)\mapsto z_1z_2$$
identifies the quotient $Z(\widehat{G})/Z(\widehat{G}_X)$, $\Gamma$-equivariantly, with $Z(\widehat{H})$ i.e.\ we have a short exact sequence of $\Gamma$-modules
\begin{equation}\label{ses Z}
\displaystyle 1\to Z(\widehat{G}_X)\to Z(\widehat{G}) \xrightarrow{N} Z(\widehat{H})\to 1.
\end{equation}
	
Some statements will be more simply formulated using the homogeneous space
\begin{equation*}
\displaystyle \widehat{X}={}^ L G/{}^L G_X=\widehat{G}/\widehat{G}_X.
\end{equation*}\index{$\widehat{X}$}
By \eqref{ses Z}, there is a natural right action of $Z(\widehat{H})$ on $\widehat{X}$. This action commutes with left translations by $\widehat{G}$ but not with the Weil action. More precisely, denoting by
	$$\sigma\in W_k\mapsto \sigma_X\in \Aut(\widehat{X})$$\index{$\sigma_X$}
	the natural action of $W_k$ on $\widehat{X}$, we have
	$$\displaystyle \sigma_X(xz)=\sigma_X(x)\sigma_H(z)$$
	for every $x\in \widehat{X}$, $z\in Z(\widehat{H})$ and $\sigma\in W_k$.
	
	Recall that $\Phi(G)$ and $\Phi(X)$ denote the sets of equivalence classes of $L$-parameters $\phi: \mathcal{L}_k\to {}^L G$ and $\psi: \mathcal{L}_k\to {}^L G_X$ respectively. Then, we have a pushforward map
	$$\displaystyle \iota_{X,*}: \Phi(X)\to \Phi(G),\;\; \iota_{X,*}(\psi)={}^L\iota_X\circ \psi.$$\index{$\iota_{X,*}$}
	
	For every $L$-parameter $\phi\in \Phi(G)$ (resp. $\psi\in \Phi(X)$), we denote by $Z_\phi=\Cent_{\widehat{G}}(\phi)$\index{$Z_\phi$} (resp. $Z_\psi=\Cent_{\widehat{G}_X}(\psi)$\index{$Z_\psi$}) the centralizer of its image in $\widehat{G}$ (resp. in $\widehat{G}_X$). These groups are well-defined up to inner automorphisms and following \S \ref{S LLC} we denote by $S_\phi=\pi_0(Z_\phi)$, $S_\psi=\pi_0(Z_\psi)$ their component groups. 
	
	For every $\psi\in \Phi(X)$, setting $\phi=\iota_{X,*}(\psi)$, we have a natural embedding $Z_\psi\to Z_\phi$. More precisely, identifying $\widehat{G}_X$ with its image in $\widehat{G}$ we have $Z_\psi=Z_{\phi}\cap \widehat{G}_X$. This induces a morphism $S_\psi\to S_\phi$ that is however not always an embedding.
	
	An $L$-parameter $\phi: \mathcal{L}_k \to {}^L G$ defines, by composition, a left action of $\mathcal{L}_k$ on $\widehat{X}$. The subset of fixed points for this action can then be naturally decomposed as
	\begin{equation}\label{eq fixed points}
		\displaystyle \widehat{X}^\phi\simeq \bigsqcup_{\psi\in \iota_{X,*}^{-1}(\phi)} Z_\phi/Z_\psi.
	\end{equation}\index{$\widehat{X}^\phi$}
	
	In \cite{PraGal}, D. Prasad has made a precise conjecture on the distinction of irreducible representations $\pi\in \Irr(G)$ with respect to $X$ and even about the individual multiplicities $m_X(\pi)=\dim \Hom_G(\cS(X),\pi)$ in terms of the morphism ${}^L \iota_X$ and the induced map between spaces of Langlands parameters $\iota_{X,*}$, at least when the parameters are {\it generic} in a suitable sense. It seems however that the finer aspects of this conjecture as stated in {\it op.\ cit.}\ are still provisional as of now, due to some subtleties at the points of ramification of the map $\iota_{X,*}$, so that we will here only state Prasad's prediction in the restricted case of discrete series, where the previous issues disappear. In particular, we will from now on assume that $k$ is non-Archimedean (as for $k=\bR$, $G(k)=H(\bC)$ has no discrete series).
\end{paragr}

\begin{paragr}
	Recall that an $L$-parameter $\phi: \mathcal{L}_k\to {}^L G$ is said to be {\em discrete} if its centralizer $Z_\phi$ in $\widehat{G}$ is finite modulo the subgroup $Z(\widehat{G})^\Gamma$ of Galois fixed elements in the center of $\widehat{G}$. We define similarly the notion of discrete $L$-parameter $\phi': \mathcal{L}_k\to {}^L G_X$. One basic requirement of the local Langlands correspondence is that $L$-packets corresponding to discrete $L$-parameters should consist entirely of representations that are essentially square-integrable and, conversely, if an $L$-packet contains an essentially square-integrable representation then it is associated to a discrete $L$-parameter. We shall denote by $\Phi_{\disc}(G)$\index{$\Phi_{\disc}(G)$} and $\Phi_{\disc}(X)$\index{$\Phi_{\disc}(X)$} the subsets of discrete $L$-parameters in $\Phi(G)$ and $\Phi(X)$ respectively. Note that for $\psi\in \Phi(X)$, $\phi\in \Phi(G)$ with $\phi=\iota_X(\psi)$, $\phi$ discrete implies that $\psi$ is discrete as well. The converse, however, is false in general.
	
	In what follows, we will assume the existence of a local Langlands correspondence as formulated in \S \ref{S LLC} and we fix the Whittaker datum $\mathfrak{w}=(N,\psi_N)$, on which the bijection \eqref{param Lpackets} depends, such that its image $\mathfrak{w}^c$ by the automorphism of $G$ induced by the nontrivial element of $\Gal(\ell/k)$ is conjugate to the {\it opposite} Whittaker datum $(N,\psi_N^{-1})$. More concretely, fixing a Borel subgroup $B=T_HN_H$ of $H$, such a Whittaker datum is given by any nondegenerate character $\psi_N: N_H(\ell)\to \bC^\times$ that is trivial on $N_H(k)$ (so that $\psi_N^c=\psi_N^{-1}$).
	
	If $\pi$ is an irreducible representation of some pure inner form $G_\alpha(k)$, $\alpha\in H^1(k,G)$, we denote by
	$$\displaystyle m_X(\pi)=\dim \Hom_{G_\alpha}(\cS(X_\alpha), \pi)$$
	the multiplicity of $\pi$ with respect to the associated pure inner form $X_\alpha$ of $X$, see \S \ref{S PIF X}.
	
	\begin{conj}[Prasad's conjecture, first version]\label{conj Prasad 1}
		Let $\phi\in \Phi_{\disc}(G)$ be a discrete $L$-parameter and $\rho\in \Irr(\pi_0(Z_\phi))$. We have
		\begin{equation*}
			\displaystyle m_X(\pi_{\mathfrak{w}}(\phi,\rho))=\sum_{\psi\in \iota_{X,*}^{-1}(\phi)} \langle \rho, \mathbf{1}\rangle_{S_\psi}
		\end{equation*}
		where $\langle .,.\rangle_{S_\psi}$ denotes the canonical pairing on the Grothendieck group of $S_\psi$ (given, for two representations $\tau_1$, $\tau_2$ of $S_\psi$, by $\langle \tau_1,\tau_2\rangle_{S_\psi}=\dim \Hom_{S_\psi}(\tau_1,\tau_2)$).\index{$\langle .,.\rangle_{S_\psi}$}
	\end{conj}
	
	\begin{rem}
		\begin{itemize}
			\item Using the description \eqref{eq fixed points}, the conjecture can be reformulated as saying that
			\begin{equation}\label{eq conj Prasad 1}
				\displaystyle m_X(\pi_{\mathfrak{w}}(\phi,\rho))=\dim \Hom_{S_\phi}(\rho,H^0(\widehat{X}^\phi,\bC))
			\end{equation} 
			where $H^0(\widehat{X}^\phi,\bC)=\bC^{\pi_0(\widehat{X}^\phi)}$ denotes the zeroth cohomology group (with complex coefficients) with its natural action of $S_\phi$.
			
			\item One basic requirement on the local Langlands correspondence is that the representation
			$$\displaystyle \sum_{\rho\in \Irr(S_\phi)} \dim(\rho)\pi_{\mathfrak{w}}(\phi,\rho)$$
			is stable, in the sense of its character being a stable distribution. We then define the {\it stable multiplicity} of $\phi$ (with respect to $X$) as the number
			\begin{equation*}
				\displaystyle m_X(\phi):=\sum_{\rho\in \Irr(\pi_0(Z_\phi))} \dim(\rho)m_X(\pi_{\mathfrak{w}}(\phi,\rho)).
			\end{equation*}
			Summing \eqref{eq conj Prasad 1} over $\rho$, we arrive at the following prediction for the stable multiplicity:
			\begin{equation}\label{eq Prasad stable 1}
				\displaystyle m_X(\phi)=\dim H^0(\widehat{X}^\phi,\bC)=\left\lvert \pi_0(\widehat{X}^\phi) \right\rvert=\sum_{\psi\in \iota_{X,*}^{-1}(\phi)} \lvert \pi_0(Z_\phi/Z_\psi) \rvert.
			\end{equation}
		\end{itemize}
	\end{rem}
	
\end{paragr}

\begin{paragr}
	For $\alpha\in H^1(k,G)$, the set of $k$-points $X_\alpha(k)$ naturally decomposes into $G_\alpha(k)$-orbits parametrized by the fiber above $\alpha$ of the map $H^1(k,H)\to H^1(k,G)$ i.e.\
	$$\displaystyle X_\alpha(k)=\bigsqcup_{H^1(k,H)\ni \beta\mapsto \alpha\in H^1(k,G)} H_\beta(k)\backslash G_\alpha(k).$$
	Thus, the multiplicity $m_X(\pi)$ can also be written as a finite sum
	$$\displaystyle m_X(\pi)=\sum_{H^1(k,H)\ni \beta\mapsto \alpha\in H^1(k,G)} m_{H_\beta}(\pi),$$
	where we have set
	$$\displaystyle m_{H_\beta}(\pi):=\dim \Hom_{G_\alpha}(\cS(H_\beta(k)\backslash G_\alpha(k)),\pi)=\dim \Hom_{H_\beta}(\pi^\vee, \bC).$$
	Prasad's original conjecture is about the refined multiplicities $m_{H_\beta}(\pi)$ and we will formulate it for discrete series below. In order to make the statement more pleasant, if $\beta\in H^1(k,H)$ does not map to $\alpha\in H^1(k,G)$, we will set $m_{H_\beta}(\pi)$ to be zero.
	
	Let us introduce the following `twisting' group
	\begin{equation*}
		\displaystyle D:=\Ker(H^1(\Gamma, Z(\widehat{G}_X))\to H^1(\Gamma, Z(\widehat{G}))).
	\end{equation*}\index{$D$}
	Then, $D$ acts naturally on $\Phi(X)$ (by multiplication of cocycles) and the map $\iota_{X,*}: \Phi(X)\to \Phi(G)$ is $D$-invariant. For $\phi\in \Phi(G)$, we denote by $\iota_{X,*}^{-1}(\phi)/D$ the set of $D$-orbits on the fiber $\iota_{X,*}^{-1}(\phi)$.

	Recall the Kottwitz isomorphism \cite{KottSTFell} $H^1(k,H)\simeq \pi_0(Z(\widehat{H})^\Gamma)^D$, where the superscript $D$ stands for the Pontryagin dual $\Hom(.,\bC^\times)$. For each $\beta\in H^1(k,H)$, we shall denote by $\chi_\beta: Z(\widehat{H})^\Gamma\to \bC^\times$ the corresponding character.
	
	Let $\psi\in \Phi(X)$ and set $\phi=\iota_{X,*}(\psi)\in \Phi(G)$. We define
	$$\displaystyle Z_{\psi}^\heartsuit=Z_{Z(\widehat{G})\widehat{G}_X}(\phi),$$\index{$Z_{\psi}^\heartsuit$}
	as the centralizer of $\phi$ in the subgroup $Z(\widehat{G})\widehat{G}_X\subset \widehat{G}$, and
	$$\displaystyle S_\psi^\heartsuit=\pi_0(Z_{\psi}^\heartsuit).$$\index{$S_\psi^\heartsuit$}
	We have $Z_{\psi}^\heartsuit\subset Z_{\phi}$ and, using that $Z(\widehat{G})\cap \widehat{G}_X=Z(\widehat{G}_X)$, there is a natural projection $Z_{\psi}^\heartsuit\to Z(\widehat{G})/Z(\widehat{G}_X)\simeq Z(\widehat{H})$ (where the latter identification comes from the norm, see the short exact sequence \eqref{ses Z}) that induces an exact sequence
	\begin{equation}\label{ses Z 2}
		\displaystyle 1\to Z_\psi\to Z_\psi^\heartsuit\xrightarrow{\theta_\psi} Z(\widehat{H})^\Gamma.
	\end{equation}
	Then, for each $\beta\in H^1(k,H)$, we denote by $\chi_\beta^\heartsuit$ the pullback of $\chi_\beta$ to $Z_{\psi}^\heartsuit$ (or $S_{\psi}^\heartsuit$) by the last morphism in the above sequence.

	\begin{conj}[Prasad's conjecture, refined version]\label{conj Prasad 2}
		Let $\phi\in \Phi_{\disc}(G)$ be a discrete $L$-parameter and $\rho\in \Irr(S_\phi)$. Then, for every $\beta\in H^1(k,H)$, we have
		\begin{equation}\label{eq conj Prasad 2}
			\displaystyle m_{H_\beta}(\pi_{\mathfrak{w}}(\phi,\rho))=\sum_{\psi\in \iota_{X,*}^{-1}(\phi)/D} \langle \rho, \chi_\beta^\heartsuit\rangle_{S_\psi^\heartsuit}.
		\end{equation}
	\end{conj}
	
	\begin{rem}
		\begin{itemize}
			\item Let $\alpha\in H^1(k,G)$ be such that $\pi_{\mathfrak{w}}(\phi,\rho)\in \Pi^{G_\alpha}(\phi)$. Then, compatibility of the local Langlands correspondence with the Kottwitz isomorphism requires that the central character of $\rho$ on $Z(\widehat{G})^\Gamma$ should be given by $\chi_\alpha$ whereas the restriction of $\chi_\beta^\heartsuit$ to $Z(\widehat{G})^\Gamma$ is the composition of $\chi_\beta$ with the norm map $N:Z(\widehat{G})^\Gamma\to Z(\widehat{H})^\Gamma$. Thus, the right hand side of \eqref{eq conj Prasad 2} is zero unless $\chi_\alpha=\chi_\beta\circ N$. As the dual of the norm morphism corresponds, under Kottwitz isomorphisms, to the natural map $H^1(k,H)\to H^1(k,G)$, we see that this last condition is equivalent to $\alpha$ being the image of $\beta$ by the latter, also a necessary condition for the left hand side of \eqref{eq conj Prasad 2} to be nonzero (by definition).
			
			\item Applying the derived functor of $\Gamma$-invariants to the short exact sequence \eqref{ses Z}, we obtain an isomorphism $D\simeq Z(\widehat{H})^\Gamma/N(Z(\widehat{G})^\Gamma)$. Moreover, for every $\psi\in \Phi(X)$, $Z_\psi^\heartsuit$ contains $Z(\widehat{G})^\Gamma$ and composing the last arrow of \eqref{ses Z 2} with the projection $Z(\widehat{H})^\Gamma\to D$ we obtain a short exact sequence
			\begin{equation}\label{ses Z 3}
				\displaystyle 1\to Z(\widehat{G})^\Gamma Z_\psi\to Z_\psi^\heartsuit\to D_\psi\to 1
			\end{equation}
			for some subgroup $D_\psi\subset D$. It is not hard to see that $D_\psi$ coincides with the stabilizer of $\psi$ in $D$ (for the twisting action). In particular, the identity \eqref{eq conj Prasad 2} can be restated as
			\begin{equation}\label{eq conj Prasad 2bis}
				\displaystyle m_{H_\beta}(\pi_{\mathfrak{w}}(\phi,\rho))=\sum_{\psi\in \iota_{X,*}^{-1}(\phi)} \langle \rho, \chi_\beta^\heartsuit\rangle_{S_\psi^\heartsuit} \frac{\lvert D_\psi\rvert}{\lvert D\rvert}.
			\end{equation}
			
			\item Similarly to \eqref{eq conj Prasad 1}, we can also reformulate the conjecture in terms of the subset of fixed points $\widehat{X}^\phi$. Indeed, the latter is not only naturally acted upon by $Z_\phi$ (by left multiplication) but also by $Z(\widehat{H})^\Gamma$ (through right multiplication by $Z(\widehat{G})/Z(\widehat{G}_X)\simeq Z(\widehat{H})$). Then, \eqref{eq conj Prasad 2bis} is equivalent to the formula
			\begin{equation}\label{eq conj Prasad 2ter}
				\displaystyle m_{H_\beta}(\pi_{\mathfrak{w}}(\phi,\rho))=\dim \Hom_{S_\phi\times Z(\widehat{H})^\Gamma}(\rho\otimes \chi_\beta, H^0(\widehat{X}^\phi,\bC)).
			\end{equation}
			
			\item Conjecture \ref{conj Prasad 1} is a consequence of the above refined version. Indeed, summing \eqref{eq conj Prasad 2ter} over $\beta\in H^1(k,H)$, we obtain
			\[\begin{aligned}
				\displaystyle m_X(\pi_{\mathfrak{w}}(\phi,\rho)) & =\dim \Hom_{S_\phi}(\rho, H^0(\widehat{X}^\phi,\bC))=\sum_{\psi\in \iota_{X,*}^{-1}(\phi)} \langle \rho,\mathbf{1}\rangle_{S_\psi}.
			\end{aligned}\]
			
			\item Similarly, summing \eqref{eq conj Prasad 2ter} over $\rho$, we obtain the following formula
			\begin{equation*}
				\displaystyle m_{H_\beta}(\phi)=\left\lvert \pi_0\left(\widehat{X}^\phi/Z(\widehat{H})^\Gamma \right)\right\rvert
			\end{equation*}
			for the stable multiplicity
			$$\displaystyle m_{H_\beta}(\phi):=\sum_{\rho\in \Irr(S_\phi)} \dim(\rho) m_{H_\beta}(\pi_{\mathfrak{w}}(\phi,\rho)).$$
		\end{itemize}
	\end{rem}
\end{paragr}

\begin{paragr}[The group case.]
	The discussion of the previous paragraphs actually can be applied verbatim to the group case $X=H$ by formally replacing $\ell$ by the split extension $k\times k$. This leads to the following prediction on the compatibility between the local Langlands correspondence and the operation of passing to the contragredient \cite{Prasdual}.
	
	\begin{conj}[Prasad conjecture for contragredients.]
		Let $H$ be a quasi-split connected reductive group over $k$ and $\mathfrak{w}=(N,\psi)$ be a Whittaker datum for $H(k)$. Then, for every $L$-parameter $\phi\in \Phi(H)$ and $\rho\in \Irr(\pi_0(Z_\phi))$ we have:
		\begin{equation}
			\displaystyle \pi_{\mathfrak{w}}(\phi,\rho)^\vee\simeq \pi_{\mathfrak{w}^\vee}(\phi^\vee, \rho^\vee)
		\end{equation}
		where $\phi^\vee$ denotes the composite of $\phi$ with the duality involution $h\mapsto h^\vee$ of ${}^L H$, $\mathfrak{w}^\vee=(N,\psi^{-1})$\index{$\mathfrak{w}^\vee$} is the opposite Whittaker datum and $\pi_{\mathfrak{w}}(\phi,\rho)^\vee$, $\rho^\vee$ denote the contragredients of the representations $\pi_{\mathfrak{w}}(\phi,\rho)$ and $\rho$ respectively. (And we have used the natural isomorphism $S_\phi\simeq S_{\phi^\vee}$ to identify $\rho^\vee$ with an irreducible representation of $S_{\phi^\vee}$.)
	\end{conj}
	
	A similar conjecture was formulated by Adams-Vogan \cite{AVdual} at the level of $L$-parameters and established in {\it op.\ cit.}\ when $k=\bR$. Furthermore, in \cite{Kaldual}, Kaletha has studied the compatibility of the above conjecture with the theory of endoscopy leading to a proof of it in the case of quasi-split real reductive groups or quasi-split symplectic and special orthogonal $p$-adic groups (based on the endoscopic characterization of the local Langlands correspondence in the latter two cases given in \cite{artbook}).

\end{paragr}

\begin{paragr}[Twist by a character.]
	We can also state analogs of Conjectures \ref{conj Prasad 1} and \ref{conj Prasad 2} with a character twist. More precisely, Langlands has constructed a natural morphism
	$$\displaystyle H^1(W_k, Z(\widehat{H}))\to \Hom_{\operatorname{cont}}(H(k),\bC^\times)$$
	which, under the running assumption that $H$ is quasi-split, is bijective (see e.g. the appendix \cite[Appendix A]{LaMao} by J.-P.Labesse and E. Lapid). Actually, the Langlands morphism exists for general reductive groups so that an element $\chi\in H^1(W_k, Z(\widehat{H}))$ yields characters of all the pure inner forms $H_\beta$ of $H$. We shall call such an element a {\it Galoisian character} of $H$.

	Let $\chi\in H^1(W_k, Z(\widehat{H}))$ be a Galoisian character. Then, for every $\beta\in H^1(k,H)$ with image $\alpha\in H^1(k,G)$ and $\pi\in \Irr(G_\alpha)$, we define a multiplicity
	\begin{equation*}
		\displaystyle m_{H_\beta,\chi}(\pi):=\dim\Hom_{G_\alpha}(\cS(H_\beta(k)\backslash G_\alpha(k),\chi),\pi)=\dim\Hom_{H_\beta}(\pi^\vee\otimes \chi,\bC)
	\end{equation*}\index{$m_{H_\beta,\chi}(\pi)$}
	where $\cS(H_\beta(k)\backslash G_\alpha(k),\chi)$ denotes the space of smooth functions $f: G_\alpha(k)\to \bC$ such that $f(hg)=\chi(h)f(g)$ for every $h\in H_\beta(k)$ and whose support has compact image in $\cS(H_\beta(k)\backslash G_\alpha(k),\chi)$. We can also define the multiplicity
	\begin{equation*}
		\displaystyle m_{X_\chi}(\pi):=\dim \Hom_{G_\alpha}(\cS(X_{\chi,\alpha}),\pi)=\sum_{H^1(k,H)\ni \beta\mapsto \alpha\in H^1(k,G)} m_{X_\chi,\beta}(\pi)
	\end{equation*}\index{$m_{X_\chi}(\pi)$}
	where
	$$\displaystyle \cS(X_{\chi,\alpha}):=\bigoplus_{H^1(k,H)\ni \beta\mapsto \alpha\in H^1(k,G)} \cS(H_\beta(k)\backslash G_\alpha(k),\chi).$$\index{$\cS(X_{\chi,\alpha})$}
	
	The $L$-group ${}^L G_X$ (in its Weil form) with the morphism $\iota_{X}: {}^L G_X\to {}^L G$ also admits versions with $\chi$-twists that can be defined as follows. Let $c_\chi: W_F\to Z(\widehat{H})$ be a cocycle representing $\chi$. Then, we introduce the subgroup 
	\begin{equation*}
		\displaystyle \mathcal{Z}_\chi:=\left\{(z,\sigma)\in Z(\widehat{G})\rtimes W_F\mid c_\chi(\sigma)=N(z)  \right\}
	\end{equation*}\index{$\mathcal{Z}_\chi$}
	of $Z(\widehat{G})\rtimes W_F$. This defines an extension of $W_F$ by $Z(\widehat{G}_X)$ i.e.\ we have a short exact sequence:
	$$\displaystyle 1\to Z(\widehat{G}_X)\to \mathcal{Z}_\chi\to W_F\to 1.$$
	Then, we let ${}^L G_{X,\chi}$ be the pushout of this extension by the inclusion $Z(\widehat{G}_X)\subset \widehat{G}_X$ that is:
	\begin{equation}
		\displaystyle {}^L G_{X,\chi}:= (\widehat{G}_X\rtimes \mathcal{Z}_\chi)/Z(\widehat{G}_X)
	\end{equation}\index{${}^L G_{X,\chi}$}
	where the action of $\mathcal{Z}_\chi$ on $\widehat{G}_X$ is given by the projection $\mathcal{Z}_\chi\to W_F$ and we quotient by the antidiagonal copy of $Z(\widehat{G}_X)$ i.e.\ the image of $Z(\widehat{G}_X)\to \widehat{G}_X\rtimes \mathcal{Z}_\chi$, $z\mapsto (z,z^{-1})$. The product induces a natural injective morphism
	\begin{equation*}
		\displaystyle \iota_{X_\chi}: {}^L G_{X,\chi}\to {}^L G.
	\end{equation*}
	As can be readily checked, changing the choice of the cocycle $c_\chi$ conjugates the subgroup ${}^L G_{X,\chi}\subset {}^L G$ by an element of $Z(\widehat{G})$.
	
	Actually, it is easier to define the twisted space $\widehat{X}_\chi={}^L G/{}^L G_{X,\chi}$\index{$\widehat{X}_\chi$}. Indeed, the latter is simply obtained from $\widehat{X}$ by twisting the action of $W_k$ by any cocycle $c_\chi: W_k\to Z(\widehat{H})$ representing $\chi$. (Recall that there is an action of $Z(\widehat{H})\rtimes W_k$ on $\widehat{X}$).
	
	Replacing ${}^L G_X$ by ${}^L G_{X,\chi}$, or $\widehat{X}$ by $\widehat{X}_\chi$, we can directly formulate analogs of Conjectures \ref{conj Prasad 1} and \ref{conj Prasad 2} for the multiplicities $m_{X_\chi}(\pi)$ and $m_{H_\beta,\chi}(\pi)$ when $\pi$ is a discrete series. We leave the details to the reader. 
	
	We emphasize that the original conjecture in \cite{PraGal} is formulated for a particular, in general nontrivial, Galoisian character $\chi=\omega_{\ell/k}$. More precisely, there is a canonical central element $\epsilon\in Z(\widehat{H})$ of order at most $2$ that can be described as the image of $-I_2$ by any principal $\SL_2$-morphism $\SL_2(\bC)\to \widehat{H}$. The Galoisian character $\omega_{\ell/k}$ is then the one associated to the cocycle $W_k\twoheadrightarrow \Gal(\ell/k)\to Z(\widehat{H})$ mapping $c$ to $\epsilon$. It turns out that for $\chi=\omega_{\ell/k}$, we can identify ${}^L G_{X,\chi}$ with the $L$-group ${}^L H^{op}$ of a certain quasi-split $\ell/k$-form $H^{op}$ of $H$\footnote{More precisely, $H^{op}$ is the form obtained by twisting the Galois action of $c\in \Gal(\ell/k)$ on $H_{\ell}$ by the Chevalley involution, see \S \ref{S Galois pairs Lgroup}.}, called the {\it opposite group}, so that the morphism $\iota_{X,\chi}:{}^L G_{X,\chi}\to {}^L G$ becomes the base-change morphism.
\end{paragr}

\begin{paragr}[Reduction to the untwisted case.]
	Actually, the twisted versions of the conjecture mentioned in the previous paragraph are not more general but consequences of the original (untwisted) conjectures. Let us briefly sketch the mechanism for this reduction. The key is to use the theory of $z$-extensions. Namely, we can find a central extension
	\begin{equation}\label{extension}
		\displaystyle 1\to T\to H'\xrightarrow{p} H\to 1
	\end{equation}
	such that:
	\begin{itemize}
		\item $T$ is an induced torus (i.e.\ it can be written as a finite product $T\simeq \prod_{i\in I} \Res_{k_i/k}\mathbb{G}_m$ for some finite extensions $k_i/k$);
		
		\item the derived subgroup of $H'$ is simply connected (equivalently: $Z(\widehat{H}')$ is connected);
		
		\item the natural map $H^1(k,H')\to H^1(k,H)$ is a bijection (equivalently $\pi_0(Z(\widehat{H})^\Gamma)\to \pi_0(Z(\widehat{H'})^\Gamma)$ is bijective).
	\end{itemize}
	Indeed, a $z$-extension is precisely an extension satisfying the two first conditions above and it always exists by \cite[Lemma 1.1]{KottratCC}. Moreover, fixing such a $z$-extension we can force it to satisfy the last condition by pushing it out along the inclusion $T\hookrightarrow \Res_{k'/k} T_{k'}$ (whose target is again an induced torus) for any finite extension $k'/k$ or degree divisible by the order of $H^1(k,H)$.
	
	Let us now pick an extension \eqref{extension} satisfying the three conditions above and set $S=\Res_{\ell/k} T_\ell$, $G'=\Res_{\ell/k} G_\ell$, $X'=H'\backslash G'$. The Galoisian character $\chi$ of $H$ also induces a Galoisian character $\chi\circ p$ of $H'$, that for simplicity we will denote also by $\chi$. Since $T$ is connected (i.e.\ is a torus), we have a short exact sequence
	$$\displaystyle 1\to {}^L G\xrightarrow{i} {}^L G'\to {}^L S\to 1$$
	of $L$-groups. This leads to a natural ${}^L G'$-equivariant projection $\widehat{X}'\to \widehat{T}$ with fiber above $1$ isomorphic to $\widehat{X}$ (${}^L G$-equivariantly). We shall summarize this situation with the following notation:
	$$\displaystyle \widehat{X}\hookrightarrow \widehat{X}'\twoheadrightarrow \widehat{T}.$$
	Then, twisting everything by $\chi$ leads to a similar sequence:
	\begin{equation}\label{sequence}
		\displaystyle \widehat{X}_\chi\hookrightarrow \widehat{X}'_\chi\twoheadrightarrow \widehat{T}.
	\end{equation}
	
	The general principle of the reduction is then to show two implications of the form
	\begin{equation}\label{reduction Prasad conjecture}
		\displaystyle \mbox{ Prasad's conjecture for } (X,\chi)\Leftarrow \mbox{ Prasad's conjecture for } (X',\chi) \Leftarrow \mbox{ Prasad's conjecture for }  X'.
	\end{equation}
	
	Let $\pi\in \Irr(G)$ and $\pi'\in \Irr(G')$ be its pullback along $p$. Since $T$ is induced, the maps $G'(k)\to G(k)$ and $H'(k)\to H(k)$ are surjective and therefore 
	\begin{equation}\label{eq mult isogeny}
		\displaystyle m_{H,\chi}(\pi)=m_{H',\chi}(\pi').
	\end{equation}
	Moreover, if $\phi: \mathcal{L}_k\to {}^L G$ be an $L$-parameter and $\phi': \mathcal{L}_k\to {}^L G'$ its composite with $i$ (so that if $\pi\in \Pi(\phi)$ then we should have $\pi'\in \Pi(\phi')$), taking invariants with respect to the corresponding actions in \eqref{sequence} we obtain a similar sequence
	$$\displaystyle \widehat{X}^\phi_\chi\hookrightarrow (\widehat{X}'_\chi)^{\phi'}\twoheadrightarrow \widehat{T}^\Gamma.$$
	On the other hand, again because $T$ is induced, $\widehat{T}^\Gamma$ is connected and combining this with the third condition on the extension this gives an exact sequence of connected groups
	$$\displaystyle 1\to Z(\widehat{H})^{\Gamma,0}\to Z(\widehat{H}')^{\Gamma,0}\to \widehat{T}^\Gamma\to 1.$$
	This shows that the previous map $\widehat{X}^\phi_\chi\to (\widehat{X}'_\chi)^{\phi'}$ induces a bijection on $\pi_0$'s:
	\begin{equation}\label{map pi0}
		\displaystyle \pi_0(\widehat{X}^\phi_\chi)\simeq \pi_0((\widehat{X}'_\chi)^{\phi'})
	\end{equation}
	which is moreover equivariant under the action of $\pi_0(Z_\phi)\times \pi_0(Z(\widehat{H})^\Gamma)=\pi_0(Z_{\phi'})\times \pi_0(Z(\widehat{H}')^\Gamma)$. The first implication in \eqref{reduction Prasad conjecture} can then be readily deduced from the equality \eqref{eq mult isogeny} and the isomorphism \eqref{map pi0}.
	
	As for the second implication, we note that the second condition on the extension implies that the natural map 
	$$\displaystyle H^1(W_k,Z(\widehat{G}'))\to H^1(W_k,Z(\widehat{H}'))$$
	is surjective. Therefore, we may choose a Galoisian character $\mu$ of $G'$ whose restriction to $H'$ is $\chi$. This implies
	\begin{equation*}
		\displaystyle m_{H',\chi}(\pi')=m_{H'}(\pi'\otimes \mu),\;\; \mbox{ for every } \pi'\in \Irr(G').
	\end{equation*} 
	On the other hand, twist by a cocycle representing $\mu\in H^1(W_k,Z(\widehat{G}'))$ should send the $L$-parameter $\phi'$ of $\pi'$ to the $L$-parameter $\phi'_{\mu}$ of $\pi'\otimes \mu$ whereas it turns the $W_k$-action on $X'$ to that on $X'_\chi$. This shows that the $\mathcal{L}_k$-spaces $(\widehat{X}',\phi')$ and $(\widehat{X}',\phi'_{\mu})$ are isomorphic from which we get a $Z_{\phi'}\times Z(\widehat{H}')^\Gamma=Z_{\phi'_\mu}\times Z(\widehat{H}')^\Gamma$-equivariant bijection
	$$\displaystyle \pi_0((\widehat{X}')^{\phi'})\simeq \pi_0((\widehat{X}'_\chi)^{\phi'_\mu}).$$
	The second implication of \eqref{reduction Prasad conjecture} immediately follows.
\end{paragr}

\begin{paragr}
	Here are two consequences of Conjecture \ref{conj Prasad 2} (suitably extended to include twists by Galoisian characters as in the previous paragraph):
	\begin{itemize}
		\item The Steinberg representation $\St$ of $G(k)$ is $(H,\chi)$-distinguished (i.e.\ $m_{H,\chi}(\St)\neq 0$) if and only if $\chi=\omega_{\ell/k}$ is the quadratic character defined by Prasad. Moreover, in this case we have $m_{H,\chi}(\St)=1$;
		
		\item For every discrete $L$-parameter $\phi\in \Phi_{\disc}(G)$ the stable multiplicity
		$$\displaystyle m_{H_\beta,\chi}(\phi):=\sum_{\rho\in \Irr(S_\phi)} \dim(\rho) m_{H_\beta,\chi}(\pi_{\mathfrak{w}}(\phi,\rho))$$
		does not depend on $\beta\in H^1(k,H)$.
	\end{itemize}
	Actually, both of these consequences have been deduced in \cite{BeuGal} (assuming that the character of the representation $\sum_{\rho\in \Irr(S_\phi)} \dim(\rho) \pi_{\mathfrak{w}}(\phi,\rho)$ is stable in a suitable sense for the second property) from a certain integral formula for the multiplicities $m_{H_\beta,\chi}(\pi)$. This formula is in turn a byproduct of some local trace formula for the varieties $X_\chi$ that we shall discuss in more details in Section \ref{Sect local trace formula}.
\end{paragr}

\begin{paragr}[The case $H=\GL_n$.]
	Let $H=\GL_n$. Then, $H^{op}=U_n$ is a quasi-split unitary group of rank $n$ (with respect to $\ell/k$) and there exists an isomorphism ${}^L G_X\simeq {}^L U_n$ through which ${}^L\iota_X$ is the base-change morphism when $n$ is odd and an unstable base-change morphism when $n$ is even. (We recall that an unstable base-change morphism is obtained by twisting the base-change morphism ${}^L U_n\to {}^L G$ by a cocycle $W_k\to Z(\widehat{G})$ that corresponds to a (Galoisian) character $\chi: G(k)=\GL_n(\ell)\to \bC^\times$ whose restriction to $\GL_n(k)$ is $\eta_{\ell/k}\circ \det$; $\eta_{\ell/k}$\index{$\eta_{\ell/k}$} being the quadratic character of $k^\times$ associated to $\ell/k$. It should be noted that, although such a character $\chi$ is not unique, the image of the unstable base-change morphism ${}^L U_n\to {}^L G$ does not depend on it and, similarly, neither does the image of the corresponding functorial lift.)
	
	By the Local Langlands Correspondence for $\GL_n$ \cite{HTLLC} \cite{HennLLC} any $\pi\in \Irr(G)$ determines an $L$-parameter $\phi_\pi\in \Phi(G)$. Then, Conjecture \ref{conj Prasad 2} in this case is a consequence of the following more general result of Kable \cite{Kabl} and Matringe \cite{Matr}.
	
	\begin{theo}[Kable, Matringe]
		Assume that $H=\GL_n$. Then, for every generic irreducible representation $\pi\in \Irr(G)$, we have
		\begin{equation*}
			\displaystyle m_X(\pi)=m_H(\pi)=\left\{\begin{array}{ll}
				1 & \mbox{ if } \phi_\pi\in \iota_{X,*}(\Phi(X)), \\
				0 & \mbox{ otherwise.}
			\end{array}\right.
		\end{equation*}
	\end{theo}
\end{paragr}

\begin{paragr}[The case $H=U_n$.]
	Let us now consider the case where $H$ is a unitary group of rank $n$ with respect to the quadratic extension $\ell/k$. Then, $G\simeq \Res_{\ell/k} \GL_n$ (so that $X=H\backslash G$ identifies with the variety of nondegenerate Hermitian forms on $\ell^n$) whereas $H^{op}\simeq \GL_n$. Moreover, we have ${}^L X={}^L H^{op}=\GL_n(\bC)\times W_k$ and $\iota_X: {}^L X\to {}^L G$ is the base-change morphism.
	
	Distinction of representations in this case has been much studied by Jacquet \cite{Jacfactorization} as a byproduct of his development of a relative trace formula approach to the, related, problem of factorization of global unitary periods. This work of Jacquet has then been greatly extended and completed by a paper of Feigon, Lapid and Offen \cite{FLO} leading to an almost complete description of the multiplicities $m_H(\pi)$, $\pi\in \Irr(G)$, when $\pi$ is generic. For some particular representations $\pi$ however, those corresponding to ramification points of the base-change map as described below, {\it op.\ cit.}\ only provided a lower bound for $m_H(\pi)$. That these lower bounds yield exactly the multiplicity has then been established in my paper \cite{BeuGLnU}.
	
	Let us now describe more precisely the resulting formula for $m_H(\pi)$. This will include Prasad's conjecture for this group as a particular case (when $\pi$ is a discrete series) but also features some extra phenomenon in general concerning the ramification of the map $\iota_{X,*}: \Phi(X)\to \Phi(G)$. To be more specific, we need to equip the set of $L$-parameters $\Phi(G)$ and $\Phi(X)$ with structures of algebraic varieties over $\bC$; we do it by choosing the most naive structure thereof. 
	
	
	Namely, let $Z^1(\mathcal{L}_k, \widehat{G})$\index{$Z^1(\mathcal{L}_k, \widehat{G})$} be the set of all $1$-cocycles $\mathcal{L}_k=W_k\times \SL_2(\bC)\to \widehat{G}$ that are continuous on $W_k$, algebraic on $\SL_2(\bC)$ but not necessarily Frobenius semisimple. We define similarly $Z^1(I_k\times \SL_2(\bC),\widehat{G})$ where $I_k$ stands for the inertia subgroup of $W_k$. Then, letting $\widehat{G}$ act in the usual way on $1$-cocycles, we equip each $\widehat{G}$-orbit in $Z^1(I_k\times \SL_2(\bC),\widehat{G})$ with its natural structure of algebraic variety making it an homogeneous variety under $\widehat{G}$; thus for $\psi\in Z^1(I_k\times \SL_2(\bC),\widehat{G})$ we have an isomorphism of varieties $\widehat{G}\cdot \psi\simeq \widehat{G}/\widehat{G}_\psi$ where $\widehat{G}_\psi\subset \widehat{G}$ denotes the (Zariski closed) stabilizer of $\psi$. We then think of $Z^1(I_k\times \SL_2(\bC),\widehat{G})$ as the disjoint union of these algebraic varieties. There is a natural restriction map $Z^1(\mathcal{L}_k, \widehat{G})\to Z^1(I_k\times \SL_2(\bC),\widehat{G})$ and, choosing a lift $\Frob\in W_k$ of the geometric Frobenius, we also have an evaluation map $\phi\in Z^1(\mathcal{L}_k, \widehat{G})\mapsto \phi(\Frob)\in \widehat{G}$. The product of those two maps gives an injection
	$$\displaystyle Z^1(\mathcal{L}_k, \widehat{G})\hookrightarrow Z^1(I_k\times \SL_2(\bC),\widehat{G})\times \widehat{G},$$
	whose image is readily seen to be a closed subvariety of the product, and we equip $Z^1(\mathcal{L}_k, \widehat{G})$ with the resulting structure of algebraic variety (which obviously does not depend on the choice of $\Frob$). Projection to $\widehat{G}$ then identifies the set of $L$-parameters $\Phi(G)$ with the set of closed $\widehat{G}$-orbits in $Z^1(\mathcal{L}_k, \widehat{G})$ for the obvious conjugation action.\footnote{More precisely, this is saying that the $\widehat{G}$-conjugacy class of $\phi\in Z^1(\mathcal{L}_k, \widehat{G})$ is closed if and only if $\phi(\Frob)$ is semisimple.}  Therefore, $\Phi(G)$ is in natural bijection with the points of the GIT quotient $Z^1(\mathcal{L}_k, \widehat{G})\sslash \widehat{G}$ which is how we equip the former with a structure of algebraic variety. The definition of the structure of algebraic variety on $\Phi(X)$ is completely similar.
	
	The pushforward map
	$$\displaystyle \iota_{X,*}: \Phi(X)\to \Phi(G),\;\;\; \psi\mapsto \phi=\iota_X\circ \psi,$$
	is then a finite morphism. In the case at hand ($H=U_n$), $\Phi(X)$ and $\Phi(G)$ are the varieties of $L$-parameters for $\GL_{n,k}$ and $\GL_{n,\ell}$ respectively and $\iota_{X,*}$ associates to every $L$-parameter $\phi: \mathcal{L}_k\to \GL_n(\bC)$ its restriction to $\mathcal{L}_{\ell}$. Since $L$-packets in this case are singletons, we can identify $\iota_{X,*}$ with the {\it base-change} map $\BC: \Irr(\GL_{n,k})\to \Irr(\GL_{n,\ell})$ of Arthur-Clozel \cite{ACBC}. For $\sigma\in \Irr(\GL_{n,k})$, we let $\deg \BC(\sigma)$ be the degree of this (finite regular) map at $\sigma$. More precisely, $\deg \BC(\sigma)$ is the cardinality of a generic nonempty fiber in a neighborhood of $\sigma$ or equivalently, and more formally, the degree of the extension $\operatorname{Frac}(\mathcal{O}_\sigma)/\operatorname{Frac}(\mathcal{O}_\pi/I)$ where $\mathcal{O}_\sigma$, $\mathcal{O}_\pi$ denote the local rings at $\sigma$ and $\pi$ respectively, $I$ is the kernel of $\BC^*:\mathcal{O}_\pi\to \mathcal{O}_\sigma$ and $\operatorname{Frac}(.)$ stands for the associated field of fractions. For $\pi\in \Irr(\GL_{n,\ell})$, we then let
	$$\displaystyle \deg \BC(\pi):=\sum_{\sigma\in \BC^{-1}(\pi)} \deg \BC(\sigma).$$\index{$\deg \BC(\pi)$}
	We are now in position to state the aforementioned formula for $m_H(\pi)$.
	
	\begin{theo}[\cite{Jacfactorization}, \cite{FLO}, \cite{BeuGLnU}]
		For every generic representation $\pi\in \Irr(G)=\Irr(\GL_{n,\ell})$, we have
		\begin{equation*}
			\displaystyle m_H(\pi)=\left\{\begin{array}{ll}
				\ceil{\frac{\deg \BC(\pi)}{2}} & \mbox{ if } H \mbox{ is quasi-split}, \\
				\\
				\floor{\frac{\deg \BC(\pi)}{2}} & \mbox{ otherwise,}
			\end{array} \right.
		\end{equation*}
		where the brackets indicate the ceiling function (less integer greater than or equal to $\frac{\deg \BC(\pi)}{2}$) on the first line and the floor function (greatest integer less than or equal to $\frac{\deg \BC(\pi)}{2}$) on the second line.
	\end{theo}
	
	\begin{rem}
		\begin{itemize}
			\item Denoting by $\pi^c$ the composition of $\pi$ with the standard action of $c$ on $\GL_n(\ell)$ (the one fixing $\GL_n(k)$ but not $H(k)$!), it is known that $\pi$ is in the image of base-change from $\GL_n(k)$ if and only if $\pi\simeq \pi^c$. In particular, when $H$ is quasi-split the above theorem implies a conjecture of Jacquet \cite{Jacfactorization} that $\pi$ is $H$-distinguished if and only if $\pi\simeq \pi^c$.
			
			\item There are only two pure inner forms of unitary groups of rank $n$, let us denote them $H$ and $H'$. The theorem implies
			\begin{equation}\label{eq mult unitary period}
				\displaystyle m_X(\pi)=m_H(\pi)+m_{H'}(\pi)=\deg \BC(\pi)
			\end{equation}
			for every generic $\pi\in \Irr(G)$.
			
			\item Any generic $\pi\in \Irr(\GL_{n,\ell})$ is parabolically induced from essentially square-integrable representations i.e.\, denoting by $\times :\Irr(\GL_a)\times \Irr(\GL_b)\to \Irr(\GL_{a+b})$ the parabolic induction functor, there exist essentially square-integrable representations $\delta_i\in \Pi_{2,\mathrm{ess}}(\GL_{n_i,\ell})$ ($1\leq i\leq t$) such that
			$$\displaystyle \pi=\delta_1\times\ldots\times \delta_t.$$
			Using this presentation, if $\pi\simeq \pi^c$ we have
			$$\displaystyle \deg \BC(\pi)=2^{\lvert \{i\mid \delta_i\simeq \delta_i^c \}\rvert}.$$
			On the other hand, the fiber $\BC^{-1}(\pi)$ has cardinality
			$$\displaystyle \lvert \BC^{-1}(\pi)\rvert=\prod_{\delta\simeq \delta^c} (1+\lvert \{i\mid \delta_i\simeq \delta \}\rvert)$$
			where the product runs over all essentially square-integrable representations of some $\GL_m(\ell)$ (but almost all the terms are equal to $1$). In particular, if the multiset $\{\delta_i\mid \delta_i\simeq \delta_i^c \}$ has multiplicities then $\lvert \BC^{-1}(\pi)\rvert<\deg \BC(\pi)$.
			
			\item Passing back to $L$-parameters, in the case at hand the centralizers $Z_\phi$, $Z_\psi$ are connected for every $\phi\in \Phi(G)$, $\psi\in \Phi(X)$. It follows that \eqref{eq mult unitary period} can be reformulated similarly to \eqref{eq Prasad stable 1}, for parameters $\phi\in \Phi(G)$ that are generic rather than discrete, provided we count connected components of $\widehat{X}^\phi$ with multiplicities corresponding to the degree of the map $\iota_{X,*}$, that is in the case at hand we have
			\begin{equation}\label{reform Prasad GLn}
				\displaystyle m_X(\phi)=\sum_{\psi\in \iota_{X,*}^{-1}(\phi)}  \deg \iota_{X,*}(\psi) \lvert \pi_0(Z_\phi/Z_\psi)\rvert.
			\end{equation}
			(where $\lvert \pi_0(Z_\phi/Z_\psi)\rvert=1$.) Similarly, the formula of Conjecture \ref{conj Prasad 2} holds for generic parameters provided it is modified by weighting the sum by the function $\psi\mapsto \deg \iota_{X,*}(\psi)$.
			
			\item Although this suggests extending Conjecture \ref{conj Prasad 2} by taking into account the degree of the map $\iota_{X,*}$ as above, formula \eqref{reform Prasad GLn} is wrong in general. Indeed, consider the case $H=\SL_2$. Let $\chi: \ell^\times\to \{\pm 1\}$ be a quadratic character whose restriction to $k^\times$ is $\eta_{\ell/k}$\footnote{Note that this implies that $\chi$ can be factored by the norm map $\operatorname{Nm}: \ell^\times\to k^\times$, i.\ e.\ that is is trivial on $\Ker(\operatorname{Nm})$. Indeed, the fact that $\chi\mid_{k^\times}=\eta_{\ell/k}$ entails that $\chi^c=\chi^{-1}$ (where $c\in \Gal(\ell/k)$ is the nontrivial element), hence, since $\chi$ is quadratic, $\chi^c=\chi$ but that last equality is equivalent to $\chi$ being trivial on $\Ker(\operatorname{Nm})$.}. Let $\phi$ the $L$-parameter of $G=\Res_{\ell/k} \SL_{2,\ell}$ corresponding to the principal series induced from $\chi$. Then, the fiber $\iota_{X,*}^{-1}(\phi)=\{\psi \}$ is a singleton\footnote{Here the dual group of $X$ is ${}^L G_X=\PGL_2(\bC)\times W_k$ and $\psi$ is given by any character $\chi_0: W_k^{ab}=k^{\times}\to \bC^\times\subset \PGL_2(\bC)$ whose composition with the norm $\ell^\times\to k^\times$ is equal to $\chi$. There are two such characters, namely $\chi_0$ and $\chi_0\eta_{\ell/k}$, but $\chi_0\eta_{\ell/k}=\chi_0^{-1}$ and they lead to conjugate parameters in $\PGL_2(\bC)$.} and $Z_\phi/Z_{\psi}=\bZ/2\bZ$, $\deg \iota_{X,*}(\psi)=2$\footnote{In a neighborhood of $\psi$ the base-change map $\iota_{X,*}$ looks like the quotient map $\bC^\times\to\bC^\times/(z\sim z^{-1})$ near $z=1$.} so that the right-hand side of \eqref{reform Prasad GLn} is equal to $4$. On the other hand, the $L$-packet $\Pi(\phi)$ contains two representations $\pi^+$, $\pi^-$ (the irreducible constituents of $I_{B(\ell)}^{\SL_2(\ell)}(\chi)$) with $m_X(\pi^+)=m_X(\pi^-)=1$.
		\end{itemize}
	\end{rem}
\end{paragr}

\subsection{The action of the center}\label{Section center}

In this subsection, we assume that the field $k$ is non-Archimedean.

\vspace{2mm}

\begin{paragr}\label{S normalization action center}
Recall that the center $Z(X)=\Aut_G(X)$ of the variety $X$ is a diagonalizable group. We let $Z(X)$ act on the left of $X$ (to emphasize that it commutes with the right $G$-action) and we will denote by $(a,f)\mapsto (a\cdot f)(x)=f(a^{-1}x)$ the corresponding action of $Z(X)(k)$ on functions on $X(k)$. However, when dealing with the $L^2$-theory it will be convienent to introduce another action, related to the previous one by a simple twist. Indeed, although we have assumed that $X(k)$ carries a $G(k)$-invariant measure, this measure is not necessarily invariant by $Z(X)(k)$ (as the example of a parabolically induced variety shows). This measure is however necessarily semi-invariant i.e.\ translation by $Z(X)(k)$ multiplies it according to a character $\eta: Z(X)(k)\to \bR_{>0}$\index{$\eta$}. The normalized action of $Z(X)(k)$ on $L^2(X)$ is then defined by
$$\displaystyle (a,f)\in Z(X)(k)\times L^2(X)\mapsto  (\mathcal{L}_a f)(x)=\eta(a)^{-1/2}f(a^{-1}x).$$
Note that for every $a\in Z(X)(k)$, the operator $\mathcal{L}_a$ is then unitary.

	It turns out that parts of the general theory developed by Sakellaridis and Venkatesh in \cite{SV} require certain basic inputs related to the action of $Z(X)(k)$ on various spaces of functions on $X(k)$ both for $X$ and its boundary degenerations. We will review below these two conditions, which concern more precisely:
	\begin{itemize}
		\item A certain finiteness property for the action of $Z(X)(k)$ on the space of smooth functions named $\mathbf{(Z-fin)}$ below. This is the basic input needed for the theory of smooth asymptotics reviewed in Section \ref{Sect smooth asym}. When $X$ is parabolically induced from a wavefront variety, it is easy to check that all of its boundary degenerations satisfy this condition (see Lemma \ref{lem Z-fin}). Strictly speaking, this condition didn't appear in \cite{SV} but there was instead a restriction to wavefront spherical varieties there.
		
		\item An assumption on how the relative discrete series of $X$ vary with their central character. This corresponds to the {\em Discrete Series Conjecture} of \cite[\S 9.4.6]{SV} and is an important input for the development of the theory of Bernstein maps (aka $L^2$-asymptotic maps see Section \ref{Sect Bernstein maps}) through its implication to the existence of certain spectral gaps \cite[Proposition 9.4.8]{SV}. This issue of variation of the central character naturally leads to distinguish certain spherical varieties named {\it factorizable} \cite[\S 9.4.1]{SV}. Parabolic induction of factorizable varieties are readily seen to satisfy the Discrete Series Conjecture and $X$ is said to be {\em strongly factorizable} if each of its boundary degenerations is of this form (that is: parabolically induced from a factorizable variety). Symmetric varieties are strongly factorizable \cite[Proposition 9.4.2]{SV} whereas strongly factorizable varieties are automatically parabolically induced from a wavefront variety. A complete characterization of strongly factorizable varieties (including some non-symmetric examples) can be found in \cite[Appendix A]{DHS}.
	\end{itemize}
	
	We still lack a proof that these two properties hold true for general spherical varieties. However, the technique of {\it unfolding} \cite[\S 9.5]{SV}, also introduced by Sakellaridis-Venkatesh, allows in practice to check both properties in any given particular case. The emphasis of {\it op.\ cit.}\ was on the Discrete Series Conjecture. However, the same method can be applied to check condition $\mathbf{(Z-fin)}$ and we will illustrate this in \S \ref{S unfolding} for the boundary degenerations of $X=\GL_2\backslash \SO_5$ (this is the first example of a spherical variety that is not induced from a wavefront one).

\end{paragr}

\begin{paragr}[Finiteness of the center over the Hecke algebra.]\label{S condition Zfin}
	
	In the following, we shall denote by $\mathcal{Z}(G)$\index{$\mathcal{Z}(G)$} the Bernstein center of $G(k)$ \cite{Bercenter} and by $\mathcal{B}(G)=\Specmax(\cZ(G))$\index{$\mathcal{B}(G)$} the {\it Bernstein variety} (i.e.\ the complex variety of supercuspidal supports). A {\em Bernstein block} is a connected component $\mathfrak{s}\subset \mathcal{B}(G)$\index{$\mathfrak{s}$} (i.e.\ a supercuspidal support up to inertial equivalence) corresponding to a direct factor $\cZ_{\mathfrak{s}}(G)\subset \cZ(G)$\index{$\cZ_{\mathfrak{s}}(G)$} (that is the algebra of regular functions on $\mathfrak{s}$) as well as an indecomposable block of the category of smooth representations. Similarly, if $\mathfrak{s}\subset \mathcal{B}(G)$ is a finite union of Bernstein blocks $\mathfrak{s}=\cup_{i=1}^n \mathfrak{s}_i$, we denote by
	$$\displaystyle \cZ_{\mathfrak{s}}(G)=\prod_{i=1}^n \cZ_{\mathfrak{s}_i}(G)$$
	the corresponding direct factor of the Bernstein center. We also denote by $\cS_{\mathfrak{s}}(X)$\index{$\cS_{\mathfrak{s}}(X)$} the sum of the projections of $\cS(X)$ to the blocks $\mathfrak{s}_i$ and by $\cH_{\mathfrak{s}}(Z(X)^0)$\index{$\cH_{\mathfrak{s}}(Z(X)^0)$} the image of the Hecke algebra $\cH(Z(X)^0)$ of $Z(X)^0(k)$ by the morphism (given by left convolution) $\cH(Z(X)^0)\to \End_{\bC}(\cS_{\mathfrak{s}}(X))$.
	
	Similarly, for $J\subseteq G(k)$ a compact-open subgroup, we write $\cZ_J(G)$\index{$\cZ_J(G)$} for the image of $\cZ(G)$ by the natural projection $\cZ(G)\to \cH_J(G)$ and $\cH_J(Z(X)^0)$ for the image of $\cH(Z(X)^0)$ by the map $\cH(Z(X)^0)\to \End_{\bC}(\cS(X)^J)$.
	
	\begin{lem}\label{lem ZX}
		The following are equivalent:
		\begin{enumerate}[(i)]
			\item For every compact-open subgroup $J\subset G(k)$, $\cH_J(Z(X)^0)$ is finite over $\cH_J(G)$ i.e.\ the $\cH_J(G)$-subalgebra of $\End_{\bC}(\cS(X)^J)$ generated by $\cH_J(Z(X)^0)$ is finitely generated as a $\cH_J(G)$-module.
			
			\item For every compact-open subgroup $J\subset G(k)$, $\cH_J(Z(X)^0)$ is finite over $\cZ_J(G)$.
			
			\item For every Bernstein block $\mathfrak{s}\subset \mathcal{B}(G)$, $\cH_{\mathfrak{s}}(Z(X)^0)$ is finite over $\cH_{\mathfrak{s}}(G)$.
			
			\item For every Bernstein block $\mathfrak{s}\subset \mathcal{B}(G)$, $\cH_{\mathfrak{s}}(Z(X)^0)$ is finite over $\cZ_{\mathfrak{s}}(G)$.
		\end{enumerate}
	\end{lem}
	\begin{proof}
		By \cite[corollaire 3.9]{Bercenter}, there exist a decreasing and separating\footnote{i.e.\ $\bigcap_n J_n=\{ 1\}$.} family $(J_n)_n$ of compact-open subgroups as well as an increasing and exhausting\footnote{i.e.\ $\bigcup_n \mathfrak{s}_n=\mathcal{B}(G)$.} family $(\mathfrak{s}_n)_n$ of finite unions of Bernstein blocks such that $\cH_{J_n}(G)=\cH_{\mathfrak{s}_n}(G)$, therefore also $\cZ_{J_n}(G)=\cZ_{\mathfrak{s}_n}(G)$ and $\cH_{J_n}(Z(X)^0)=\cH_{\mathfrak{s}_n}(Z(X)^0)$, for every $n$. This shows that (i) is equivalent to (iii) and (ii) to (iv). On the other hand, (i) is equivalent to (ii) since $\cH_J(G)$ is finite over $\cZ_J(G)$ \cite[corollaire 3.4]{Bercenter}.
	\end{proof}
	
	\begin{rem}
		Because of the inclusions $\cS(X)\subset L^2(X)^\infty\subset C^\infty(X)$ (where $L^2(X)^\infty$ denotes the subspace of smooth vectors in $L^2(X)$) and since $C^\infty(X)$ can be identified with the smooth contragredient of $\cS(X)$ (through the invariant measure on $X(k)$), we see that replacing $\cS(X)$ by either $L^2(X)^\infty$ or $C^\infty(X)$ in the above lemma, we obtain statements that are equivalent to the original ones. Thus, for the condition $\mathbf{(Z-fin)}$ introduced below, we can as well freely use one of these two spaces instead of $\cS(X)$.
	\end{rem}
	
	The following condition on the variety $X$ plays a crucial role in the construction of smooth asymptotic maps (see Section \ref{Sect smooth asym}):
	
	\begin{equation*}
		\displaystyle \mbox{The equivalent conditions of Lemma \ref{lem ZX} are satisfied.} \leqno \mathbf{(Z-fin)_X}
	\end{equation*}
	
	When obvious, we will drop the index $X$ in the condition $\mathbf{(Z-fin)_X}$. As we will see in the next section, assuming the above condition holds for $X$ and its boundary degeneration, we can construct certain {\it smooth asymptotic maps} whose first main consequence is the much stronger finiteness property that for every compact-open subgroup $J\subset G(k)$, $\cS(X)^J$ is a finitely generated $\cH_J(G)$-module (Theorem \ref{theo f.g. Schwartz space}).
	
	An analog of the property $\mathbf{(Z-fin)_X}$ in the Archimedean case can be stated as follows: the enveloping algebra $\mathcal{U}(\mathfrak{z}(X))$ of the Lie algebra $\mathfrak{z}(X)$ of $Z(X)$ is finite over the center $\mathcal{Z}(\mathfrak{g})$\index{$\mathcal{Z}(\mathfrak{g})$} of $\mathcal{U}(\mathfrak{g})$ (both being mapped into $\End(\cS(X))$ by differentiating the actions of $Z(X)(k)$ and $G(k)$ respectively). This finiteness property can actually be deduced from the Harish-Chandra isomorphism of Knop \cite{KnopHC} for the algebra $\mathbb{D}(X)$ of invariant differential operators on $X$. Indeed, the images of $\mathcal{U}(\mathfrak{z}(X))$ and $\mathcal{Z}(\mathfrak{g})$ both land in $\mathbb{D}(X)$ and the aforementioned result of Knop immediately implies that $\mathbb{D}(X)$ is finite over $\mathcal{Z}(\mathfrak{g})$. We also note that this finiteness property was used in \cite{DKSBP} to construct asymptotic maps between spaces of admissible tempered functions for general spherical varieties over $\mathbb{R}$.
	
	
	Let us say that a spherical variety $X$ is {\em central} if the cokernel of the natural morphism $Z(G)^0\to Z(X)^0$ is anisotropic (or equivalently if the image of this morphism contains the maximal split subtorus of $Z(X)^0$).
	The following lemma gives an elementary criterion to check that a spherical variety satisfies property $\mathbf{(Z-fin)}$. In particular, it holds for boundary degenerations of wavefront varieties (see point (iii) below).
	
	\begin{lem}\label{lem Z-fin}
		Let $X$ be a spherical variety.
		\begin{enumerate}[(i)]
			\item Let $J\subset G(k)$ be a compact-open subgroup. Then, there exists a compact-open subgroup $J_Z\subset Z(X)^0(k)$ such that the quotient map $\cH(Z(X)^0)\to \cH_J(Z(X)^0)$ factors through $\cH_{J_Z}(Z(X)^0)$.
			
			\item Assume that $X$ is central. Then, the condition $\mathbf{(Z-fin)_X}$ is satisfied.
			
			\item More generally, if $X$ is parabolically induced from a central spherical variety, i.e.\ if there exists a parabolic subgroup $P^-=LU^-$ and a central spherical $L$-variety $X^L$ such that $X\simeq X^L\times^{P^-} G$, then $X$ satisfies condition $\mathbf{(Z-fin)_X}$.
		\end{enumerate} 
	\end{lem}
	
	\begin{proof}
		(i) Thanks to the weak Cartan decomposition \eqref{weak Cartan}, we are immediately reduced to the case where $G=A$ is a torus and $X=A_X$ a quotient torus. In this case however, we have $Z(X)=A$ and we can just take $J_Z=J$.
		
		(ii) Let $J\subset G(k)$ be a compact-open subgroup, $J_Z$ be as in (i) and $J'_Z\subset Z(G)^0(k)$ be a compact-open subgroup mapping into $J_Z$. Then, the condition that $X$ is central implies that the Hecke algebra $\cH_{J_Z}(Z(X)^0)$ is finite over $\cH_{J'_Z}(Z(G)^0)$. As the latter is contained in $\cH_J(G)$, this immediately implies condition $\mathbf{(Z-fin)_X}$.
		
		(iii) is easier to see using the formulation of $\mathbf{(Z-fin)_X}$ in terms of the Berstein center and indeed follows from (ii) and the fact that the natural Harish-Chandra morphism $\cZ(G)\to \cZ(L)$ between Bernstein centers is finite. The last part of the lemma is then a consequence of the description of boundary degenerations of wavefront spherical varieties as parabolic inductions given in \S \ref{S wavefront case}.
	\end{proof}
\end{paragr}

\begin{paragr}[Finiteness of the center over the Hecke algebra: reduction to the supercuspidal blocks.]\label{S supercuspidal reduction}
	
	Let $P\subset G$ be a parabolic subgroup with unipotent radical $U_P$ and Levi quotient $L_P$. It is known \cite{Brionsph} that $X$ has finitely many $P$-orbits. Let $Y\subset X$ be one of them. Then, $Y$ is stable under the action of $Z(X)^0$, the quotient $Y_P=Y\sslash U_P$ is $L_P$-spherical\footnote{Indeed, if $B\subset P$ is a Borel subgroup, the finiteness of $B$-orbits in $X$ implies the finiteness of $B_L$-orbits in $Y_P$ where $B_L$ denotes the image of $B$ in $L_P$, a Borel subgroup of the latter.} and the resulting action of $Z(X)^0$ on $Y_P$ factors through a morphism $Z(X)^0\to Z(Y_P)^0$.
	
	A block $\mathfrak{s}\subset \mathcal{B}(G)$ is said supercuspidal if it represents the inertial equivalence class of a cuspidal support of the form $(G,\sigma)$ (with $\sigma$ a supercuspidal representation of $G(k)$).
	
	\begin{prop}
		Assume that for every parabolic subgroup $P\subset G$, every $P$-orbit $Y\subset X$ and every supercuspidal block $\mathfrak{s}_L\in \cB(L_P)$, $\cH_{\mathfrak{s}_L}(Z(Y_P))$ is finite over $\cZ_{\mathfrak{s}_L}(L_P)$. Then, $X$ satisfies condition $\mathbf{(Z-fin)}$.
	\end{prop}
	
	\begin{proof}
		Let $\mathfrak{s}\subset \cB(G)$ be a Bernstein block corresponding to the inertial equivalence class of a pair $(L,\sigma)$ where $L\subset G$ is a Levi subgroup and $\sigma$ a supercuspidal representation of $L(k)$. Set $B=\bC[L(k)/L(k)^1]$, where $L(k)^1\subset L(k)$ is the subgroup of $\ell\in L(k)$ such that $\lvert \chi(\ell)\rvert_k=1$ for every algebraic character $\chi\in X^*(L)$, and $\sigma_B=\sigma\otimes B$ equipped with the smooth $L(k)$-action given by $\ell\cdot (v\otimes f)=\sigma(\ell)v\otimes \ell f$. Let $P=LU_P$ be a parabolic subgroup with Levi factor $L$ and $P^-=LU_P^-$ be the opposite parabolic. Then, the parabolic induction $\Pi_{\mathfrak{s}}=I_{P^-}^G(\sigma_B)$ is a progenerator of the block of smooth representations corresponding to $\mathfrak{s}$ \cite{Bercoursepadic}. In particular, $\cZ_{\mathfrak{s}}(G)$ can be identified with the center of the algebra $A_{\mathfrak{s}}=\End_G(\Pi_{\mathfrak{s}})$.  We have a natural inclusion $\cZ_{\mathfrak{s}}(G)\subset B\subset A_{\mathfrak{s}}$ and $A_{\mathfrak{s}}$ is finite over $\cZ_{\mathfrak{s}}(G)$.
		
		Thus, it suffices to check that the action of $Z(X)(k)$ on the intertwining space $\Hom_G(\Pi_{\mathfrak{s}},\cS(X))$ is finite over that of $B$. By Bernstein's second adjunction theorem,
		$$\displaystyle \Hom_G(\Pi_{\mathfrak{s}},\cS(X))\simeq \Hom_L(\sigma_B,J_P\cS(X))$$
		where $J_P$ denotes the normalized Jacquet functor. Moreover, $J_P\cS(X)$ admits a filtration whose associated graded pieces are isomorphic to $J_P\cS(Y)$, for $Y\subset X$ a $P$-orbit, and the latter is itself isomorphic (as an $L(k)$-representation) to $\cS(Y_P)\otimes \chi_Y$ for some unramified character $\chi_Y: L(k)\to \mathbb{R}_{>0}$. Furthermore, this filtration is readily seen to be $Z(X)^0(k)$-stable with $Z(X)^0(k)$ acting on $\cS(Y_P)\otimes \chi_Y$ through the morphism $Z(X)^0\to Z(Y_P)^0$. Finally, the assumption entails that the $Z(Y_P)^0(k)$-action on $\Hom_L(\sigma_B,\cS(Y_P)\otimes \chi_Y)$ is finite over $B$ and the proposition follows.
	\end{proof}
	
	Thus, by the above proposition, showing that condition $\mathbf{(Z-fin)}$ holds for every spherical variety amounts to proving that for every connected reductive group $G$, every $G$-spherical variety $X$ and every supercuspidal block $\mathfrak{s}\subset \cB(G)$, $\cH_{\mathfrak{s}}(Z(X)^0)$ is finite over $\cZ_{\mathfrak{s}}(G)$. This would actually follow at once from the following stronger conjecture which is in turn a consequence of Conjecture \ref{conj1 SV}\footnote{More precisely, this follows from Conjecture \ref{conj1 SV} and the following two facts:
		-if $X$ is not central, then the morphism ${}^L \iota_X$ factors through a Levi subgroup of ${}^L G$, so that all the $X$-distinguished Arthur packets should be parabolically induced,
		-if a supercuspidal representation is $X$-distinguished it necessarily embeds in $L^2(X)$.}
	
	\begin{conj}\label{conj supercuspidal}
		If $X$ is not central then there is no $X$-distinguished supercuspidal representation of $G(k)$ i.e.\ for every supercuspidal irreducible representation $\pi$ of $G(k)$, we have
		$$\displaystyle \Hom_{G(k)}(\cS(X),\pi)=0.$$
	\end{conj}
	
	Let us also remark that, under the assumption that $X$ is not central, this conjecture is also a consequence of the seemingly weaker property that for every irreducible supercuspidal representation $\pi$ of $G(k)$, the $X$-multiplicity
	$$\displaystyle m_X(\pi)=\dim \Hom_G(\cS(X),\pi)$$
	is finite. Indeed, denoting by $\omega: Z(G)(k)\to \bC^\times$ the central character of $\pi$, we have the equality
	$$\displaystyle \Hom_G(\cS(X),\pi)=\Hom_G(\cS(X,\omega),\pi)$$
	where $\cS(X,\omega)$ stands for the space of smooth functions $f: X(k)\to \bC$ satisfying $f(zx)=\omega(z)f(x)$ for every $z\in Z(G)^0(k)$ and whose support is compact modulo $Z(G)^0(k)$. By the injectivity of $\pi$ in the category of smooth representations with central character $\omega$, if this Hom space is nonzero then so is $\Hom_G(\pi,\cS(X,\omega))$. However, the latter space embeds in $\Hom_G(\pi,C^\infty(X,\omega))=\Hom_G(\cS(X),\pi^\vee)$. Therefore, if $m_X(\pi^\vee)$ is finite dimensional then so is $\Hom_G(\pi,\cS(X,\omega))$ but the latter admits a natural action of $Z(X)^0(k)$ (coming from the translation action on $\cS(X,\omega)$) and there is no non-trivial $Z(X)^0(k)$-stable finite dimensional subspace of $\cS(X,\omega)$ (since by assumption the quotient $Z(X)^0(k)/Z(G)^0(k)$ is non compact) which forces $\Hom_G(\pi,\cS(X,\omega))=0$ and therefore $\Hom_G(\cS(X),\pi)=0$ too.
\end{paragr}

\begin{paragr}[The discrete series conjecture.]\label{S DSC}
	For every unitary character $\chi: Z(X)(k)\to \bC^\times$, we let $L^2(X,\chi)$\index{$L^2(X,\chi)$} be the Hilbert space of functions $f: X(k)\to \bC$ satisfying $\mathcal{L}_a f=\chi(a) f$ for every $a\in Z(X)(k)$ and such that $\lvert f\rvert^2$ is integrable on $Z(X)(k)\backslash X(k)$. By the elementary theory of Mellin transform, we have a $Z(X)(k)\times G(k)$-equivariant direct integral decomposition (in the sense of \S \ref{sect Plancherel})
	$$\displaystyle L^2(X)=\int_{\Unit(Z(X))}^\oplus L^2(X,\chi) d\chi.$$
	Let $L^2_{\disc}(X,\chi)$\index{$L^2_{\disc}(X,\chi)$} be the discrete part of $L^2(X,\chi)$ i.e.\ the (completed) direct sum of all its irreducible subrepresentations. As explained in \cite[Sect. 9.3]{SV}, it is a priori not clear how the discrete subspaces $L^2_{\disc}(X,\chi)$ vary with $\chi$. Actually, even the property that these subspaces vary ``measurably'' with $\chi$  is not clear and this basic fact was actually proved in \cite[Proposition 9.3.3]{SV} using the theory of smooth asymptotic maps which we will review in the next section. (And whose development eventually depends on the property $\mathbf{(Z-fin)}$ for all the boundary degenerations of $X$ as we will explain there.) For the sake of this paragraph, we will assume that this basic fact holds which allows to introduce the relatively discrete subspace
	$$\displaystyle L^2_{\disc}(X)=\int_{\Unit(Z(X))}^\oplus L^2_{\disc}(X,\chi) d\chi$$\index{$L^2_{\disc}(X)$}
	of $L^2(X)$.
	
	The situation should be compared to the group case where such issue does not appear as we may always twist irreducible representations by the torus of unramified characters of {\it the group} and the latter maps surjectively to the torus of unramified characters of {\it the center}. Building on this example, it is natural to introduce the notion of a {\it factorizable} spherical variety $X$ by asking that the kernel of the action of $Z(X)^0$ on the ``abelianization'' $X_{\mathrm{ab}}:=X\sslash G_{\der}$\index{$X_{ab}$} (where $G_{\der}$ denotes the derived subgroup of $G$) be anisotropic\footnote{Note that this in particular entails that the cokernel of $Z(G)^0\to Z(X)^0$ is anisotropic.}. Indeed, if $X$ is factorizable, fixing a base-point in $X_{ab}$ so that it can be identified with a torus quotient of $Z(X)$, multiplication by the unramified characters of $X_{\mathrm{ab}}(k)$ allows to vary central characters over full connected components of $\Unit(Z(X))$ (since the torus of unramified characters of $X_{\mathrm{ab}}(k)$ surjects onto that of $Z(X)(k)$). All symmetric varieties are factorizable but there are many examples of spherical varieties that are not e.g. $\GL_1\backslash \GL_2$. The Discrete Series Conjecture of \cite[\S 9.4.6]{SV} states roughly that the discrete subspaces $L^2_{\disc}(X,\chi)$ vary in the same way as for varieties parabolically induced from a factorizable one. More precisely, here is a slightly modified form of the conjecture:
	
	\begin{itemize}
		\item $\mathbf{(DSC')_X}$ There exist a parabolic subgroup $P=LU$ of $G$, a subgroup $P(k)'\subset P(k)$ containing $P(k)^1$ (in particular containing the unipotent radical $U(k)$)\footnote{Recall that $P(k)^1$ denotes the intersection of the kernels of all unramified characters $P(k)\to \mathbb{C}^\times$.}, a unitary representation $\Sigma$ of $L(k)$ that is discrete (i.e.\ a direct sum of countably many irreducible representations), a surjective morphism $Z(L)^0\to Z(X)^0$ and an isomorphism of unitary representations
		$$\displaystyle I_{P(k)'}^{G(k)}(\Sigma)\simeq L^2_{\disc}(X)$$
		which is $Z(L)^0$-equivariant.
	\end{itemize}
	
	As explained in \cite[Proposition 9.4.8]{SV}, this property of a spherical variety $X$ has a direct consequence to the set of real exponents of the set of (relative) discrete series which in turn implies the existence of certain spectral gaps that are crucial for the construction of Bernstein maps (Section \ref{Sect Bernstein maps}). To be more specific, let us denote by $\cA^*=X^*(A)_{\overline{k}}\otimes \bR$ be the set of real exponents, where we recall that $A$ stands for the universal Cartan of $G$ (taken over an algebraic closure if $G$ is not split for simplicity). Then, to every infinitesimal character $\chi: \cZ(G)\to \bC$ corresponding to a point $z_\chi\in \cB(G)$, we can associate a $W_{\overline{k}}$-orbit\index{$W_{\overline{k}}$} $\Re(\chi)\in \cA^*/W_{\overline{k}}$\index{$\Re(\chi)$} (where $W_{\overline{k}}$ denotes the absolute Weyl group of $G$) as follows. Take a cuspidal pair $(L,\sigma)$ representing $z_\chi$. Then, the absolute value of the central character of $\sigma$, defines an unramified character $\lvert \omega_\sigma\rvert: Z(L)^0(k)\to \bR_{>0}$ which can then be lifted, uniquely, to a character $L(k)\to \bR_{>0}$. The restriction of this lifting to any maximal torus inside $L$ then defines a point of $\cA^*$ up to the action of $W_{\overline{k}}$, it is $\Re(\chi)$. Let $\Pi_{\disc}(X)$ be the set of all isomorphism classes of irreducible representations of $G(k)$ appearing in $L^2_{\disc}(X,\chi)$ for some $\chi\in \Unit(Z(X))$ and, for $J\subset G(k)$ a compact-open subgroup, let $\Pi_{\disc}(X)^J$ be the subset of $\pi\in \Pi_{\disc}(X)$ such that $\pi^J\neq 0$. Every $\pi\in \Irr(G)$ defines an infinitesimal character $\chi_\pi: \cZ(G)\to \bC$ and with these definitions, we can now state the consequence of $\mathbf{(DSC')}$ to the space of real exponents of $\Pi_{\disc}(X)$ on which the construction of Bernstein maps eventually rests as follows:
	
	\begin{itemize}
		\item $\mathbf{(SG)_X}$ For every compact-open sugroup $J\subseteq G(k)$, the image of the map
		$$\displaystyle \pi\in \Pi_{\disc}(X)^J\mapsto \Re(\chi_\pi)\in \cA^*/W_{\overline{k}}$$
		is finite.
	\end{itemize}
	
	As a final remark, let us note that condition $\mathbf{(SG)_X}$ is automatically satisfied when $X$ is tempered (because the set $\Temp(G)^J$ has finitely many real exponents).
	
\end{paragr}

\begin{paragr}[Unfolding.]\label{S unfolding}
	
	The {\it unfolding} technique of Sakellaridis and Venkatesh \cite[\S 9.5]{SV} is a useful tool to check that a given spherical variety $X$ satisfies condition $\mathbf{(Z-fin)}$ or the discrete series conjecture $\mathbf{(DSC')}$ whenever it does not obviously do so (e.g. if $X$ is not parabolically induced from a central or factorizable variety). We will not give a systematic presentation of this technique here, referring the reader to {\em op.\ cit.}\ for a thorough theoretical description of this process, a local counterpart to the well-known unfolding method of global period integrals, and in particular how it can help to verify the Discrete Series Conjecture. (Indeed, this was the emphasis of {\em op.\ cit.}) In a nutshell, unfolding is a way to produce isomorphisms of unitary representations
	$$\displaystyle L^2(X)\simeq L^2(Y)$$
	where $X$ is the spherical variety of interest whereas, in general, $Y$ stands for the Whittaker induction (in the sense of \S \ref{S Whittaker induction}) of some auxilliary spherical variety. These isomorphisms are obtained by a series of partial Fourier transforms and are typically $Z(X)^0$-equivariant via an isomorphism of centers $Z(X)^0\simeq Z(Y)^0$ (where the latter has to be suitably interpreted by the normalizer of a subgroup with a character). The main point is then to produce such an unfolding where the action of $Z(Y)^0$ on $L^2(Y)$ is easier to analyze e.g. where $Y$ is parabolically induced from a central or factorizable variety (or Whittaker induced of such). There is no known example of a spherical variety that does not unfold to the parabolic induction of a factorizable one and it would be very desirable to show that it is always the case.

	Let us illutrate the unfolding method on the first example of a non-wavefront spherical variety (and not parabolically induced from such) that is $X=\GL_2\backslash \SO_5$ (where each group represents its split form). In this case, writing $\Delta=\{\alpha,\beta \}$ with $\alpha$ (resp. $\beta$) the short (resp. long) root, we have $\Delta_X=\{\alpha,\alpha+\beta \}$ and there is exactly one boundary degeneration that is not parabolically induced from a factorizable variety: it is $X_\alpha$. More precisely, let $P\subset \SO_5$ be a parabolic subgroup stabilizing an isotropic plane, $U$ be its unipotent radical and $L$ a Levi factor that we identify with $\GL_2$ (via the choice of a basis of the isotropic plane). Then, the boundary degeneration $X_\alpha$ is isomorphic to $\SL_2\ltimes Z(U)\backslash \SO(5)$ where $Z(U)$ stands for the center of $U$ whereas $Z(X_\alpha)=\GL_2/\SL_2\simeq \mathbb{G}_m$. Note that $U$ is a Heisenberg group with $Z(U)$ one-dimensional and that the quotient $U/Z(U)$ is isomorphic to the standard representation of $L=\GL(2)$. Fixing such isomorphism $U(k)/Z(U)(k)\simeq k^2$ with the standard representation, the unfolding technique can be applied in this situation to yield an isomorphism of unitary representations
	$$\displaystyle \mathcal{U}: L^2(X_\alpha)\simeq L^2(V(k)\ltimes U(k)\backslash \SO_5(k),\psi_U)$$
	where $\psi_U$ is the character $u=\begin{pmatrix} x \\ y \end{pmatrix}\in U(k)/Z(U)(k)\mapsto \psi(y)$, for some non-trivial character $\psi: F\to \bC^\times$, and $V=\begin{pmatrix} 1 & \star \\ 0 & 1 \end{pmatrix}$ is the standard unipotent subgroup of $L=\GL(2)$ (which is precisely the stabilizer of $\psi_N$ in $\SL_2$). On the dense subspace $\cS(X_\alpha)\subset L^2(X_\alpha)$ the isomorphism is given by
	$$\displaystyle f\in \cS(X_\alpha)\mapsto (\mathcal{U}f)(g)=\int_{U(k)/ Z(U)(k)} f(ug) \psi_U(u)^{-1}du.$$
	Moreover, the right-hand side of the above isomorphism is readily seen to be induced from a parabolic $Q$ with Levi factor $M=\GL_1\times \SO_3$. More precisely:
	$$\displaystyle L^2(V(k)\ltimes U(k)\backslash \SO_5(k),\psi_N)\simeq I_Q^{\SO_5}L^2(U_M(k)\backslash M(k),\psi_{U_M})$$
	where $U_M=U\cap M$ is a maximal unipotent subgroup of $M$, $\psi_{U_M}$ a nondegenerate character on it (i.e.\ the inducing representation is the $L^2$ Whittaker model of $M$) and the induction functor is implicitely taken to be unitary induction. Moreover, as is readily checked, by the above chain of isomorphisms, the action of $Z(X_\alpha)$ is transported to the natural action of $Z(M)$ on the last induced representation by the natural identification $Z(X_\alpha)=Z(M)$. Then, by the very same argument as for Lemma \ref{lem Z-fin} (iii), for every compact-open subgroup $J\subset G(k)$ the action of $Z(M)(k)$ on $I_Q^{\SO_5}(L^2(U_M(k)\backslash M(k),\psi_{U_M}))^J$ is finite over the Bernstein center $\cZ(G)$, and therefore so does the action of $Z(X)(k)$ on $\cS(X_\alpha)^J\subset L^2(X_\alpha)^J$ i.e.\ we have $\mathbf{(Z-fin)}_{X_\alpha}$. Similarly, the ``Whittaker variety'' $Y=(U_M\backslash M,\psi_{U_M})$ is factorizable in a suitable sense from which we can readily deduce that the Discrete Series Conjecture $\mathbf{(DSC')}$ holds for $X_\alpha$ as well.
	
	The above discussion implies that all the boundary degenerations of $X=\GL_2\backslash \SO_5$ satisfy both $\mathbf{(Z-fin)}$ and $\mathbf{(DSC')}$. Therefore, as will be explained in the next sections, the theory of smooth asymptotics and Bernstein maps can be applied to $X$.
\end{paragr}

\subsection{Smooth asymptotics}\label{Sect smooth asym}

In this subsection and the next we review (part of) the theory developed by Sakellaridis-Venkatesh in \cite{SV} broadly relating harmonic analysis on $X$ to that on its boundary degenerations. It can roughly be divided into two parts, one concerning the smooth theory and one on the $L^2$-theory. In both cases, they are best stated in terms of certain {\it asymptotic maps} between function spaces. In the smooth setting these are simply refered to as {\it smooth asymptotics} whereas in the unitary or $L^2$ setting these were named {\it Bernstein maps} by Sakellaridis and Venkatesh.
	
All of these results have been obtained under the assumptions that $k$ is non-Archimedean (still of characteristic zero) and that the group $G$ is split and these will be running assumptions until Section \ref{Sect most cont spectrum}. As of now, we also need to impose some restrictions on the spherical variety $X$ (the most stringent one being that it is wavefront) that will be described below in more details. However, we certainly expect that a similar story (suitably interpreted) can be told without those restrictions and since the appearance of the book \cite{SV} the theory of asymptotic maps has been extended to $p$-adic symmetric varieties (by the work of Delorme \cite{Delpadicsymm}) whereas Berstein maps have been constructed for all real spherical varieties (see \cite{DKKS}). For arbitrary $p$-adic spherical varieties, the main missing ingredient for the extension of Sakellaridis-Venkatesh theory was a lack of a suitable theory of compactifications over the base field (when the group $G$ is nonsplit) but this has recently been developed by Knop and Kr\"otz \cite{KKRGA}. On the other hand, finding the correct analog of smooth asymptotics for Archimedean fields, even in the group case, looks like an outstanding question (see however \cite{WangSL2R} for some hints in that direction).
	
For every compact-open subgroup $J\subseteq G(k)$, we will write $\cH_J(G)$\index{$\cH_J(G)$} for the Hecke algebra of compactly supported bi-$J$-invariant measures on $G(k)$ whose multiplication is given by convolution $\star$. 

\vspace{2mm}

\begin{paragr}[Neighborhoods of $\infty_\Theta$.]\label{S neighbds inftytheta}
	Recall that to every subset $\Theta\subset \Delta_X$, we have associated in Section \ref{sect boundary degenerations} a boundary degeneration $X_\Theta$ which is supposed to model the geometry at infinity of $X$ in certain directions.
	
	Let us fix from now on a smooth complete toroidal embedding $X\hookrightarrow \overline{X}$ as in \S \ref{S toroidal embeddings}. Recall that we have associated to every $\Theta\subset \Delta_X$ a closed $G$-stable subvariety $\infty_\Theta\subset \overline{X}$. In the following, we will implicitly identify $\infty_\Theta$ with its set of $k$-points.
	
	By definition, a {\it neighborhood of $\infty_\Theta$} in $X(k)$ is the intersection of the latter with a neighborhood of $\infty_\Theta$ in $\overline{X}(k)$. It can be shown that this notion does not depend on the chosen toroidal embedding. (Essentially because any morphism $\overline{X}'\to \overline{X}$ of toroidal embeddings sends $\Theta$-infinity $\infty_\Theta'\subset \overline{X}'$ for $\overline{X}'$ onto $\infty_\Theta\subset \overline{X}$ and any pair of toroidal embeddings can be covered by a third one.)
	
	Following \S \ref{S ident normal bundles}, for each subset $\Theta\subseteq \Delta_X$, $\overline{X}$ gives rise to smooth (but not complete in general) toroidal embedding $X_\Theta\hookrightarrow \overline{X}_\Theta$ with $\overline{X}_\Theta$ containing a natural copy of $\infty_\Theta$. We define similarly a {\it neighborhood of $\infty_\Theta$ in $X_\Theta(k)$} to be the intersection of a (genuine) neighborhood of $\infty_\Theta$ in $\overline{X}_\Theta(k)$ with $X_\Theta(k)$. For any subset $\Omega\subset \Theta$, we have $\infty_\Omega\subset \infty_\Theta$ and we can define more generally the notion of {\it neighborhood of $\infty_\Omega$ in $X_\Theta(k)$} in a similar way. We should however warn the reader that this definition is not equivalent to that of neighborhoods of $\infty_\Omega$ in $X(k)$ as defined above by simply replacing $X$ by $X_\Theta$. (We recall from \S \ref{S invts Xtheta} that we have $\Delta_{X_\Theta}=\Theta$.) Indeed, we can always find a complete toroidal embedding $\overline{X}_\Theta\hookrightarrow \overline{\overline{X}}_\Theta$ and the corresponding closed subvariety $\infty^{X_\Theta}_\Omega\subset \overline{\overline{X}}_\Theta$ of ``$\Omega$-infinity'' is usually strictly larger than $\infty_\Omega$. When we think that these two notions may enter in conflict (e.g. when arguing by induction assuming that some result holds for all the proper boundary degenerations) we will distinguish the two notions of $\Omega$-infinity with superscripts as follows: $\infty_\Omega^X$ and $\infty_{\Omega}^{X_\Theta}$.
	
	Recall that there is a left action of $A_{X,\Theta}$ on $X_\Theta$ coming from the identification $Z(X_\Theta)^0=A_{X,\Theta}$. Set
	$$\displaystyle A_{X,\Theta}^-=\{a\in A_{X,\Theta}(k)\mid \lvert \alpha(a)\rvert\geq 1, \; \forall \alpha\in \Delta_X\setminus \Theta \}.$$\index{$A_{X,\Theta}^-$}
	Then, we have the following important properties relating neighborhoods of infinity in $X_\Theta(k)$ and the action of the monoid $A_{X,\Theta}^-$:
	\begin{itemize}
		\item For every neighborhood $N_\Theta\subseteq X_\Theta(k)$ of $\infty_\Theta$ and $x\in X_\Theta(k)$, we can find $a\in A_{X,\Theta}(k)$ such that $A_{X,\Theta}^- ax\subseteq N_\Theta$;
		
		\item Let $N_\Omega\subseteq X_\Theta(k)$ be neighborhoods of $\infty_\Omega$ for every proper subset $\Omega\subsetneq \Theta$. Then, for every sufficiently small neighborhood $N_\Theta\subseteq X_\Theta(k)$ of $\infty_\Theta$, the difference
		$$\displaystyle N_\Theta\setminus \bigcup_{\Omega\subsetneq \Theta} N_\Omega$$
		is relatively compact modulo $A_{X,\Theta}^-$ (i.e.\ there exists a compact subset $C\subseteq X_\Theta(k)$ such that the difference is included in $A_{X,\Theta}^-C$).
	\end{itemize}
\end{paragr}

\begin{paragr}[Description using a weak Cartan decomposition.]\label{S weak Cartan}
	There is another alternative, perhaps more concrete, description of neighborhoods of $\infty_\Theta$ which rests on the choice of a weak Cartan decomposition. Indeed, fixing a compact subset $\mathcal{K}\subset G(k)$ such that \eqref{weak Cartan} holds, we can describe neighborhoods of $\infty_\Theta$ as the subsets of $X(k)$ containing $A_X^-(\Theta,\geq C)\mathcal{K}$ for some $C\geq 1$, where we have set 
	$$\displaystyle A_X^-(\Theta,\geq C):=\{a\in A_X^-\mid \lvert \alpha(a)\rvert\geq C,\; \forall \alpha\in \Delta_X\setminus \Theta \}.$$
	
	For $\Omega\subseteq \Theta$, we can similarly describe neighborhoods of $\infty_\Omega$ in $X_\Theta(k)$ as follows. First, we pick a weak Cartan decomposition for $X_\Theta(k)$, it takes the form $X_\Theta(k)=A_X^{\Theta,-}\mathcal{K}$ where $\mathcal{K}\subset G(k)$ is a suitable compact subset, $A_X^{\Theta,-}=\{a\in A_X(k)\mid \lvert \alpha(a)\rvert\geq 1, \; \forall \alpha\in \Theta \}$ and we have chosen an embedding $A_X\subset X_\Theta$ as in \S \ref{S parabolic type}. Then, neighborhoods of $\infty_\Omega$ in $X_\Theta(k)$ are the subsets containing $A_X^-(\Omega,\geq C)\mathcal{K}$ for some $C\geq 1$.
\end{paragr}

\begin{paragr}[Exponential maps.]\label{S exponential maps}
	The fundamental device to relate harmonic analysis on $X(k)$ to its boundary degenerations are {\it exponential maps}. More precisely, for any $\Theta\subseteq \Delta_X$, a {\em $(X,\Theta)$-exponential map}, or simply a {\em $\Theta$-exponential map} if the variety $X$ is clear from the context, in the sense of Sakellaridis-Venkatesh \cite[\S 4.3]{SV} is a bijection
	$$\displaystyle \exp_\Theta: N_\Theta\simeq N'_{\Theta}$$\index{$\exp_\Theta$}
	where $N_\Theta\subset \overline{X}(k)$ and $N'_{\Theta}\subset \overline{X}_\Theta(k)$ are open neighborhoods of $\infty_\Theta$ that extends (necessarily uniquely) to a $F$-analytic isomorphism 
	$$\displaystyle \overline{\exp}_\Theta: \overline{N}_\Theta\simeq \overline{N}'_\Theta$$\index{$\overline{\exp}_\Theta$}
	between the closures of $N_\Theta$, $N'_\Theta$ in $\overline{X}(k)$ and $\overline{X}_{\Theta}(k)$ respectively, satisfying the following conditions:
	\begin{itemize}
		\item With the identification of \S \ref{S ident normal bundles}, the restriction of $\overline{\exp}_\Theta$ to $\infty_\Theta$ is the identity and for every $G$-orbit closure $Z\subset \infty_\Theta$, the normal derivatives of $\overline{\exp}_\Theta$ induce the canonical identification
		$$\displaystyle N_{Z(k)}(\overline{X}(k))\simeq N_{Z(k)}(\overline{X}_\Theta(k))$$
		of \eqref{can iden normal bundles}.
		
		\item The map $\overline{\exp}_\Theta$ preserves the natural correspondence between orbits that is to say: if $\cF$ denotes the fan giving rise to the toroidal embedding $\overline{X}$, for every cone $C\in \cF$, $\exp_\Theta$ sends $Z_C(k)\cap N_\Theta$ onto $Z'_C(k)\cap N'_\Theta$ where we have denoted by $Z_C\subset \overline{X}$, $Z'_C\subset \overline{X}_\Theta$ the corresponding closures of $G$-orbits (see \S \ref{S ident normal bundles}).
	\end{itemize}
	
	The basic properties on existence and canonicity of exponential maps is summarized in the following statement (see \cite[Proposition 4.3.1, Proposition 4.3.3 and \S 4.3.4]{SV}).
	
	\begin{prop}\label{prop exponential maps}
		Let $\Theta\subset \Delta_X$ and $J\subseteq G(k)$ be a compact-open subgroup. Then,
		\begin{enumerate}[(i)]
			\item $\Theta$-exponential maps exist, more precisely for any small enough neighborhood $N_\Theta\subset X(k)$ of $\infty_\Theta$, there exists a $\Theta$-exponential map $\exp_\Theta: N_\Theta\simeq N'_{\Theta}$;
			
			\item Let $\exp_\Theta: N_\Theta\simeq N'_\Theta$ be a $\Theta$-exponential map. Then, we may find $J$-stable neighborhoods $U_\Theta\subset N_\Theta$ and $U'_\Theta\subset N'_\Theta$ of $\infty_\Theta$ such that the restriction of $\exp_\Theta$ descends to a bijection
			\begin{equation}\label{J exp map}
				\displaystyle \exp_{\Theta,J}: U_\Theta/J\simeq U'_\Theta/J.
			\end{equation}
			We shall call such a bijection a \textbf{$(X,\Theta,J)$-exponential map}, or simply a \textbf{$(\Theta,J)$-exponential map} if $X$ is clear from the context.
			
			\item The germ of $(\Theta,J)$-exponential maps at $\infty_\Theta$ is unique. More precisely, if $\exp_{\Theta,J}: N_\Theta/J \simeq N'_\Theta/J$ and $\exp'_{\Theta,J}: U_\Theta/J\simeq U'_\Theta/J$ are two $(\Theta,J)$-exponential maps then we can find further $J$-stable neighborhoods $V_\Theta\subset N_\Theta\cap U_\Theta$, $V'_\Theta\subset N'_\Theta\cap U'_\Theta$ of $\infty_\Theta$ such that $\exp_{\Theta,J}$ and $\exp'_{\Theta,J}$ restrict to a common bijection
			\begin{equation*}
				\displaystyle \exp_{\Theta,J}\mid_{V_\Theta}=\exp'_{\Theta,J}\mid_{V_\Theta}: V_\Theta/J\simeq V'_\Theta/J.
			\end{equation*}
			
			\item (Eventual equivariance) Let $\exp_{\Theta,J}: N_\Theta/J\simeq N'_\Theta/J$ be a $(\Theta,J)$-exponential map and $h\in \mathcal{H}_J(G)$. Then, we can find $J$-invariant neighborhoods $U_\Theta\subset N_\Theta$ and $U'_\Theta\subset N'_\Theta$ of $\infty_\Theta$ such that:
			\begin{itemize}
				\item $\exp_{\Theta,J}$ restricts to a bijection $U_\Theta/J\simeq U'_\Theta/J$;
				
				\item we have $U_\Theta \Supp(h)\subset N_\Theta$ and $U'_\Theta \Supp(h)\subset N'_\Theta$;
				
				\item for every function $f\in C^\infty(U'_\Theta)^J$ we have $\exp^*_{\Theta}(f\star h)=\exp^*_{\Theta}(f)\star h$ where $\exp_\Theta^*(f)$ denotes the pullback function $f\circ\exp_\Theta\in C^\infty(U_\Theta)$ (and similarly for $\exp^*_{\Theta}(f\star h)$).
			\end{itemize}
			
			\item (Transitivity of exponential maps) Let $\Omega\subseteq \Theta$, and let $\exp_{\Theta,J}: N_\Theta/J\simeq N'_\Theta/J$, $\exp_{\Omega,J}^\Theta: N_\Omega^\Theta/J\simeq N_\Omega^{\Theta,'}/J$ be respectively $(X,\Theta,J)$- and $(X_\Theta,\Omega,J)$-exponential maps. (Thus, here $N_\Omega^\Theta$ and $N_\Omega^{\Theta,'}$ are $J$-invariant neighborhoods of $\infty_\Omega^{X_\Theta}$.) Then, we can find a $J$-invariant neighborhood $U_\Theta\subset N_\Theta$ of $\infty_\Theta$ such that $\exp_{\Theta,J}(U_\Theta/J)\subset N_\Omega^\Theta/J$ and the composition $\exp_{\Omega,J}^\Theta\circ \exp_{\Theta,J}\mid_{U_\Theta/J}$ is a $(X,\Omega,J)$-exponential map.
		\end{enumerate}
	\end{prop}
\end{paragr}

\begin{paragr}[Smooth asymptotics: statement.]
	
	The following theorem contains the basic existence result on smooth asymptotic maps. 	
	
	\begin{theo}\label{theo smooth asymptotic}
		Let $\Theta\subset \Delta_X$ and assume that the property $\mathbf{(Z-fin)_{X_\Omega}}$ holds for every subset $\Omega\subseteq \Theta$. Then, there exists a $G(k)$-equivariant morphism
		\begin{equation}
			\displaystyle e_\Theta: \cS(X_\Theta)\to \cS(X)
		\end{equation}\index{$e_\Theta$}
		that is uniquely characterized by the following property: for every compact-open subgroup $J\subset G(k)$, we can find an exponential map $\exp_\Theta: N_\Theta\to N'_\Theta$, that descends to a bijection between $J$-orbits $N_\Theta/J\simeq N'_\Theta/J$, such that for every $f\in \cS(X_\Theta)^J$ we have $e_\Theta(f)=\exp_\Theta^*(f)$.
	\end{theo}
	
	Note that from Lemma \ref{lem Z-fin}, the assumptions of the above theorem are always satisfied when $X$ is (parabolically induced from) a wavefront spherical variety. This is actually the original assumption made by Sakellaridis and Venkatesh in \cite{SV}. However, we prefer the above formulation as, despite the fact that we don't yet know a systematic way to check the assumption outside the wavefront setting, we feel that it makes clearer what is currently missing to deal with the general case (that is a better understanding of the center for arbitrary spherical varieties). The proof we will sketch in the subsequent paragraphs will actually follow a strategy introduced by Bezrukavnikov and Kazhdan in the group case \cite{BK} (where the existence of smooth asymptotic maps is closely related to Bernstein second adjunction theorem) which takes as input some finite generation property of the Schwartz space $\cS(X_\Theta)$. Thus, we need to establish at the same time this finiteness property for the boundary degenerations and this will be achieved by induction using the smooth asymptotic maps $\cS(X_\Omega)\to \cS(X_\Theta)$ for $\Omega\subset \Theta\subsetneq \Delta_X$. More formally, we will apply the following theorem that will be established simultaneously to Theorem \ref{theo smooth asymptotic}.
	
	\begin{theo}\label{theo f.g. Schwartz space}
		Assume that the property $\mathbf{(Z-fin)_{X_\Theta}}$ holds for every subset $\Theta\subseteq \Delta_X$. Then, for every compact-open subgroup $J\subset G(k)$, the $\cH_J(G)$-module $\cS(X)^J$ is finitely generated.
	\end{theo}
\end{paragr}

\begin{paragr}[Application to finiteness of multiplicities.]\label{S smooth asymptotics cons}
	The smooth asymptotic maps have as immediate consequence the finiteness of multiplicities for the spherical variety $X$ (see \cite[Theorem 5.1.5]{SV}):
	
	\begin{cor}\label{cor multiplicities}
		Assume that the property $\mathbf{(Z-fin)_{X_\Theta}}$ holds for every subset $\Theta\subseteq \Delta_X$. Then, for every irreducible representation $\pi\in \Irr(G)$ the multiplicity
		$$\displaystyle m_X(\pi):=\dim_{\bC} \Hom_G(\cS(X),\pi)$$
		is finite. More generally, all the higher extension groups $\Ext^i_G(\cS(X),\pi)$, for $i\geq 0$ and $\pi\in \Irr(G)$, are of finite dimension.
	\end{cor}
	
	\begin{proof}
		Indeed, by \cite[corollaire 3.9]{Bercenter}, we can find a compact-open subgroup $J\subset G(k)$ such that the subgategory of smooth representations $V$ of $G(k)$ generated by their $J$-invariant subspace $V^J$ is a direct summand which is equivalent to the category of $\cH_J(G)$-modules by the functor $V\mapsto V^J$ and $\pi^J\neq 0$. Therefore, $\Ext^i_G(\cS(X),\pi)=\Ext^i_{\cH_J(G)}(\cS(X)^J,\pi^J)$ for all $i\geq 0$ and the result is now a direct consequence of Theorem \ref{theo f.g. Schwartz space} (and the fact that $\cH_J(G)$ is Noetherian).
	\end{proof}
\end{paragr}

\begin{paragr}[Smooth asymptotics: proof.]
	We sketch a proof of Theorem \ref{theo smooth asymptotic} that is different from the one given in \cite{SV} (which is based on Bernstein's stabilization theorem) and basically follows the strategy used by Bezrukavnikov and Kazhdan in the group case \cite{BK}. The key ingredient is the following multidimensional version of \cite[Lemma 4.4]{BK}. For convenience, we reproduce the proof of {\it op.\ cit.}\ below.
	
	\begin{lem}\label{lemma BK}
		Let $A$ be a Noetherian $\bC$-algebra, $\Lambda\simeq \bZ^r$ be a free $\bZ$-module of finite rank and $\Lambda^-\subset \Lambda$ be a saturated finitely generated submonoid that generates $\Lambda$ as a $\bZ$-module\footnote{Concretely, this means that we can find a subset $S\subset \Hom(\Lambda,\bZ)$ generating a strictly convex subcone such that $\Lambda^-=\{\lambda\in \Lambda\mid \langle \alpha,\lambda\rangle \leq 0 \; \forall \alpha\in S \}$.}. We denote by $A[\Lambda]$ the algebra of Laurent polynomials $\sum_{\lambda\in \Lambda} a_\lambda T^\lambda$, where $(a_\lambda)_{\lambda\in \Lambda}\in A^\Lambda$ is a family with finite support, and $A[\Lambda^-]$ the subalgebra corresponding to families supported in $\Lambda^-$.
		Assume given:
		\begin{itemize}
			\item An $A[\Lambda]$-module $M$ that is of finite type as an $A$-module and $N$ an (arbitrary) $A$-module,
			
			\item $M_0\subset M$ a $\bC$-vector subspace and $e_0: M_0\to N$ a $\bC$-linear map,
		\end{itemize}
		satisfying the following condition:
		\begin{itemize}
			\item For every $m\in M$ and $h\in A$, there exists $\lambda_0=\lambda_0(m,h)\in \Lambda$ such that for every $\lambda\in \lambda_0+\Lambda^-$ we have $T^\lambda m\in M_0$, $T^\lambda hm\in M_0$ and $e_0(T^\lambda hm)=he_0(T^\lambda m)$.
		\end{itemize}
		Then, there exists a unique $A$-linear map $e:M\to N$ satisfying: for every $m\in M$, there exists $\lambda_0=\lambda_0(m)\in \Lambda$ such that for every $\lambda\in \lambda_0+\Lambda^-$ we have $T^\lambda m\in M_0$ and $e(T^\lambda m)=e_0(T^\lambda m)$.
	\end{lem}
	
	\begin{proof}
		Pick a finite set $\{m_1,\ldots,m_s \}$ generating $M$ as a $A$-module.
		The unicity follows readily from noting that if $e$ satisfies the condition of the lemma, we can find $\lambda\in \Lambda$ such that $e(T^\lambda m_i)=e_0(T^\lambda m_i)$ for every $1\leq i\leq s$ while the family $\{T^\lambda m_1,\ldots,T^\lambda m_s \}$ still generates $M$ as a $A$-module. 
		
		For the existence, let $\tilde{M}$ be a free $A$-module of rank $s$ together with a basis denoted by $\tilde{m}_1,\ldots,\tilde{m}_s$ and let $\Pi: \tilde{M}\to M$ be the $A$-linear surjection sending $\tilde{m}_i$ to $m_i$ for every $i$. We will also use the convenient, although slightly imprecise, notation $\lambda \ll_{\Lambda^-} 0$ to mean that there exists $\lambda_0\in \Lambda$ such that a certain property holds for every $\lambda\in \lambda_0+\Lambda^-$.

		By assumption, for $\lambda\ll_{\Lambda^-} 0$ we have $T^\lambda m_i\in M_0$ for each $1\leq i\leq s$. For such a $\lambda$, we define $\tilde{e}_\lambda\in \Hom_A(\tilde{M},N)$ by $\tilde{e}_\lambda(\tilde{m}_i)=e_0(T^\lambda m_i)$, $1\leq i\leq s$.
		
		Let $x=\sum_{i=1}^s h_i \tilde{m}_i\in \Ker(\Pi)$. Then, still by the assumption of the lemma, for $\lambda \ll_{\Lambda^-} 0$ we have
		$$\displaystyle \tilde{e}_\lambda(x)=\sum_i h_i e_0(T^\lambda m_i)=e_0(T^\lambda \Pi(x))=0.$$
		Since $\tilde{M}$ is of finite type and $A$ is Noetherian, the $A$-module $\Ker(\Pi)$ is finitely generated and it therefore follows that for $\lambda\ll_{\Lambda^-} 0$, $\tilde{e}_\lambda$ vanishes on $\Ker(\Pi)$ and thus descends to an homomorphism $e_\lambda\in \Hom_A(M,N)$.
		
		Let $\mu\in X^+$ and choose decompositions
		$$\displaystyle T^\mu m_i=\sum_{j=1}^s h_{ij} m_j,\;\; h_{ij}\in A,$$
		for every $1\leq i\leq s$. Then, for $\lambda\ll_{\Lambda^-} 0$ we have
		$$\displaystyle e_{\lambda}(T^\mu m_i)=\sum_j h_{ij}e_0(T^\lambda m_j)=e_0(T^{\lambda+\mu}m_i)=e_{\lambda+\mu}(m_i)$$
		for every $i$. Since $\Lambda^-$ is finitely generated, this shows that for $\lambda\ll_{\Lambda^-}0$ we have $e_\lambda\circ T^\mu=e_{\lambda+\mu}$ for every $\mu\in \Lambda^-$ and then it readily follows that $e:=e_\lambda\circ T^{-\lambda}$ satisfies the desired property for $\lambda\ll_{\Lambda^-}0$.
	\end{proof}
	
	We now explain the proof of Theorems \ref{theo smooth asymptotic} and \ref{theo f.g. Schwartz space}. We prove both by induction on $\lvert \Delta_X\rvert$. Actually, we will first prove Theorem \ref{theo smooth asymptotic} and then explain how to deduce Theorem \ref{theo f.g. Schwartz space} from it.
	
	For the proof of Theorem \ref{theo smooth asymptotic}, obviously, it suffices to show that for every compact-open subgroup $J\subset G(k)$ there exists a unique  $\mathcal{H}_J(G)$-equivariant map $e_{\Theta,J}: \cS(X_\Theta)^J\to \cS(X)^J$ with the required asymptotic property. (Indeed, unicity will entail that for $J'\subset J$ the restriction of $e_{\Theta,J'}$ to $\cS(X_\Theta)^J$ coincides with $e_{\Theta,J}$ and therefore that the family $(e_{\Theta,J})_J$, with $J$ running over all compact-open subgroups, comes from a $G(k)$-equivariant morphism $e_\Theta: \cS(X_\Theta)\to \cS(X)$.) 
	Let us now fix a compact-open subgroup $J\subset G(k)$. We may suppose that $\Theta\neq \Delta_X$ (otherwise the statement is trivial with $e_{\Delta_X}=Id$). Then, by the induction hypothesis on Theorem \ref{theo f.g. Schwartz space}, we know that $\cS(X_\Theta)^J$ is a finitely generated $\cH_J(G)$-module.
	
	Let $\exp_{\Theta,J}: N_\Theta/J\simeq N'_\Theta/J$ be a $(\Theta,J)$-exponential map as in Proposition \ref{prop exponential maps}. We are going to apply Lemma \ref{lemma BK} to the following objects:
	$$\displaystyle A=\cH_J(G),\;\; M=\cS(X_\Theta)^J,\;\; \Lambda=\Lambda_\Theta=X_*(A_{X,\Theta}),\;\; \Lambda^-=\Lambda_\Theta^-=\{\lambda\in \Lambda\mid \langle \alpha,\lambda\rangle\leq 0 \; \forall \alpha\in \Delta_X\setminus \Theta \},$$
	$$\displaystyle N=\cS(X)^J,\;\; M_0=\cS(N'_\Theta)^J,\;\; e_0=\exp_{\Theta,J}^*:M_0\to N.$$
	Here, implicitely, we have fixed an uniformizer $\varpi\in k$ and the action of $\Lambda$ on $M=\cS(X_\Theta)^J$ is deduced from the (left) action of $A_{X,\Theta}(k)$ on $X_\Theta(k)/J$ via the embedding $\Lambda_\Theta\subset A_{X,\Theta}(k)$, $\lambda\mapsto \lambda(\varpi)$. We will denote this action by $(\lambda,f)\in \Lambda\times \cS(X_\Theta)\mapsto L_\lambda f:=L(\lambda(\varpi))f$. Note that for every function $f\in \cS(X_\Theta)$ and every neighborhood $N''_\Theta\subset X_\Theta(k)$ of $\infty_\Theta$ we can find $\lambda_0\in \Lambda_\Theta$ such that $\Supp(L_\lambda f)=\lambda(\varpi)\Supp(f)\subset N''_\Theta$ for all $\lambda\in \lambda_0+\Lambda_\Theta^-$ (see \S \ref{S neighbds inftytheta}). Thus, the assumption of Lemma \ref{lemma BK} is satisfied by the eventual equivariance property of exponential maps (Proposition \ref{prop exponential maps}(iv)) and we may conclude that there exists a unique $\cH_J(G)$-equivariant morphism $e_{\Theta,J}:\cS(X_\Theta)^J\to \cS(X)^J$ satisfying the following property:
	\begin{num}
		\item\label{eq first charact etheta} for every $f\in \cS(X_\Theta)^J$, we have $e_{\Theta,J}(L_\lambda f)=\exp_{\Theta,J}^*(L_\lambda f)$ for $\lambda\ll_{\Lambda_\Theta^-} 0$ (the implicit constant depending on $f$).
	\end{num}
	It now remains to prove that the morphism $e_{\Theta,J}$ so obtained coincides with $\exp^*_{\Theta,J}$ on some $J$-invariant neighborhood $N''_\Theta\subset N'_\Theta$ of $\infty_\Theta$ since the above characterizing condition is in general weaker than the desired one.
	
	For this, it is convenient to use the asymptotic maps $e_\Omega: \cS(X_\Omega)\to \cS(X)$ and $e_\Omega^{\Theta}: \cS(X_\Omega)\to \cS(X_\Theta)$ for all the proper subsets $\Omega\subsetneq \Theta$. Indeed, by a suitable induction hypothesis (on $\lvert \Theta\rvert$ this time), we may assume that the latter maps are already known to exist. The crux is then to check that
	\begin{equation}\label{eq composition asymp maps}
		\displaystyle e_\Omega=e_\Theta\circ e_\Omega^\Theta \mbox{ for every } \Omega\subsetneq \Theta.
	\end{equation}
	For this, we first we note that, as follows readily from unicity, $e_\Omega^\Theta$ is $A_{X,\Theta}(k)$-equivariant (for the left action; note that $A_{X,\Theta}\subset A_{X,\Omega}$) and thus in particular, with our current notation, $e_\Omega^\Theta\circ L_\lambda=L_\lambda\circ e_\Omega^\Theta$ for every  $\lambda\in \Lambda_\Theta$. Let $f_1,\ldots,f_s$ be a generating family of $\cS(X_\Omega)^J$ as a $\cH_J(G)$-module. Then, by the characterizing properties of $e_\Omega$ and $e_\Omega^{\Theta}$ we can find $\lambda_0\in \Lambda_\Omega$ such that for every $\lambda\in \lambda_0+\Lambda_\Omega^-$ and $1\leq i\leq s$ we have
	$$\displaystyle e_\Omega(L_\lambda f_i)=\exp_{\Omega,J}^*(L_\lambda f_i),\;\; e_\Omega^{\Theta}(L_\lambda f_i)=\exp_{\Omega,J}^{\Theta,*}(L_\lambda f_i)$$
	where we have fixed $(\Omega,J)$-exponential maps $\exp_{\Omega,J}: X(k)/J\supset N_\Omega/J\simeq N'_\Omega/J\subset X_\Omega(k)/J$ and $\exp_{\Omega,J}^\Theta: X_\Theta(k)/J\supset N^\Theta_\Omega/J\simeq N^{\Theta,'}_\Omega/J$. Also, fixing such a $\lambda_0$, for $\mu\ll_{\Lambda_\Theta^-} 0$ (where the bounds depends a priori on $\lambda_0$) we have
	$$\displaystyle e_\Theta\circ e_\Omega^\Theta(L_{\mu+\lambda_0}f_i)=e_\Theta(L_\mu e_\Omega^\Theta(L_{\lambda_0}f_i))=\exp_{\Theta,J}^*(\exp_{\Omega,J}^{\Theta,*}L_{\mu+\lambda_0}f_i)$$
	for each $1\leq i \leq s$. By the transitivity of exponential maps (Proposition \ref{prop exponential maps} (iv)), up to shrinking the domains of the various exponential maps, we may further assume that $\exp_{\Theta,J}^*\circ \exp_{\Omega,J}^{\Theta,*}=\exp_{\Omega,J}^*$\footnote{This entails in particular that $N'_\Theta\subset N_\Omega^\Theta$.}. Then, choosing $\mu$ in $\Lambda_\Theta^-\subset \Lambda_\Omega^-$, as we may, we have $\lambda:=\lambda_0+\mu\in \lambda_0+\Lambda_\Omega^-$ and it follows that
	$$\displaystyle e_\Omega(L_{\lambda} f_i)=\exp_{\Omega,J}^*(L_\lambda f_i)=e_\Theta\circ e_\Omega^\Theta(L_{\lambda}f_i)$$
	for each $1\leq i \leq s$, i.e.\ $e_\Omega$ and $e_\Theta\circ e_\Omega^\Theta$ coincides on the family $\{ L_\lambda f_1,\ldots, L_\lambda f_s\}$. As the latter also generates $\cS(X_\Omega)^J$ as $\cH_J(G)$-module the claim \eqref{eq composition asymp maps} follows.
	
	We can now deduce the desired property for $e_\Theta$. Indeed, from \eqref{eq composition asymp maps}, and the fact that $e_{\Omega,J}$, $e_{\Omega,J}^\Theta$ coincide on suitable neighborhoods of $\infty_\Omega$ with pushforwards along exponential maps,  we deduce that, for every $\Omega\subsetneq \Theta$, $e_{\Theta,J}$ coincides with $\exp_{\Theta,J}^*$ on some $J$-invariant and $A_{X,\Theta}^-$-invariant neighborhood $N'_\Omega\subset X_\Theta(k)$ of $\infty_\Omega$. Moreover, as explained in \S \ref{S neighbds inftytheta}, up to shrinking $N_\Theta'$ we may assume that the difference $N'_\Theta\setminus \bigsqcup_{\Omega\subsetneq \Theta} N'_\Omega$ is relatively compact modulo $A_{X,\Theta}^-$ i.e.\ that it is included in the union of a finite number of orbits $\Lambda_\Theta^- x_0J\cup \ldots\cup \Lambda_\Theta^- x_rJ$. Then, applying the property \eqref{eq first charact etheta}  to the characteristic functions $f_i=\mathbf{1}_{x_iJ}$, we see that $e_\Theta$ coincides with $\exp_{\Theta,J}^*$ on some translate $aN'_\Theta$ where $a\in A_{X,\Theta}^-$ but the latter is again a neighborhood of $\infty_\Theta$ and this ends the proof of Theorem \ref{theo smooth asymptotic}.
	
	Given Theorem \ref{theo smooth asymptotic} for $X$ and its boundary degenerations, the proof of Theorem \ref{theo f.g. Schwartz space} (which is needed to complete the induction) is basically the same as that of \cite[Theorem 5.1.5]{SV}. For convenience, we repeat the argument here. By the assumption of this theorem and Theorem \ref{theo smooth asymptotic} we have at our disposal the asymptotic maps $e_\Theta: \cS(X_\Theta)\to \cS(X)$ for every $\Theta\subset \Delta_X$. Moreover, by the induction hypothesis, the $\cH_J(G)$-modules $\cS(X_\Theta)^J$, for every proper subset $\Theta\subsetneq \Delta_X$, are all finitely generated. Thus, it suffices to check that the quotient module
	\begin{equation}\label{eq1 proof theo f.g. Schwartz space}
		\displaystyle \cS(X)^J/\sum_{\Theta\subsetneq \Delta_X} e_\Theta \cS(X_\Theta)^J
	\end{equation}
	is also finitely generated. Let us fix $J$-exponential maps $\exp_{\Theta,J}: N_\Theta/J\simeq N'_\Theta/J$, for every $\Theta\subsetneq \Delta_X$, such that $e_{\Theta,J}$ coincides with $\exp_{\Theta,J}^*$ on $\cS(N'_\Theta)^J$. In particular, the submodule $\sum_{\Theta\subsetneq \Delta_X} e_\Theta \cS(X_\Theta)^J$ contains all functions $f\in \cS(X)^J$ that are supported in $\bigcup_{\Theta\neq \Delta_X} N_\Theta$ and it follows that the above quotient is also a quotient of the $\cH_J(G)$-submodule generated by $\cS(N^c)^J$ where $N^c$ denotes the complement of $\bigcup_{\Theta\neq \Delta_X} N_\Theta$ in $X(k)$. By the properness of the toroidal compactification $\overline{X}$, $N^c$ is compact modulo $Z(X)(k)$. Therefore, up to enlarging $N^c$ to make it $Z(X)(k)$-invariant, we see that $\cS(N^c)^J$ is finitely generated as a $Z(X)(k)$-representation. But this, together with the assumption $\mathbf{(Z-fin)_X}$, implies that the $\cH_J(G)$-module spanned by $\cS(N^c)^J$ is finitely generated. Thus, the quotient \eqref{eq1 proof theo f.g. Schwartz space}, and therefore also $\cS(X)^J$, are finitely generated over $\cH_J(G)$ too.
	
\end{paragr}

\subsection{Bernstein maps and scattering operators}\label{Sect Bernstein maps}

In this subsection, we continue to assume that $k$ is a non-Archimedean local field (of characteristic zero) and the group $G$ is split.

\vspace{2mm}

\begin{paragr}[Bernstein maps.]
The Bernstein maps are $L^2$ analogs of the smooth asymptotic maps and roughly correspond, in a global setting, to the integration of wave-packets of Eisenstein series on the unitary axis (as opposed to integration far from the unitary axis, leading to the so-called pseudo-Eisenstein series of \cite[\S II.1.10]{MWlivre} and which are the analogs of the smooth asymptotic maps). They have been constructed quite generally by Sakellaridis and Venkatesh in \cite{SV} however not as integrals of local analogs of Eisenstein series (which in the literature on symmetric varieties come under the name of normalized Eisenstein integrals) but rather as abstract maps between $L^2$ spaces characterized by their asymptotic properties. We summarize the main result of {\em op.\ cit.}\ in the following statement.
	
	\begin{theo}[Sakellaridis-Venkatesh]\label{theo Bernstein maps}
		Assume that all the boundary degenerations $X_\Theta$, $\Theta\subset \Delta_X$, of $X$ satisfy the condition $\mathbf{(Z-fin)}$ of \S \ref{S condition Zfin} as well as the Discrete Series Conjecture $\mathbf{(DSC')}$ of \S \ref{S DSC}. Then, for every $\Theta\subset \Delta_X$, there exists a unique equivariant continuous map
		$$\displaystyle \iota_\Theta: L^2(X_\Theta)\to L^2(X),$$\index{$\iota_\Theta$}
		called the \textbf{Bernstein map}, satisfying the following property: for any $\Theta$-exponential map $\exp_\Theta: N_\Theta\simeq N'_\Theta$, $f\in L^2(X_\Theta)$ and $a\in A_{X,\Theta}^{--}$ we have
		\begin{equation*}
			\displaystyle \lim\limits_{n\to \infty} \lVert \iota_\Theta(\mathcal{L}_{a^n}f)-\exp_{\Theta}^*(\mathcal{L}_{a^n}f)\rVert=0.
		\end{equation*}
		Moreover, these maps satisfy the following properties:
		\begin{enumerate}[(i)]
			\item For every subsets $\Omega\subset \Theta\subset \Delta_X$, denoting by $\iota_\Omega^{\Theta}: L^2(X_\Omega)\to L^2(X_\Theta)$ the Bernstein map with respect to $X_\Theta$, we have $\iota_\Theta\circ \iota_\Omega^\Theta=\iota_\Omega$;
			
			\item $$\displaystyle L^2(X)=\sum_{\Theta\subset \Delta_X} \iota_\Theta L^2_{\disc}(X_\Theta).$$
		\end{enumerate}
	\end{theo}
	
	\begin{rem}
		\begin{itemize}
			\item The Bernstein maps are not isometric in general and their failure to be so is in some sense accounted for by the scattering operators that we shall introduce in the next paragraph. Similarly, the decomposition in (ii) is not orthogonal.
			
			\item A direct consequence of (ii) is that the Plancherel measure for $L^2(X)$ is in the same class as the sum of the Plancherel measures of $L^2_{\disc}(X_\Theta)$ for $\Theta\subset \Delta_X$. In particular, the spectrum of $L^2(X)$ (i.e.\ the support of its Plancherel measure) is the union of the discrete spectra of the boundary degenerations $X_\Theta$. This is pleasantly aligned with Conjecture \ref{conj1 SV}, and in particular the natural partition of $\Phi_{\temp}(X)$ into sets of discrete $L$-parameters valued in (conjugacy classes of) standard Levi subgroups of ${}^L G_X$ (which are exactly the $L$-groups ${}^L G_{X_\Theta}$ of the boundary degenerations). This parallel also suggests that the subspaces $\iota_\Theta L^2_{\disc}(X_\Theta)$ and $\iota_\Omega L^2_{\disc}(X_\Omega)$ should be the same when the subsets $\Theta,\Omega\subset \Delta_X$ are conjugate under the Weyl group $W_X$. Under additional assumptions, this is also a consequence of the theory of scattering operators.
			
			\item In \cite{SV}, for the above theorem, Sakellaridis and Venkatesh actually assume that $X$ is wavefront (which automatically entails that its boundary degenerations satisfy $\mathbf{(Z-fin)}$ cf. \S \ref{S condition Zfin}). However, a closer look at the proof shows that all that is used are the following properties:
			\begin{itemize}
				\item The existence of smooth asymptotic maps $e_\Theta: \cS(X_\Theta)\to \cS(X)$ (which is guaranteed by Theorem \ref{theo smooth asymptotic});
				
				\item The spectral gap property $\mathbf{(SG)}$ of \S \ref{S DSC}, which as explained there is a consequence of the Discrete Series Conjecture;
				
				\item For every $\Theta\subset \Delta_X$, a uniform bound on the number of (unitary) $A_{X,\Theta}$-exponents for relative discrete series of $X_\Theta$ i.e.\ the existence of a constant $C>0$ such that for every $\pi\in \Irr(G)$, the number of $\chi\in \Unit(A_{X,\Theta})$ such that $\pi\hookrightarrow L^2(X_\Theta,\chi)$ is bounded by $C$. But this is also readily seen to be a consequence of the Discrete Series Conjecture.
			\end{itemize}
		\end{itemize}
	\end{rem}
\end{paragr}

\begin{paragr}[Scattering operators]
	
	Let, for $\Theta\subset \Delta_X$, $\iota_{\Theta,\disc}$\index{$\iota_{\Theta,\disc}$} denote the restriction of the Bernstein map $\iota_\Theta$ to the discrete subspace $L^2_{\disc}(X_\Theta)$ of $L^2(X_\Theta)$. By Theorem \ref{theo Bernstein maps} all that is missing to arrive at a Plancherel decomposition of $L^2(X)$ in terms of the discrete spectra of the boundary degenerations is a description of the kernel of the map $\sum_{\Theta} \iota_{\Theta,\disc}$. Such a description is provided by the theory of {\it scattering operators} which is summarized in the next theorem. This however requires so-far much more stringent assumptions on $X$ than before. Namely, not only do we have to asuume that $X$ is wavefront but also that it satisfies the condition of
	\begin{equation*}
		\displaystyle \mbox{generic injectivity on faces of } \cA_X^*/W_X\to \cA^*/W.
	\end{equation*} 
	This means that after restricting the natural map $\cA_X^*/W_X\to \cA^*/W$ to any facet of a Weyl chamber in $\cA_X^*$ it is injective on a subset with negligible complement (see \cite[Sect. 14.2]{SV} for a more detailed discussion of this condition). Although satisfied for symmetric varieties, it should be noted that it is a very strong assumption that fails for most non-wavefront spherical varieties (e.g. for $\GL_2\backslash \SO_5$).
	
	Before stating the theorem we need to introduce some notation. Namely, for every subsets $\Theta,\Omega\subset \Delta_X$ we let
	$$\displaystyle W_X(\Theta,\Omega)=\{w\in W_X\mid w\Theta=\Omega \}$$\index{$ W_X(\Theta,\Omega)$}
	be the set of elements in the Weyl group conjugating $\Theta$ to $\Omega$ (or equivalently, on the dual side, conjugating the standard Levi $\widehat{G}_{X_\Theta}$ of $\widehat{G}_X$ to $\widehat{G}_{X_\Omega}$). For $\Theta\subset \Delta_X$, we also set
	$$\displaystyle W_X(\Theta):=\bigcup_{\Omega\subset \Delta_X} W_X(\Theta,\Omega)$$\index{$W_X(\Theta)$}
	which is also the representatives of minimal length of the cosets $W_X/W_{X_\Theta}$. 
	
	Finally, a technical annoyance requires to introduce the images $A_{X,\Theta}'$\index{$A_{X,\Theta}'$} of the $k$-points of the center $Z(L_\Theta)^0$ by the natural morphism $Z(L_\Theta)^0\to A_{X,\Theta}=Z(X_\Theta)$. (Recall that the Levi subgroup was introduced in \S \ref{S wavefront case} for wavefront varieties and that $X_\Theta$ is parabolically induced from an $L_\Theta$-variety, hence the morphism $Z(L_\Theta)^0\to Z(X_\Theta)$.) Note that every $w\in W_X(\Theta,\Omega)$, seen for example as an element of the Weyl group $W$ of $G$ via the embedding of \S \ref{embedding WX}, conjugates $L_\Theta$ to $L_\Omega$ and therefore sends the subgroup $A'_{X,\Theta}\subset A_{X,\Theta}(k)$ onto $A'_{X,\Omega}$.
	
	\begin{theo}[Sakellaridis-Venkatesh]\label{theo scattering}
		Assume that $X$ is wavefront and satisfies the spectral gap property $\mathbf{(SG)}$ as well as the ``generic injectivity on faces of $\cA_X^*/W_X\to \cA^*/W$'' stated above. Then, there exist unitary $G$-equivariant isomorphisms
		$$\displaystyle S_w: L^2(X_\Theta)\to L^2(X_\Omega),\;\; \Theta,\Omega\subseteq \Delta_X, w\in W_X(\Theta,\Omega),$$\index{$S_w$}
		called \textbf{scattering operators} that are characterized by the following properties:
		\begin{enumerate}[(i)]
			\item For every $\Theta,\Omega\subseteq \Delta_X$ and $w\in W_X(\Theta,\Omega)$, $S_w$ is $A'_{X,\Theta}$-equivariant with respect to the isomorphism $w: A'_{X,\Theta}\simeq A'_{X,\Omega}$;
			
			\item For every $\Theta,\Omega, Z\subseteq \Delta_X$, $w\in W_X(\Theta,\Omega)$ and $w'\in W_X(\Omega,Z)$, we have
			$$\displaystyle S_{w'}S_w=S_{w'w}.$$

			\item The discrete Berstein maps $(\iota_{\Theta,\disc})_\Theta$ induce a $G$-equivariant isometric isomorphism
			$$\displaystyle \left(\bigoplus_{\Theta\subset \Delta_X} L^2_{\disc}(X_\Theta)\right)^{inv}\simeq L^2(X), (f_\Theta)_{\Theta}\mapsto \sum_{\Theta} \iota_\Theta(f_\Theta)$$
			where $\left(\bigoplus_{\Theta\subset \Delta_X} L^2_{\disc}(X_\Theta)\right)^{inv}$ stands for the subspace of tuples $\mathbf{f}=(f_\Theta)_{\Theta}\in \bigoplus_{\Theta\subset \Delta_X} L^2_{\disc}(X_\Theta)$ satisfying the symmetry relations
			$$\displaystyle S_wf_\Theta=f_\Omega, \mbox{ for every } \Theta,\Omega\subseteq \Delta_X, w\in W_X(\Theta,\Omega),$$
			and it is equipped with the Hilbert norm
			$$\displaystyle \lVert \mathbf{f}\rVert^2_{inv}:=\sum_{\Theta}\lvert W_X(\Theta)\rvert \cdot \lVert f_\Theta\rVert^2.$$
			
			\item For every $w\in W_X(\Theta,\Omega)$, we have $\iota_\Omega S_w=\iota_\Theta$. More generally, if $\Omega\subseteq \Omega'\subset \Delta_X$, then
			$$\displaystyle \iota_\Omega^{\Omega'}S_w=S_{w'}\iota_{\Theta}^{\Theta'}$$
			where $\Theta'\supset \Theta$ and $w'\in W_X(\Theta',\Omega')$ are characterized by $w\in W_{X_{\Omega'}}w'$.
		\end{enumerate}
		Moreover, the scattering operators have the following additional property for every subsets $\Theta,\Omega\subset \Delta_X$:
		\begin{enumerate}[(i)]
			\setcounter{enumi}{4}
			
			\item $$\displaystyle \iota_{\Omega,\disc}^*\iota_{\Theta,\disc}=\sum_{w\in W_X(\Theta,\Omega)} S_w\mid_{L^2_{\disc}(X_\Theta)}.$$
		\end{enumerate}
	\end{theo}
	
	\begin{rem}
		To make precise the unitary part of the statement, we need to fix in a coherent way the invariant measures on the various boundary degenerations. This can be done as in \cite[\S 4.2]{SV}, more precisely starting with an invariant measure on $X(k)$ we can deduce invariant measures on all the boundary degenerations $X_\Theta(k)$, $\Theta\subset \Delta_X$, characterized in particular by the fact that the exponential maps are eventually measure-preserving.
	\end{rem}
	
	Let us explain the unicity part of the above theorem. Namely, property (iv) entails that the image of the map
	$$\displaystyle (\iota_{\Theta,\disc}^*)_{\Theta\subset \Delta_X}: L^2(X)\to \bigoplus L^2_{\disc}(X_\Theta)$$
	lands in $(\bigoplus L^2_{\disc}(X_\Theta))^{inv}$. Combining this with (iii), this readily implies property (v) which in turn, together with property (i), characterizes the restriction $S_w\mid_{L^2_{\disc}(X_\Theta)}$ for $w\in W_X(\Theta,\Omega)$. Property (iii) also entails that $S_w \iota_Z^\Theta=\iota_{wZ}^\Omega S_w$ for every $Z\subset \Theta$, $w\in W_X(\Theta,\Omega)$. Since (by Theorem \ref{theo Bernstein maps}), $L^2(X_\Theta)=\sum_{Z\subset \Theta} \iota_Z^\Theta L^2_{\disc}(X_Z)$, this shows the desired unicity.
	
	It turns out that the scattering operators are related to standard intertwining operators. More precisely, let $\Theta,\Omega\subset \Delta_X$. The Plancherel decompositions of $L^2_{\disc}(X_\Theta)$ and $L^2_{\disc}(X_\Omega)$ for the actions respectively of $A_{X,\Theta}'$ and $A_{X,\Omega}'$ read
	$$\displaystyle L^2_{\disc}(X_\Theta)=\int^\oplus_{\Unit(A_{X,\Theta}')} L^2_{\disc}(X_\Theta,\chi) d\chi \mbox{ and } L^2_{\disc}(X_\Omega)=\int^\oplus_{\Unit(A_{X,\Omega}')} L^2_{\disc}(X_\Omega,\chi) d\chi$$
	Then, for each $w\in W_X(\Theta,\Omega)$, since the scattering operator $S_w$ is equivariant with respect to the isomorphism $w: A_{X,\Theta}'\simeq A_{X,\Omega}'$, it desintegrates into unitary equivariant maps
	$$\displaystyle S_{w,\chi}: L^2_{\disc}(X_\Theta,\chi)\to L^2_{\disc}(X_\Omega,w\chi)$$\index{$S_{w,\chi}$}
	well-defined for almost all $\chi\in \Unit(A_{X,\Theta}')$. Furthermore, for all $\chi$ inducing the isotypic decompositions of $L^2_{\disc}(X^L_\Theta,\chi)$ and $L^2_{\disc}(X^L_\Omega,w\chi)$ yields direct sums decompositions
	\begin{equation}\label{decomp L2disc}
		\displaystyle L^2_{\disc}(X_\Theta,\chi)\simeq \bigoplus_{\sigma} M_\Theta(\sigma)\otimes I_{P_\Theta^-}^G(\sigma) \mbox{ and } L^2_{\disc}(X_\Omega,w\chi)\simeq \bigoplus_{\tau} M_\Omega(\tau)\otimes I_{P_\Omega^-}^G(\tau)
	\end{equation}
	where $\sigma$ and $\tau$ runs over countable sets of irreducible representations of $L_\Theta(k)$ and $L_\Omega(k)$ with central characters $\chi$ and $w\chi$ respectively and $M_\Theta(\sigma)$, $M_\Omega(\tau)$ are some multiplicity spaces. Assume moreover that the Levi varieties $X^L_\Theta$, $X^L_\Omega$  satisfy the Discrete Series Conjecture $\mathbf{(DSC')}$ (see \S \ref{S DSC}), so that the $\sigma$ and $\tau$ appearing above belong to a countable number of orbits in $\Irr(L_\Theta)$, $\Irr(L_\Omega)$ for the actions by twisting of unramified characters. It can then be shown\footnote{see the proof of \cite[Corollary 15.3.5]{SV}}, that for almost all $\chi$ all induced representations appearing in the decompositions \eqref{decomp L2disc} are irreducible and that there is a nonzero intertwining map $I_{P_{\Theta}^-}^G(\sigma)\to I_{P_{\Omega}^-}^G(\tau)$ only if $\tau\simeq w\sigma$. Therefore, for almost all $\chi$, the desintegration of the scattering operator decomposes further into a sum of intertwinings
	$$\displaystyle S_{w,\sigma}: M_\Theta(\sigma)\otimes I_{P_\Theta^-}^G(\sigma)\to M_\Omega(w\sigma)\otimes I_{P_\Omega^-}^G(w\sigma).$$\index{$S_{w,\sigma}$}
	Comparing it with the standard intertwining operator (whenever it is regular) $M(\sigma,w): I_{P_\Theta^-}^G(\sigma)\to I_{P_\Omega^-}^G(w\sigma)$\index{$M(\sigma,w)$} we further obtain a factorization
	$$\displaystyle S_{w,\sigma}=C(w,\sigma)\otimes M(\sigma,w)$$
	with $C(w,\sigma)\in \Hom(M_\Theta(\sigma), M_\Omega(w\sigma))$. When we are moreover in a multiplicity one setting and we can find coherent trivializations of the line $M_\Theta(\sigma), M_\Omega(w\sigma)$, this yields a function $\sigma\mapsto C(w,\sigma)\in \bC$ which, experience shows, is often given by some local $\gamma$-factors. E.g. in the case of the Whittaker model (to which most of the previous discussion, including Theorem \ref{theo scattering} applies modulo slight modifications), we recover the so-called local constants of the Langlands-Shahidi method \cite{ShaFE}. Other examples in rank one have been computed in \cite[Section 3]{SaktransOpI}. For symmetric spaces, the coefficients $C(w,\sigma)$ are closely related to functional equations of so-called local intertwining periods; see \cite[Theorem 12.4(4)]{FLO} and \cite{Matringegammafact} for a computation of those for some Galois symmetric varieties.
\end{paragr}

\begin{paragr}[The group case.]
	Finally, let us explain the relations of Theorems \ref{theo Bernstein maps} and \ref{theo scattering} to Harish-Chandra Plancherel formula in the group case $X=H$, $G=H\times H$ \cite{WaldPlanch}. For a similar discussion we also refer the reader to \cite[\S 15.7]{SV}. Recall that in this case we have $\Delta_X=\Delta_H$, the parabolic $P_\Theta$ is a product $P_\Theta=Q_\Theta\times Q_\Theta^-$ where $Q_\Theta\subset H$ is the standard parabolic subgroup associated to $\Theta$, $Q_\Theta^-$ the opposite antistandard parabolic, and we have an identification
	$$\displaystyle X_\Theta=M_\Theta\times^{P_\Theta^-} G=(N_\Theta^-\times M_\Theta^{diag}\times N_\Theta)\backslash H\times H$$
	where $M_\Theta=Q_\Theta\cap Q_\Theta^-$, $Q_\Theta=M_\Theta N_\Theta$, $Q_\Theta^-=M_\Theta N_\Theta^-$.
	
	Let us fix Haar measures on $N_\Theta(k)$, $M_\Theta(k)$, $N_\Theta^-(k)$, $G(k)$ such that the product map $N_\Theta(k)\times M_\Theta(k)\times N_\Theta^-(k)\to G(k)$ is measure preserving.
	
	Fixing furthermore a Haar measure on the Pontryagin dual $\Unit(A_{X,\Theta})$, the Plancherel decomposition for $L^2_{\disc}(M_\Theta)$ reads
	\begin{equation}\label{Planch Mtheta}
		\displaystyle L^2_{\disc}(M_\Theta)\simeq \int^{\oplus}_{\Pi_2(M_\Theta)} \sigma\widehat{\otimes} \sigma^\vee d\sigma
	\end{equation}
	where $d\sigma$ is the unique measure such that the central character map $\Pi_2(M_\Theta)\to \Unit(A_{X,\Theta})$, $\sigma\mapsto \omega_\sigma$, is locally measure preserving,  the backwards isomorphism is given by the formation of matrix coefficients:
	$$\displaystyle \int_{\Pi_2(M_\Theta)} f_\sigma\otimes f_\sigma^\vee d\sigma\mapsto \left( m\mapsto \int_{\Pi_2(M_\Theta)} \langle \sigma(m)f_\sigma,f_\sigma^\vee\rangle d\sigma\right).$$
	and $\sigma\otimes \sigma^\vee\simeq HS(\sigma)$ (the space of Hilbert Schmidt operators on $\sigma$) is equipped with the unitary structure given by the Hilbert-Schmidt norm $\lVert .\rVert^2_{\mathrm{HS}}$ times the formal degree $\deg(\sigma)\in \bR_{>0}$. (This normalization is necessary for the isomorphism \eqref{Planch Mtheta} to be unitary.)
	
	Inducting from $P_\Theta^-$ to $G$, we get an isomorphism
	\begin{equation}
		\displaystyle L^2(X_\Theta)\simeq \int_{\Pi_2(M_\Theta)} I_{Q^-_\Theta}(\sigma)\widehat{\otimes} I_{Q_\Theta}(\sigma^\vee) d\sigma.
	\end{equation}
	Then, on the dense subspace $\cS(H)\subset L^2(X)$, the adjoint of the discrete Bernstein map $\iota_{\Theta,\disc}^*$ is given by
	$$\displaystyle \iota_{\Theta,\disc}^*f=\int_{\Pi_2(M_\Theta)} (M_{Q_\Theta\mid Q_\Theta^-}(\sigma)^{-1}\otimes I)I_{Q_\Theta}(\sigma,f) d\sigma$$
	where $I_{Q_\Theta}(\sigma,f)$ is the operator associated to $f$ on the induced representation $I_{Q_\Theta}(\sigma)$, a finite rank operator that we identify with an element of $I_{Q_\Theta}(\sigma)\otimes I_{Q_\Theta}(\sigma^\vee)$ via an isomorphism $I_{Q_\Theta}(\sigma^\vee)\simeq I_{Q_\Theta}(\sigma)^\vee$ (given by the choice of Haar measure on $Q_\Theta(k)$), and
	$$\displaystyle M_{Q_\Theta\mid Q_\Theta^-}(\sigma): I_{Q^-_\Theta}(\sigma)\to I_{Q_\Theta}(\sigma)$$\index{$M_{Q_\Theta\mid Q_\Theta^-}(\sigma)$}
	stands for the standard intertwining operator given by (the meromorphic continuation of) $e\mapsto \int_{N_\Theta(k)} e(ug)du$.
	
	Note that the norm of $(M_{Q_\Theta\mid Q_\Theta^-}^{-1}\otimes I)I_{Q_\Theta}(\sigma,f)$ for the unitary structure on the induced representation $I_{Q^-_\Theta}(\sigma)\otimes I_{Q_\Theta}(\sigma^\vee)$ is given by
	$$\displaystyle \lVert (M_{Q_\Theta\mid Q_\Theta^-}(\sigma)^{-1}\otimes I)I_{Q_\Theta}(\sigma,f)\rVert^2=\mu_H(\sigma) \lVert I_{Q_\Theta}(\sigma,f)\rVert_{H-S}^2$$
	where $\lVert . \rVert_{H-S}$ denotes the Hilbert-Schmidt norm and
	$$\displaystyle \mu_H(\sigma):=(M_{Q^-_\Theta\mid Q_\Theta}(\sigma)M_{Q_\Theta\mid Q_\Theta^-}(\sigma))^{-1} \deg(\sigma)$$
	is the so-called Harish-Chandra Plancherel density. Here, the composition $M_{Q^-_\Theta\mid Q_\Theta}(\sigma)M_{Q_\Theta\mid Q_\Theta^-}(\sigma)$ appears by the adjunction property ${}^tM_{Q_\Theta\mid Q_\Theta^-}(\sigma)=M_{Q^-_\Theta\mid Q_\Theta}(\sigma^\vee)$ and is identified with a number as it is generically (at the holomorphy points of the intertwining operators) a scalar operator, whereas $\deg(\sigma)>0$ stands, as above, for the formal degree of $\sigma$.
	
	The Plancherel formula of Theorem \ref{theo scattering}(ii) then reads
	$$\displaystyle \lVert f\rVert_{L^2}^2=\sum_{\Theta} \frac{1}{\lvert W_H(\Theta)\rvert}\int_{\Pi_2(M_\Theta)} \lVert I_{Q_\Theta}(\sigma,f)\rVert_{H-S}^2 \mu_H(\sigma) d\sigma,\mbox{ for } f\in \cS(H),$$
	which is nothing but the Plancherel formula of Harish-Chandra \cite{WaldPlanch}.
\end{paragr}

\subsection{The most continuous part of the spectrum}\label{Sect most cont spectrum}

In this section, we continue to assume that $k$ is non-Archimedean and $G$ is split.

\vspace{2mm}

\begin{paragr}[A relative Satake isomorphism.]
Let $B=AN$ be a Borel subgroup and set $K=G(\mathfrak{o})$. We are interested in the $K$-unramified spectrum of $X$ i.e.\ in $\cS(X)^K$ as a module over the unramified Hecke algebra $\cH(G,K)$. Let us recall the classical Satake isomorphism \cite{CartierCorvallis}
	\begin{equation*}
		\displaystyle \cH(G,K)\simeq \bC[\widehat{A}]^W
	\end{equation*}
	where we recall that $\widehat{A}$ is the dual complex torus to $A$ and $\bC[\widehat{A}]$\index{$\bC[\widehat{A}]$} is the ring of regular functions on the latter. In order to have a nicer formulation of the main results, we further assume that
	\begin{equation*}
		\mbox{ The } k\mbox{-points } X_B(k) \mbox{ of the open } B\mbox{-orbit contains a unique } B(k)\mbox{-orbit.}
	\end{equation*}
	Indeed, this condition is equivalent to the injectivity of the dual morphism $\widehat{A}_X\to \widehat{A}$ (or equivalently of $\iota_X: \widehat{G}_X\to \widehat{G}$) so that lifts of unramified $L$-parameters $\phi: W_k/I_k\to \widehat{G}$ to $\widehat{G}_X$ are automatically also unramified. Therefore, according to the spirit of Conjecture \ref{conj1 SV}, only under this assumption is $\cS(X)^K$ really capturing ``the unramified spectrum of $X$''. 
	
	As a side remark, we note that the above condition implies that $X$ has no type $N$-roots (see \S \ref{S spherical roots}). Indeed, it implies that the cokernel of the morphism between character lattice $X^*(A_X)\to X^*(A)$ doesn't have torsion, hence $\Delta_X\subset X^*(A_X)$ (as in any event we have $\Delta_X\subset X^*(A_X)\otimes \bQ\cap X^*(A)$).
	
	Let $\delta^{1/2}_{(X)}$\index{$\delta^{1/2}_{(X)}$} be the square root of the modular character of the parabolic $P_X$. It is an unramified character of $A(k)$ and as such can be seen as an element of $\widehat{A}$. Let $\cH_X$\index{$\cH_X$} be the image of the unramified Hecke algebra by the Satake isomorphism followed by the restriction map
	$$\displaystyle \bC[\widehat{A}]^W\to \bC[\delta^{1/2}_{(X)}\widehat{A}_X]$$
	and $\mathcal{K}_X$\index{$\mathcal{K}_X$} be its fraction field. Actually, under the assumption that $X$ is unimodular we have $\delta^{1/2}_{(X)}\in \widehat{A}_X$ so that $\delta^{1/2}_{(X)}\widehat{A}_X=\widehat{A}_X$. However, the theorem below doesn't require this assumption. Note that the Weyl group $W_X$ acts on $\delta^{1/2}_{(X)}\widehat{A}_X$ and we have the containment $\cH_X\subset \bC[\delta^{1/2}_{(X)}\widehat{A}_X]^{W_X}$.

	\begin{theo}[Sakellaridis, \cite{Sakunr} \cite{SakSph}]\label{theo relative Satake}
		\begin{enumerate}[(i)]
			\item The $\cH(G,K)$-module $\cS(X)^K$ is finitely generated, torsion free and supported on $\cH_X$ (i.e.\ the action factors through $\cH(G,K)\to \cH_X$).  Moreover, there exists an isomorphism
			\begin{equation}\label{iso Satake}
				\displaystyle \cS(X)^K\otimes_{\cH_X} \mathcal{K}_X\simeq \bC(\delta^{1/2}_{(X)}\widehat{A}_X)^{W_X}
			\end{equation}
			that is canonical up to multiplication by an element of $\bC(\delta^{1/2}_{(X)}\widehat{A}_X)^{W_X}$.
			
			\item Assume furthermore that the $G$-scheme $X$ is defined over $\mathfrak{o}$, is affine, satisfies the axioms stated in \cite[Section 2]{SakSph} \footnote{If $(G,X)$ is defined over a number field, these axioms are always satisfied at all except finitely many places.} as well as the two technical conditions of \cite[Theorem 7.2.1]{SakSph}\footnote{These conditions concern the combinatorics of so-called colors of the spherical variety $X$ and are conjectured to be satisfied under the affineness condition.}. Then, the isomorphism in (i) can be chosen to restrict to an isomorphism
			\begin{equation*}
				\displaystyle \cS(X)^K\simeq \bC[\delta^{1/2}_{(X)}\widehat{A}_X]^{W_X}
			\end{equation*}
			sending $\mathbf{1}_{X(\mathfrak{o})}\in \cS(X)^K$ to $1$. (It then becomes canonical.)
		\end{enumerate}
		
	\end{theo}
	
	\begin{rem}
		More precisely, the construction of the isomorphism \eqref{iso Satake} given in \cite[Section 6.2]{Sakunr} depends on the choice of a nonzero function $f_0\in C_c^\infty(X)^K$. In the situation of point (ii) of the theorem this can be fixed by choosing $f_0=\mathbf{1}_{X(\mathfrak{o})}$.
	\end{rem}
\end{paragr}

\begin{paragr}[The unramified Plancherel formula.]\label{S unramified Plancherel}
	In this paragraph, we explain how to describe explicitly the smooth asymptotic map as well as the Bernstein map for the most continuous part of the spectrum (corresponding to $\Theta=\emptyset$) following \cite[Section 15]{SV}. As an application, we get an explicit Plancherel decomposition for the unramified subspace $L^2(X)^K$ whose understanding, as we will see in Section \ref{Sect global conjecture} is crucially related to global conjectures in the relative Langlands program.
	
	Let $X^h$\index{$X^h$} be the space of {\em generic horocycles} on $X$ \cite[\S 2.8]{SV}. More precisely, it is the variety classifying pairs $(Q,\mathcal{O})$ (or just $\mathcal{O}$ for short) where $Q$ is a parabolic subgroup in the same class as $P_X$ and $\mathcal{O}$ is an orbit under the unipotent radical $N_Q$ of $Q$ in the unique open $Q$-orbit $X_{Q}\subset X$. All of these orbits are $G$-conjugates so that fixing one $N_X$-orbit $\mathcal{O}_0\subset X_{P_X}=X_B$ we get an identification
	\begin{equation}\label{iso horocycles}
		\displaystyle X^h\approx A_X\times^{P_X} G
	\end{equation}
	where $P_X$ acts transitively on $A_X$ via the quotient $P_X\twoheadrightarrow L_X\twoheadrightarrow A_X$. We equip $X^h(k)$ with the $G(k)$-equivariant $\bC$-local system of rank one $\mathcal{L}$ whose fiber $\mathcal{L}_{\mathcal{O}}$ over an horocycle $\mathcal{O}$ is the space of invariant complex measures on $\mathcal{O}$ (or equivalently the space of complex Haar measures on $N_Q(k)$ where $Q$ is the underlying parabolic). Integration along horocycles yields a $G$-morphism
	$$\displaystyle \mathcal{R}: \cS(X)\to C^\infty(X^h,\mathcal{L}^\vee),$$\index{$\mathcal{R}$}
	called the {\it Radon transform}, whose target is the space of smooth (with respect to the $G(k)$-action) sections of the dual local system $\mathcal{L}^\vee$. More explicitly, the image of $\varphi\in \cS(X)$ by the Radon transform is the section sending $\mathcal{O}$ to the linear form
	$$\displaystyle \mu\in \mathcal{L}_{\mathcal{O}}\mapsto \int_{\mathcal{O}} \varphi \mu.$$
	(Note that, since $X$ is quasi-affine, horocycles are automatically closed in it, by \cite[Exercise 8, p.115]{Humphreys}, so that the above integral makes sense.)
	Actually, under our assumption that the variety $X$ is unimodular, the $G(k)$-equivariant local system $\mathcal{L}$ can be trivialized i.e.\ we can choose invariant measures on all horocycles in a coherent way, namely by fixing a measure $\mu$ on one horocycle $\mathcal{O}\subset X$ and equipping any other horocycle $\mathcal{O}'$ with the translated measure by any $g\in G$ such that $\mathcal{O}'=\mathcal{O}g$ (since $X$ is unimodular the resulting measure on $\mathcal{O}'$ does not depend on the choice of $g$). In what follows, for the sake of notational simplicity, we will fix such a trivialization and suppress the local system $\mathcal{L}$.
	
	We can make a similar construction for the asymptotic cone $X_\emptyset$ whose space of horocycles will be denoted by $X_\emptyset^h$\index{$X_\emptyset^h$}. The isomorphism \eqref{eq3} then shows that there is a canonical isomorphism
	$$\displaystyle X^h\simeq X_\emptyset^h.$$
	More concretely, this isomorphism sends an horocycle $\mathcal{O}\subset X$ to the intersection of the closure of $\mathcal{O}\times A_X$ in the affine degeneration $\mathfrak{X}^{\aff}$ with $X_\emptyset$. Moreover, our trivialization of the bundle $\mathcal{L}$ automatically fixes invariant measures on all the horocycles in $X_\emptyset$. (Indeed, as can readily be seen, two horocycles corresponding by the above isomorphism are orbits under the unipotent radical $N_Q$ of the same parabolic subgroup $Q$ and fixing a measure on any of them amounts to choose a Haar measure on $N_Q(k)$.)
	
	Introducing the Radon transform $\mathcal{R}_\emptyset$\index{$\mathcal{R}_\emptyset$} as above we thus have a diagram
	$$\displaystyle \xymatrix{ \cS(X_\emptyset) \ar[rd]^{\mathcal{R_\emptyset}} & & \cS(X) \ar[ld]^{\mathcal{R}} \\ & C^\infty(X_\emptyset^h)\simeq C^\infty(X^h) & }.$$
	The adjoint smooth asymptotic map $e_\emptyset^*$ and $\iota_\emptyset^*$ are then essentially obtained by inverting the left arrow of the above diagram. More specifically, this inversion can be done spectrally almost everywhere as follows. Note that there is a natural action of $A_X$ on $X^h\simeq X_\emptyset^h$ since translation by $L_Q$ of any $Q$-horocycle factors through the quotient $L_Q\twoheadrightarrow A_X$. Since it also commutes with the $G$-action, we will make this $A_X$-action a left action. The Radon transform $\mathcal{R}_\emptyset$ is then $A_X(k)$-equivariant if we normalize the $A_X(k)$-action on $C^\infty(X^h)$ by
	$$\displaystyle (\mathcal{L}_a\varphi)(.)=\delta_{(X)}(a)^{1/2}\varphi(a^{-1}.),\;\; \varphi\in C^\infty(X^h), a\in A_X(k)$$
	where $\delta_{(X)}$ stands as before for the modular character of $P_X(k)$. Thus, this normalization is the inverse of the normalization of the $A_X(k)$-action on $C^\infty(X_\emptyset)$ defined in \S \ref{S normalization action center} (for $X_\emptyset$ it is easy to see that $\eta=\delta_{(X)}$). This stems from the fact that the (left) $A_X(k)$-action on horocycles of $X_\emptyset$ doesn't respect our choice of measures and multiplies them by $\delta_{(X)}(a)$.
	
	For every character $\chi: A_X(k)\to \bC^\times$, we define the $\chi$-Radon transform for $X_\emptyset$ by
	$$\displaystyle (\mathcal{R}_{\emptyset,\chi}\varphi)(y):=\int_{A_X(k)} \delta_{(X)}(a)^{1/2}(\mathcal{R}_\emptyset\varphi)(a^{-1}y) \chi(a) da,\;\; \varphi\in \cS(X_\emptyset), y\in X^h(k).$$\index{$\mathcal{R}_{\emptyset,\chi}$}
	The above integral can be shown to converge absolutely whenever $\Re(\chi)$ is in some translate of the negative Weyl chamber $(\cA_X^*)^-$ and to admit a rational extension to all $\chi$. Actually, this boils down to well-known results on standard intertwining operators. Indeed, the map $\mathcal{R}_{\emptyset,\chi}$ descends to an intertwining map
	$$\displaystyle \mathcal{R}_{\emptyset,\chi}: \cS(X_\emptyset, \chi)\to \cS(X^h,\chi)$$
	where $\cS(X_\emptyset, \chi)$ and $\cS(X^h,\chi)$ denote the spaces of smooth functions $\varphi: X_\emptyset(k)\to \bC$ (respectively $\varphi: X^h(k)\to \bC$) satisfying $\varphi(ax)=\delta_{(X)}(a)^{-1/2}\chi(a)\varphi(x)$ (resp. $\varphi(ax)=\delta_{(X)}(a)^{1/2}\chi(a)\varphi(x)$) for all $a\in A_X(k)$. Using the isomorphisms \eqref{iso horocycles} and $X_\emptyset\simeq A_X\times^{P_X^-} G$ (which is a particular case of \eqref{eq boundary deg parabind}), $\cS(X_\emptyset, \chi)$ and $\cS(X^h,\chi)$ can be (non canonically) identified with the normalized parabolic inductions $I_{P_X^-}^{G}(\chi)$ and $I_{P_X}^{G}(\chi)$ respectively through which $\mathcal{R}_{\emptyset,\chi}$ is simply a nonzero scalar multiple of the standard intertwining operator $M_{P_X^-\mid P_X}: I_{P_X^-}^{G}(\chi)\to I_{P_X}^{G}(\chi)$. In particular, $\mathcal{R}_{\emptyset,\chi}$ yields an isomorphism $\cS(X_\emptyset, \chi)\simeq \cS(X^h,\chi)$ for almost all $\chi$.
	
	Similarly, the $\chi$-Radon transform $\mathcal{R}_\chi$ for $X$ is defined by\index{$\mathcal{R}_{\chi}$}
	$$\displaystyle (\mathcal{R}_{\chi}\varphi)(y):=\int_{A_X(k)} \delta_{(X)}(a)^{1/2}(\mathcal{R}\varphi)(a^{-1}y) \chi(a) da,\;\; \varphi\in \cS(X), y\in X^h(k).$$
	It converges for $\Re(\chi)$ in a translate of $(\cA_X^*)^-$, extends rationally to all $\chi$ (see \cite[Proposition 4.5.1]{Sakunr}) and, whenever it is regular, defines a $G$-equivariant morphism
	$$\displaystyle \mathcal{R}_\chi: \cS(X)\to \cS(X^h,\chi).$$
	
	For all $\chi$ for which it makes sense we then define the intertwining map\index{$e^*_{\emptyset,\chi}$}
	$$\displaystyle e^*_{\emptyset,\chi}:= \mathcal{R}_{\emptyset,\chi}^{-1}\circ \mathcal{R}_\chi:\cS(X)\to \cS(X_\emptyset,\chi).$$
	This furnishes explicit decompositions of $e_\emptyset$ and $\iota_\emptyset$ by the following theorem \cite[Theorem 15.4.2, Theorem 15.6.1]{SV}.
	
	\begin{theo}[Sakellaridis-Venkatesh]
		Under additional assumptions on $X$, that are for example satisfied if $X$ is a symmetric variety, we have the following disintegration of the adjoint asymptotic map $e_{\emptyset}^*$ and Bernstein map $\iota_\emptyset^*$: for every $f\in \cS(X)$ and $x\in X_\emptyset(k)$,
		\begin{equation*}
			\displaystyle (e_\emptyset^*f)(x)=\int_{\Re(\chi)=\lambda_0} (e^*_{\emptyset,\chi}f)(x) d\chi
		\end{equation*}
		and
		\begin{equation*}
			\displaystyle (\iota_\emptyset^*f)(x)=\int_{\Unit(A_X)} (e^*_{\emptyset,\chi}f)(x) d\chi
		\end{equation*}
		where $\lambda_0$ is a point of $\cA_X^*$ that is sufficiently far in the negative Weyl chamber $(\cA_X^*)^-$, the subscript ``$\Re(\chi)=\lambda_0$'' indicates that we integrate over the (torsor under $\Unit(A_X)$) subset of characters $\chi: A_X(k)\to \bC^\times$ whose real part is given by $\chi$ and the measures are (translated from) the measure on $\Unit(A_X)$ dual to the measure we have fixed on $A_X(k)$.
	\end{theo}
	
	Let us denote by $L^2(X)_{\emptyset}$ the image in $L^2(X)$ of the Bernstein map $\iota_\emptyset$. Then, by Theorem \ref{theo scattering}, the above theorem can be seen as giving an explicit Plancherel decomposition for $L^2(X)_{\emptyset}$. Namely, the Plancherel formula for $L^2(X)_{\emptyset}$ reads as follows:
	$$\displaystyle \lVert f\rVert^2=\int_{\Unit(A_X)/W_X} J_\chi(f,f) d\chi$$
	where the relative character $J_\chi$ is given by $J_\chi(f,f)=$ the $L^2$-norm of the function $e_{\emptyset,\chi}^* f\in L^2(X_\emptyset,\chi)$.
	
	By \cite[Section 9]{SakSph}, under the same assumptions as in Theorem \ref{theo relative Satake} (ii), we have $L^2(X)^K=L^2(X)_{\emptyset}^K$ and therefore the above specializes to a Plancherel formula for $L^2(X)^K$:
	\begin{equation*}
		\displaystyle \lVert f\rVert^2=\int_{\widehat{A}_X^1/W_X} J_\chi(f,f) d\chi
	\end{equation*}
	where $\widehat{A}_X^1$ is the maximal compact subgroup of the dual torus $\widehat{A}_X$ identified in the usual way with the unitary unramified characters of $A_X(k)$.
\end{paragr}

\subsection{The local trace formula approach}\label{Sect local trace formula}

The local trace formula is an important tool to study the spectrum of local spherical varieties. For simplicity, we shall assume in this section that $X(k)$ has a unique $G(k)$-orbit i.e.\ $X(k)=H(k)\backslash G(k)$ and that $k$ is non-Archimedean (the Archimedean case requires, as often, more care on certain analytic aspects) although these assumptions can be relaxed in most of the statements.

\vspace{2mm}

\begin{paragr}[Kernels.]
The basic idea is to try to compute the character of the unitary representation $L^2(X)$. More precisely, for every $f\in \cS(G)$, we aim to compute the trace of the operator of right convolution $R(f): L^2(X)\to L^2(X)$. This operator is readily seen to be given by a kernel function that is:\index{$K_f$}
$$\displaystyle (R(f)\varphi)(x)=\int_{X(k)} K_f(x,y)\varphi(y) dy,\;\;\; \varphi\in L^2(X),$$
where the ``kernel'' $K_f$ is a function $X(k)\times X(k)\to \bC$. Under our assumption that $X(k)=H(k)\backslash G(k)$, it is given by the explicit formula
	$$\displaystyle K_f(x,y)=\int_{H(k)} f(x^{-1}hy)dh.$$
	
	Thinking formally of $K_f$ as some ``continuous matrix'' representing the operator $R(f)$, the trace should be given by
	$$\displaystyle \Tr(R(f))=\int_{X(k)} K_f(x,x)dx.$$
	However, as soon as $X$ has some continuous spectrum this does not make sense: the operator $R(f)$ has no well-defined trace and the above integral usually diverges.
	
	As a side remark, when $X$ is tempered (see \S \ref{S tempered varieties}), we can define more generally an operator $R(f)$ for $f$ in the Harish-Chandra Schwartz space $\cC(G)$ and the above discussion applies verbatim.
\end{paragr}

\begin{paragr}[Arthur local trace formula.]\label{S Arthur LTF}
	There are at least two ways to circumvent this original difficulty. One of them was succesfully implemented by Arthur \cite{ArtLTF} in the group case $X=H$, $G=H\times H$. Simply put, it consists in truncating the previous divergent integral and studying the behavior of the resulting expression when some truncation parameter (usually called $T$) goes to infinity. Let us us discuss very briefly the mechanism developed by Arthur in this case, assuming for simplicity that the group is split and referring the reader to {\it op.\ cit.}\ for details.
	
	Since $G=H\times H$, by linearity we may assume that $f=f_1\otimes f_2$ where $f_1,f_2\in \cS(H)$. The kernel is then given by
	$$\displaystyle K_f(x,y)=\int_{H(k)} f_1(x^{-1}hy) f_2(h)dh,\;\; x,y\in H(k).$$
	Arthur's truncation starts with the choice of a Cartan decomposition $H(k)=KA^+K$, where $A^+=\{a\in A(k)\mid \lvert \alpha(a)\rvert\geq 1 \forall \alpha\in \Delta_0 \}$ is the positive Weyl chamber in $A(k)$ and $K$ is a suitable maximal compact subgroup\footnote{More precisely, the stabilizer of a special point in the apartment associated to $A$ in the (extended) Bruhat-Tits bulding of $H$.}. For $T\in \cA^+$, let $u_T$\index{$u_T$} be the characteristic function of those elements $h\in H(k)$ that can be written as $h=k_1ak_2$ where $k_1,k_2\in K$ and $a\in A^+$ satisfies the inequalities $\lvert \varpi_\alpha(a)\rvert \leq \langle \varpi_\alpha, T\rangle$ for every $\alpha\in \Delta_0$, where $\{ \varpi_\alpha\mid \alpha\in \Delta_0\}$ denotes the set of fundamental weights. This set is compact modulo the center and the expression\index{$I^{H,T}$}
	$$\displaystyle I^{H,T}(f)=\int_{A_H(k)\backslash H(k)} u_T(h) K_f(h,h)dh$$
	is therefore well-defined.
	
	The main results of \cite{ArtLTF} can then be summarized very roughly as follows. First, for a fixed function $f\in \cS(G)$, the function $T\mapsto I^{H,T}(f)$ is asymptotic to a polynomial as $T\to \infty$ by staying ``sufficiently far from the walls'' of $\cA^+$\footnote{More precisely, this means as $T\to \infty$ while staying in the domain $\{ T\in \cA^+\mid \lVert T\rVert\geq \epsilon \alpha(T)\; \forall \alpha\in \Delta_0\}$ for any given $\epsilon>0$.} whose constant terms we denote by $I^H(f)$\index{$I^H$}. Secondly, and most importantly, Arthur provides two different expansions for the distribution $I^H$. One of them is ``geometric'' in nature and involves generalizations of orbital integrals known as {\it weighted orbital integrals} that are distributions of the form\index{$WO_M(\gamma,f)$}
	\begin{equation}\label{def WOI}
		\displaystyle f\mapsto WO_M(\gamma,f)=\int_{H_\gamma(k)\backslash H(k)} f(g^{-1}\gamma g) v_M(g)dg
	\end{equation}
	where $\gamma\in H(k)$ is a regular semisimple conjugacy class, $H_\gamma$ stands for the centralizer of $\gamma$, $M$ is a Levi subgroup containing $\gamma$ and the ``weight'' $v_M$ is a certain function on $H_\gamma(k)\backslash H(k)$ (that also depends on the choice of the Haar measure on $H_\gamma(k)$ so that the above expression only depends on the chosen measure on $H(k)$). The other expansion is ``spectral'' and features generalizations of characters known as {\it weighted characters} which are distributions of the form\index{$W\Theta_M(\pi,f)$}
	\begin{equation}\label{def WC}
		\displaystyle f\mapsto W\Theta_M(\pi,f)= \Tr(\mathcal{R}_M(\pi)\pi(f))
	\end{equation}
	where $\pi\in \Temp(G)$ is an irreducible tempered representation, $M$ is again a Levi subgroup and the ``weight'' $\mathcal{R}_M(\pi)$ is this time a certain operator on (the space of) $\pi$ that actually depends on the auxilliary choice of a representation $\pi^M\in \Temp(M)$ such that $\pi\simeq I_M^H(\pi^M)$ (in particular the weighted character $W\Theta_M(\pi,.)$ can only be defined for a Levi subgroup $M$ from which it is induced).
	
	Needless to say the discussion so far has been on a very rough and approximative level, in great part due to our very vague description of weighted orbital integrals and characters, and we will not say much more, refering the interested reader to \cite{ArtWOI} and \cite{ArtIntertwining} for precise definitions (which involve, among other things, the theory of $(G,M)$-families as introduced and developed by Arthur). Let us just mention that the distributions \eqref{def WOI} and \eqref{def WC}, although not made transparent in the notation, are actually also dependent on the choice of the maximal compact subgroup $K$ as well as on the choice of a normalization of the standard intertwining operators $M_{P\mid Q}(\pi_M): I_P^H(\pi_M)\to I_Q^H(\pi_M)$ (in the case of the weighted characters).
\end{paragr}

\begin{paragr}[Strongly cuspidal functions.]\label{S strongly cuspidal}
	Another strategy, going back at least to Kazhdan in a global setting, is to restrict oneself to particular test functions $f$ for which the kernel function $K_f$ is integrable over the diagonal. We then hope to expand this integral as in the previous case in two different ways: a ``geometric'' expression featuring distributions of geometric origin (orbital integrals and the like) and a ``spectral'' expansion involving terms closely related to characters of irreducible representations. The resulting identity is then usually called a ``simple'' trace formula. 
	
	Originally, these simple trace formulas have been mostly used in a global setting by imposing on the test function $f$ two different conditions having the effect to drastically simplify both the geometric and spectral expansions (as well as their proofs) respectively; hence the word ``simple''. More precisely, these two conditions restrict the possible support of $f$ both geometrically and spectrally (that is of its Fourier transform $\pi\to \pi(f)$) and are typically incompatible locally (as a form of the Heisenberg uncertainty principle). Thus, only in a global situation, where the test functions are typically products $f=\prod_v f_v$ of local test functions, can we really impose these restrictions in a meaningful way (at two different places $v_0$, $v_1$).
	
	In his work on the local Gross-Prasad conjecture \cite{WaldGP1}, Waldspurger was able to develop such a simple local trace formula, that is however not as simple as in the above sense anymore. The particular test functions he used are named {\it strongly cuspidal functions} and are characterized by the following property: for every proper parabolic subgroup $P=MU$ we have
	$$\displaystyle \int_{U(k)} f(mu)du=0,\;\; m\in M(k).$$
	One reason we could expect a simple trace formula for strongly cuspidal functions is the following result.
	
	\begin{lem}\label{lem wavefront kernel}
		Assume that $G$ is split, $k$ is non-Archimedean and $X$ is wavefront. Then, for every strongly cuspidal $f\in \cS(G)$, the function
		$$\displaystyle x\in X(k)\mapsto K_f(x,x)$$
		is compactly supported modulo $Z(X)^0(k)$.
	\end{lem}
	\begin{proof}
		Let $\Theta\subset \Delta_X$. Since $G$ is split, we can consider the boundary degeneration $X_\Theta$ (see Section \ref{sect boundary degenerations}). Let $\exp_\Theta: N_\Theta\to N'_\Theta$ be a $\Theta$-exponential map (see \S \ref{S exponential maps}) and let $K_f^\Theta$ be the kernel of $f$ acting on $L^2(X_\Theta)$. Then, as a consequence of the eventual equivariance of exponential maps (Proposition \ref{prop exponential maps} (iv)), up to shrinking the neighborhood $N_\Theta$ of $\infty_\Theta$, we have
		$$\displaystyle K_f(x,y)=K_f^\Theta(\exp_\Theta(x),\exp_\Theta(y)),\;\; \mbox{ for every } x,y\in N_\Theta.$$
		Moreover, since $X$ is wavefront and $\Theta\neq \Delta_X$, the variety $X_\Theta$ is induced from a proper parabolic subgroup $P_\Theta^-=L_\Theta U_\Theta$ i.e.\ there exists a subgroup $H^L_\Theta\subset L_\Theta$ such that all the isotropy groups of $X_\Theta$ are conjugated to $H_\Theta=H^L_\Theta\ltimes U_\Theta$ (see \S \ref{S wavefront case}). In particular, as $f$ is strongly cuspidal and since for every $x\in X_\Theta(k)$, $K_f^\Theta(x,x)$ is the integral of $f$ over the stabilizer of $x$, this implies that the kernel $K_f^\Theta$ vanishes on the diagonal. Combining this with the above asymptotic identity, this shows that the function $x\in X(k)\mapsto K_f(x,x)$ vanishes in a neighborhood $N_\Theta$ of $\infty_\Theta$ for every $\Theta\subsetneq \Delta_X$. As $X(k)\setminus \bigcup_{\Theta\subsetneq \Delta_X} N_\Theta$ is compact modulo $Z(X)^0(k)$, this proves the lemma.
	\end{proof}
	
	\begin{rem}
		\begin{itemize}
			\item In the group case ($X=H$, $G=H\times H$), the same statement holds for test functions $f=f_1\otimes f_2$ where only one of $f_1$ and $f_2$ is strongly cuspidal. Indeed, this can be shown by the same argument noting that for $\Theta\neq \Delta_X$, since the boundary degeneration $X_\Theta$ is in this case induced from a parabolic subgroup $P_\Theta^-=Q_\Theta\times Q_\Theta^-$ with both $Q_\Theta$ and $Q^-_\Theta$ proper, this weaker assumption also implies $K_f^\Theta(x,x)=0$ for all $x\in X_\Theta(k)$.
			
			\item Of course, the split assumption in the lemma shouldn't be necessary. However, this would require extending the theory of boundary degenerations and exponential maps to the nonsplit case and this hasn't been done yet. Similarly, the lemma should hold for $k$ Archimedean up to remplacing ``compactly supported'' by ``rapidly decreasing'' (in a suitable sense). In the setting of the unitary Gan-Gross-Prasad conjectures, this was established by a direct argument in \cite[Theorem 8.1.1]{BP3}.
			
			\item One simple example of a strongly cuspidal function is a function supported in the elliptic locus $G(k)_{\mathrm{ell}}$ (where we use the following definition of ellipticity: an element $\gamma\in G(k)$ is elliptic if $G_\gamma(k)/A_G(k)$ is compact or equivalently if it does not belong to any proper parabolic subgroup.)
		\end{itemize}

	\end{rem}
\end{paragr}

\begin{paragr}[Weighted orbital integrals for strongly cuspidal functions.]\label{S geom sides}
	Assume from now on that $X$ is wavefront. Given Lemma \ref{lem wavefront kernel}, we can now define a functional
	$$\displaystyle f\in \cS_{\mathrm{\scusp}}(G)\mapsto I^X(f)=\int_{X(k)} K_f(x,x)dx$$
	on the subspace $\cS_{\scusp}(G)\subset \cS(G)$\index{$\cS_{\scusp}(G)\subset \cS(G)$} of strongly cuspidal functions. By analogy with the local trace formula of Arthur, we could expect that this distribution admits an expansion in terms of weighted orbital integrals. For the Gross-Prasad variety $X=\SO_n\backslash \SO_n\times \SO_{n+1}$ this is indeed true as shown in a striking paper of Waldspurger \cite{WaldGP1}. However, the geometric expansion obtained by Waldspurger in this case is more complicated than in the group case due to the appearance of certain singular contributions which don't show up in the later (even for general test functions).
	
	Before describing Waldspurger's formula, we first give an ad hoc definition of weighted orbital integrals for strongly cuspidal functions. Namely, for $f\in \cS_{\scusp}(G)$ and $\gamma\in G(k)$ regular and semisimple, we can define the weighted orbital of $f$ at $\gamma$ as a limit\index{$WO(\gamma,f)$}
	$$\displaystyle WO(\gamma,f)=\lim\limits_{N\to \infty} \int_{A_G(k)\backslash G(k)} u_N(g) f(g^{-1}\gamma g) dg$$
	where the $(u_N)_N$ are the characteristic function of a sequence $(\mathcal{K}_N)_N$ of compact subsets of $A_G(k)\backslash G(k)$ that is increasing, exhaustive (i.e.\ $\bigcup_N \mathcal{K}_N=G(k)$) and uniformly smooth (in the sense that there exists a compact-open subgroup $J\subset G(k)$ such that $J\mathcal{K}_N=\mathcal{K}_N$ for all $N$). For example, we may take for $u_N$ the truncation function of Arthur $u_{T_N}$ for some parameters $T_N\in \cA^+$ going to infinity as $N\to \infty$. Let us make few remarks on this definition:
	\begin{itemize}
		\item The fact that the limit exists and does not depend on the sequence $(u_N)$ really uses the assumption that $f$ is strongly cuspidal (there is no such description of weighted orbital integrals for other test functions);
		
		\item When $\gamma$ is elliptic (i.e.\ $G_\gamma(k)$ is compact modulo $A_G(k)$), $WO(\gamma, f)$ coincides with the usual orbital integral $O(\gamma,f)=\int_{A_G(k)\backslash G(k)} f(g^{-1}\gamma g) dg$.
		
		\item As briefly discussed in \S \ref{S Arthur LTF}, there are usually more than one weighted orbital integral supported on the orbit of $\gamma$, as they are basically indexed by Levi subgroups containing $\gamma$, and moreover they a priori depend on the auxilliary choice of a special maximal compact subgroup $K$. However, for strongly cuspidal test functions all these weighted orbital integrals vanish identically except possibly for the one indexed by the unique minimal Levi $M$ containing $\gamma$ and it is independent on the choice of $K$ \cite[lemme 5.2]{WaldGP1}. Moreover, the above definition relates to Arthur's one for this particular $M$ up to a sign:
		$$\displaystyle WO(\gamma,f)=(-1)^{a_M-a_G}\Phi_M(\gamma,f)$$
		(where for a reductive group $H$, $a_H$ stands for the split rank of its center.)
		
		\item The distribution $f\mapsto WO(\gamma,f)$ is invariant in the sense that $WO(\gamma,{}^g f)=WO(\gamma,f)$ for every $g\in G(k)$, where we have set ${}^gf(x)=f(g^{-1}xg)$. Indeed, this follows by passing from truncation along the sequence $(\mathcal{K}_N)$ to truncation along the translated sequence $(g\mathcal{K}_N)$.
		
		\item A similar definition can be made in the Archimedean case if instead of characteristic functions we use sequence of functions $u_N\in C_c^\infty(G)$ that converge pointwise to $1$ and are {\it uniformly smooth and bounded} by which we mean that for every $X\in \mathcal{U}(\mathfrak{g})$ (the enveloping algebra of the Lie algebra of $G$), we can find $C_X>0$ such that $L(X)u_N$ is bounded in absolute value by $C_X$ for all $N$.
	\end{itemize}
\end{paragr}

\begin{paragr}[Geometric side of Arthur local trace formula.]
	The definition of weighted orbital integrals for stringly cuspidal functions given in the previous paragraph is actually already sufficient to state the geometric side of Arthur local trace formula when one of the two test functions is strongly cuspidal (in which case it takes a particularly simple form).
	
	\begin{theo}[Arthur]\label{theo Arthur LTF geometric}
		Let $f_1\in \cS_{\scusp}(H)$ be a strongly cuspidal test function. Then, for every $f_2\in \cS(H)$ we have
		$$\displaystyle I^H(f_1\otimes f_2)=\int_{H(k)} WO(\gamma,f_1) f_2(\gamma) d\gamma.$$
	\end{theo}
	
	Note that, using Weyl's integration formula, the above expansion can be rewritten as
	$$\displaystyle I^H(f_1\otimes f_2)=\int_{\Gamma_{\mathrm{rs}}(H)} WO(\gamma,f_1) O(\gamma,f_2) d\gamma.$$
	where $\Gamma_{\mathrm{rs}}(H)$\index{$\Gamma_{\mathrm{rs}}(H)$} denotes the set of regular semisimple conjugacy classes in $H(k)$ that we equip with a measure $d\gamma$ of the form
	$$\displaystyle \int_{\Gamma_{\mathrm{rs}}(H)} \varphi(\gamma)d\gamma=\sum_{T\in \mathcal{T}} \lvert W_T\rvert^{-1} \int_{T(k)} D^H(t) \varphi(t)dt$$
	the sum running over a set $\mathcal{T}$ of representatives for the conjugacy classes of maximal tori $T\subset H$, with $W_T=\Norm_{H(k)}(T)/T(k)$ the corresponding Weyl group, $dt$ some fixed Haar measure on $T(k)$ and $D^H(t)=\lvert \det(1-\Ad(t)\mid_{\mathfrak{h}/\mathfrak{t}})\rvert$\index{$D^H$} is the Weyl discriminant whereas $O(\gamma,f_2)=\int_{H_\gamma(k)\backslash H(k)} f(g^{-1}\gamma g)dg$\index{$O(\gamma,.)$} are the usual orbital integrals\footnote{To be specific about measures, $H_\gamma$ is conjugated to a unique $T\in \mathcal{T}$ and we need to choose the Haar measures on these two groups compatibly.}. Specializing further to the case where $f_2$ is also strongly cuspidal (so that the tensor product $f=f_1\otimes f_2$ is a strongly cuspidal function on $G(k)$), we have $O(\gamma,f_2)=0$ unless $\gamma$ is elliptic and therefore
	$$\displaystyle I^H(f_1\otimes f_2)=\int_{\Gamma_{\mathrm{ell}}(H)} O(\gamma,f_1) O(\gamma,f_2) d\gamma$$
	where $\Gamma_{\mathrm{ell}}(H)\subset \Gamma_{\mathrm{rs}}(H)$ denotes the subset of elliptic (and regular semisimple) conjugacy classes.
\end{paragr}

\begin{paragr}[Geometric side in the Gross-Prasad case.]
	Let us now consider a Gross-Prasad variety $X=\SO_n\backslash \SO_n\times \SO_{n+1}$ as in Section \ref{Sect GGP}. Let $\Theta_f$\index{$\Theta_f$} be the function $\gamma\mapsto WO(\gamma,f)$ which is defined on the open dense subset of regular semisimple elements. To state the geometric expansion of Waldspurger in this case, we need first to extend this function to certain singular elements $\gamma$. This can be done using the following result \cite[corollaire 5.9]{WaldGP1}: near every semisimple element $\gamma\in G(k)$, the function $\Theta_f$ admits an expansion\index{$c_{\mathcal{O}}(f)$}\index{$\widehat{j}(\mathcal{O},.)$}
	\begin{equation}\label{expansion qcharacters}
		\displaystyle \Theta_f(\gamma e^X)=\sum_{\mathcal{O}\in \mathrm{Nil}(\mathfrak{g}^*_\gamma)} c_{\mathcal{O}}(f) \widehat{j}(\mathcal{O},X),\;\; X\in \omega_{\mathrm{rs}}
	\end{equation}
	where:
	\begin{itemize}
		\item $\omega$ is a small neighborhood of $0$ in the Lie algebra $\mathfrak{g}_\gamma(k)$ of $G_\gamma(k)$ on which the exponential map $\exp: \omega\to G_\gamma(k)$, $X\mapsto e^X$ is well-defined and $\omega_{\mathrm{rs}}\subset \omega$ is the subset of regular semisimple elements;
		
		\item $\mathrm{Nil}(\mathfrak{g}^*_\gamma)$\index{$\mathrm{Nil}(\mathfrak{g}^*)$} denotes the set of nilpotent coadjoint orbits in $\mathfrak{g}^*_\gamma(k)$ and $\widehat{j}(\mathcal{O},.)$ is the unique smooth function on $\mathfrak{g}_{\gamma,\mathrm{rs}}(k)$ that represents the Fourier transform of the orbital integral over $\mathcal{O}$, that is:
		$$\displaystyle \int_{\mathfrak{g}_\gamma(k)} \varphi(X)\widehat{j}(\mathcal{O},X)dX=\int_{\mathcal{O}} \widehat{\varphi}(Y) dY \mbox{ for every } \varphi\in \cS(\mathfrak{g}_\gamma)$$
		where $\widehat{\varphi}(Y)=\int_{\mathfrak{g}_\gamma(k)} \varphi(X) \psi(\langle Y,X\rangle) dX$ is the Fourier transform of $\varphi$ (with $\psi: k\to \bC^\times$ some fixed non-trivial additive character) and $dY$ denotes the Kirillov-Kostant measure on the coadjoint orbit $\mathcal{O}$\footnote{More precisely, this measure comes from the natural symplectic structure on the coadjoint orbit $\mathcal{O}$ (which yields a canonical volume form) and the Haar measure on $k$ that is self-dual with respect to $\psi$.}. That such a function exists is a deep result of Harish-Chandra \cite{HCadminv}.
		
		\item The coefficients $c_{\mathcal{O}}(f)$ are simply complex numbers.
	\end{itemize}
	This kind of expression is not unfamiliar since it was shown by Harish-Chandra that function-characters of irreducible representations have similar local expansions. In the terminology of \cite{WaldGP1}, we say that the function $\Theta_f$ is a {\it quasi-character}.
	
	Waldspurger also defines a certain set $\Gamma_{X-\elli}(G)$\index{$\Gamma_{X-\elli}(G)$} of semisimple conjugacy classes in $G(k)$ as well as a measure on it. For the detailed definition we refer the reader to {\it op.\ cit.}, let us just say that $\Gamma_{X-\elli}(G)$ is the image of a finite union $\bigsqcup_{T\in \mathcal{T}} T(k)$ of compact tori and that the measure is defined in a way similar to the measure on $\Gamma_{\mathrm{rs}}(H)$, namely it is the unique measure that is locally equal to the normalized Haar measures on this various tori times a Weyl discriminant. However, in this case most of the tori $T\in \mathcal{T}$ are sontained in the singular locus of $G$ and consequently the singular conjugacy classes in $\Gamma_{X-\elli}(G)$ are not negligible e.g., in the case where $G$ is quasi-split, the trivial conjugacy class $1\in \Gamma_{X-\elli}(G)$ is an atom of mass one for this measure. 
	
	In complement to this construction, Waldspurger also associates to every $\gamma\in \Gamma_{X-\elli}(G)$ a particular regular nilpotent coadjoint orbit $\mathcal{O}_{\gamma,X}\in Nil(\mathfrak{g}^*_\gamma)$ (whose true meaning remains to be elucidated). Set\index{$c_{f,X}(\gamma)$}
	\begin{equation}\label{coeff cf}
		\displaystyle c_{f,X}(\gamma)=c_{\mathcal{O}_{\gamma,X}}(f),\mbox{ for every } \gamma\in \Gamma_{X-\elli}(G).
	\end{equation}
	We can now state Waldspurger's geometric expansion.
	
	\begin{theo}[Waldspurger]\label{theo geom side Waldspurger}
		Assume that $X=SO_n\backslash SO_n\times SO_{n+1}$ is a Gross-Prasad variety and that $k$ is non-Archimedean. Then, for every strongly cuspidal function $f\in \cS_{\scusp}(G)$ we have
		$$\displaystyle I^X(f)=\int_{\Gamma_{X-\elli}(G)} c_{f,X}(\gamma) d\gamma.$$
	\end{theo}
	
\end{paragr}

\begin{paragr}[Weighted characters and elliptic tempered representations.]\label{S spectral expansions}
	Given the original motivation that lead to the distribution $f\mapsto I^X(f)$, it is also natural to hope that this can in some way be expressed in terms of the characters of the irreducible representations in the spectrum of $X$ or the weighted characters introduced in \S \ref{S Arthur LTF}. 
	
	As for the geometric side, we will first discuss what happens for Arthur's local trace formula but before we need to say a word on weighted characters for strongly cuspidal functions following \cite[\S 2.2]{WaldGGP2}. The situation is a bit similar to what happens for weighted orbital integrals. Namely, for $f\in \cS_{\scusp}(G)$, $\pi\in \Temp(G)$ and a suitable Levi subgroup $M\subset G$, the weighted character $W\Theta_M(\pi,f)$ vanishes unless $M$ is ``a minimal Levi subgroup from which $\pi$ can be induced''\footnote{Although it is not immediately clear, all such Levi subgroups are conjugate, at least for tempered representations.}. Furthermore, for such a choice of $M$, $W\Theta_M(\pi,f)$ does not depend on the (implicit) choices of $K$, of the representation $\pi_M\in \Temp(M)$ such that $\pi\simeq I_M^G(\pi_M)$ and of the particular normalization of standard intertwining operators that enter into the definition. It doesn't either depend on the choice of $M$ and therefore, it is legitimate to simply set\index{$W\Theta(\pi,f)$}
	$$\displaystyle W\Theta(\pi,f):=W\Theta_M(\pi,f).$$
	Although this shows the existence of a well-defined functional $W\Theta(\pi,.)$ on $\cS_{\scusp}(G)$ associated to every $\pi\in \Temp(G)$, I don't know of any alternative direct definition of those (other than Arthur's original) contrary to what happens for weighted orbital integrals.
	
	In complement to these weighted characters, we also need to introduce a certain set $\mathcal{X}(G)$\index{$\mathcal{X}(G)$} of virtual tempered representations of $G(k)$. It is built on the set of {\it elliptic tempered representations} $\Pi_{\mathrm{ell}}(G)$\index{$\Pi_{\mathrm{ell}}(G)$} introduced by Arthur in \cite{Artelliptic}\footnote{We warn the reader that the set $\Pi_{\mathrm{ell}}(G)$ is actually denoted by $T_{\elli}(G)$ in \cite{Artelliptic}.} and which represents in some sense the spectral counterpart of the elliptic conjugacy classes. More precisely, $\Pi_{\mathrm{ell}}(G)$ is a certain set of {\it virtual} tempered representations containing the discrete series $\Pi_2(G)$ and which forms the basis of a complementary subspace in the Grothendieck group $\mathcal{R}_{\temp}(G)$\index{$\mathcal{R}_{\temp}(G)$} of all virtual tempered representations of finite length of the subspace $\mathcal{R}^{\mathrm{ind}}_{\temp}(G)$\index{$\mathcal{R}^{\mathrm{ind}}_{\temp}(G)$} spanned by (tempered) representations that are properly parabolically induced. The precise definition of $\Pi_{\mathrm{ell}}(G)$ involves again the choice of a normalization of some standard intertwining operators\footnote{More specifically, we need this normalization to construct actions of some (extensions of) $R$-groups on parabolic inductions of the form $I_M^G(\sigma)$, $\sigma\in \Pi_2(M)$.} which makes elements of $\Pi_{\mathrm{ell}}(G)$ canonical only up to a scalar of module $1$. In particular, denoting by $\pi\mapsto \overline{\pi}$ the unique antilinear automorphism of $\mathcal{R}_{\temp}(G)$ extending the contragredient, for every $\pi\in \Pi_{\mathrm{ell}}(G)$ the tensor product $\pi\otimes \overline{\pi}\in \mathcal{R}_{\temp}(G\times G)$ doesn't depend on any choice.
	
	The set $\mathcal{X}(G)$ then consists of all tempered virtual representations (up to multiplication by $\mathbb{S}^1$) of the form $I_M^G(\tau)$ where $M\subset G$ is a Levi subgroup and $\tau\in \Pi_{\mathrm{ell}}(M)$. It forms a basis of $\mathcal{R}_{\temp}(G)$ and is naturally equipped with a measure coming from Haar measures on tori of unitary unramified characters of Levi subgroups that we shall multiply by a certain spectral analog of the Weyl discriminant $\pi\mapsto D(\pi)$\index{$D(\pi)$}\footnote{In the notation of \cite{Artelliptic}, $D(\pi)$ equals the inverse of $d(\pi)$ times the cardinality of a certain finite group.}. 
\end{paragr}		

\begin{paragr}[Spectral side of Arthur local trace formula.]
	We can now state the spectral side of Arthur's local trace formula, still in the case where one of the two test functions is strongly cuspidal.
	
	\begin{theo}[Arthur]\label{theo Arthur LTF spectral}
		Let $f_1\in \cS_{\scusp}(H)$ be a strongly cuspidal test function. Then, for every $f_2\in \cS(H)$ we have
		$$\displaystyle I^H(f_1\otimes f_2)=\int_{\mathcal{X}(H)} W\Theta(\pi,f_1) \Theta(\overline{\pi},f_2) d\pi.$$
	\end{theo}
	
	Once again, the formula simplifies even more if $f_2$ is also strongly cuspidal. Indeed, we then have $\Theta(\overline{\pi},f_2)=0$ unless $\pi\in \Pi_{\elli}(H)$, in which case $W\Theta(\pi,f_1)=\Theta(\pi,f_1)$ is the usual character, so that the above identity becomes
	\begin{equation}
		\displaystyle I^H(f_1\otimes f_2)=\int_{\Pi_{\elli}(H)} \Theta(\pi,f_1) \Theta(\overline{\pi},f_2) d\pi.
	\end{equation}
\end{paragr}

\begin{paragr}[Spectral side for strongly tempered varieties.]
	
	\begin{theo}\label{theo spectral expansion strongly tempered}
		Assume that $X$ is strongly tempered and wavefront and that $k$ is non-Archimedean. Then, for every strongly cuspidal test function $f\in \cS_{\scusp}(G)$, we have
		\begin{equation*}
			\displaystyle I^X(f)=\int_{\mathcal{X}(G)} W\Theta(\pi,f) m_X(\overline{\pi}) d\pi
		\end{equation*}
		where we have extended the multiplicity function $\pi\in \Temp(G)\mapsto m_X(\pi)$ to $\mathcal{R}_{\temp}(G)$ by linearity.
	\end{theo}
	
	\begin{proof}
		We only give a sketch assuming for simplicity that $X(k)=H(k)\backslash G(k)$ contains only one $G(k)$-orbit. The idea is to reduce everything to the spectral side of Arthur's local trace formula by means of (the proof of) Theorem \ref{theo nondeg canonical form SV}. Let $J\subset G(k)$ be a compact-open subgroup leaving $f$ both left and right-invariant. We shall use freely the notation introduced in the proof of Theorem \ref{theo nondeg canonical form SV}. In particular, we let $J'$ be another compact-open subgroup such that the map $P_{J'}:\cC(J'\backslash G/J)\to \cC(X)^J$ is surjective and we let $f_{J'}\in \cC(J'\backslash G/J')$ be the function such that $P_{J'}^*P_{J'}f'=f_{J'}\star f'$ for every $f'\in \cC(J'\backslash G)$. Recall from \eqref{eq7 theo nondeg canonical form SV} that we have a decomposition
		$$\displaystyle \cC(J'\backslash G/J)=f_{J'}\star \cC(J'\backslash G/J)\oplus \Ker(L(f_{J'})\mid_{\cC(J'\backslash G/J)})$$
		where $L(f_{J'})$ stands for the operator of left convolution by $f_{J'}$. Then, up to shrinking $J$ we may assume that it is ``good'' in the sense of \cite[\S 3.7]{Bercenter} i.e.\ that the full subcategory of smooth representations of $G(k)$ generated by their $J$-invariants is a product of Bernstein blocks $\mathfrak{s}$. Let $\cC(J'\backslash G)_{\mathfrak{s}}$ be the direct summand of the smooth $G(k)$-representation $\cC(J'\backslash G)$ corresponding that product of blocks. Then, the above decomposition extends to $\cC(J'\backslash G)_{\mathfrak{s}}$ i.e.\
		\begin{equation}\label{eq1 spectral side strongtemp}
			\displaystyle \cC(J'\backslash G)_{\mathfrak{s}}=f_{J'}\star \cC(J'\backslash G)_{\mathfrak{s}}\oplus \Ker(L(f_{J'})\mid_{\cC(J'\backslash G)_{\mathfrak{s}}}).
		\end{equation}
		The crux is to consider a partial inverse $g\in \cC(J'\backslash G)_{\mathfrak{s}}$ of $f_{J'}$ for the convolution product i.e.\ a function such that
		\begin{equation}\label{eq2 spectral side strongtemp}
			\displaystyle f_{J'}\star g\star f_{J'}\star f'=f_{J'}\star f' \mbox{ for every } f'\in \cC(J'\backslash G)_{\mathfrak{s}}.
		\end{equation}
		The existence of such a function can be inferred from the decomposition \eqref{eq1 spectral side strongtemp}. Indeed, this decomposition implies that the endomorphism
		$$\displaystyle f_{J'}\star \cC(J'\backslash G)_{\mathfrak{s}}\to f_{J'}\star \cC(J'\backslash G)_{\mathfrak{s}},\; f'\mapsto f_{J'}\star f'$$
		is bijective. Thus, denoting by $e_{J',\mathfrak{s}}$ the projection of $e_{J'}$ to $\cC(J'\backslash G)_{\mathfrak{s}}$, there exists $g\in \cC(J'\backslash G)_{\mathfrak{s}}$ such that
		$$\displaystyle f_{J'}\star f_{J'}\star g=f_{J'}\star e_{J',\mathfrak{s}}.$$
		Taking the convolution with $f_{J'}\star f'$, where $f'\in \cC(J'\backslash G)_{\mathfrak{s}}$, on the right this gives $f_{J'}\star f_{J'}\star g\star f_{J'}\star f'=f_{J'}\star f_{J'}\star f'$ and therefore, since $f_{J'}\star$ is injective on $f_{J'}\star \cC(J'\backslash G)_{\mathfrak{s}}$, $g$ indeed satisfies \eqref{eq2 spectral side strongtemp}.
		
		Since $L(f_{J'})=P_{J'}^*P_{J'}$ (on $\cC(J'\backslash G/J)$) the equality \eqref{eq2 spectral side strongtemp} implies $P_{J'}^*P_{J'}L(g)P_{J'}^*P_{J'}=P_{J'}^*P_{J'}$ or even, since $P_{J'}$ is surjective and $P_{J'}^*$ injective, that $P_{J'}L(g)P_{J'}^*$ is the identity of $\cC(X)^J$. Therefore, multiplying everything by the operator of right convolution by $f$, we obtain a commutative diagram
		\begin{equation}
			\displaystyle \xymatrix@=2cm{\cC(J'\backslash G/J) \ar[r]^{L(g)\otimes R(f)} & \cC(J'\backslash G/J) \ar[d]^{P_{J'}} \\ \cC(X)^J \ar[u]^{P_{J'}^*} \ar[r]^{R(f)} & \cC(X)^J.}
		\end{equation}
		In terms of kernels, the above diagram reads
		\begin{equation}
			\displaystyle K_f^X(x,y)=\int_{H(k)\times H(k)} K^G_{f,g}(h_1x,h_2y)dh_1 dh_2,\;\; x,y\in X(k)=H(k)\backslash G(k).
		\end{equation}
		
		Therefore, we have
		$$\displaystyle I^X(f)=\int_{H(k)\backslash G(k)}\int_{H(k)\times H(k)} K^G_{f,g}(h_1x,h_2x)dh_1 dh_2 dx.$$
		It turns out that the above triple integral is absolutely convergent (but we will not prove it; in the case of unitary Gan-Gross-Prasad this is proved in (the proof of) \cite[Proposition 9.2.2]{BP3}), so that
		$$\displaystyle I^X(f)=\int_{H(k)}\int_{G(k)} K^G_{f,g}(\gamma,h\gamma) d\gamma dh=\int_{H(k)} I^G(f\otimes L(h)g)dh$$
		where we have identified the inner integral with Arthur's local trace formula for the test function $f\otimes L(h)g$. (Recall that the test function $f$ is strongly cuspidal so that $I^G(f\otimes L(h)g)$ is absolutely convergent.) Therefore, applying Theorem \ref{theo Arthur LTF spectral}, we obtain
		\begin{equation*}
			\displaystyle I^X(f)=\int_{H(k)}\int_{\mathcal{X}(G)} W\Theta(\pi,f) \Theta(\overline{\pi}, L(h)g)d\pi dh.
		\end{equation*}
		But, since $X$ is strongly tempered, this integral is again absolutely convergent, essentially because the functions $\gamma\in G(k)\mapsto \Theta(\overline{\pi}, L(\gamma)g)$ are sums of tempered matrix coefficients and therefore integrable over $H(k)$. Therefore, to conclude it only remains to show that for every $\pi\in \Temp(G)$ we have
		$$\displaystyle \int_{H(k)} \Theta(\pi, L(h)g) dh=m_X(\pi).$$
		By definition of the form $\mathcal{B}_\pi^H$ (see \S \ref{S strongly tempered}), we have
		\[\begin{aligned}
			\displaystyle \int_{H(k)} \Theta(\pi, L(h)g) dh=\sum_{v} \mathcal{B}_\pi^H(\pi(g)v,v)
		\end{aligned}\]
		where the sum runs over an orthonormal basis of $\pi^{J'}$. By \eqref{eq1 theo nondeg canonical form SV}, this can be rewritten as
		$$\displaystyle \int_{H(k)} \Theta(\pi, L(h)g) dh=\sum_{v} (\pi(f_{J'})\pi(g)v,v)= \Tr(\pi(f_{J'})\pi(g)).$$
		However, by \eqref{eq2 spectral side strongtemp} we have $\pi(f_{J'})\pi(g)\pi(f_{J'})=\pi(f_{J'})$ which implies that $\pi(g)$ coincides with the inverse of $\pi(f_{J'})$ on $\Ima(\pi(f_{J'}))=\Ker(\pi(f_{J'}))^{\perp}$. Hence,
		$$\displaystyle \Tr(\pi(f_{J'})\pi(g))=rk(\pi(f_{J'}))=m_X(\pi).$$
	\end{proof}
\end{paragr}

\begin{paragr}[Multiplicity formulas.]
	One important application of the simple local trace formulas we have been discussing is to provide formulas for the multiplicities $m_X(\pi)$ in terms of the Harish-Chandra character $\Theta_\pi$ of the representation $\pi$. This was discovered by Waldspurger in the setting of the Gross-Prasad conjecture but we can also revisit some previous work of Arthur in this light.
	
	Assume that the group $H(k)$ has a compact center. Recall that a function $f\in \cS(H)$ is called {\it cuspidal} if it satisfies the following equivalent conditions:
	\begin{itemize}
		\item $O(\gamma,f)=0$ for every regular semisimple element $\gamma\in H(k)$ that is not elliptic;
		
		\item $\Tr \pi(f)=0$ for every tempered representation $\pi\in \Temp(G)$ that is properly parabolically induced.
	\end{itemize}
	We emphasize that strongly cuspidal functions (as defined in \S \ref{S strongly cuspidal}) are cuspidal but not conversely.
	
	Plugging in Arthur's geometric and spectral expansions (that is Theorems \ref{theo Arthur LTF geometric} and \ref{theo Arthur LTF spectral} respectively) a cuspidal test function $f=f_2\in \cS(H)$ we obtain the identity
	$$\displaystyle \int_{\Gamma_{\mathrm{ell}}(H)} O(\gamma, f_1)O(\gamma,f)d\gamma=\sum_{\pi\in \Pi_{\elli}(H)} D(\pi)\Theta(\pi,f_1) \Theta(\overline{\pi},f_2)$$
	for every function $f_1\in \cS(H)$ that is strongly cuspidal where $D(.)$ is a certain non-vanishing function on $\Pi_{\elli}(H)$ that formally ressembles the (inverse of) Weyl discriminant and which is identically $1$ on the subset of discrete series $\Pi_2(H)$. In particular, taking for $f=f_\pi$ a pseudo-coefficient of some elliptic representation $\pi\in \Pi_{\mathrm{ell}}(H)$ i.e.\ a test function satisfying $\Tr \pi'(f)=\delta_{\pi,\pi'}$ for every $\pi'\in \mathcal{X}(H)$ (the existence of pseudo-coefficients follows from the trace Paley-Wiener theorem of \cite{tracePW}) and for $f_1$ any test function supported in the elliptic locus (recall that those are automatically strongly cuspidal), we get the identity
	\begin{equation}
		\displaystyle O(\gamma,f_\pi)=D(\pi)\overline{\Theta_{\pi}(\gamma)}, \mbox{ for every } \gamma\in \Gamma_{\mathrm{ell}}(H),
	\end{equation}
	where $\Theta_{\pi}$\index{$\Theta_{\pi}$} is the {\it Harish-Chandra character} of $\pi$ (the only locally integrable function representing the character-distribution $f'\mapsto \Tr \pi(f')$). Plugging this back into the defining equality $\Tr \pi'(f_\pi)=\delta_{\pi',\pi}$, we end up with the so-called {\it elliptic orthogonality relations}
	\begin{equation*}
		\displaystyle \int_{\Gamma_{\mathrm{ell}}(H)} \Theta_{\pi'}(\gamma) \overline{\Theta_\pi(\gamma)} d\gamma=\left\{\begin{array}{ll}
			D(\pi)^{-1} & \mbox{ if } \pi'=\pi,\\
			0 & \mbox{ otherwise}
		\end{array}\right.
	\end{equation*}
	for every $\pi,\pi'\in \Pi_{\mathrm{ell}}(H)$. To put this in perspective with Waldspurger's multiplicity formula below, we note that, whenever $\pi,\pi'\in \Pi_2(H)$ are discrete series so that $D(\pi)=1$, the right hand side of the above identity can also be rewritten as $m_X(\pi'\otimes \overline{\pi})$ where $X=H$ is the group variety.

	Let us now consider the case of a Gross-Prasad variety $X=\SO_n\backslash \SO_n\times \SO_{n+1}$. Recall the space of semisimple conjugacy classes $\Gamma_{X-\elli}(G)$ that was introduced in \S \ref{S geom sides}. According to Harish-Chandra, the characters $\Theta_\pi$, $\pi\in \Irr(G)$, admit local expansions of the form \eqref{expansion qcharacters} and therefore the definition of the coefficients $\gamma\in \Gamma_{X-\elli}(G)\mapsto c_{f,X}(\gamma)$ given by \eqref{coeff cf} for $f$ strongly cuspidal, which only depends on those local expansions for $\Theta_f$, can be naturally adapted to the character $\Theta_\pi$ instead yielding functions $\gamma\in \Gamma_{X-\elli}(G)\mapsto c_{\pi,X}(\gamma)$\index{$c_{\pi,X}(\gamma)$} for every $\pi\in \Irr(G)$. Then, by an argument in the same spirit as above, Waldspurger was able to deduce from the geometric and spectral expansions of Theorems \ref{theo geom side Waldspurger} and \ref{theo spectral expansion strongly tempered}, the following striking formula.
	
	\begin{theo}[Waldspurger]\label{theo multiplicity Waldspurger}
		Assume that $X=\SO_n\backslash \SO_n\times \SO_{n+1}$ is a Gross-Prasad orthogonal variety and that $k$ is non-Archimedean. Then, for every tempered representation $\pi\in \Temp(G)$, we have
		$$\displaystyle m_X(\pi)=\int_{\Gamma_{X-\elli}(G)} c_{\pi,X}(\gamma) d\gamma.$$
	\end{theo}
	
	This result is arguably the main and crucial ingredient in Waldspurger's approach to the local Gross-Prasad conjecture. Combining it with the theory of (twisted) endoscopy \cite{WalGP2}, this leads to a full proof of the conjecture in the tempered case.
	
	\begin{rem}
		\begin{itemize}
			\item Similar formulas have been obtained in the setting of the local unitary Gan-Gross-Prasad conjecture \cite{BPlocalGGPpadic}, the Ginzburg-Rallis model \cite{ChenGR} and Galois symmetric varieties \cite{BeuGal}. There have been also works extending the work of Waldspurger to Archimedean fields \cite{BP3}, \cite{Luo}.
			
			\item According to a conjecture of Prasad, the right-hand side of Waldspurger's formula should represent in general the $X$-Euler Poincar\'e characteristic i.e.\ it is conjectured that\index{$\mathrm{EP}_X(\pi)$}
			$$\displaystyle \mathrm{EP}_X(\pi)=\int_{\Gamma_{X-\elli}(G)} c_{\pi,X}(\gamma) d\gamma \mbox{ for every } \pi\in \Irr(G)$$
			where we have set $\mathrm{EP}_X(\pi):=\sum_{i\geq 0} (-1)^i \dim \Ext^i_G(\cS(X),\pi)$. The analog in the group case is a conjecture of Kazhdan which was proved by Schneider-St\"uhler \cite{SSbuilding} and Bezrukavnikov \cite{Bezthesis}.
			
			\item It is natural to wonder how general Waldspurger's formula can be. As a very first step in that direction, it seems important to understand in a more conceptual way the space of conjugacy classes $\Gamma_{X-\elli}(G)$ as well as the coefficients $c_{\pi,X}(\gamma)$, or equivalently the nilpotent coadjoint orbits $\mathcal{O}_{\gamma,X}\in Nil(\mathfrak{g}^*_\gamma)$, for $\gamma\in \Gamma_{X-\elli}(G)$. Motivated by some specific steps in Waldspurger's proof of the geometric expansion in Theorem \ref{theo geom side Waldspurger} (which, more precisely, builds on Harish-Chandra method of semisimple descent), Chen Wan \cite{Chenmultformula} has proposed a general construction of such a set of conjugacy classes $\Gamma_{X-\elli}(G)$. More precisely, rephrasing slightly Wan's construction, $\Gamma_{X-\elli}(G)$ should consists of the $G(k)$-conjugacy classes of pairs $\gamma=(h,Y)$ where $h\in G_{\mathrm{ss}}(k)$ is a semisimple element and $Y$ is a connected component of the subvariety $X^h$ of $h$-fixed points (it can be shown that $Y$ is then automatically a homogeneous spherical $G_h$-variety) satisfying the three following conditions:
			\begin{itemize}
				\item The connected centralizer $G_h$ is quasi-split;
				
				\item The pair $(G_h,Y)$ is {\it minimal} in the sense of {\em op.\ cit.}\ i.e.\ that $Y/Z(G_h)$ has the same dimension as that of a Borel subgroup of $G_h/Z(G_h)$. (Note that this is the maximal possible dimension for a $G_h/Z(G_h)$-spherical variety.)
				
				\item The pair $(G_h,Y)$ is {\it elliptic} in the following sense: the morphism between centers $Z(G_h)^0\to Z(X)^0$ is surjective and the connected component of its kernel is anisotropic modulo $Z(G)^0$.
			\end{itemize}
			We note that by a result of Knop, the minimality condition is equivalent to: the image of the moment map $T^*(Y/Z(G_h))\to (\mathfrak{g}_h/\mathfrak{z}(G_h))^*$ contains an open dense subset. Specializing to the Gross-Prasad, this recovers the definition of $\Gamma_{X-\elli}(G)$ given by Waldspurger. (In this case, if $X^h$ is nonempty it is connected so that we can really identify $\Gamma_{X-\elli}(G)$ with a set of semisimple conjugacy classes in $G(k)$). Wan has also proposed a generalization of Waldspurger's definition of the nilpotent coadjoint orbits $\mathcal{O}_{\gamma,X}$ as a linear combination (with rational coefficients) of regular nilpotent coadjoint orbits in $\mathfrak{g}^*_\gamma)$ (where if $\gamma=[h,Y]$ we have written $\mathfrak{g}^*_\gamma$ for $\mathfrak{g}^*_h$) and using it he has made very general predictions on multiplicity formulas of the same shape as in Theorem \ref{theo multiplicity Waldspurger}. However, Wan's proposal is not very explicit as it relies on the knowledge of certain Shalika germs that are difficult to compute. It would be nice to see another (even conjectural) more explicit description of these coadjoint orbits.
			
		\end{itemize}
	\end{rem}
\end{paragr}

\section{Global aspects}\label{Part global}

Here are some general notations and conventions that we will adopt in this chapter:
\begin{itemize}
	\item We will use $F$\index{$F$} to denote a global field which, for simplicity of exposition, we assume to be a number field. The main reason is that the geometry of spherical varieties can be more subtle in positive characteristic, but most of the story should adapt verbatim to function fields. We write $V_F$\index{$V_F$} for the set of places of $F$. For $v\in V_F$, we denote by $F_v$\index{$F_v$} the corresponding completion and, if the place in non-Archimedean, by $\mathcal{O}_v\subset F_v$\index{$\mathcal{O}_v$} its ring of integers and $\mathbb{F}_{q_v}$\index{$\mathbb{F}_{q_v}$} its residue field (therefore $q_v$ will be its cardinal!). The absolute Galois group of $F$ (resp. $F_v$) will be denoted by $\Gamma_F$\index{$\Gamma_F$} (resp. $\Gamma_{F_v}$)\index{$\Gamma_{F_v}$} and we will write $\Frob_v\in \Gamma_{F_v}$\index{$\Frob_v$} for any (geometric) Frobenius element.
	
	\item For every algebraic variety $Y$ over $F$ and $v\in V_F$, we will usually write $Y_v$\index{$Y_v$} for its set of $F_v$ points: $Y_v=Y(F_v)$. We will denote by $\mathbb{A}=\prod_v' F_v$\index{$\mathbb{A}$} the adele ring of $F$ and, for any linear algebraic group $G$ over $F$, by $[G]:=G(F)\backslash G(\mathbb{A})$\index{$[G]$} its {\em automorphic quotient};
	
	\item For any linear algebraic reductive group $G$ over $F$, we will equip $G(\bA)$ with its {\em Tamagawa measure}. Let us recall briefly how this canonical measure is constructed. We pick an invariant volume form $\omega$ on $G$, defined over $F$, as well as a nontrivial additive character $\psi: \bA/F\to \bC^\times$\index{$\psi$} and we denote by $\psi_v$\index{$\psi_v$} its restriction to $F_v$ for every place $v\in V_F$. We equip $G_v$ with the {\em local Tamagawa measure}\index{$d_{\Tam}g_v$}
	$$\displaystyle d_{\Tam}g_v=\lvert \omega\rvert_{\psi_v}$$
	that is to say the measure given locally as follows: if in local coordinates
	$$\displaystyle \omega=f(x_1,\ldots,x_d) dx_1\wedge\ldots\wedge dx_d$$
	then
	$$d_{\Tam}g_v=\lvert f(x_1,\ldots,x_d)\rvert_v d_{\psi_v}x_1\ldots d_{\psi_v}x_d$$
	where $d_{\psi_v}x$ stands for the $\psi_v$-selfdual Haar measure on $F_v$. Note that when $\psi_v$ is unramified (which, of course, happens at all but finitely many places), the self-dual measure $d_{\psi_v}x$ is simply the one giving volume $1$ to $\mathcal{O}_v\subset F_v$. Moreover, for almost all places $v$, by results of Weil \cite{Wei} and Steinberg \cite{Steinb} we have
	$$\displaystyle \vol(G(\mathcal{O}_v),d_{\Tam}g_v)=\frac{\lvert G(\mathbb{F}_{q_v})\rvert}{q_v^{d_G}}=L_v(M_G^\vee(1))^{-1}$$
	where $d_G=\dim(G)$\index{$d_G$}, $M_G$\index{$M_G$} is a certain Artin-Tate motive with nonnegative weight naturally associated to $G$ (the so-called {\em motive of $G$} introduced by Gross \cite{Gros}) and $L_v(M_G^\vee(1))$\index{$L_v(M_G^\vee(1))$} is the value at $s=0$ of the local $L$-factor of the Tate-twisted dual motive. We will use the notation $\Delta_{G,v}=L_v(M_G^\vee(1))$\index{$\Delta_{G,v}$} to denote this local $L$-value. We can then form the local normalized measures $dg_v=\Delta_{G,v} d_{\Tam}g_v$, which satisfy $\vol(G(\mathcal{O}_v),dg_v)=1$ for almost all $v$. The Tamagawa measure on $G(\bA)$ is defined as the product\index{$d_{\Tam}g$}
	$$\displaystyle d_{\Tam}g=(\Delta_{G}^*)^{-1}\prod_{v\in V_F}dg_v$$
	where $\Delta_G^*$\index{$\Delta_G^*$} denotes the leading coefficient in the Taylor expansion of the global Artin $L$-function $L(s,M_G^\vee(1))$ at $s=0$ i.e.\
	$$\displaystyle \Delta_{G}^*=\lim\limits_{s\to 0} s^{a_G} L(s,M_G^\vee(1))$$
	where $a_G=\dim(A_G)$\index{$a_G$}, the dimension of the maximal central split torus of $G$. (Note that since $M^\vee(1)$ has weights $\leq -2$, $L(s,M^\vee(1))$ can't have a zero at $s=0$, moreover this pole is of order $a_G$.) We emphasize that the local Tamagawa measure $d_{\Tam}g_v$ depends on the choices of $\omega$ and $\psi_v$ but that, as readily follows from the product formula, the global Tamagawa measure is independent of these choices. 
	Note that for every finite set $S\subset V_F$, we also have the decomposition
	$$\displaystyle d_{\Tam}g=(\Delta_{G}^{S,*})^{-1}\prod_{v\in S} d_{\Tam} g_v \times \prod_{v\notin S}dg_v$$
	where $\Delta_{G}^{S,*}$\index{$\Delta_{G}^{S,*}$} is defined similarly as the leading term in the Taylor expansion of the partial $L$-function $L^S(s,M_G^\vee(1))$ at $s=0$. In particular, taking the limit $S\to \infty$, we formally have a decomposition $d_{\Tam}g=\prod_v d_{\Tam} g_v$.
	
	\item Let $H_G: G(\bA)\to \mathcal{A}_G:=\Hom(X^*(G),\bR)$\index{$H_G$} be the homomorphism defined by $\langle \chi,H_G(g)\rangle=\log \lvert \chi(g)\rvert$, where $X^*(G)$ denotes the lattice of algebraic characters of $G$ defined over $F$, $\lvert .\rvert=\prod_v \lvert .\rvert_v$\index{$\lvert .\rvert$}\index{$\lvert .\rvert_v$} is the product of the local normalized absolute values and $\log$ denotes the logarithm in the natural basis. We set $G(\bA)^1=\Ker(H_G)$\index{$G(\bA)^1$} and we equip $G(\bA)^1$ with the unique Haar measure such that the quotient of the Tamagawa measure $d_{\Tam}g$ of $G(\bA)$ by it induces on $\mathcal{A}_G=H_G(G(\bA))$ the measure for which the lattice $\Hom(X^*(G),\bZ)$ is of covolume $1$. We will refer to this measure as the {\em Tamagawa measure on $G(\bA)^1$} and denote it by $d_{\Tam}g^1$\index{$d_{\Tam}g^1$}. Note that our normalization is such that, for $G=\mathbb{G}_m$, the total volume of $G(F)\backslash G(\bA)^1=F^\times\backslash \mathbb{A}^1$ is $1$.
	
	\item The above definition extends naturally to any unimodular homogeneous space $X=H\backslash G$, thus yielding local Tamagawa measures $d_{\Tam}x_v$\index{$d_{\Tam}x_v$} on $X_v$ (that depend on the choice of an invariant volume form $\omega_X$ on $X$ and of an additive character $\psi$) as well as a global Tamagawa measure $d_{\Tam}x$\index{$d_{\Tam}x$} on $X(\bA)$, that is independent of any choice and that admits a formal decomposition $d_{\Tam}x=\prod_v d_{\Tam} x_v$. More precisely, we have
	$$\displaystyle d_{\Tam}x=(\Delta_X^{*})^{-1}\prod_{v} \Delta_{X,v}d_{\Tam}x_v$$
	where $\Delta_{X,v}:=\Delta_{G,v}\Delta_{H,v}^{-1}$\index{$\Delta_{X,v}$} and $\Delta_X^{*}=\Delta_G^*(\Delta_H^*)^{-1}$\index{$\Delta_X^{*}$} are the ratios between the (regularized) Tamagawa factors for $G$ and $H$.
	
	\item For $G$ a connected reductive group over $F$, we write $\cA(G)$\index{$\cA(G)$} for the space of automorphic forms for $G$: it is a space of functions on $[G]$ satisfying certain conditions of regularity (it is smooth at every place and of {\em moderate growth} in the number field case), finiteness (roughly speaking equivalent to asking that they generate, by right translation, representations of finite length of $G(\bA)$). We refer the reader to \cite{BoJa} for the official definition, but we emphasize that for our purpose it is more natural to weaken a bit the conditions at Archimedean places by not asking the functions to be $K_\infty$-finite (for some choice of maximal compact subgroup $K_\infty\subset G(F_\infty)$); this makes the theory slightly more analytic and in particular requires to work with certain topological representations at Archimedean places rather than the usual $(\mathfrak{g}_\infty,K_\infty)$-modules; see \cite{Grob} or \cite[\S 2.7]{BCZ} for more details.
	\item  The subspace of cuspidal forms, consisting of those $\phi\in \cA(G)$ with vanishing constant terms, will be denoted by $\cA_{\cusp}(G)$\index{$\cA_{\cusp}(G)$}.
	
	\item The {\em cuspidal automorphic representations} of $G(\bA)$ are the irreducible subrepresentations of $\cA_{\cusp}(G)$. Similarly, an {\em automorphic representation} of $G(\bA)$ is an irreducible subquotient of $\cA(G)$.
	
	\item Every unitary cuspidal automorphic representation $\pi$ of $G(\bA)$ is equipped with a canonical invariant inner product namely the {\em Petersson inner product} defined by\index{$\langle .,.\rangle_{\Pet}$}
	$$\displaystyle \langle \phi,\phi'\rangle_{\Pet}=\int_{G(F)\backslash G(\bA)^1} \phi(g^1) \overline{\phi'(g^1)} d_{\Tam}g^1$$
	where $d_{\Tam}g^1$ stands for the quotient of the Tamagawa measure on $G(\bA)^1$ by the counting measure on $G(F)$.
	
	\item Let $r: {}^L G\to \GL(M)$ be an algebraic representation of the $L$-group of $G$. Then, for every automorphic representation $\pi$ of $G(\bA)$, we can define following Langlands a partial $L$-function\index{$L^S(s,\pi,r)$}
	$$\displaystyle L^S(s,\pi,r)=\prod_{v\notin S} L(s,\pi_v,r)$$
	where $S\subset V_F$ is a finite subset containing the Archimedean places and such that for every $v\notin S$, the representation $\pi_v$ is unramified (in the sense that it admits a nonzero vector invariant under a hyperspecial maximal compact subgroup), the local $L$-factors are given by $L(s,\pi_v,r)=\det(1-q_v^{-s} r(\mathcal{S}_v(\pi)))^{-1}$ where $\mathcal{S}_v(\pi)$\index{$\mathcal{S}_v(\pi)$} is the Satake parameter of $\pi_v$ (a $\widehat{G}$-conjugacy class in ${}^L G$) and the product converges for $s\in \bC$ of sufficiently large real part. We refer the reader to \cite{Bolf} for more details. In particular, we emphasize that our $L$-functions are always normalized so that $1/2$ ought to be the center of symmetry of the expected functional equation.
	
	\item It will be convenient to use the {\em global Langlands group} $L_F$\index{$L_F$} of $F$. When $F$ is a function field, it is just the global Weil group $W_F$ of $F$. When $F$ is a number field, $L_F$ ought to be an extension of the Weil group by a compact group but there is no satisfactory definition of it yet. However, it should satisfy the crucial property that the (continuous) $n$-dimensional irreducible representations of $L_F$ are in natural bijection with the set of cuspidal automorphic representations of $\GL_n(\bA)$. The global Langlands correspondence, as refined by Arthur, should then give a decomposition of the space of cuspidal automorphic forms (or more generally, of the space of discrete automorphic forms)
	$$\displaystyle \mathcal{A}_{\cusp}(G)=\bigoplus_{\psi} \mathcal{A}_{\cusp}(G)_\psi$$
	indexed by global $A$-parameters
	$$\displaystyle \psi: L_F\times \SL_2(\bC)\to {}^L G$$
	which are {\em discrete} in the sense that the centralizer $Z_{\widehat{G}}(\psi)$ is finite modulo $Z(\widehat{G})^\Gamma$. We refer the reader to \cite[\S 16.1]{SV} for a nice and concise summary of some basic expectations on this conjectural decomposition. We will also sometimes talk about the global $A$-parameter of a cuspidal representation $\pi$ by which we mean a $A$-parameter $\psi$ such that $\pi$ appears in $\mathcal{A}_{\cusp}(G)_\psi$. Strictly speaking it is a bit abusive as it might happen that $\pi$ appears in different pieces of the above decomposition but in all cases where we will do so, either the choice of $\psi$ will not matter or there will be only one such parameter (e.g. this happens if $\pi$ appears with multiplicity one in $\mathcal{A}_{\cusp}(G)$, which we know is always true for all classical groups except even special orthogonal groups, when the multiplicity can be $2$). We also recall that for classical groups we can, following Arthur \cite{artbook}, define unconditionally certain substitute for $A$-parameters using cuspidal automorphic representations of general linear groups to replace irreducible representations of the hypothetical group $L_F$.
	
	\item As usual, the $A$-parameter $\psi$ of a cuspidal automorphic representation is said to be {\em generic} if it is trivial on $\SL_2(\bC)$. For classical groups, this corresponds to a very explicit condition on Arthur's substitute to the $A$-parameter which can be succinctly described as the requirement that the functorial lift $\pi^{\GL}$ of $\pi$ to the suitable general linear group (corresponding to the standard representation of the $L$-group of $G$) is generic (in the usual, representation-theoretic, sense).
\end{itemize}

\subsection{The global Gan-Gross-Prasad and Ichino-Ikeda conjectures}\label{S global GGP}
As an illustrative introduction to the general conjectures of Sakellaridis-Venkatesh, we first recall the formulation of the global {\em Gan-Gross-Prasad} \cite{GGP} and its refinement due to {\em Ichino-Ikeda} \cite{IIk} in the setting of orthogonal groups. We adopt notation similar to the one in Section \ref{Sect GGP} but over a number field. Thus, $(V,q)$ is a quadratic space over a number field $F$, $W\subset V$ is a nondegenerate subspace of it and we set
$$\displaystyle H=\SO(W) \hookrightarrow G=\SO(W)\times \SO(V).$$

\vspace{2mm}

\begin{paragr}
The {\em Gross-Prasad period} $\mathcal{P}_H$\index{$\mathcal{P}_H$} is the automorphic period associated to that subgroup $H$ of $G$ i.e.\ it is the following linear functional on $\mathcal{A}_{\cusp}(G)$:
$$\displaystyle \mathcal{P}_H(\varphi):=\int_{[H]} \varphi(h) dh,\;\;\; \varphi\in \mathcal{A}_{\cusp}(G).$$
\end{paragr}

\begin{paragr}
	Let $\pi\subset \mathcal{A}_{\cusp}(G)$ be a cuspidal automorphic representation of $G(\mathbb{A})$. The Gan-Gross-Prasad conjecture relates the restriction of the automorphic period $\mathcal{P}_H$ on $\pi$ to the central value of a certain (Rankin-Selberg kind of) $L$-function. More precisely, as recalled in Section \ref{Sect GGP}, each of ${}^L \SO(W)$ and ${}^L \SO(V)$ can be naturally identified with subgroups of either an even complex orthogonal group or a complex symplectic group. More precisely, we have
	$$\displaystyle {}^L \SO(W)\times {}^L \SO(V) \subset \left\{\begin{array}{ll}
		O_{2m}(\mathbb{C})\times \Sp_{2m}(\mathbb{C}) \mbox{ if } \dim(V)=2m+1 \mbox{ is odd,} \\
		\Sp_{2m-2}(\mathbb{C})\times O_{2m}(\mathbb{C}) \mbox{ if } \dim(V)=2m \mbox{ is even.}
	\end{array} \right.$$
	Composing these embeddings with the standard representations of the orthogonal and symplectic groups involved, we obtain complex representations $R_W$ and $R_V$ of ${}^L \SO(W)$ and ${}^L \SO(V)$ respectively. Set $R=R_W\otimes R_V$, a complex symplectic representation of ${}^L G$ and form the (partial) $L$-function
	$$\displaystyle L^S(s,\pi,R)=\prod_{v\notin S} L(s,\pi_v,R),\;\; \Re(s)\gg 0.$$
	Decomposing $\pi=\pi_W\boxtimes \pi_V$ where $\pi_W$, $\pi_V$ are cuspidal representations of $\SO(W)$ and $\SO(V)$ respectively, we will also write $L^S(s,\pi_W\times \pi_V)$ instead of $L^S(s,\pi,R)$.
\end{paragr}

\begin{paragr}
	For the statement of the conjecture we need to assume that this $L$-function admits, as expected, holomorphic continuation. We recall that this follows from the existence of functorial lifts $\pi_W^{\GL}$, $\pi_V^{\GL}$\index{$\pi_W^{\GL}$, $\pi_V^{\GL}$} of $\pi_W$, $\pi_V$ to $\GL_{d_W}$, $\GL_{d_V}$ respectively, where we have denoted by $d_W$, $d_V$ the dimensions of the standard representations $R_W$ and $R_V$ respectively. More precisely, the local components of the representations $\pi_W^{\GL}$ and $\pi_V^{\GL}$ should be related to those of $\pi_W$ and $\pi_V$ at almost all places $v$ by the following requirement on Langlands-Satake parameters: 
	$$\displaystyle \mathcal{S}_v(\pi^{\GL}_{W})=R_W(\mathcal{S}_v(\pi_{W})) \mbox{ and } \mathcal{S}_v(\pi^{\GL}_{V})=R_V(\mathcal{S}_v(\pi_{V,v})).$$
	This makes the automorphic representations $\pi_W^{\GL}$ and $\pi_V^{\GL}$ unique provided we require that they are isobaric sums of cuspidal representations on smaller general linear groups. In the case where $G$ is split, the existence of these lifts was established by Arthur in his seminal work \cite{artbook}.
	
	Assuming the existence of $\pi_W^{\GL}$, $\pi_V^{\GL}$ from now on, $L^S(s,\pi_W\times \pi_V)$ coincides with the (partial) Rankin-Selberg $L$-function of the pair $(\pi^{GL}_W, \pi^{GL}_V)$:
	$$\displaystyle L^S(s,\pi_W\times \pi_V)=L^S(s,\pi^{GL}_W\times \pi^{GL}_V)$$
	and as is well-known \cite{JPSS}, this $L$-function admits a holomorphic continuation to $\mathbb{C}$ and can be completed, by adding suitable local factors at all ramified places (including Archimedean places), to a completed $L$-function $L(s,\pi^{GL}_W\times \pi^{GL}_V)$, that we shall henceforth denote by $L(s,\pi_W\times \pi_V)$\index{$L(s,\pi_W\times \pi_V)$}, satisfying a functional equation of the form $s\leftrightarrow 1-s$.
\end{paragr}

\begin{paragr}
	We can now state the conjecture of Gross-Prasad in the case considered as follows (see \cite{GP}).

	\begin{conj}[Gross-Prasad]\label{conj global GGP}
		Assume that the functorial lift $\pi^{\GL}=\pi_W^{\GL}\boxtimes \pi_V^{\GL}$ is a generic automorphic representation of $\GL_{d_W}(\mathbb{A})\times \GL_{d_V}(\mathbb{A})$. Then, the following assertions are equivalent:
		\begin{enumerate}[(i)]
			\item $L(\frac{1}{2}, \pi_W\times \pi_V)\neq 0$ and for every place $v$ of $F$, $\Hom_{H(F_v)}(\pi_v,\mathbb{C})\neq 0$;
			
			\item there exists $\varphi\in \pi$ such that $\mathcal{P}_H(\varphi)\neq 0$.
		\end{enumerate}
	\end{conj}
	
	\begin{rem}
		\begin{itemize}
			\item The condition that the lift $\pi^{\GL}$ be generic should be seen as a substitute for the condition that the global A-parameter $\psi: L_F\times \SL_2(\bC)\to {}^LG$ be generic (that is trivial on the $\SL_2(\bC)$ factor). According to the generalized Ramanujuan conjecture, this should also be equivalent to the condition: the local component $\pi_v$ is tempered for every $v\in V_F$.
			
			\item In the above statement, it is important to work with the completed $L$-function $L(s,\pi_W\times \pi_V)$ (as opposed to a partial $L$-function). Indeed, without the Ramanujan conjecture for $\pi$, we can't ensure that the local factors $L(s,\pi_{W,v}\times \pi_{V,v})$ don't have poles at $s=1/2$ (whereas, if $\pi_{W,v}$, $\pi_{V,v}$ are tempered, all the local poles will be in the half-plane $\{ \Re(s)\leq 0\}$) which could mean that $L(\frac{1}{2},\pi_W\times \pi_V)\neq 0$ whereas $L^{\{v\}}(\frac{1}{2},\pi_W\times \pi_V)=0$.
			
			\item Similarly to the local conjecture, we can also give a statement that includes all the {\em pure inner forms} $(G_\alpha, H_\alpha)$ ($\alpha\in H^1(F,H)$) of the pair $(G,H)$. Indeed, we also conjecturally have an equivalence between (see \cite[\S 26]{GGP}):
			\begin{enumerate}
				\item $L(\frac{1}{2}, \pi_W\times \pi_V)\neq 0$;
				
				\item There exists a pure inner form $(G_\alpha, H_\alpha)$ of $(G,H)$, a cuspidal representation $\pi'\subset \cA_{\cusp}(G_\alpha)$ with the same $\GL$-lift $(\pi')^{\GL}=\pi^{\GL}$ as $\pi$ and such that $\mathcal{P}_{H_\alpha}\mid_{\pi'}\neq 0$.
			\end{enumerate}
			Actually, we can show using the local GGP conjecture and Arthur's multiplicity formula (which is again known for $G$ quasi-split case by \cite{artbook}) that this version is equivalent to Conjecture \ref{conj global GGP}. Furthermore, the condition that $(\pi')^{\GL}=\pi^{\GL}$ can be read as saying that $\pi$ and $\pi'$ have the same {\em global $L$-parameter} (modulo the usual ambiguity of conjugacy by the full orthogonal group for $\SO_{2m}$).
			
			\item As for the local conjecture, there is a variant of the conjecture for groups of the form $G=\SO_{2n+1}\times \SO_{2m}$ as well as variants for other classical groups. We refer the reader again to \cite{GGP} for a (very) complete discussion.
		\end{itemize}
	\end{rem}
\end{paragr}

\begin{paragr}[Local periods.]
	We now turn to a refinement of the Gross-Prasad conjecture that has been proposed by Ichino and Ikeda \cite{IIk}. This refinement takes the form of an exact identify relating (the square of) the automorphic period $\mathcal{P}_H \mid_\pi$ to the central value $L(\frac{1}{2},\pi_W\times \pi_V)$.
	
	This requires the introduction of local periods which are most easily defined when the local components $\pi_v$ of $\pi$ are all tempered representations which we assume from now on. (Recall that the Ramanujan conjecture predicts that this happens if and only if the $A$-parameter of $\pi$ is generic, which is needed in the GGP conjecture.)
	
	We fix a factorization $\pi=\bigotimes_v' \pi_v$ (which in particular depends on a choice of normalized spherical vectors $\phi_v^\circ\in \pi_v$\index{$\phi_v^\circ$} for almost all places $v$) as well as an invariant inner product $\langle .,.\rangle_v$\index{$\langle .,.\rangle_v$} on $\pi_v$ for every place $v$ of $F$. As discussed in Section \ref{Sect GGP}, the variety $X=H\backslash G$ is strongly tempered which entails that matrix coefficients of the tempered representation $\pi_v$ are integrable over $H(F_v)$. The local period $\mathcal{P}_{H_v}$ is then defined as the sequilinear form
	$$\displaystyle \mathcal{P}_{H_v}: \pi_v\times \overline{\pi_v}\to \mathbb{C},\;\; \mathcal{P}_{H_v}(\varphi_v,\varphi'_v)=\int_{H(F_v)} \langle \pi_v(h_v) \varphi_v,\varphi'_v \rangle_v dh_v.$$
	
	We assume that the local inner products as well as the Haar measures $dh_v$ have been chosen so that they factorize the Petersson inner product and the Tamagawa measure on $H(\bA)$:
	$$\displaystyle \langle.,.\rangle_{\Pet}=\prod_v \langle .,.\rangle_v,\;\; d_{\Tam}h=\prod_v dh_v$$
	with $\langle \phi^\circ_v,\phi_v^\circ\rangle_v=1$ (where $\phi^\circ_v\in \pi_v$ stands for the normalized spherical vector) and $\vol(H(\mathcal{O}_v),dh_v)=1$ for almost all $v$.
	
	Some (nontrivial) unramified computation done in \cite{IIk}, yields that
	\begin{equation}\label{unr computation}
		\displaystyle \mathcal{P}_{H_v}(\phi^\circ_v,\phi^\circ_v)=\Delta_{v} \frac{L(\frac{1}{2},\pi_{W,v}\times \pi_{V,v})}{L(1,\pi_v,\Ad)}
	\end{equation}
	for almost all places $v$, where $L(s,\pi_v,\Ad)$ is the local adjoint $L$-factor of $\pi_v$ and $\Delta_{v}=\Delta_{\SO(W),v}$ is the Tamagawa local unramified factor of the group $\SO(W)$, the same as the Tamagawa local unramified factor for the homogeneous variety $X=H\backslash G$ (and, in this case, a product of abelian $L$-factors that depend on the discriminant of $W$ see \cite{IIk} for an explicit description). (As a side note, we emphasize that, by the assumption that $\pi_v$ is tempered, the $L$-factor $L(s,\pi_{W,v}\times \pi_{V,v})$ doesn't have a pole at $s=\frac{1}{2}$.)
\end{paragr}

\begin{paragr}
	The Ichino-Ikeda conjecture can now be stated as follows:	
	\begin{conj}[Ichino-Ikeda]\label{global Ichino-Ikeda conj}
		Assume that at every place $v$, $\pi_v$ is tempered and that $\pi$ appears with multiplicity one in $\mathcal{A}_{\cusp}(G)$. Then, for every factorizable vector $\varphi=\bigotimes_v \varphi_v\in \pi=\bigotimes'_v \pi_v$, we have a factorization
		$$\displaystyle \left\lvert \mathcal{P}_H(\varphi)\right\rvert^2=\frac{1}{\lvert S_\pi\rvert}\Delta^S\frac{L^S(\frac{1}{2},\pi_{W}\times \pi_{V})}{L^S(1,\pi,\Ad)}\prod_{v\in S} \mathcal{P}_{H_v}(\varphi_v,\varphi_v)$$
		for any sufficiently large finite set of places $S$, where $\Delta^S=\Delta_{\SO(W)}^S$ denotes the  partial $L$-value appearing in the definition of the Tamagawa measures for $\SO(W)\simeq H\backslash G$ and $S_\pi$ ought to be the centralizer of the hypothetical $L$-parameter $\phi_\pi: L_F\to {}^L G$ of $\pi$ but for which we have the alternative, unconditional description in terms of the isobaric decompositions of the functorial lifts $\pi_W^{\GL}$, $\pi_V^{\GL}$. In particular, its order $\lvert S_\pi\rvert$ is always a power of $2$ (see \cite[\S 2]{IIk} for details).
	\end{conj}
	
	\begin{rem}\label{rem II conj}
		\begin{enumerate}[(1)]
			\item The formula does not depend on the choice of $S$ (as soon as it is sufficiently large) thanks to the unramified computation \eqref{unr computation}. In particular, it can formally (i.e.\ replacing partial $L$-values by the corresponding Euler products, even if not convergent) be rewritten as
			$$\displaystyle \left\lvert \mathcal{P}_H(\varphi)\right\rvert^2=\frac{1}{\lvert S_\pi\rvert}\prod'_{v} \mathcal{P}_{H_v}(\varphi_v,\varphi_v)$$
			where the notation $\prod_v'$ indicates that the product has to be regularized by interpreting the product over almost all places as a suitable ratio of special $L$-values.
			
			\item This conjecture is a direct generalization of Waldspurger's formula for toric periods on $\GL_2$ \cite{Waldtoric}, to which it essentially reduces when $G=\SO_2\times \SO_3$.
			
			\item The case of $\SO_3\times \SO_4$ is essentially equivalent to Ichino's triple product formula proven in \cite{Ichtriple}.
			
			\item Some cuspidal representation $\pi$ may appear with multiplicity $2$ in $\cA_{\cusp}(G)$. This can happen as soon as $n\geq 3$. More precisely, write $\{W,V\}=\{U,U'\}$ where $U$ is the even dimensional quadratic space and $U'$ the odd dimensional one. Then, $\pi_{U'}$ always appear with multiplicity one in $\cA_{\cusp}(\SO(U'))$ so that if $\pi$ has multiplicity two (in $\cA_{\cusp}(G)$) exactly when $\pi_U$ has multiplicity two (in $\cA_{\cusp}(\SO(U))$). The conjecture cannot hold as stated in this situation (as was pointed in \cite[Sect. 6]{Xue2}), because the period $\mathcal{P}_H\mid_\pi$ depends on the way $\pi$ is embedded in $\cA_{\cusp}(G)$. There are (at least) two ways around to circumvent this difficulty. First, an analysis based on Arthur's (conjectural) multiplicity formula for even special orthogonal groups suggests that $\pi_U$ appears with multiplicity two in $\cA_{\cusp}(\SO(U))$ exactly when there exists two equivalence classes of global $L$-parameters $\phi_1,\phi_2: L_F\to {}^L \SO(U)$ that are locally everywhere equivalent to the local $L$-parameter $\phi_{\pi_{U,v}}$ of $\pi_U$ but not globally equivalent. Thus, the $\pi_U$-isotypic component $\cA_{\cusp}(\SO(U))_{\pi_U}$ should admit the canonical decomposition $\cA_{\cusp}(\SO(U))_{\pi_U}=\pi_U\oplus \pi_U$ where the first (resp. second) copy of $\pi_U$ is contained in the conjectural subspace of cusp forms $\cA_{\cusp}(\SO(U))_{\phi_1}$ (resp. $\cA_{\cusp}(G)_{\phi_2}$) associated to the global parameter $\phi_1$ (resp. $\phi_2$). This should then give two canonical copies of $\pi$ inside $A_{\cusp}(G)$ and the conjecture should hold as stated for the restriction of the period $\mathcal{P}_H$ to any of these two particular copies of $\pi$. This particular case of the Arthur decomposition seems however still difficult to define in the number field case (for function fields however, the spaces of automorphic forms $\cA_{\cusp}(\SO(U))_{\phi_1}$, $\cA_{\cusp}(\SO(U))_{\phi_2}$ can be defined thanks to the work of V. Lafforgue). Another way to make a precise statement, that does not assume the full Arthur decomposition for even special orthogonal groups, has been proposed by H. Xue in \cite[Sect. 6]{Xue2}. It consists in considering a certain particular copy of $\pi_U$ inside $\cA_{\cusp}(\SO(U))$ that is moreover invariant by conjugation by the ($F$-points of) full orthogonal group $O(U)$. The above conjecture should then hold for the corresponding copy of $\pi$ inside $\cA_{\cusp}(G)$ up to an extra factor of $2$.
			
			\item This conjecture has been extended to unitary groups by Neal Harris \cite{NHar} and then more generally to groups of the form $\SO_{2n+1}\times \SO_{2m}$, $U_{2n+1}\times U_{2m}$ (where the relevant periods are often called {\em Bessel periods}) by Y.Liu \cite{LiuBessel}. Finally, in the remaining cases of product of unitary groups $U_n\times U_m$ with $n$, $m$ of the same parity as well as symplectic/metaplectic groups (the so-called {\em Fourier-Jacobi} cases), a similar statement was proposed by H. Xue \cite{Xue1},\cite{Xue2}.
			
			\item The conjecture for unitary groups has now been fully established in \cite{BCZ}, \cite{BPC} in the case of Bessel periods and more recently in \cite{BLX} in the case of Fourier-Jacobi periods. The basic method employed is that of a comparison of relative trace formulas. We refer the reader to Chaudouard's contribution in these proceedings for an introduction to this circle of ideas as well as to the introduction of the aforementioned papers for an overview of previous results and contributions on the global Gan-Gross-Prasad and Ichino-Ikeda conjectures (which are numerous).
			
			\item Part of the proof of the local Gross-Prasad conjecture \cite{WalGP2} requires to show that for any local tempered irreducible representation $\pi_v$ we have the equivalence
			$$\displaystyle \mathcal{P}_{H,v}\mid_{\pi_v}\neq 0 \Leftrightarrow \Hom_{H_v}(\pi_v,\bC)\neq 0.$$
			Therefore, at least under the seemingly stronger (but conjecturally not) assumption that $\pi$ is tempered everywhere, the (global) Gross-Prasad conjecture can be deduced from the Ichino-Ikeda conjecture.
		\end{enumerate}
	\end{rem}
\end{paragr}

\subsection{Normalized relative characters and unramified computations}\label{S normalized rc}

In order to generalize the Ichino-Ikeda conjecture we reinterpret, following Sakellaridis and Venkatesh, the local periods that it involves in terms of {\em local Plancherel formulas}. More precisely, as recalled in \S \ref{S strongly tempered}, the dual periods are ``dual'' to the relative characters $J_{\pi_v}^X: \mathcal{S}(X_v)\times \mathcal{S}(X_v)\to \bC$, where we have set $X=\SO(W)\backslash \SO(W)\times \SO(V)$, that appear in the Plancherel decomposition for $X_v$:\index{$J_{\pi_v}^X$}
\begin{equation}\label{eq Plancherel formula II}
\displaystyle \langle f_1,f_2\rangle_{X_v}=\int_{\Temp(G_v)} J_{\pi_v}^X(f_1,f_2) d\mu_{G_v}(\pi_v).
\end{equation}

\vspace{2mm}

\begin{paragr}\label{S theta series}
	The Ichino-Ikeda conjecture can then be completely reformulated in terms of these relative characters, a point of view that we will fully adopt here. More precisely, the reformulation requires to replace similarly the automorphic period $\mathcal{P}_H$ by its ``dual'' object, the {\em $X$-theta series map}, namely the automorphic realization morphism\index{$\Theta_X$}
	$$\displaystyle \Theta_X: \mathcal{S}(X(\bA))=\bigotimes_v'\mathcal{S}(X_v)\to C^\infty([G])$$
	given by summation over rational points:
	$$\displaystyle \Theta_{X}(f,g)=\sum_{x\in X(F)} f(xg),\;\;\; g\in G(\bA).$$
	
	For $\pi\subset \cA_{\cusp}(G)$ and $f\in \mathcal{S}(X(\bA))$, we let $\Theta_X(f)_\pi$\index{$\Theta_X(f)_\pi$} be the orthogonal projection of $\Theta_X(f)$ to $\pi$. Then, the Ichino-Ikeda conjecture can be reformulated as the identity, for $\pi$ tempered everywhere,
	\begin{equation}\label{eq reformulation II}
		\displaystyle \langle \Theta_X(f)_\pi,\Theta_X(f)_\pi\rangle_{\Pet}=\frac{1}{\lvert S_\pi\rvert} \Delta^S_X \frac{L^S(\frac{1}{2},\pi_W\times \pi_V)}{L^S(1,\pi,\Ad)}\prod_{v\in S} J_{\pi_v}^X(f_v,f_v)=\frac{1}{\lvert S_\pi\rvert}\prod'_{v} J_{\pi_v}^X(f_v,f_v)
	\end{equation}
	for every $f=\prod_v f_v\in \mathcal{S}(X(\bA))$ and where the last product has to be interpreted ``in the sense of $L$-functions'' as in Remark \eqref{rem II conj}.
\end{paragr}

\begin{paragr}
	Let us be a bit more precise on the implicit normalization of measures in \eqref{eq reformulation II}. Indeed, from the local Plancherel formula \eqref{eq Plancherel formula II}, we see that the relative characters $J_{\pi_v}^X$ actually depend on the choice of an invariant measure on $X_v$ (needed to fix the $L^2$-inner product $\langle .,.\rangle_{X_v}$) and of the Plancherel measure, which in turn is inversely proportional to the choice of a Haar measure on $G_v$.
	
	To fit with the Ichino-Ikeda normalization, we pick local invariant measures $dx_v$, $dg_v$ that factorize the Tamagawa measures
	$$\displaystyle d_{\Tam}x=\prod_v dx_v,\;\; d_{\Tam} g=\prod_v dg_v$$
	and are normalized such that $\vol(X(\mathcal{O}_v),dx_v)=\vol(G(\mathcal{O}_v),dg_v)=1$ for almost all places $v$. In particular, with these choices, the analog of the formula \eqref{unr computation} reads:
	\begin{equation*}
		\displaystyle J_{\pi_v}^{X}(\mathbf{1}_{X(\mathcal{O}_v)},\mathbf{1}_{X(\mathcal{O}_v)})=\Delta_{X,v}\frac{L(\frac{1}{2},\pi_{W,v}\times \pi_{V,v})}{L(1,\pi_v,\Ad)}
	\end{equation*}
	for almost all $v$.
	
	Note that, we could have used local Tamagawa measures $d_{\Tam}x_v$ $d_{\Tam}g_v$ instead. (Recall that, at least formally, they also factorize the global Tamagawa measure.) This would lead to a different normalization of the local relative characters $J_{\pi_v}^{X,\Tam}$ with unramified values
	\begin{equation}\label{unr computation local RC II}
		\displaystyle J_{\pi_v}^{X,\Tam}(\mathbf{1}_{X(\mathcal{O}_v)},\mathbf{1}_{X(\mathcal{O}_v)})=\Delta^{-1}_{G,v}\frac{L(\frac{1}{2},\pi_{W,v}\times \pi_{V,v})}{L(1,\pi_v,\Ad)}
	\end{equation}
	for almost all $v$. Note that this has just the effect of replacing the normalizing factor $\Delta_{X,v}$ by $\Delta^{-1}_{G,v}$. Finally, we could also decide to use the local Tamagawa measures $d_{\Tam}x_v$ on $X_v$ but normalized local measures $dg_v$ on $G_v$, in which case the normalizing factor will just disappear, but this seems a bit artificial.
\end{paragr}

\begin{paragr}[Normalized local relative characters.]
	More generally, for any unimodular homogeneous spherical variety $X=H\backslash G$ without type $N$ root, the local Conjecture \ref{conj1 SV} should give rise to relative characters\index{$J_{\phi_v}^X$}
	$$\displaystyle J_{\phi_v}^X: \mathcal{S}(X_v)\times \mathcal{S}(X_v)\to \bC,$$
	indexed by the tempered $L$-parameters $\phi_v: L_{F_v}\to {}^L G_X$, such that the Plancherel formula for $X_v$ reads
	$$\displaystyle \langle f_1,f_2\rangle_{X_v}=\int_{\Phi_{\temp, v}(X)} J_{\phi_v}^X(f_1,f_2) d\mu_X(\phi_v),\;\; f_1,f_2\in \mathcal{S}(X_v),$$
	where  $d\mu_X$ is in the natural class of measures on $\Phi_{\temp,v}(X)$ (as explained in \S \ref{S local conjecture weak form}). (To emphasize the dependence on the place $v$, we are denoting here by $\Phi_{\temp,v}(X)$ the space of tempered $L$-parameters $L_{F_v}\to {}^L G_X$.) We remind the reader that, although the $L$-parameters $\phi_v$ are tempered, the relative characters $J_{\phi_v}^X$ should be supported on the Arthur packet of $G$ associated to the $A$-parameter the composition
	$$L_{F_v}\times \SL_2(\bC)\xrightarrow{\phi_v\times \Id} {}^L G_X\times \SL_2(\bC)\xrightarrow{{}^L\iota_X} {}^L G$$
	where ${}^L\iota_X: {}^L G_X\times \SL_2(\bC)\xrightarrow{{}^L\iota_X} {}^L G$ is the strongly distinguished morphism introduced in \S \ref{S dual group}. In particular, if the restriction $\iota_X^{\SL_2}: \SL_2(\bC)\to {}^L G$ of the latter to the $\SL_2$-factor is nontrivial, the relative characters $J_{\phi_v}^X$ are usually not supported on tempered representations of $G(F_v)$.
	
	In order to normalize these relative characters precisely (at least for $\mu_X$-almost every $\phi_v$), we need to specify an invariant measure on $X_v$ as well as pick a specific Plancherel measure $d\mu_X$. The first point is easily dealt with: we will choose the local Tamagawa invariant measures $d_{\Tam}x_v$ on the $X_v$'s. They eventually depend on the choice of a global volume form $\omega_X$ together with a character $\psi:\bA/F\to \bC^\times$ and satisfy $\vol(X(\mathcal{O}_v),dx_v)=\Delta_{X,v}^{-1}$ for almost all $v$. For the Plancherel measures $d\mu_X$ on the other hand, the idea is to use the local Plancherel measures for the quasi-split group $G_X$ whose $L$-group is ${}^L G_X$. It turns out that, according to a conjecture of Hiraga, Ichino and Ikeda \cite{HII}, this Plancherel measure should naturally descend, up to certain rational factors, to the space of tempered Langlands parameters $\Phi_{\temp,v}(X)$ where it is given by an explicit formula involving adjoint gamma factors. 
	
	We will recall this expected description in the next paragraph but, before doing so, let us remark on one technical issue that arises. Namely, ${}^L G_X$ is not always an $L$-group (see \S \ref{S L-groups}) and therefore a quasi-split group $G_X$ as above does not necessarily exist. We can still define $G_X$ by looking at the outer Galois action on $\widehat{G}_X$ that ${}^L G_X$ induces, but there is no (conjectural) direct relation between tempered representations of $G_X(F_v)$ and $L$-parameters $L_{F_v}\to {}^L G_X$. This situation actually also arises in the theory of endoscopy and the way around is well-known: we should introduce a certain central extension $\tilde{G}_X\to G_X$ (a so-called $z$-extension) and work with representations of $\tilde{G}_X(F_v)$ with a fixed central character on the kernel $\omega: [Z]\to \mathbb{S}^1$. For our purpose, this means that we should consider the Plancherel measure for $L^2(Z_v\backslash \tilde{G}_{X,v},\omega_v)$ rather than the Plancherel measure for $G_{X,v}$. However, we will ignore this technical issue in the discussion below, not only because it can be dealt with very easily using such a $z$-extension, but also because the final formula for the resulting canonical choice of $d\mu_X$ will be completely uniform.
\end{paragr}

\begin{paragr}[Canonical Plancherel measures.]\label{S canonical Planch measure}
	Below we recall the conjectural description of the Plancherel measure $d\mu_{G_X}$\index{$d\mu_{G_X}$} for $G_{X,v}=G_X(F_v)$ from \cite{HII} but first let us comment on the normalization of measures (again!). As already pointed out, a Plancherel measure for $G_{X,v}$ a priori depends on the choice of a Haar measure on the latter (and is actually inversely proportional to this choice). We would like to take the local Tamagawa measure, but that depends on the choice of a volume form $\omega_{G_X}$ as well as that of an additive (unitary) character $\psi_v: F_v\to \bC^\times$, recall that $d_{\Tam}g_{X,v}=\lvert \omega_{G_X}\rvert_{\psi_v}$. However, it turns out that there is, up to a sign, a canonical choice of such a volume form (see \cite{HIIcorr}): namely that of a non-vanishing invariant volume form $\omega_{G_X}^{\can}$ on the split form $G_{X,\bZ}^s$ of $G_X$ over $\bZ$. Note that $\omega_{G_X}^{\can}$ can be transferred to $G_{X}$ over an algebraic closure $\overline{F}_v$ via the choice of an isomorphism $G_X\times \overline{F}_v\simeq G_{X,\bZ}\times \overline{F}_v$ and then turned into a Haar measure using Weil's construction, and the additive character $\psi_v$, namely by setting $d^{\can}_{\Tam}g_{X,v}=\lvert \lambda\rvert_{v}^{-1}\lvert \lambda \omega_{G_X}^{\can}\rvert_{\psi_v}$\index{$d^{\can}_{\Tam}g_{X,v}$} for any $\lambda\in \overline{F}_v^\times$ for which $\lambda \omega_{G_X}^{\can}$ is defined over $F_v$ and where $\lvert .\rvert_v$ denotes the (unique) extension of the normalized absolute value on $F_v$ to $\overline{F}_v$. This ``canonical'' local Tamagawa measure still depends on the additive character $\psi_v$ but as we will see below the spectral Plancherel density of Hiraga, Ichino and Ikeda also depends on such a choice.
	
	According to Harish-Chandra, the Plancherel measure $d\mu_{G_X}$ is supported on the tempered dual $\Temp(G_{X,v})$ and furthermore decomposes as a product
	$$\displaystyle d\mu_{G_X}(\pi_v)=\mu_{G_X}(\pi_v) d\pi_v$$
	where $\mu_{G_X}(\pi_v)$ is the so-called {\em Plancherel density} and $d\pi_v$ is a certain ``standard'' measure on $\Temp(G_{X,v})$, roughly coming from Haar measures on the groups of unitary unramified characters of Levi subgroups, but for which we would like to propose the following explicit normalization. Recall that any tempered representation $\pi_v$ can be embedded into the (normalized) parabolic induction $I_{M_v}(\sigma_v)$ of a discrete series $\sigma_v$ of a Levi subgroup $M_v\subset G_{X,v}$ with the pair $(M_v,\sigma_v)$ being unique up to conjugacy. The standard measure $d\pi_v$ then factors through the space of such ``discrete data'' $[M_v,\sigma_v]$ (where the brackets indicate that it is considered up to conjugacy) and decomposes on the latter as the sum of measures on inertial orbits
	$$\displaystyle \mathcal{O}=\{[M_v,\sigma_v\otimes \lambda]\mid \lambda\in \mathcal{X}^{\unit}(M_v)\},$$
	where $\mathcal{X}^{\unit}(M_v)$ stands for the group of unitary unramified characters on the {\em maximal split torus quotient} $A^{M_v}$ of $M_v$\footnote{Note that $\mathcal{X}^{\unit}(M_v)$ is not necessarily the same as the group of unitary unramified characters of $M_v$, essentially because the natural morphism $M_v\to A^{M_v}$ is not always surjective (on $F_v$-points). However, it is the torus of unramified characters of $A^{M_v}$ that seems the more natural on the dual side.} given by ``transferring'' a certain Haar measure $d\lambda$ on $\mathcal{X}^{\unit}(M_v)$ normalized as explained below. Let us be more precise on how to ``transfer'' the measure from $\mathcal{X}^{\unit}(M_v)$ to $\mathcal{O}$: fixing a base-point identifies the inertial orbit $\mathcal{O}$ with a quotient of $\mathcal{X}^{\unit}(M_v)$ by a finite group $W_{\pi_v}'$ of affine automorphisms and we take the restriction $d\pi_v\mid_{\mathcal{O}}$ to be the pushforward of the measure $d\chi$ on $\mathcal{X}^{\unit}(M_v)$ (still to be normalized) divided by $\lvert W_{\pi_v}'\rvert$. In other words, this is the unique measure in the natural measure class on $\mathcal{O}$ such that near a point in general position (corresponding to a pair $(M_v,\sigma_v)$ with normalizer $M_v$ in $G_{X,v}$) the action is locally measure preserving. Finally, the normalization we take for the Haar measure $d\chi$ on $\mathcal{X}^{\unit}(M_v)$ is as follows:
	\begin{itemize}
		\item If $v$ is non-Archimedean, we have a natural identification $\mathcal{X}^{\unit}(M_v)=A_{\widehat{M}_v}^1$, where $A_{\widehat{M}_v}$ denotes the connected center of the $L$-group ${}^L M_v$ (that is $(Z(\widehat{M}_v)^{\Gamma_v})^0$) and $A_{\widehat{M}_v}^1$ its maximal compact subgroup. We then equip $A_{\widehat{M}_v}^1$ with the Haar measure of total mass $1$.
		
		\item If $v$ is Archimedean, by differentiation at $1$ we get an identification $\mathcal{X}^{\unit}(M_v)=\Lie(A_{\widehat{M}_v}^1)$ and we transfer the Haar measure on $A_{\widehat{M}_v}^1$ to its Lie algebra by requiring that the Jacobian of the exponential map is $1$ at the origin.
	\end{itemize}
	
	The Plancherel density on the other hand, itself decomposes as a product of the formal degree of $\sigma_v$ with the inverse of a certain product of standard intertwining operators (see e.g. \cite{WaldPlanch}). According to a conjecture of Hiraga-Ichino-Ikeda on formal degrees and of Langlands on the normalization of standard intertwining operators, it should be given by the following formula
	\begin{equation}\label{HII conjecture}
		\displaystyle \mu_{G_X}(\pi_v)=\frac{\rho_{\pi_v}(1)}{\lvert S_{\pi_v}\rvert} \lvert \gamma^*(0,\pi_v,\Ad_{X},\psi_v)\rvert
	\end{equation}
	where:
	\begin{itemize}
		\item $\Ad_{X}$\index{$\Ad_{X}$} stands for the adjoint representation of ${}^LG_X\curvearrowright \widehat{\mathfrak{g}}_X$,
		$$\displaystyle \gamma(s,\pi_v,\Ad_{X},\psi_v)=\epsilon(s,\Ad_X\circ \phi_{\pi_v},\psi_v)\frac{L(1-s,\Ad_X\circ \phi_{\pi_v})}{L(s, \Ad_X\circ \phi_{\pi_v})}$$
		is the corresponding $\gamma$-factor (where the $\epsilon$-factor is normalized as in \cite[\S 5]{GGP}), with $\phi_{\pi_v}: L_{F_v}\to {}^L G_X$ the $L$-parameter of $\pi_v$, and, when $\pi_v\subset I_{M_v}(\sigma_v)$ with $\sigma_v$ a square-integrable re presentation of the Levi $M_v$,\index{$\gamma^*(0,\pi_v,\Ad_{X},\psi_v)$}
		$$\displaystyle \gamma^*(0,\pi_v,\Ad_{X},\psi_v):=\zeta_{F_v}(s)^{a_{M_v}}\gamma(s,\pi_v,\Ad_{X},\psi_v)\mid_{s=0}$$
		is a regularization of its value at $s=0$, $\zeta_{F_v}(s)$ being the local zeta factor of $F_v$ and $a_{M_v}$ the dimension of the split center of $M_v$, over $F_v$.
		
		\item $S_{\pi_v}=\pi_0(Z_{\widehat{G}_X}(\phi_{\pi_v}))$\index{$S_{\pi_v}$} denotes the component group of the centralizer of the $L$-parameter of $\pi_v$ and $\rho_{\pi_v}$ is the irreducible character of that finite group conjecturally parametrizing $\pi_v$ inside its $L$-packet. (This irreducible character a priori depends on an auxilliary choice, namely that of a Whittaker datum, but only up to a twist so that its degree $\rho_{\pi_v}(1)$ shouldn't depend on that choice.)
	\end{itemize}
	
	All the ingredients of the formula \eqref{HII conjecture}, except for $\rho_{\pi_v}(1)$, only depend on the $L$-parameter of $\pi_v$. This suggests defining the following measure on the set $\Phi_{\temp,v}(X)$ of tempered $L$-parameters:\index{$d\mu_{X}(\phi)$}
	$$\displaystyle d\mu_{X}(\phi):= \mu_{X}(\phi) d\phi$$
	where:
	\begin{itemize}
		\item $$\displaystyle \mu_{X}(\phi)=\lvert S_{\phi}\rvert^{-1} \lvert \gamma^*(0,\Ad_{X}\circ \phi,\psi_v)\rvert$$
		with $S_\phi=\pi_0(Z_{\widehat{G}_X}(\phi))$ and
		$$\displaystyle \gamma^*(0,\Ad_{X}\circ \phi,\psi_v)=\zeta_{F_v}(s)^{a_M}\gamma(s,\Ad_{X}\circ \phi,\psi_v)\mid_{s=0},$$
		$a_M$ denoting the dimension of the connected center $A_{\widehat{M}_v}$ of a minimal Levi subgroup ${}^L M_v\subset {}^L G_{X,v}$ through which $\phi$ factors;
		
		\item $d\phi$ is the ``standard'' measure on $\Phi_{\temp,v}(X)$ described as follows: in the neighborhood of any $\phi\in \Phi_{\temp,v}(X)$, picking a minimal Levi subgroup ${}^L M_v\subset {}^L G_{X,v}$ through which $\phi$ factors, $d\phi$ is given by transferring the normalized Haar measure on $A_{\widehat{M}_v}^1$ ($=$ the maximal compact subgroup in $A_{\widehat{M}_v}$) or its Lie algebra $\mathfrak{a}_{\widehat{M}_v}^1)$ (depending whether $v$ is non-Archimedean or Archimedean) to $\Phi_{\temp,v}(X)$ via the natural twisting action $\lambda\mapsto \phi_\lambda$.
	\end{itemize}
	
	As an illustration of the relevance of this measure, we note that by the conjectural formula \eqref{HII conjecture}, the Plancherel formula for $G_{X,v}$ (here considered as a spectral decomposition of the Dirac at $1$) can be rewritten in the stable form
	$$\displaystyle f(1)=\int_{\Phi_{\temp,v}(X)} \Theta_\phi(f) d\mu_X(\phi)$$
	where $\Theta_\phi=\sum_{\pi\in \Pi_\phi} \rho_\pi(1) \Theta_{\pi}$ is the stable  linear combination of characters from the $L$-packet $\Pi_\phi$ of $\phi$.
\end{paragr}

\begin{paragr}[Spectral density for the basic function.]
	Taking our choice of Plancherel measures $d\mu_X$ for the (conjectural) Plancherel decompositions of $L^2(X_v)$ parametrized by $\Phi_{\temp,v}(X)$ as in the previous paragraph, we get well-defined local relative characters $J_{\phi_v}^X$\index{$J_{\phi_v}^X$} for ($\mu_X$-almost) all $\phi_v\in \Phi_{\temp,v}(X)$. The last, but crucial, ingredient to formulate the generalization of Ichino-Ikeda formula due to Sakellaridis-Venkatesh, is that of an {\em unramified computation} similar to \eqref{unr computation local RC II}. More precisely, for almost all places $v$ the Plancherel decomposition of the basic characteristic function $f^\circ_v=\mathbf{1}_{X(\mathcal{O}_v)}$ should look like
	\begin{equation}\label{Plancherel decomposition basic vector}
		\displaystyle \langle f_v^\circ, f_v^\circ\rangle_{X_v}=\lvert W_{X,v}\rvert^{-1} \Delta_{G,v}^{-1}\int_{\widehat{A}_{X,v}^1} \Omega(\chi_v)d\chi_v,
	\end{equation}
	where:
	\begin{itemize}
		\item $\widehat{A}_{X,v}$\index{$\widehat{A}_{X,v}$} stands for the torus of unramified characters of $A_X(F_v)$ identified with the $\Frob_v$-coinvariants in $\widehat{A}_X$, that is $\widehat{A}_{X,v}=\widehat{A}_X/(1-\Frob_v)\widehat{A}_X$, in particular we are assuming that at the places $v$ considered the local Galois action $\Gamma_{v}\curvearrowright \widehat{G}_X$ factors through the unramified quotient $\Gamma_v/I_v=\overline{\langle \Frob_v\rangle }$, $\widehat{A}^1_{X,v}\subset \widehat{A}_{X,v}$\index{$\widehat{A}^1_{X,v}$} is its maximal compact subgroup and $d\chi_v$ the Haar measure on the latter with total mass $f^{-1}$ where $f$ denotes the cardinality of the kernel of the natural morphism $(\widehat{A}_X^{\Frob_v})^0\to \widehat{A}_{X,v}$\footnote{Note that, under the identification of $\widehat{A}_{X,v}/W_{X,v}$ with the space of unramified tempered $L$-parameters $W_v/I_v\to {}^L G_X$, $d\chi_v$ corresponds to the ``standard'' measure $d\phi$ on $\Phi_{\temp,v}(X)$.}.
		\item $W_{X,v}=W_X^{\Frob_v}$\index{$W_{X,v}$} is the subgroup of $\Frob_v$-fixed points in the Weyl group of $X$.
		\item $\Omega(\chi_v):=\Delta_{G,v} J_{\chi_v}^X(f_v^\circ,f_v^\circ) \mu_X(\chi_v)$\index{$\Omega(\chi_v)$} where we are identifying $\chi_v$ with the $L$-parameter $W_v/I_v\to {}^L G_X$ sending $\Frob_v$ to $(\chi_v, \Frob_v)$. We emphasize that the introduction of the Tamagawa factor $\Delta_{G,v}$ here is purely cosmetic as it eventually leads to a simpler expression, see Conjecture \ref{local unramified conjecture} below and more specifically comment \eqref{comment BZSV} after it. We emphasize that this constant eventually depends on our choice of local measures (recall that we are endowing  $X_v$ with local Tamagawa measures).
	\end{itemize}
	When the group $G$ is split over $F_v$, and under extra assumptions, such a spectral decomposition follows from the work of Sakellaridis \cite{SakSph} or Sakellaridis and Venkatesh \cite{SV} and can be described explicitly through the {\em Radon transform} of $f_v^\circ$ cf. \S \ref{S unramified Plancherel}. More precisely, $\Omega(\chi_v)$ is given, for a suitable normalization of measures, by the (square of the) Mellin transform of the push-forward of $f_v^\circ$ to $A_{X,v}=A_X(F_v)$ along the quotient map $X_{B,v}\to A_{X,v}$, where $X_{B,v}\subset X_v$ denotes the $F_v$-points of an open Borel orbit. Note that this push-forward is readily seen (for almost all $v$) to be $A_X(\mathcal{O}_v)$ invariant, which implies that its Mellin transform is indeed supported on the torus of unramified characters $\widehat{A}_{X,v}^1$ of $A_{X,v}$.
	
	
	When the map $\widehat{A}_{X,v}/W_{X,v}\to \widehat{A}_v/W_v$ (where, as above, $\widehat{A}_v$ denotes the torus of $\Frob_v$-coinvariants in $\widehat{A}$ and $W_v=W^{\Frob_v}$) is injective (implying in particular that $\widehat{G}_X\subset \widehat{G}$), the unramified measure $\Omega(\chi_v)d\chi_v$ can also be characterized by the following property: for every unramified Hecke operator $T\in \mathcal{H}(G_v,G(\mathcal{O}_v))$ we have
	$$\displaystyle \langle T\star f_v^\circ,f_v^\circ\rangle_{X_v}=\lvert W_{X,v}\rvert^{-1} \int_{\widehat{A}_{X,v}^1} \widehat{T}(\chi_v)\Omega(\chi_v)d\chi_v$$
	where $\widehat{T}$ denotes the Satake transform of $T$, identified with a regular function on $\widehat{A}_v$, and $T\star f_v^\circ$ the convolution of $f_v^\circ$ by $T$. In general, the unramified spectrum might not be multiplicity free and, although the above identity of course still holds, it does not characterize the measure $\Omega(\chi_v)d\chi_v$ anymore. Rather, as explained in \S \ref{S unramified Plancherel}, the decomposition \eqref{Plancherel decomposition basic vector} is directly related to the Bernstein map $\iota_\emptyset$ from the asymptotic cone and therefore should be characterized by the asymptotic of the corresponding decomposition of $f_v^\circ$ in eigenfunctions.
\end{paragr}

\begin{paragr}[The local unramified conjecture.]
	Of great importance for the global conjecture, is the expression of the local unramified density $\Omega(\chi_v)$ or, equivalently, of the unramified relative characters $J_{\chi_v}^X(f_v^\circ,f_v^\circ)$, in terms of special values of unramified $L$-factors. More specifically, we have:
	
	\begin{conj}[Sakellaridis-Venkatesh]\label{local unramified conjecture}
		Assume that $X$ is affine. Then, there exists a graded algebraic representation $\rho_X=\bigoplus_{d\geq 1} \rho_{X,d}$\index{$\rho_X$} of the $L$-group ${}^L G_X$ such that for almost all places $v$, we have
		\begin{equation}
			\displaystyle \Omega(\chi_v)=\frac{L(\chi_v,\rho_X)}{L(0,\chi_v,\overline{\Ad}_{X})},\;\; \chi_v\in \widehat{A}_{X,v}^1,
		\end{equation}
		the notation $L(\chi_v,\rho_X)$\index{$L(\chi_v,\rho_X)$} being a short-hand for the product of $L$-values
		$$\displaystyle L(\chi_v,\rho_X)=\prod_{d\geq 1} L(\frac{d}{2},\rho_{X,d}\circ \chi_v)=\prod_{d\geq 1}\det(1-q_v^{-d/2}\rho_{X,d}(\chi_v))^{-1},$$
		where in the middle term we have considered $\chi_v$ as an $L$-parameter $W_{F_v}/I_{F_v}\to {}^L G_X$, and $\overline{\Ad}_{X}$\index{$\overline{\Ad}_{X}$} denoting the adjoint representation of ${}^L A_{X,v}$ on $\widehat{\mathfrak{g}}_X/\widehat{\mathfrak{a}}^{\Gamma_{F_v}}_X$ so that
		$$\displaystyle L(0,\chi_v,\overline{\Ad}_{X})=L(0,\overline{\Ad}_{X}\circ \chi_v)=\det(1-\Ad(\chi_v)\mid\widehat{\mathfrak{g}}_X/\widehat{\mathfrak{a}}^{\Gamma_{F_v}}_X)^{-1}.$$
		Equivalently, since the Plancherel density $\mu_X$ on unramified $L$-parameters is given by
		$$\displaystyle \mu_X(\chi_v)=\frac{L(1,\chi_v,\Ad_{X})}{L(0,\chi_v,\overline{\Ad}_{X})},$$
		where $\Ad_X: {}^L G_X\to \GL(\widehat{\mathfrak{g}}_X)$\index{$\Ad_X$} denotes the full adjoint representation of ${}^L G_X$, we have the following expression for the local relative characters:
		\begin{equation}\label{formula unramified LRC}
			\displaystyle J^X_{\chi_v}(f_v^\circ,f_v^\circ)=\Delta_{G,v}^{-1}\frac{L(\chi_v,\rho_X)}{L(1,\chi_v,\Ad_{X})},\;\; \chi_v\in \widehat{A}_{X,v};
		\end{equation}
		
	\end{conj}
	
\end{paragr}

\begin{paragr}[Remarks on the local unramified conjecture.]\label{S rmks local unr conjecture}
	\begin{enumerate}[(1)]
		\item The ``for almost all places $v$'' is not very explicit but makes the statement easier to write. There should be a set of natural conditions, automatically satisfied at almost all places, and under which the above formula should be verified e.g. we certainly need to impose that $G_v$ is unramified, $G(\mathcal{O}_v)$ is the group of integral point of a reductive smooth model over $\mathcal{O}_v$ and similarly that $X(\mathcal{O}_v)$ is the set of integral points of a ``nice enough'' integral model. We haven't tried to formulate such a list but see \cite{SakSph} for the case where $G$ is split.
		
		\item The condition ``$X$ affine'' is not easy to motivate but is what appears in practice. In the case at hand, where $X=H\backslash G$ is homogeneous, this just means that $H$ is reductive. Although we didn't include it systematically in our discussion, the conjecture should also holds for Whittaker induction of spherical varieties (in the sense of \S \ref{S Whittaker induction}), provided the notion of ``affine Whittaker induction'' is properly interpreted.
		
		\item The conjecture of course specializes to the unramified computation \eqref{unr computation local RC II} in the Gan-Gross-Prasad case $X=\SO_n\backslash \SO_n\times \SO_{n+1}$. In the Whittaker case, i.e.\ $X=(N,\psi)\backslash G$, such a formula is very closely related to the Casselman-Shalika computation of unramified Whittaker functions \cite{CasSha}. There are many other similar computations in the literature, to cite just a few \cite{Macdo}, \cite{HiSa}, \cite{KMSorth}, \cite{Offsph}, \cite{SakCS}. In \cite{SakSph}, Sakellaridis makes the first attempt to perform those unramified computations systematically using some combinatorial procedure based on the geometry of $B$-orbits in $X$ and, under some technical assumptions including that the group $G$ is split, he obtains a formula like \eqref{formula unramified LRC} but where $\rho_X$ is only a priori a graded representation of the maximal torus. Thus, the question becomes in this case of whether $\rho_X$ comes from a $\widehat{G}_X$-representation. In particular, the results from \cite{SakSph} describe explicitly the multiset of weights of $\rho_X$ in terms of the so-called {\em colors} of $X$ (that is the set of $B$-invariant divisors). This multiset is $W_X$-invariant from which it follows that $\rho_X$ extends at least to a {\em virtual} representation of $\widehat{G}_X$ and that the corresponding $L$-factor is a ratio of $L$-factors for $\widehat{G}_X$. 
		
		\item We refer the reader to the last page of \cite{SakSph} for a table with many examples of such unramified computation. There are at least two particular examples that are worth mentioning here. The first one is the {\em group case}, that is $G=H\times H\curvearrowright H$, where ${}^L G_X={}^L H$ and $\rho_X=\Ad_H$ so that the right hand side of \eqref{formula unramified LRC} reduces to a constant (which here is non-trivial but could be made so by another choice of Haar measures), which is mostly obvious once we remember that we have normalized $d\mu_X$ to be the Plancherel measure for $H_v$, so that $J^X_{\chi_v}(f_v^\circ,f_v^\circ)$ is in this case the Hilbert-Schmidt inner product for the action of $\mathbf{1}_{H(\mathcal{O}_v)}$ on the induced representation $I(\chi_v)$ i.e.\, because the subspace of spherical vectors is one-dimensional, just $\vol(H(\mathcal{O}_v))^2=\Delta_{G,v}^{-1}$.
		
		\item\label{rmk unr Whittaker} Another important instance is the {\em Whittaker case} where, by the Casselman-Shalika formula \cite{CasSha}, \eqref{formula unramified LRC} holds for $L(\chi_v,\rho_X)=1$ i.e.\ $\rho_X=0$. This in particular suggests that the Whittaker case is even more fundamental than the group case, being in some sense the greatest common divisor of all other unramified computations.
		
		\item The conjecture actually admits a generalization in another direction (see \cite{Sakunr}): namely it should extend to every, not necessarily homogeneous, affine spherical variety $X$. When $X$ is smooth, the statement is similar with $f_v^\circ$ the characteristic function $\mathbf{1}_{X(\mathcal{O}_v)}$ of its integral points up to one technical point: we need to take on $X_v$ a measure that is not necessarily $G_v$-invariant but only semi-invariant, that is transforming under the action of the group by a certain character $\eta_v: G_v\to \bR_{>0}$. For example, in the situation of Tate's thesis, that is $\mathbb{G}_m\curvearrowright \mathbb{A}^1$, we take the additive measure on $X_v=F_v$ and $\eta_v(t)=\lvert t\rvert_v$. We then have to twist the action of $G_v$ on functions by the square-root of $\eta_v$ to preserve unitarity of the representation on $L^2(X_v)$ (and thus to be able to talk about the local Plancherel decomposition). One classical instance of this setting is the unramified computation behind Godement and Jacquet's theory \cite{GodJac}, which directly relates to the case where $G=\GL_n\times \GL_n\curvearrowright \Mat_n$ and leads in particular a formula like \eqref{formula unramified LRC} with
		$$\displaystyle L(\chi_v,\rho_X)=\lvert L(\frac{1}{2},\chi_v,\Std)\rvert^2 L(1, \chi_v,\Ad_{\GL_n})=L(\frac{1}{2},\chi_v,\Std) L(\frac{1}{2},\chi_v,\Std^*)L(1, \chi_v,\Ad_{\GL_n})$$
		where $\Std$ stands for the standard representation of $\GL_n(\bC)$, and $\Std^*$ for its dual.
		
		\item For $X$ an affine but non-smooth spherical variety, we still expect to have an analog of \eqref{formula unramified LRC}. In the case where $X$ is a so-called {\em $L$-monoid} (a special case of reductive monoids i.e.\ affine equivariant embeddings of a reductive group), this is related to a proposal of Braverman-Kazhdan \cite{BravKaz} aiming to generalize Godement-Jacquet theory to all automorphic $L$-functions. To be more specific, in the case where $X$ is non-smooth, we expect to be able to introduce suitable local Schwartz spaces $\mathcal{S}(X_v)$ (replacing $C_c^\infty(X_v)$) as well as, for almost all $v$, a certain basic function $f_v^\circ\in \mathcal{S}(X_v)$ (in general different from the characteristic function of $X(\mathcal{O}_v)$). In the function field setting, the general expectation is that this basic function could be constructed by taking the Frobenius trace of a suitably defined IC sheaf on the formal arc space $\mathcal{L}X$ of $X$ (assumed, for simplification, to be defined over the field of constants $\mathbb{F}_q$). However, since $\mathcal{L}X$ is an infinite dimensional scheme it is not obvious how to make sense of this IC sheaf. However, thanks to a theorem of Grinberg-Kazhdan and Drinfeld according to which the singularities of $\mathcal{L}X$ at ``generic'' points are essentially finite dimensional, it is still possible to define an {\em IC function} $f_v^\circ$, see \cite{BNSak} or \cite{SakWang} for details.
		
		\item More precisely, in \cite{BNSak} a suitable basic function $f_v^\circ$ is introduced, following the procedure roughly described above, and an unramified computation of the form \eqref{formula unramified LRC} is performed in the case of $L$-monoids. Actually, they rather compute the trace of $f_v^\circ$ on unramified representations, which essentially corresponds to the square root of the relative characters, and the main result of {\em op.\ cit.}\ states that this is equal to the unramified $L$-function of a suitable irreducible representation of the dual group naturally associated to the $L$-monoid as conjectured previously in \cite{Ngomonoid}.
		
		\item The work of Sakellaridis-Wang \cite{SakWang} on the other hand deals with the case where $G$ is unramified (that is defined over the finite field of constants $\mathbb{F}_q$) and ${}^L G_X={}^L G$ (i.e.\ essentially when $X$ is strongly tempered in the sense of \S \ref{S strongly tempered}). This corresponds in some sense the most interesting case arithmetically, as we expect the special $L$-factors appearing in the unramified computation to be at the central point $\frac{1}{2}$. The main result of {\it op.\ cit.}\ is a formula like \eqref{formula unramified LRC} where $\rho_X$ is a representation of ${}^L A=\widehat{A}\rtimes \langle \Frob\rangle$ whose multiset of weights is explicitly described in terms of the {\em colors} of $X$. It also goes one step further towards showing that $\rho_X$ actually comes from a representation of ${}^L G$, by endowing the multiset of weights with the structure of a {\em Kashiwara crystal} isolating in particular precisely what the highest weights of $\rho_X$ are. In particular, they are able to deduce the conjecture when these highest weights are all minuscule, which is always the case when $X=H\backslash G$ (hence, in particular, $H$ is reductive). One striking consequence of this work is for $X$ the affine closure of the quotient of $\SL_2^n$ by the subgroup
		$$\displaystyle \left\{\begin{pmatrix} 1 & x_1 \\ 0 & 1 \end{pmatrix}\times \ldots \times \begin{pmatrix} 1 & x_n \\ 0 & 1 \end{pmatrix}\mid x_1+\ldots+x_n=0 \right\}$$
		considered as a spherical variety under the action of $G=(\mathbb{G}_m\times \SL_2^n)/\mu_2^{\diag}$, for which we obtain the analog of \eqref{formula unramified LRC} with $\rho_X$ the tensor product of standard representations of $\widehat{G}=\GL_2\times_{\det}\ldots\times_{\det} \GL_2$ see \cite[Example 1.1.3]{SakWang}. In particular, for $n=1,2,3$ this essentially recovers unramified computations behind the Hecke integral representation of the standard $L$-function for $\GL_2$, the Rankin-Selberg integral representation for $\GL_2\times \GL_2$ and Garrett's triple product integral representation for $\GL_2^3$ \cite{Garr} respectively.
		
		\item\label{comment BZSV} Following the recent work of Ben-Zvi, Sakellaridis and Venkatesh \cite{BZSV} we can be a little more precise on the graded representation $\rho_X$. Namely, it should admit a decomposition
		$$\displaystyle \rho_X=S_X\oplus \widehat{\mathfrak{g}}^e/\widehat{\mathfrak{g}}_X$$
		where:
		\begin{itemize}
			\item  $S_X$ is a {\em symplectic} representation of ${}^L G_X$, placed in degree $1$;
			
			\item $\widehat{\mathfrak{g}}^e$ denotes the centralizer, in $\widehat{\mathfrak{g}}=\Lie(\widehat{G})$, of the image of the nilpotent element $e=\begin{pmatrix}0 & 1 \\ 0 & 0 \end{pmatrix}\in \mathfrak{sl}_2(\bC)$ by the morphism $d\iota\mid_{\SL_2(\bC)}: \mathfrak{sl}_2(\bC)\to \widehat{g}$ ;
			
			\item The representation of ${}^L G_X$ on $\widehat{\mathfrak{g}}^e/\widehat{\mathfrak{g}}_X$ is induced by the adjoint action and the grading comes from the decomposition along highest weights for $\mathfrak{sl}_2(\bC)$, that is:
			$$\displaystyle (\widehat{\mathfrak{g}}^e/\widehat{\mathfrak{g}}_X)_d=(d-2)-\mbox{eigenspace for } \ad(d\iota(h))$$ 
			where $h=\begin{pmatrix}1 & 0 \\ 0 & -1 \end{pmatrix} \in \mathfrak{sl}_2(\bC)$.
		\end{itemize}
		This is actually related to the {\em Hamiltonian duality} proposal of {\it op.\ cit.}\ where starting from the symplectic representation $S_X$ and the dual group morphism $\iota: \widehat{G}_X\times \SL_2\to \widehat{G}$, the authors construct a dual $\widehat{G}$-Hamiltonian variety $\widehat{M}$ to $M=T^*X$. We refer the reader to the article of Venkatesh in this proceeding for an introduction to this circle of ideas.
	\end{enumerate}
\end{paragr}

\subsection{The Sakellaridis-Venkatesh global conjecture}\label{Sect global conjecture}

We are now in position to state the main global conjecture from \cite{SV}.

\vspace{2mm}

\begin{paragr}[The conjecture.]
Recall that for $f\in \mathcal{S}(X(\mathbb{A}))$, $\Theta_X(f)$ denotes the corresponding theta series (see \S \ref{S theta series}). For $\psi: L_F\times \SL(2)\to {}^L G$ a global $A$-parameter for $G$, we will denote by $\Theta_X(f)_\psi$\index{$\Theta_X(f)_\psi$} its orthogonal projection to the subspace of cuspidal automorphic forms $\cA_{\cusp}(G)_{\psi}$ associated to $\psi$.
	
We assume that the local Conjecture \ref{conj1 SV} holds for $X$. This yields in particular local relative characters $J^X_{\phi_v}$ for $\mu_X$-almost all $\phi_v\in \Phi_{\temp,v}(X)$ that we will normalize by choosing the Plancherel measure as in \S \ref{S canonical Planch measure}. Here we will need to assume that these relative characters can actually be defined precisely for every local parameter $\phi_v$, which is a consequence of the following natural expectation:
	
\begin{center}
\textbf{(REG)} (regularity of the local relative characters): For every $f_v\in \mathcal{S}(X_v)$, the function $\phi_v\in \Phi_{\temp,v}(X)\mapsto J^X_{\phi_v}(f_v)$ (a priori defined $\mu_X$-almost everywhere) is analytic.
\end{center}
	
	We also assume the local unramified Conjecture \ref{local unramified conjecture} for a certain graded representation $\rho_X=\bigoplus_{d\geq 1} \rho_{X,d}$ of ${}^L G_X$. For any finite set of places $S$ and global $L$-parameter $\phi: L_F\to {}^L G_X$, we can then form the product of partial $L$-functions\index{$ L^S(s,\phi,\rho_X)$}
	\begin{equation*}
		\displaystyle L^S(s,\phi,\rho_X)=\prod_{d\geq 1} L^S(s+\frac{d}{2},\rho_{X,d}\circ \phi).
	\end{equation*}
	We will assume below that this $L$-function has a meromorphic continuation that is holomorphic at $s=0$ whenever the corresponding $A$-parameter $\psi=\iota_X\circ (\phi\times \Id_{\SL(2)})$ is discrete, namely:
	
	\begin{center}
		\textbf{(MER)} (meromorphic extension of the $L$-function): For every tempered $L$-parameter $\phi: L_F\to {}^L G_X$ such that the corresponding $A$-parameter $\psi=\iota_X\circ (\phi\times \Id_{\SL(2)})$ is {\em discrete} (i.e.\ such that $Z_{\widehat{G}}(\psi)$ is finite modulo $Z(\widehat{G})^\Gamma$) and every sufficiently large finite set of places $S$, the partial $L$-function $L^S(s,\phi,\rho_X)$, a priori only defined for $\Re(s)\gg 0$, admits a meromorphic extension to $\bC$ that is holomorphic at $s=0$. Moreover, the adjoint partial $L$-function $L^S(s,\Ad_X\circ \phi)$ also admits a meromorphic continuation with a pole at $s=1$ of order $a_G=\dim(A_G)$.
	\end{center}
	Actually, looking at the conjectural description of $\rho_X$ from \S \ref{S rmks local unr conjecture} point \eqref{comment BZSV}, we see that the only source of poles for the $L$-function $L^S(s,\phi,\rho_X)$ should be from the trivial summand $\mathfrak{z}(\widehat{\mathfrak{g}})^\Gamma/\mathfrak{z}(\widehat{\mathfrak{g}}_X)^\Gamma$ of the adjoint representation of ${}^L G_X$ on $\widehat{\mathfrak{g}}^e/\widehat{\mathfrak{g}}_X$. (Here $\mathfrak{z}(\widehat{\mathfrak{g}})$, $\mathfrak{z}(\widehat{\mathfrak{g}}_X)$ denotes the center of the Lie algebras of $\widehat{\mathfrak{g}}$, $\widehat{\mathfrak{g}}_X$ and the $\Gamma$ superscript indicates that we are taking fixed points under the Galois action.) Thus, we expect the above condition to hold as soon as $\mathfrak{z}(\widehat{\mathfrak{g}})^\Gamma\subset \widehat{\mathfrak{g}}_X$ which is equivalent to the condition that the intersection $A_G\cap H$ is finite. Note that this condition can always readily be achieved, up to replacing $G$ by its quotient by the neutral component of $A_G\cap H$. Similarly, assuming that $(A_G\cap H)^0=1$, the condition on the adjoint $L$-function should be equivalent to the equality $\mathfrak{z}(\widehat{\mathfrak{g}}_X)^\Gamma=\mathfrak{z}(\widehat{\mathfrak{g}})^\Gamma$, which is automatically satisfied since otherwise the $A$-parameter $\psi$ wouldn't be discrete.

	Based on the unramified computation of Conjecture \ref{local unramified conjecture}, Sakellaridis and Venkatesh have made the following striking prediction generalizing the Ichino-Ikeda formula (see Conjecture \ref{global Ichino-Ikeda conj}).
	
	\begin{conj}[Sakellaridis-Venkatesh]\label{global SV conj}
		Assume that $X=H\backslash G$ is an affine spherical variety without type $N$ roots with $A_G\cap H$ finite. Assume also that the local Conjecture \ref{conj1 SV} as well as the local unramified Conjecture \ref{local unramified conjecture} both hold for $X$. Assume finally the two conditions \textbf{(REG)} and \textbf{(MER)} above.
		
		Let $\psi: L_F\times \SL(2)\to {}^L G$ be a global discrete $A$-parameter for $G$ whose restriction to $\SL(2)$ is the same (up to $\widehat{G}$-conjugacy) as that of the morphism $\iota_X: {}^L G_X\times \SL(2)\to {}^L G$. Then, for every function $f=\prod_v f_v\in \mathcal{S}(X(\bA))$, we have:
		\begin{equation}\label{factorization global periods}
			\displaystyle \langle \Theta_X(f)_\psi,\Theta_X(f)_\psi\rangle_{\Pet}=\sum_{\phi} q_\phi J^X_\phi(f,f)
		\end{equation}
		where the sum runs over $\widehat{G}_X$-conjugacy classes of tempered $L$-parameters $\phi: L_F\to {}^L G_X$ such that $\iota_X\circ(\phi\times \Id_{\SL(2)})$ is equivalent to $\psi$, $q_\phi$\index{$q_\phi$} are certain nonzero rational numbers and\index{$J^X_\phi$}
		$$\displaystyle J^X_\phi(f,f)=(\Delta_G^{S,*})^{-1}\frac{L^S(0, \rho_X\circ \phi)}{L^{S,*}(1,\Ad_X\circ \phi)} \prod_{v\in S} J^X_{\phi_v}(f_v,f_v)=\prod'_v J^X_{\phi_v}(f_v,f_v)$$
		stands for the regularized product of the local relative characters associated to the restrictions $\phi_v$ of $\phi$ to the local Langlands groups $L_{F_v}$. More precisely, in the above product, $S$ is any sufficiently large finite set of places, $\Delta_G^{S,*}$ denotes the global ($S$-partial) Tamagawa factor for $G$, $\rho_X$ is the global representation of ${}^L G_X$ appearing in Conjecture \ref{local unramified conjecture} and 
		$$\displaystyle L^{S,*}(1,\Ad_X\circ \phi):=\lim\limits_{s\to 1} (s-1)^{a_G}L^S(s,\Ad_X\circ \phi)$$
		is the leading term in the Taylor expansion of the adjoint $L$-function $L^S(s,\Ad_X\circ \phi)$ at $s=1$. We recall that the notation $\prod_v'$ means that the last product has to be interpreted ``in the sense of $L$-functions'' i.e.\ is just a convenient shorthand for the middle expression.
		
	\end{conj}
\end{paragr}

\begin{paragr}[Remarks on the global conjecture.]
	\begin{enumerate}[(1)]
		\item This in particular implies that $\Theta_X(f)_\psi=0$ should be zero unless $\psi$ factors through $\iota_X$ i.e.\ if the representations appearing in $\cA_{\cusp}(G)_\psi$ are in a certain sense functorial lifts from $G_X$. 
		
		\item The sum over $\phi$ in \eqref{factorization global periods} can in general have more than one term. However, it is often the case that $X$ verifies the following local multiplicity one property: $\dim \Hom_{H_v}(\pi_v,\bC)\leq 1$ for every place $v$ and every irreducible representation $\pi_v$ of $G_v$. (This for example happens in the Gan-Gross-Prasad setting.) In this case, the sum has always at most one term, i.e.\ $\psi$ factors in at most one way through $\iota_X$, and the above conjecture predicts in particular that the inner product $f\mapsto \langle \Theta_X(f)_\psi,\Theta_X(f)_\psi\rangle_{\Pet}$ is factorizable (canonically) into local inner products.
		

		\item The conjecture can also be stated for individual representations (rather than packets thereof). This requires decomposing the local relative characters $J_{\phi_v}^X$ as sums
		$$\displaystyle J_{\phi_v}^X=\sum_{\pi_v} J^X_{\pi_v}$$
		where $\pi_v$ runs over the local $A$-packet associated to $\psi_v:=\iota_X\circ (\phi_v\times \Id_{\SL(2)})$, and $J^X_{\pi_v}$ is a relative character, aka a $G_v$-invariant inner product on $\mathcal{S}(X_v)$, that factors through a quotient of the form $\mathcal{S}(X_v)\otimes \mathcal{S}(X_v)\to \pi_v\otimes \pi_v^\vee$. Note that the notation $J^X_{\pi_v}$ is ambiguous, because there might be more than one $L$-parameter $\phi_v: L_{F_v}\to {}^L G_X$ whose associated $A$-packet of $G_v$ contains $\pi_v$, but we will only use it whenever this $L$-parameter is clear from the context. Then, \eqref{factorization global periods} implies a similar decomposition
		$$\displaystyle \langle \Theta_X(f)_\pi,\Theta_X(f)_\pi\rangle_{\Pet}=\sum_{\phi} q_\phi J^X_\pi(f,f)$$
		for the Petersson norm of the orthogonal projection $\Theta_X(f)_\pi$ to the $\pi$-isotypic component in $\cA_{\cusp}(G)_{\psi}$. Here, the global relative character $J^X_\pi$ is defined in a similar way to $J^X_\phi$ i.e.\ as a regularized product of the local relative characters $J^X_{\pi_v}$.
		
		\item This conjecture can be restated directly in terms of the automorphic period
		$$\displaystyle \mathcal{P}_H: \varphi\in \cA_{\cusp}(G)\mapsto \int_{[H]} \varphi(h) d_{\Tam}h.$$
		More precisely, assuming for simplicity that we are in a multiplicity one setting, the conjecture implies that $\mathcal{P}_H$ vanishes identically in $\cA_{\cusp}(G)_\psi$ unless $\psi$ factors through $\iota_X$ in which case, for every $\pi\subset \cA_{\cusp}(G)_\psi$, we have an identity
		\begin{equation}\label{factorization periods2}
			\displaystyle \lvert \mathcal{P}_H(\varphi)\rvert^2=q (\Delta_{H}^{S,*})^{-2}\Delta_G^{S,*}\frac{L^S(\rho_X\circ \phi)}{L^{S,*}(1,\Ad_X\circ \phi)} \prod_{v\in S} \mathcal{P}_{H_v}(\varphi_v,\varphi_v)
		\end{equation}
		for every factorizable vector $\varphi=\otimes_v\varphi_v\in \pi$, where $q\in \mathbb{Q}^\times$ and the local periods $\mathcal{P}_{H_v}: \pi_v\otimes \overline{\pi}_v\to \bC$ are $H_v$-invariant sesquilinear forms that are obtain by composing the contragredient $\pi_v\otimes \overline{\pi}_v\to C^\infty(X_v)\otimes C^\infty(X_v)$ of the quotient map $\mathcal{S}(X_v)\otimes \mathcal{S}(X_v)\to \pi_v\otimes \overline{\pi}_v$ through which $J^X_{\pi_v}$ factors with evaluation at the base point $x_0=H1\in X_v$. Note that the factorization of $J_{\pi_v}^X$ through $\pi_v\otimes \overline{\pi}_v$ requires fixing an invariant inner product $\langle .,.\rangle_v$ on $\pi_v$, and they should be chosen to decompose the Petersson inner product:
		$$\displaystyle \langle .,.\rangle_{\Pet}=\prod_v \langle .,.\rangle_{v}.$$
		We emphasize that the Tamagawa factors appearing in \eqref{factorization periods2} are different from that in \eqref{factorization global periods}, but this comes from our normalization of local measures: since we are equipping the $X_v$'s with their local Tamagawa measure, at all unramified places, we have the following relation between the relative character and the local period
		$$\displaystyle P_{H_v}(\varphi_v^\circ,\varphi_v^\circ)=\Delta_{X,v}^{2} J_{\pi_v}(f_v^\circ,f_v^\circ)=\Delta_{H,v}^{-2}\Delta_{G,v}^2 J_{\pi_v}(f_v^\circ,f_v^\circ)$$
		where $\varphi_v^\circ\in \pi_v$ denotes a normalized spherical vector (i.e.\ such that $\langle \varphi_v^\circ,\varphi_v^\circ\rangle_v=1$). Note that, the factor $(\Delta_{H}^{S,*})^{-1}$ also appears in the normalization of the Tamagawa measure on $H(\bA)$, through which the period $\mathcal{P}_H$ is defined, whereas, by our assumption that $A_G\cap H$ is finite, the two $L$-functions $\Delta_G^S(s):=L(s,M_G^\vee(1))$ and $L^{S}(s+1,\Ad_X\circ \phi)$ both have a pole of order $a_G$ at $s=0$, so that in \eqref{factorization periods2} we may replace the quotient $\Delta_G^{S,*} L^{S,*}(1,\Ad_X\circ \phi)^{-1}$ by the limit
		$$\displaystyle \lim\limits_{s\to 0} \frac{\Delta_G^S(s)}{L^{S}(s+1,\Ad_X\circ \phi)}.$$

		\item As for the local unramified Conjecture \eqref{factorization global periods}, the conjecture should admit an extension to the non-homogeneous case (still assuming that $X$ is affine). This extension is essentially formal up to remplacing the theta series map $\Theta_X$ by a sum over the rational points {\em in the open $G$-orbit} $X^{\bullet}\subset X$. Note that in this case the relation to automorphic periods is less transparent. In the case of $G=\GL_n\times \GL_n\curvearrowright \Mat_n$, this essentially recovers the global theory of Godement-Jacquet \cite{GodJac}.
		
		\item Note that $\langle \Theta_X(f)_\psi,\Theta_X(f)_\psi\rangle_{\Pet}$ can be seen as the contribution of $\psi$ to the relative trace formula for $X\times X/G=H\backslash G/H$. Indeed, at a formal level (i.e.\ ignoring convergence issues), the relative trace formula for $X\times X$ is given by the scalar product of two theta series, that is:
		$$\displaystyle RTF_X(f_1,f_2)=\langle \Theta_X(f_1),\Theta_X(f_2)\rangle_{\Pet}$$
		where $f_1,f_2\in \mathcal{S}(X(\bA))$, see in particular the contribution of P.-H. Chaudouard in these proceedings. Going a bit further, we expect that the rational factors $q_\phi$ appearing in the conjecture could be understood via a suitable theory of endoscopy for the relative trace formula.  In particular, we can expect that, in  analogy with the Arthur-Selberg trace formula, the relative trace formula admits a ``stable'' form $SRTF_X$ whose {\em most tempered part}, i.e.\ the part corresponding to the automorphic spectrum of $G$ with Arthur $\SL(2)$ given by $\iota_X\mid_{\SL(2)}$, decomposes as a sum of global relative characters indexed by tempered $L$-parameters $\phi: L_F\to {}^L G_X$, in formula:
		$$\displaystyle SRTF_X^{\temp}(f_1,f_2)=\int_{\Phi_{\temp}(X)} \frac{1}{\lvert S_\phi\rvert}J^X_\phi(f_1,f_2) d\phi$$
		where $J^X_\phi$ denotes the regularized product of local relative characters as in the conjecture, $S_\phi=\pi_0(Z_{\widehat{G}_X}(\phi))$ the component group of the centralizer of $\phi$ and $d\phi$ a natural measure on $\Phi_{\temp}(X)$ (similar to the one considered locally in \S \ref{S canonical Planch measure}). A first step in to make this more precise would be to develop an analog of the theory of endoscopy in the relative setting. There has been interesting work of Spencer Leslie \cite{SLeslie}, in the case of symmetric varieties, providing some evidence in that direction.
		
		\item The conjecture should admit an extension to automorphic representations $\pi$ appearing in the continuous spectrum . A precise formulation would however require to define the analog of the orthogonal projection $\Theta_X(f)_\pi$, which is a delicate analytic issue e.g. because the theta series $\Theta_X(f)$ is usually not in $L^2([G])$ and thus cannot be spectrally decomposed in an obvious way. This is kind of related to the spectral expansion of the corresponding relative trace formula mentioned in the previous comment. A similar problem (dually) is the {\em regularizion} of the period integral $\mathcal{P}_H$ on Eisenstein series, see however \cite{Z4} for some recent work in that direction. Finally, we should mention that one important evidence for the above conjecture is the case of {\em principal Eisenstein series} for which a formula like \eqref{factorization periods2}, suitably interpreted (i.e.\ including a regularization of the period involved), can be established in large generality see \cite[\S 18.1]{SV}.
		
		\item Apart from the Ichino-Ikeda conjecture and its numerous variants, another important instance where a precise conjecture (including explicit determination of the constants $q_\phi$) has been formulated is that of the Whittaker model $X=N,\psi\backslash G$. This conjecture is due to Lapid and Mao \cite{LaMao}. The conjecture is known for $G=\GL_n$ (see \cite[Sect. 4]{LaMao} and \cite[\S 18.3]{SV}), where it can essentially been deduced from work of Jacquet and Shalika, as well as for the metaplectic group \cite{LaMao2} and unitary groups \cite{Morimoto24}.
		
		\item Concerning the determination of the rational constants $q_\phi$, a conjecture of Ben Zvi, Sakellaridis and Venkatesh \cite[Conjecture 14.2.1]{BZSV} relates it to a fixed points count on the dual side. See in particular Example 14.2.3 in {\em op.\ cit.}\ .
		
		\item Some of the best known instances of integral representations for automorphic $L$-functions can be used to verify Conjecture \ref{global SV conj} in particular cases. The main extra ingredient to do so is a local analog of the {\em global unfoldings} that are behind these integral representations. Namely, the {\em local unfolding} method, introduced by Sakellaridis and Venkatesh in \cite[\S 9.5]{SV} (see also \S \ref{S unfolding} where this makes a short appearance), allows to related the Plancherel decomposition of different spherical varieties via certain integral transforms, basically given as partial Fourier transforms, that mimick the mechanism of the corresponding global unfolding. See in particular \cite[\S 18.4]{SV}, for an explicit application of this method in the case of Rankin-Selberg integral representations of the tensor product $L$-functions on $\GL_n\times \GL_{n+1}$ and $\GL_n\times \GL_n$ 
		
		\item Conjecture \ref{global SV conj} has also been established in few more cases. This includes the Galois symmetric variety $X=\GL_n\backslash \Res_{E/F}\GL_n$, where $E/F$ is a quadratic extension, for which an explicit factorization of the period was known thanks to work of Rallis and Flicker \cite{Flicker}. More precisely, the $H$-period in this case can be essentially realized as the residue at $s=1$ of some Zeta integrals that represents Asai $L$-functions; in particular the $H$-period can only be nonzero on cuspidal representations that come by base-change from a unitary group. In \cite{BPPlanch}, a method is developed to obtain in a similar manner an explicit Plancherel decomposition for $X_v$ by taking the residue at $s=1$ of an expression that is closely related to the Flicker-Rallis local Zeta integrals. Combining this with the work of Flicker-Rallis then yields a special case of Conjecture \ref{global SV conj}. We note that similar methods have since been applied to obtain explicit Plancherel decompositions for the Shalika model \cite{Duh} as well as for the symmetric variety $X=\Sp_{2n}\backslash \GL_{2n}$ \cite{LaOffsymp}.
		
		\item Another interesting example of Conjecture \ref{global SV conj} is that of unitary periods on $\GL_n$. More precisely, let $E/F$ be a quadratic extension and let us consider the symmetric variety $X=U(n)\backslash \Res_{E/F} \GL_n$, where $U(n)$ stands for the unitary group of any Hermitian form of rank $n$ over $E$. (Note that the choice of the Hermitian form does not matter since they all lead to the same symmetric variety; more canonically we could define $X$ as the space of nondegenerate Hermitian forms of rank $n$). The $U(n)$-periods on $\Res_{E/F}\GL_n$ have been much studied by Jacquet (see inparticular \cite{Jacfactorization}) which have in particular proposed an approach via a comparison of relative trace formulas. The final steps of this program were achieved by Feigon, Lapid and Offen in \cite{FLO} leading to explicit decompositions of these unitary periods, not only for cuspidal representations but also for (generic) Eisenstein series. One notable feature here is that, for Eisentein series, these periods are not directly factorizable but rather {\em finite sums} of factorizable functionals. Indeed, this is related to the fact that the distinguished morphism $\iota_X:{}^L G_X\to {}^L G$ corresponds in this case to base-change functoriality from $\GL_n$ to $\Res_{E/F}\GL_n$ which is not one to one in general \cite{ACBC}. In \cite{BeuGLnU}, combining the work of Feigon-Lapid-Offen with a similar (but much simpler) comparison of local relative trace formulas, I showed that the (square of the) local functionals appearing in these factorizations also give a Plancherel decomposition for $X_v$ indexed by the tempered dual $\Temp(\GL_n(F_v))$ of $\GL_n(F_v)$ which verifies Conjecture \ref{global SV conj} in this case.
		
		\item A natural question is to extend the above conjecture to $A$-parameters $\psi: L_F\times \SL(2)\to {}^L G$ that don't have the same $\SL(2)$ type as $\iota_X$. When the $\SL(2)$ type of $\psi$ is ``more tempered'' than that of $\iota_X$, i.e.\ when this parametrizes representations that are closer of being tempered in a suitable sense\footnote{We won't try to make this precise here, it is possibly related to the partial order on conjugacy classes of morphisms $\SL(2)\to \widehat{G}$ induced by closure relations in the nilpotent cone.}, we certainly expect the corresponding automorphic representations $\pi$ to not be $H$-distinguished (i.e.\ to have vanishing $H$-periods) for purely representation-theoretic reasons (i.e.\ because the local components $\pi_v$ of irreducible representations $\pi\subset A_{\cusp}(G)_\psi$ are not even $H_v$-distinguished for almost all $v$). On the other hand, when the $\SL(2)$ type of $\psi$ is ``less tempered'' than $\iota_X$ (e.g. when $\iota_X$ is trivial on $\SL(2)$), we don't have yet a clear picture of what to expect even in the setting of the Gan-Gross-Prasad conjectures, see however \cite{GGPnontempered}.
	\end{enumerate}
\end{paragr}

	\bibliographystyle{amsplain}

\printindex	
	
\providecommand{\bysame}{\leavevmode\hbox to3em{\hrulefill}\thinspace}
\providecommand{\MR}{\relax\ifhmode\unskip\space\fi MR }
\providecommand{\MRhref}[2]{%
	\href{http://www.ams.org/mathscinet-getitem?mr=#1}{#2}
}
\providecommand{\href}[2]{#2}

\end{document}